\documentclass[11pt,a4paper,reqno]{article}

\usepackage[utf8]{inputenc} 
\usepackage{lmodern}

\usepackage{caption}
\usepackage[left=2.2cm,
right=2.2cm,top=2.3cm,
bottom=3cm]{geometry}

\usepackage{graphicx} 

\usepackage{booktabs} % for much better looking tables
\usepackage{array} % for better arrays (eg matrices) in maths
\usepackage{paralist} % very flexible & customisable lists (eg. enumerate/itemize, etc.)
\usepackage{verbatim} % adds environment for commenting out blocks of text & for better verbatim
\usepackage{subfig} 
\usepackage{amsmath}
\usepackage{amsfonts}
\usepackage{amssymb}
\usepackage{mathrsfs}
\usepackage{amsthm, math}
\usepackage{soul}
\usepackage[colorinlistoftodos]{todonotes}
\usepackage{tikz-cd}
\usepackage{mathtools}
\usepackage{scrextend}
\usepackage{dsfont}

\usepackage{setspace}
\makeatletter
\renewcommand{\paragraph}{%
  \@startsection{paragraph}{4}{\z@}%
  {0.8ex}  % spazio sopra
  {-1em}   % testo subito inline
  {\normalfont\bfseries}%
}
\makeatother

\AtBeginDocument{%
  \setlength{\abovedisplayskip}{6.5pt} %Spazio sopra le formule
  \setlength{\belowdisplayskip}{6.5pt} %Spazio sotto le formule
    \setlength{\jot}{2pt} %Spazio nelle formule
}

\usepackage[rightcaption]{sidecap} 

\usepackage{blindtext}
\usepackage{hyperref}

\usepackage[nottoc,notlof,notlot]{tocbibind} % Put the bibliography in the ToC
\usepackage[titles,subfigure]{tocloft} % Alter the style of the Table of Contents

\DeclareMathOperator{\interior}{int}

\DeclareMathOperator{\rank}{rank} 
\DeclareMathOperator*{\toup}{\longrightarrow} 
\DeclareMathOperator*{\ttoup}{\llongrightarrow} 
 
\DeclareMathOperator*{\leftttoup}{\llongleftarrow}

\DeclareMathOperator*{\tttoup}{\xrightarrow{\hspace*{25pt}}}

\DeclareRobustCommand{\llongrightarrow}{\relbar\joinrel\relbar\joinrel\rightarrow}
\DeclareRobustCommand{\llongleftarrow}{\leftarrow\joinrel\relbar\joinrel\relbar}

\numberwithin{equation}{section}

\newtheorem{theorem}{Theorem}[section]
\newtheorem{proposition}[theorem]{Proposition}
\newtheorem{lemma}[theorem]{Lemma}
\newtheorem{corollary}[theorem]{Corollary}
\theoremstyle{definition}
\newtheorem{definition}[theorem]{Definition}
\newtheorem{remark}[theorem]{Remark}

\newtheorem{example}[theorem]{Example}

\title{\bf Equivariant critical point theory and bifurcation of $3d$ gravity-capillary Stokes waves}
\author{Tommaso Barbieri, Massimiliano Berti,  Marco Mazzucchelli}

\date{} 

\begin{document}

\maketitle

\begin{abstract}
We establish novel existence 
results of $3d$ gravity-capillary periodic traveling waves. 
In particular we prove the   
bifurcation of multiple, geometrically distinct truly $3d$ Stokes waves 
having  the {\it same} momentum of {\it any} non-resonant 
$2d$ Stokes wave.
This unexpected 
clustering phenomenon of Stokes waves, 
observed in physical fluids, is 
 a 
fundamental consequence of the  Hamiltonian nature of the water waves  
equations, their  symmetry groups,  and novel % purely 
topological   arguments. 
We first employ a variational Lyapunov-Schmidt reduction and we construct a natural constraint for the existence of critical points giving rise to Stokes waves with fixed momentum.
Although this construction appears singular 
near the $2d$-wave hyperplanes, we circumvent  this difficulty by fully 
exploiting the symmetries of the nonlinearity.  
Subsequently,  we implement  
equivariant  Morse-Conley theory 
for a Hamiltonian on a joined 
topological space invariant under a $2$-torus action. A primary challenge is distinguish  critical points that are genuinely $ 3d $ from $2d$-waves. 
 The present 
approach yields a complete  bifurcation picture of  $3d $ gravity-capillary Stokes waves.
\end{abstract}

    \smallskip
    

\smallskip


{  \footnotesize
\tableofcontents
}

\section{Introduction}

In 1847 Stokes \cite{stokes} discovered the nonlinear 
bifurcation of $ 2 d $-space  
% small amplitude 
periodic   traveling water waves solutions 
$  u(x-ct) $, nowadays referred to as Stokes waves, 
that are   
constant along the  direction
orthogonal to the propagation speed. 
These waves 
were the first global-in-time solutions ever discovered for dispersive PDEs.
Their first  rigorous mathematical existence 
proof  
was then achieved by Nekrasov \cite{Nek}, Levi-Civita \cite{LC}
and Struik \cite{Struik}   almost one century ago,  in \cite{goyon,dubreil} with vorticity and  in  \cite{Zei} including the effects of capillarity.   
Since then, numerous results 
have been obtained using a variety of innovative 
approaches, including global bifurcation theory, see e.g.
\cite{AFT,P1982,CSV,WW24,KK} and references therein. 
The literature
regarding  traveling water waves solutions 
% Stokes waves 
is immense and we 
refer to \cite{HHSTWWW} for extended presentations.

Mathematically rigorous results about genuinely $ 3d $-Stokes waves (namely non constant along any direction)
are less numerous; we refer to  \cite{HHS} for striking experimental realizations in  laboratory.  
The first bifurcation theorems of $ 3d$ gravity-capillary Stokes waves are due to Reeder-Shinbrot \cite{RS} for symmetric diamonds, under non-resonance
assumptions on the % resonant 
linear wave vectors,   extended  by Groves-Mielke \cite{GrovesMielke}, 
Groves-Haragus \cite{GrovesHaragus},  
Bagri-Groves \cite{BG}, Nilsson \cite{Nil}, via a spacial dynamics approach, also for non symmetric diamonds. We mention that 
Theorem 5 in  
\cite{GrovesHaragus}
deals with also a 
resonant case,  % ultimately 
applying the Weinstein-Moser %  Lyapunov
center theorem.
% to a reduced Hamiltonin. 

The general resonant  case
when multiple, arbitrarily many,
 linear wave vectors correspond to the same % bifurcation 
resonant speed $c_* $ 
 %--that we refer to as resonant speed--, 
has been addressed by 
Craig-Nicholls \cite{CN,CN2} by  
a variational approach,  demonstrating the occurrence of % bifurcating
small amplitude 
$3d$ % solutions 
gravity-capillary Stokes  waves 
with  
 momentum {\it non} collinear to any resonant wave vector;  
this elegant work  constitutes the most relevant precursor to our results as we comment below.  We also quote 
the   recent bifurcation results \cite{LSW, GNPW} of  doubly periodic gravity-capillary non-resonant  Beltrami flows
and with small non-zero  vorticity in \cite{SVW}.
Bifurcation of $3d$ Stokes waves in the pure gravity case is a small divisor problem and has been addressed in  Iooss-Plotnikov \cite{IP-Mem-2009,IP2}. 

\smallskip

In  this  paper we prove new bifurcation results of $3d$ gravity-capillary  Stokes waves periodic over 
an arbitrary   $2$-dimensional lattice $ \Gamma $ of $ \R^2 $. 
The main one is the following (see Theorem \ref{Existence of 3d solutions collinear nonresonant} for a precise formulation): 
\begin{quote}
\emph{For any value of 
gravity $ g >0 $, surface tension $ \kappa > 0 $, depth 
$ \tth \in (0,+\infty] $, any 
arbitrary lattice 
$ \Gamma \subset \R^2 $, 
and any resonant speed $ c_* $, 
there exist multiple genuinely
$3d$ Stokes waves  spatially  $ \Gamma $-periodic,  having the \emph{same} momentum 
 of a non-resonant $ 2 d $ Stokes wave,
thus  \emph{collinear} to some linear  wave vector.}   
\end{quote}
The emergence  
of such traveling waves  
follows by 
purely topological arguments based on  equivariant Morse-Conley theory for a Hamiltonian  invariant under the ($ \mathbb{T}^2_{\Gamma}
\rtimes \Z_2 $)-action of the  group generated by the  space translations
$\T^2_\Gamma:=\R^2/\Gamma $   and  the $ \Z_2 $-symmetry 
induced 
by the 
reversible nature of the water waves equations,  
see comments   below \eqref{XI}. 
Theorem  \ref{th:noncollinear}
improves 
\cite{CN} about existence and multiplicity  of 
Stokes waves with momentum non-collinear with any linear wave vector; actually, 
in this case,  Theorem 
\ref{th:(u,c) and cI}-($\mathfrak{I}$)
implies that 
any Stokes wave is $ 3$-dimensional.
Theorem \ref{th:(u,c) and cI} provides non-existence results of 
$ 2 d $ or $ 3d $ 
small amplitude Stokes waves,  according to the value of its momentum.
Altogether these results provide the  complete 
bifurcation scenario 
   of    $3d$  gravity-capillary  Stokes waves, as we 
  detail  in Section \ref{sec:main results}.

\smallskip 

The variational 
principle behind the
proof of 
Theorem \ref{Existence of 3d solutions collinear nonresonant}
is roughly the following. 
After a  variational Lyapunov-Schmidt procedure we construct a ``natural constraint"  by a ``Weinstein-Moser" choice of the  speed $ c $, reducing  
the problem  
to the search of stationary points of a Hamiltonian function 
$  H : \MA \to \R  $ 
defined on the 
level sets $ \MA $ 
of the ``reduced momentum".
To construct $ c(v) $  
in a full neighborhood of zero 
we must  exploit  the symmetries of the nonlinearity to compensate for  the singularities  of the linear part near the 
$2d$-waves hyperplanes.
For values  of the momentum $a$ collinear with a resonant wave vector, we show that 
$ \MA $ is equivariantly homeomorphic to 
the  compact $ \T^2_\Gamma$-topological space 
\begin{equation}
\label{casoin}
\MA  \cong \big( \, 
\underbrace{S^{2d_- -1} \times 
S^{2d_+-1}}_{3d-\text{waves}} \, \big)  \star \underbrace{S^{1}}_{2d-\text{waves}} \, , \quad d_-,d_+ \in \N  \, ,
\end{equation}
where $ S^{2 d_\pm -1}$
are unit spheres, 
$\star$ denotes the topological {\it join} 
product (cf. \eqref{defjoin}), and 
the  circle    
$ S^{1} $ corresponds to the celebrated non-resonant  $ 2 d $ Stokes orbit.
The waves of $ \MA $ which 
are 
mapped into 
$ S^{2d_- -1} \times 
S^{2d_+-1} $ in \eqref{casoin}  are genuinely $ 3d$-dimensional.
The compactness of $ \MA $ and its rich equivariant  topology under the group $ \T^2_\Gamma $  
imply a lower bound 
for the number of geometrically distinct
$ \T^2_\Gamma \rtimes \Z_2 $-Stokes  waves which are  truly $3$-dimensional. 
Several % new
topological and analytical difficulties are posed 
by the coexistence of  $2d$- and $3d$-waves inside $\MA $, overcome 
in  Theorem \ref{teo:ast}.   
On a more technical side
we  mention that, since $ \MA $ in \eqref{casoin} is not a manifold, 
we  
develop a Conley type approach 
to  equivariant L\"usternik-Schinerlman theory 
for 
the search of rest points of  a gradient-like flow for $ H $ on $ \MA $, 
%The coexistence of  $2d$- and $3d$-waves inside $\MA $ introduces several % new
%topological and analytical %difficulties,  
% with respect to  \cite{CN},  
as we explain in Section \ref{sec:ideas}. 
Theorem \ref{th:noncollinear} is simpler since  in the non-collinear case
(as in \cite{CN}) the space 
$ \MA $
is % equivariantly 
diffeomorphic to the manifold $ S^{2d_- -1} \times 
S^{2d_+-1} $ which consists of only $ 3 d$-waves.

Variational bifurcation theory with symmetry originated in the seminal works of 
Weinstein \cite{We2},  Moser \cite{Mo},  Fadell-Rabinowitz 
\cite{FR}, and Fadell-Husseini \cite{Fadell:1988aa} %(see also  \cite{R}), 
with main applications to the 
multiplicity of periodic 
orbits of finite dimensional Hamiltonian systems. Later Floer-Zehnder \cite{Floer:1988aa} reinterpreted Fadell-Rabinowitz's work in the context of Conley theory and Bartsch \cite{Bart1} generalized \cite{We2,Mo} under weaker assumptions. Most of these works deal with $S^1$ symmetries. More general group actions were also considered by Bartsch  and Bartsch-Clapp \cite{Bartsch:1990aa}.
In the present work,
we encounter  the natural  actions of  2-tori $ \T^2_\Gamma $, and semidirect products $ \T^2_\Gamma \rtimes \Z_2 $.

 We expect these robust ideas and techniques may lead to many other important bifurcation results in Hamiltonian PDEs.  
 For this reason, 
 Part \ref{part:equivariant_critical_point_theory} provides a rigorous and self-contained 
presentation
of equivariant critical  
 point theory in metric spaces. Unlike in the ordinary non-equivariant setting, here the arguments à la L\"usternik-Schinerlmann crucially employ the annihilator of the equivariant local cohomology.
  We now rigorously present the results.

\subsection{Formulation of the problem}

We consider the motion of a $ 3d$ incompressible  and inviscid fluid in irrotational regime under the action of  gravity  and capillarity forces
    at the free surface. The fluid occupies  the time dependent region
    \begin{equation}\label{Domain}
        \cD_{\eta(t, \cdot) }:= \big\{ (x,y)\in \T^2_\Gamma \times \R 
        \ | \ -\tth\leq y\leq \eta(t,x) \big\} \, , \quad 
        \T^2_\Gamma = \R^2/\Gamma \, ,   
    \end{equation}
where $ \tth\in(0,\infty]$ is the, possibly infinite, bottom depth, and $\Gamma\subset \R^2 $ is the $2$-dimensional lattice
\begin{equation}\label{eq:lattice}
\Gamma = % 2 \pi A \Z^2 = 
\big\{ m_1 2 \pi {\bf v}_1 +  
m_2 2 \pi {\bf v}_2 
\, , \ m_1, m_2 \in 
\Z \big\}
\end{equation}
with linearly independent generators
$ 2 \pi {\bf v}_1 $, 
$ 2 \pi {\bf v}_2  \in \R^2  $. 
 The 
irrotational and divergence free velocity field is the gradient $ \nabla_{x,y} \Phi (t,x,y)  $ of a scalar velocity potential $ \Phi (t,x,y)  $ which is  harmonic  in $ \cD_{\eta} $. 
    Inside $ \cD_{\eta}$ the velocity field  evolves according to
the Euler equations and the water waves 
dynamics is  
determined by 
two boundary conditions at the free surface.
The first condition is that   the pressure of the fluid 
plus the gravity and capillary forces at the free surface  is equal to the constant 
atmospheric pressure (dynamic boundary condition). 
The second one  is 
that the fluid particles at the free surface remain on it along the evolution
(kinematic boundary condition).   
 Denoting by $\psi(t,x) $ the value of $\Phi  $ on the free  surface $(x,\eta(t,x))$, the water waves equations are governed by 
 the Zakharov-Craig-Sulem system  \cite{Z,CS} of quasi-linear non-local equations 
 \begin{equation}\label{Zakharov formulation}\small
        \begin{cases}
            \partial_t \eta  = G(\eta)\psi\\
            \displaystyle{\partial_t \psi = - g\eta  - \frac{|\partial_x \psi|^2}{2} + 
            \frac{(  \partial_x \eta \cdot \partial_x \psi + G(\eta)\psi)^2}{2(1+|\pa_x \eta|^2)}  + \kappa\, \text{div} \Big(\frac{\partial_x \eta}{\sqrt{1+|\partial_x \eta|^2}}\Big)} 
        \end{cases}
    \end{equation}
    where 
    $g > 0 $ is gravity, $ \kappa> 0 $ is the surface tension and   $G(\eta)$ is the  Dirichlet-Neumann operator
    $$    G(\eta)\psi  :=
     %\sqrt{1+\eta_x^2} \, (\partial_{\vec n} \Phi )\vert_{y = \eta(x)}= 
     ( \pa_y \Phi  - \partial_x \eta \cdot \partial_x\Phi)|_{y = \eta(x)} \, .
    $$ 
    {\bf Hamiltonian structure.}
    The equations \eqref{Zakharov formulation} are the Hamiltonian system 
    \begin{equation}\label{Zakharov formulation Hamiltonian}
            \partial_t  u = X_{\cH} (u) \, , \ u = \vect{\eta}{\psi}  \, , 
            \ X_{\cH} (u) = J \nablal \cH (u) = 
             \vect{\nablal_\psi \cH}{-\nablal_\eta \cH} 
             (\eta,\psi)   \, , 
    \end{equation}
    where $J$ is the Poisson tensor
    \begin{equation}\label{def J}
        J := \begin{pmatrix}
            0 & \text{Id} \\
            - \text{Id} & 0
        \end{pmatrix} \, , 
    \end{equation}
$ \nablal = (\nablal_\eta, \nablal_\psi ) $ denotes  the $L^2$ gradient and the Hamiltonian 
    \begin{equation}\label{Hamiltonian}
        \cH(u) :=\frac{1}{2|\Gamma|}\int_{\T^2_\Gamma} ( \psi \, G(\eta)\psi +  g\, \eta^2 ) \, \di x 
        + \frac{\kappa}{|\Gamma|}\int_{\T^2_\Gamma} \sqrt{1+|\partial_x \eta|^2}
        \, \di x 
    \end{equation}
    is the sum of 
    kinetic,    potential and  
capillary energies. 
Introducing the symplectic form
    \begin{equation}\label{Def W}
        \boldsymbol{\Omega}\big(u,u_1\big)=
        \big( J^{-1} u , u_1 \big) := \frac{1}{|\Gamma|}\int_{\T^2_\Gamma}\big(\eta
        (x) \,\psi_1 (x) -
        \psi (x) \,\eta_1 (x) \big)
        \, \di x 
    \end{equation}
    where 
    $
    ( f  , g ) :=
    \frac{1}{|\Gamma|}    \int_{ \T^2_\Gamma} f (x) \cdot g (x) \, \di x $  
    denotes the $L^2(\mathbb{T}^2_\Gamma, \mathbb{R}^2)$ real scalar product, 
 the Hamiltonian vector field
    $X_{\cH}(u)   $  is the symplectic gradient of $ \cH$, namely 
    \begin{equation}\label{Def Hamiltonian vector field}
    \di_u \cH(u) =
    \boldsymbol{\Omega} ( X_\cH (u), \cdot ) \, . 
    \end{equation}
    \noindent
    {\bf $ \Z_2 $  and $ \T^2_\Gamma $-symmetry.} 
The water waves system 
\eqref{Zakharov formulation}  possesses both a
$ \Z_2 $  {\it and} a $ \T^2_\Gamma $-symmetry.  
    Indeed \eqref{Zakharov formulation} is {\it reversible}, namely
\begin{equation}\label{Z2 symmetry}
        \cH \circ 	\rho = \cH 
        \qquad \text{where} \qquad 
        \rho
    \begin{psmallmatrix}{}
            \eta\\
            \psi
        \end{psmallmatrix}
    (x) :=
    \begin{psmallmatrix}{}
            \eta(-x)\\
            -\psi(-x)
        \end{psmallmatrix}  
\end{equation}
is a linear  involution ($ \rho = \rho^{-1} $).
%thus 
%the Hamiltonian $ \cH $ is invariant under the 
%action of 
%$ \cS := \{ \text{Id}, \rho \}\cong \Z_2 $. 
Furthermore, 
since the bottom of \eqref{Domain} is flat, \eqref{Zakharov formulation} is also space invariant, namely
\begin{equation}\label{T2 symmetry}
    \cH \circ \tau_\theta = \cH 
    \qquad \text{where} \qquad 
            (\tau_\theta
            u) (x) :=
          u(x-\theta)  \, , \quad  \forall \theta\in \T^2_\Gamma \, .  
\end{equation} 
Since each   
$ \tau_\theta $ is 
symplectic, 
according to N\"oether theorem, 
the Hamiltonian system \eqref{Zakharov formulation Hamiltonian} 
possesses  
two prime  integrals,  
notably the momentum 
\begin{equation}\label{DefI}
\mathcal{I}(u) :=
\vect{\mathcal{I}_1}{\mathcal{I}_2} (u)  \quad 
    \text{where} \quad \mathcal{I}_i(u):=\frac{1}{|\Gamma|}\int_{ \T^2_\Gamma}\psi\,\partial_{x_i}\eta\, \di x \, . 
\end{equation}
 Actually  the  Hamiltonian vector field
    generated by $\cI_i (u) $ 
    is  
the generator $ \pa_{x_i} $ 
of the translations 
    \begin{equation}\label{XI}
     X_{\cI_i} (u) =J\nablal \mathcal{I}_i (u) =\partial_{x_i} u  \, ,  \quad i=1,2\, . 
    \end{equation}
Overall $ \cH $ is invariant  under the group of linear invertible maps 
generated by the translations 
$ 
\cT := \{ \tau_\theta \}_{\theta \in \T^2_\Gamma} \cong \T^2_\Gamma$ 
% \eqref{T2 symmetry} 
{\it and} reflection 
$ \cS := \{ \text{Id}, \rho \}\cong \Z_2 $,
%reflection \eqref{Z2 symmetry}, 
which is 
equal to $ \cT  \,  \cS =  \cS \, \cT $,  because  $ \tau_{\theta} \circ \rho=\rho \circ \tau_{-\theta}$ for any $ \theta \in \T^2_\Gamma $.
Note that  $ \cT \cap \cS = \{\text{Id} \} $ and  $\cT  $ is a normal subgroup of $\cT  \cS  $.
%hence $\cT\cS$ is the semi-direct product of $\cT $ and $\cS$.
%Identifying $\cT=\T^2_\Gamma$ and $\cS=\Z_2$ we shall denote $\cT\cS$ with the semi-direct product $\T^2_\Gamma\rtimes \Z_2$.
Furthermore  $ \cT \cS $ is isomorphic to the semi-direct product $\T^2_\Gamma \rtimes \Z_2$ given by the set of pairs $(\theta,\sigma)\in \T^2_\Gamma\times \Z_2$ (where $\Z_2=\{0,1\}$) with the product $(\theta,\sigma)\cdot (\theta',\sigma') :=(\theta+(-1)^\sigma\theta',\sigma+\sigma')$.
The isomorphism  is  $\T_\Gamma^2\rtimes \Z_2\to \cT\cS, (\theta,\sigma)\mapsto \tau_\theta \circ \rho^\sigma$.
\\[1mm]
    \noindent
    {\bf Traveling waves.}
    We seek for  
    $\Gamma$-periodic traveling  solutions 
    $  u (x-ct) $, 
    $ u : \T^2_\Gamma \to \R $,
    of  the water waves equations \eqref{Zakharov formulation Hamiltonian},  
    for some speed vector $c\in \R^2$.
    Thus 
   $ -c\cdot \partial_x u  = X_\cH (u)  $ 
   and, recalling \eqref{XI}, 
   we are reduced to 
  solve  
the equation 
    \begin{equation}\label{Bifurcation problem}
        \mathcal{F}(c, u) = 0
        \qquad \text{where} \qquad    \cF\colon \R \times X \to 
 Y 
    \end{equation}
     is  the nonlinear map 
    \begin{equation}\label{spaziY}
        \mathcal{F}(c, u):=X_{\cH + c\cdot \cI}(u)=
    c\cdot \partial_x u + 
    X_{\cH}(u )=  J\nablal( \cH + c\cdot \cI)(u) \, .
    \end{equation}
    The domain $ X $ is the dense subset  
    of $  L^2( \T^2_\Gamma) \times L^2_0(\T^2_\Gamma)$ defined in \eqref{spaceX} below, and 
      $Y \subset L^2_0( \T^2_\Gamma) \times L^2( \T^2_\Gamma)$ is the space in \eqref{spaceY}. 
    Here $L^2_0(\T^2_\Gamma)  $ is the subspace of $L^2 (\T^2_\Gamma)$ consisting of  zero average functions. 
Since $ J$ is invertible, the equation \eqref{Bifurcation problem} is equivalent 
 to search 
 critical  points,
 i.e. equilibria, of 
 the Hamiltonian 
	\begin{equation}\label{Variational}
\Psi (c, \cdot ) : X \to \R \, , \quad 	
 \Psi (c, u ) := 
 \big( \cH + c \cdot \cI \big) (u) \, , 
	\end{equation}
for some value of the 
speed $ c \in \R^2 $. 

By the group symmetries  \eqref{Z2 symmetry}, 
\eqref{T2 symmetry}, 
if $ u $ is a Stokes solution of
\eqref{Bifurcation problem} then each  translated  
$ \tau_\theta u $ and each reflected 
$ \rho u $ are solutions as well.  
Two Stokes waves   
are {\it geometrically  distinct}  if they are not obtained 
by translations and reflections of the other one, namely if they are not in the same $ \T^2_\Gamma \rtimes \Z_2 $-orbit.\\[1mm]
    \noindent
    {\bf Functional setting.}
    Let $  \lattice := ({\bf v}_1 \, | \,  {\bf v}_2 )$  be the 
    $ 2 \times 2 $ invertible matrix of 
    the lattice $\Gamma$ in \eqref{eq:lattice}, so that  $\Gamma=2\pi  \lattice\Z^2$. 
    The {\it dual lattice} associated to $ \Gamma $ is the $ 2 $-dimensional lattice 
    \begin{equation}\label{eq:Gamma'}
    \Gamma':=  \lattice^{-\top}\Z^2 =
\Big\{ k_1  {\bf v}_1' +  
k_2  {\bf v}_2' 
\, , \ k_1, k_2 \in 
\Z \Big\} 
    \end{equation}
where $ {\bf v}_1' :=  \lattice^{-\top} 
{\bf e}_1$ and 
$ {\bf v}_2' :=  \lattice^{-\top}
{\bf e}_2 $ 
and
$ {\bf e}_1, {\bf e}_2 $ denote the canonical 
basis of $ \R^2 $.

    Any $\Gamma$-periodic function $f:\T^2_\Gamma\to \C$ can be expanded in Fourier series $ 
        f(x)=\sum_{j \in \Gamma'} f_j \,  e^{\im j \cdot x}  $. We    define for any $ \sigma > 0 $ and 
$ s \in \R $ 
the Hilbert space  of $ \Gamma $-periodic 
real analytic functions
\begin{equation}\label{spaceX}
X := H^{\sigma,s} \times H^{\sigma,s}_0  := 
H^{\sigma,s}(\T^2_\Gamma) \times H^{\sigma,s}_0(\T^2_\Gamma)   
\end{equation}
where  $ H^{\sigma,s} := H^{\sigma,s}(\T^2_\Gamma)$ is the space 
\begin{equation}\label{def:Hs}
        H^{\sigma,s} := \Big\{f\in L^2(\T^2_\Gamma,\R) \  \big| \ 
        \|f\|_{\sigma,s}^2 := \sum_{j\in \Gamma'} |  f_j|^2 \, \langle j\rangle^{2 s}e^{2\sigma |j|_1} <+\infty\Big\}
    \end{equation}
with    $ \langle j \rangle := \max(1, |j|_1) $, $|j|_1 :=|j_1|+|j_2|$, and  
$ H^{\sigma,s}_0 := H^{\sigma,s}_0(\T^2_\Gamma)
$ is the subspace of zero average functions.
For any $ s > 1 $ 
each space $ H^{\sigma,s}(\T^2_\Gamma) $ 
is an algebra with respect to the product of functions. 
We assume that 
%the space $ X $ in  \eqref{spaceX} has   
$ s \geq 7 / 2 $ with $ s + \tfrac12 \in \N $, to directly apply  the analyticity result  
\cite{BMV2}[Theorem 1.2] 
for the Dirichlet-Neumann operator.
This  is 
actually  
not restrictive since $ \sigma > 0 $ is arbitrary.
%and, for any   
%$ \sigma' < \sigma $ and  $s' %\in \R$, we have  
%$ H^{\sigma',s'} (\T^2_\Gamma) %\subset H^{\sigma,s} (\T^2_\Gamma) $.

The target space in \eqref{spaziY} is 
\begin{equation}\label{spaceY}
Y := 
H^{\sigma,s-1}_0 \times H^{\sigma,s-2}
:= H^{\sigma,s-1}_0(\T^2_\Gamma) \times H^{\sigma,s-2}(\T^2_\Gamma)   \, . 
\end{equation}
Note that  $ \cF (c,0) = 0 $ for any $ c \in \R^2 $.
We are going to prove 
% by means of variational methods 
 new existence and multiplicity bifurcation results of $ 3d $-Stokes waves according to the following definition. 
\\[1mm]
{\bf Definition}
{\bf ($ 2d$ and $3d$-Stokes waves)}
A Stokes  wave  solution  $u(x) $
of \eqref{Bifurcation problem} is  $2d$ if  
it is  constant along one direction,  namely
  $   u(x)=  u\big((x\cdot j )j  \big) $
for some 
  $  j  \in \Gamma'  $.
  In this case 
  we say that $ u(x) $
  is $ 2d$-along the direction $j $. 
 Otherwise 
we say that 
$ u(x) $ is 
 {\it truly $3$-dimensional}. % Stokes wave.

Note that  % \label{eq:cparallel intro}
for a $2d$-Stokes wave $u(x) $
%as in  \eqref{2d Stokes waves}
it results 
\begin{equation}\label{eq:cparallel intro}
u(x-ct) = u( x-c^{\parallel_{\hat \jmath}}\hat \jmath \, t ) \qquad 
\text{where}  
\qquad  c^{\parallel_{\hat \jmath}} := c\cdot \hat \jmath \, , 
\ \  \hat \jmath := j / |j| \, , 
\end{equation}
and therefore  its  ``observed  velocity" 
reduces to only the parallel component 
$    c^{\parallel_{\hat \jmath}}\hat \jmath $. 

\subsection{Main results} \label{sec:main results}
As we show at 
the beginning of  Part \ref{part:II}, 
all the possible speeds  of bifurcation, namely the vectors $c_*\in\R^2$ such that $ {\mathcal L}_{c_*} := \di_u \cF(c_*,0)$ is not invertible,  
form the web 
\begin{equation}
\label{bifspeed}
    \mathsf C := 
    \bigcup_{j \in \Gamma'\setminus \{0\}} {\mathsf R}_j     \qquad
\text{where} \qquad  
{\mathsf R}_j :=  \big\{ c\in \R^2 \ | \ \omega(j)=c \cdot j \big\}
\end{equation}
and
% $\omega(\xi) $ %:=\omega(\xi;g,\tth, \kappa, \Gamma)$ 
\begin{equation}\label{omega}
		\omega(\xi) :=\omega(\xi;g, \kappa,\tth) 
        :=
   % \begin{cases} 
                \sqrt{(g+\kappa|\xi|^2) |\xi|\tanh(\tth |\xi|)} \, , 
            \qquad \xi \in \R^2 \, , 
	\end{equation}    
is the {\it dispersion relation} 
(if   $ \tth=+\infty $ then  
$\tanh(\tth |\xi|)$ is replaced by $ 1  $).
%Note that $ \omega (\xi ) $  is a % radial
%symbol  of order $ 3/ 2 $. 
Since $ \omega (j) >0 $,  each straight line $ {\mathsf R}_j $ in \eqref{bifspeed} lies in the half plane $ \{ c \in \R^2 \ | \  c \cdot j > 0 \} $ 
at a distance $ \omega (j) |j|^{-1} $ from the origin.

%\begin{figure}[h!]
 %       \centering
        %\includegraphics[width=0.35\textwidth]{Retta Rj.pdf}
        %\includegraphics[width=0.35\textwidth]{ValoriDiBiforcazione.eps}
   %     \caption{On the left: the line ${\mathsf R}_j$ in  \eqref{bifspeed} associated to a wave vector $j\in \Gamma'$.
   %     The ``web" on  the right is the union $ \mathsf C $  of all the resonant lines ${\mathsf R}_j$  
 %       for $ j \in %\Gamma'=\Z^2$.
  %      } \label{fig:bifspeeds}
 %   \end{figure}

For any % bifurcation 
speed 
   $c_*\in \mathsf C$
we  denote 
 \begin{equation}\label{Def cV}
        \cV:= \big\{
        j\in \Gamma'\setminus\{0\} \ | \ \omega(j)=c_*\cdot j \big\} 
    \end{equation}
the set of {\it resonant wave vectors}. 
The corresponding $ \ker {\mathcal L}_{c_*} $ 
    is the span of the  resonant plane waves  
\begin{equation}\label{vkern}
v = \sum_{j \in \cV }
\alpha_j 
\vect{M_j\cos(j\cdot x)}{M_{j}^{-1} \sin(j \cdot x)} + 
\beta_j \vect{-M_j\sin(j \cdot x)}{M_{j}^{-1}\cos(j \cdot x)}  \ ,  \quad  \alpha_j, \beta_j \in \R \, ,
\end{equation}
where $ M_j $ are constants defined in \eqref{eq:Mj}. 
Due to capillarity,  
    $ \ker {\mathcal L}_{c_*} $ 
   is  always finite dimensional, cf. 
    \eqref{count of Js}. 
A wave $ v $ in \eqref{vkern} is genuinely $ 3d$ if there are at least two  
non parallel resonant wave vectors $ j \not \parallel j'  $  with non-zero 
amplitudes $ (\alpha_j, \beta_j), (\alpha_{j'}, \beta_{j'}) \neq 0 $; the 
$ j, j' $ are called
``active frequencies".   
    The  number of collinear vectors in $ \cV$ is at most $ 2 $, 
    as  well known for 
    $2d $-fluids, cf.~Lemma \ref{lem:coll}. 
    Note that a   linear wave $ v $ in \eqref{Def cV} has momentum 
\begin{equation}\label{momeintro}
 \cI(v) =-\frac12 \sum_{j\in {\mathcal V}} j \, (\alpha_j^2+\beta_j^2) \, . 
\end{equation}
If $ \dim \ker(\cL_{c_*}) = 2 $ the problem reduces to the classical bifurcation of $ 2 d$ Stokes waves,
usually 
tackled  by the 
Crandall-Rabinowitz bifurcation theorem, see e.g. \cite{Wh0,M}. 
% and  \cite{BBMM}[Section 4].  
% cf.  paragraph \ref{dim2NR}. 
When $ \dim \ker(\cL_{c_*}) = 4 $ and the two  vectors in $ \cV $ are parallel, 
this is  % the problem reduces to
the resonant bifurcation problem of $ 2 d$-Stokes waves discussed e.g. in %  Jones-Toland 
\cite{JT,JT2,RS1,BBMM}.
If  the $ 2 $ 
wave vectors in $ \cV $  
are not-parallel,
bifurcation of 
 smooth families of $ 3d$  Stokes 
waves has been  
proved  in  \cite{RS,GrovesMielke,BG,Nil}; 
we discuss the 
non-resonant case
%$3d$ Stokes waves 
in  \eqref{casoNR}-\eqref{le2dNR}.

\smallskip 

In the resonant scenario when $ \dim \ker(\cL_{c_*}) \geq 6 $ the situation is much more complex and 
the only rigorous known results %we are aware of 
are due to Craig-Nicholls \cite{CN,CN2} and Groves-Haragus \cite{GrovesHaragus}[thm 5].
In particular  \cite{CN}  proves, using variational methods,  that any $ c_* $ is a 
bifurcation speed,   
constructing $ 3d $-Stokes waves, 
if the  momentum
$ \cI (u) = a $ is {\it not} parallel to any resonant wave vector 
$ j \in \cV $ (and sufficiently small). 
\begin{itemize}
\item {\sc Question:}
Are there  $ 3d $-Stokes waves if $ a $ is parallel 
to some resonant wave vector 
$ j_0 \in \cV $?   
\end{itemize}

\begin{SCfigure}[1.0]
[ht] % L'argomento opzionale [1.0] regola la larghezza relativa 
  \caption{  The closed convex cone $\mathcal{C}$ in \eqref{defC} associated to three resonant wave vectors $j_1$, $j_2$ and $j_3$.  According to Theorem \ref{th:(u,c) and cI}
 bifurcation of Stokes waves $ u $ 
  with momentum $  \cI (u) = a  $ occurs if and only if   $ a \in {\cal C }$, while $ 3d $-Stokes waves emerge when  $ a \in \interior({\cal C})$ 
  (if $ a \in \partial {\cal C} $  then 
  any $ u $ is $ 2 d $).
  The scenario $ a $ is not  collinear  with any resonant wave vector was addressed in  \cite{CN}, see the improved Theorem \ref{th:noncollinear}. The case $ a $ is collinear with  exactly
  one resonant wave vector is the subject of 
  Theorem \ref{Existence of 3d solutions collinear nonresonant} (in this figure, the vector
  $ j_0 \in \text{int}(\cC) $ in \eqref{assumthm} is $ j_2 $). 
  }
  \includegraphics[width=0.61\textwidth]{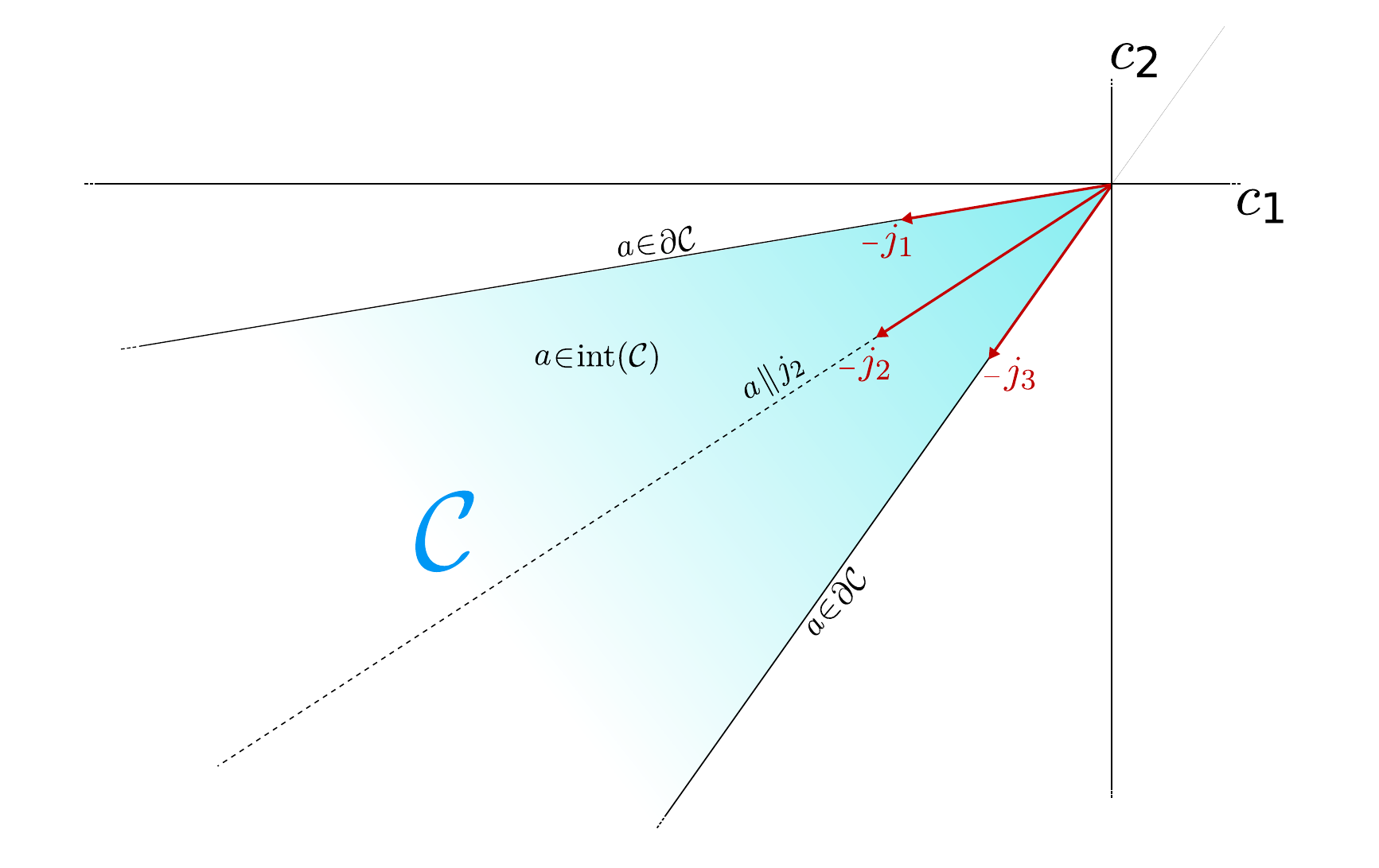} \label{fig:Cono}
\end{SCfigure}

This is a subtle 
question because, in this case,  
 $ 2 d $ Stokes waves  with momentum $ a $ 
 always exist. 
The first main result of this paper is to actually  prove   bifurcation of  multiple 
other truly $3d$-Stokes waves $ u $ having the same momentum 
$ \cI (u) = a $,  {\it collinear} with some 
$ j_0 \in \cV $.

To state precisely 
the result we introduce
the closed  cone
\begin{equation}\label{defC}
    \cC := \Big\{ \sum_{j \in \cV } 
    \tau_j \, j  \ \big| \ \tau_j \leq 0 \Big\}
\end{equation}
generated by the resonant wave vectors $ \cV $, see Figure \ref{fig:Cono},
which consists of
all  values of the momentum   \eqref{momeintro} of the resonant linear waves in $ V $.
Note that the cone $ \cC $ is convex 
and proper (its angle is  $ < \pi $), 
because $ c_* \cdot j = \omega (j) > 0 $ for any $ j \in \cV $, and  
\begin{equation}\label{eq:ilcono}
%\cC \subset \big\{ a \in \R^2 \ | \ a \cdot c_* \leq 0  \big\} \, ,
%\quad 
\cC \setminus \{0\} \subset \big\{ a \in \R^2 \ | \ a \cdot c_* < 0  \big\} \, . 
\end{equation} 

\begin{theorem}
{\bf (Multiplicity of $3d$-Stokes waves: collinear case)}
\label{Existence of 3d solutions collinear nonresonant}
Assume that the number  of resonant wave vectors is $ \# \cV \geq 3 $ (resonant case). 
       Then there is $\varepsilon >0$  such that 
    \begin{equation}\label{assumthm}
        \forall \,  
       a\in \text{int}(\cC) \, , \ 
       |a| \leq \varepsilon \, ,  \    \text{collinear 
       to a unique resonant wave vector} \ j_0 \in \cV \, , 
\end{equation}
       there exist,
       in addition to a unique 
      $S^1$-orbit of 
      $\twod $  Stokes wave $ u_*  $  
      along the direction $ j_0 $ with momentum $ \cI(u_*) = a $,  
      at least
          $ N := \#\cV-{2} $ 
         geometrically distinct 
         %$(\T^2_\Gamma\rtimes \Z_2)$-orbits of 
         truly $3d$ analytic Stokes waves 
         $$ 
         u_1, \ldots, 
         u_{N} \in X \, , 
         \quad  \text{with 
         momentum} \quad 
          \cI (u_i) = a \, , 
         $$ 
         having  size $ \| u_i \|_X = \cO (\sqrt{|a|}) $
         and 
         speeds   
        $ c_i (a)  = c_* + \cO (\sqrt{  |a|}) $.
    \end{theorem}

%The novel result is clearly the lower bound $ \# \cV - 2 $ of $ 3d$-Stokes waves having the {\it same} momentum of   the 
%$2d$-Stokes wave $ u_* $, whose  existence follows by the classical Crandall-Rabinowitz bifurcation theorem. 
\noindent
Note that $  N  \geq 1 $ since 
$ \# \cV \geq 3 $.
Theorem \ref{Existence of 3d solutions collinear nonresonant} is proved in Section \ref{sec:final}.
Let us make some comments: 
\begin{itemize}
\setlength{\itemsep}{5pt}
\item[(i)] {\bf ($2d$-Stokes wave)} The $ 2 d $-Stokes wave
$ u_* := u_* (a)  $ with momentum $ \cI(u_*) = a $
(whose  existence follows by Crandall-Rabinowitz bifurcation theorem)
has the form 
\begin{equation}
\label{laustar}
u_* = \varepsilon \vect{M_{j_0}\cos(j_0\cdot x+ \vartheta  )}{M_{j_0}^{-1} \sin(j_0 \cdot x+ \vartheta  )} + O(\varepsilon^2) \, , 
\ \vartheta \in S^1 \, ,  \  M_j  := (g + \kappa |j|^2)^{-\frac14}(|j | \tanh (\tth |j|))^{\frac14} \, ,
\end{equation} with amplitude 
$ \varepsilon  = \sqrt{\frac{2|a|}{ |j_0|}}(1+\cO(\sqrt{|a|}))$ 
and speed $ c_*^{\parallel_{\hat \jmath_0} }(a) = 
c_* \cdot {\hat \jmath_0} + \cO(|a|)$. %  cf. \eqref{2dstokesNR}.
\item[(ii)] {\bf ($ 3d$-Stokes waves shape)} 
The  $3d$-Stokes 
waves $u_i := u_i (a) $, $ i=1, \ldots, \#\cV-{2} $, are  
$(\T^2_\Gamma\rtimes \Z_2)$-orbits of  the form 
$$u_i = v_i + \cO(\| v_i\|_X^2)$$
where each $v_i := v_i (a) \in   \ker(\cL_{c_*}) \setminus \{ 0\} \, ,  \ \text{cf.} \   \eqref{vkern} \, , 
$ 
is a $3d$ wave
obtained  by variational/topological  arguments, 
and the higher order term depends analytically on $ v_i $.
A central  difficulty 
is proving  that each  
$v_i$ is  genuinely $3$-dimensional, i.e. it contains  at least two non parallel ``active frequencies",  unlike the first order term of $u_* $ 
% the $2d$-Stokes wave 
in \eqref{laustar}.  
The dependence of the $ v_i  $, thus $ u_i $, 
  with respect to  their momentum $ a $, may vary rather irregularly.  
 % as critical points %stationary points of a gradient-like flow 
 % may vary % highly 
 % erratically 
 %with respect to arbitrarily small perturbations.
\item[(iii)]{\bf ($ 3d$-Stokes waves speeds)} The velocity $ c_i (a) $ of each 
$3d$ traveling Stokes wave $u_i$
has  
component 
$$ 
c^{\bot_{\hat \jmath_0}}_i (a) 
:=c_i (a) \cdot \hat \jmath_0^\bot = 
c_*\cdot \hat \jmath_0^\bot  + 
\cO(|a|) \, , 
\quad
$$
in the direction orthogonal to $j_0$.
If  $ c_*\cdot \hat \jmath_0^\bot \neq 0 $, 
the motion of the 
$ 3d$-Stokes wave $ u_i (x- c_i (a) t )$  looks rather different from the  
$ 2d$-Stokes wave $ u_* (x -c_*^{\parallel_{\hat \jmath_0} }(a) \hat \jmath_0 \, t  )$ whose 
``observed velocity" is 
$ c_*^{\parallel_{\hat \jmath_0} }(a) \hat \jmath_0
$, accordingly to  comment
\eqref{eq:cparallel intro}. 
\item[(iv)]   
{\bf (Necessity of assumption $a\in \text{int}(\cC)  $)}
As shown in Theorem \ref{th:(u,c) and cI}  below, the assumption  $a\in \cC$   is a necessary condition for the existence of Stokes waves with momentum  $ a $,  
cf. \eqref{Identroa}; while  the assumption $a\in \interior(\cC)$ is necessary for the existence of 
truly $ 3 d $
Stokes waves having momentum  $ a $ since, by $(\mathfrak{B})$, if $ a \in \partial \cC $  any 
Stokes wave with momentum $ a $ is $ 2 $-dimensional.
% see Figure \ref{fig:Cono}.
\item[(v)]  
{\bf (Collinearity assumption)}
The assumption % \eqref{assumthm} 
that 
$ a $ is collinear to a {\it unique} resonant wave vector $j_0 \in \cV$ forces the  
existence 
of $3d$-Stokes waves by purely topological  arguments, as we explain in Section \ref{sec:ideas}. 
If there are {\it two} distinct 
%resonant wave vectors 
$j_0, j_1 \in \cV$ parallel to $ a $ then the 
existence of other $ 3d $-Stokes waves seems to depend on the nonlinearity and specific properties of the resonant wave vectors. %   $ \cV $. 
\end{itemize}

In the noncollinear case we improve the result of \cite{CN} as follows. 

    \begin{theorem}{\bf (Non-collinear case)}\label{th:noncollinear}
            For any $a\in \text{int}(\cC) \cap \overline{B_\varepsilon}$ non-collinear to any $j\in \cV$ there exist at least $ \#\cV-{1} $ 
            geometrically distinct
            %$(\T^2_\Gamma\rtimes \Z_2)$-orbits of 
            truly $3d$ analytic Stokes waves $ (u_i)_{i=1, \ldots, \#\cV-{1} } $ in $X $ with momentum  $\cI (u_i) = a $
            having  size $ \| u_i \|_X = \cO (\sqrt{|a|}) $
         and 
         speeds   
        $ c_i (a)  = c_* + \cO (\sqrt{  |a|}) $.
    \end{theorem}

Theorem \ref{th:noncollinear}  requires the uniform bound on $|a| \leq \varepsilon $ while  \cite{CN} requires  to take $ a $ smaller and smaller,  as  $a$ aligns to some resonant wave vector. In addition 
    Theorem \ref{th:noncollinear} 
    proves the existence of $ \# \cV - 1 $  
    $(\T^2_\Gamma\rtimes \Z_2)$-critical orbits,
    not just $ \T^2_\Gamma $-orbits,   
    %Thus we improve, 
improving \cite{CN} by a factor $ 2 $ (which obtained 
$ (\#\cV-{1})/2 $ geometrically distinct Stokes waves).

\smallskip

We complete the bifurcation picture 
of $ 3d$-Stokes waves of 
Theorems \ref{Existence of 3d solutions collinear nonresonant} and 
\ref{th:noncollinear} 
by the 
following 
result  valid for {\it any} small amplitude 
Stokes wave   $ u$ 
with  speed  $ c $ 
near $ c_* $.  

\begin{theorem} \label{th:(u,c) and cI}
    There exists $ \delta > 0 $ 
    such that if $ u \in X $ is an analytic 
    Stokes wave solution 
    of $ \cF (c,u) = 0 $ 
    with speed 
 $ c\in B_\delta(c_*)$ and 
    $\| u\|_{X}<\delta$,  
then its
momentum  
\begin{equation}\label{Identroa}
\cI(u)=a \in \cC  
\end{equation}
belongs 
to the convex cone $ \cC $ defined in \eqref{eq:ilcono},  and 
    \begin{itemize}
        \item[$(\mathfrak{B})$]
        {\bf (Boundary $ \partial \cC $)} 
        If $a\in \partial\cC$ then the Stokes wave $u $ is $2$-dimensional; 
        if $ a = 0 $ then $ u = 0 $.
        \item[$(\mathfrak{I})$]
        {\bf (Interior  $ \cC $)} 
        If $a\in \interior(\cC)$
        is not collinear to any  resonant wave vector $j\in \cV$,  then the Stokes wave $u (x) $ is $3$-dimensional.
    \end{itemize}
\end{theorem}

Before describing the main ideas of proof we mention that in the last years also the existence of quasi-periodic  traveling Stokes waves --which are the nonlinear superposition of Stokes waves moving with rationally independent speeds--
 has been proved  by means of KAM methods: in \cite{BFM,FG,BFM2} for 
 both $ 2 d $ gravity-capillary and pure gravity waves. 
Existence of $ 3 d $-quasi-periodic 
pure gravity Stokes waves  has been 
recently achieved in \cite{FMT}. The gravity-capillary case is still  completely open. 

\subsection{Ideas of the proof}
\label{sec:ideas}

We first perform in Section \ref{sec:LS}  a variational Lyapunov-Schmidt decomposition for  \eqref{Bifurcation problem} along the Kernel 
$ V := \ker(\cL_{c_*}) $ and  its symplectic orthogonal $ W := V^{\bot_{\boldsymbol{\Omega}}} $. The Kernel  \eqref{vkern} is the symplectic subspace  
$$
V := \ker(\cL_{c_*})  = 
\bigoplus_{j \in \cV} 
V_j \qquad \text{where}\qquad 
V_j := {\rm span}\{v_j^{(1)},v_j^{(2)} \} 
$$
are the $2$-dimensional real  subspaces in  \eqref{VJsym} with basis 
$ v_j^{(1)}, v_j^{(2)}$  in \eqref{symplectic base}, 
 and $ \cV $ are the resonant wave vectors
in \eqref{Def cV}. 
Decomposing the space 
$ X = V \oplus (W \cap X) $ 
in \eqref{spaceX}, 
we first solve in Lemma \ref{lem:range} the range equation  
$ \Pi_W {\mathcal F}(c, 
v + w) = 0 $ finding 
$ w(c,v) $ 
for any $ |c - c_* | < r $ and any  
$ v  \in B_r^V $ for some $ r > 0 $. In this way  problem \eqref{Bifurcation problem}  is reduced to solve the bifurcation equation 
$ \Pi_V {\mathcal F}(c, v + w(c,v)) = 0 $
which amounts to look for critical points of the $ \T^2_\Gamma \rtimes \Z_2 $-invariant % Hamiltonian 
functional  
$$ 
\Phi (c, \cdot ) : B_r^V \to \R \, , \quad 
\Phi (c, v) := 
\cH (c, v + w(c,v) ) +c\cdot \cI(v+w(c,v))\, .
$$ 
%The Lyapunov-Schmidt reduction  preserves the variational structure of  \eqref{spaziY}-\eqref{Variational} and  the
%symmetries 
%$ \T^2_\Gamma \rtimes \Z_2 $. %invariant. 
%The main properties of $ \Phi (c, \cdot )$  are summarized in Proposition 
%\ref{lem:expa}. 
In particular 
critical points of each restricted functional  
$ \Phi (c, \cdot) : V_{\hat \jmath} := 
\oplus_{j' \in \cV, j' \parallel j} V_{j'} \to \R $ are critical points on the whole $ V $
 and  give rise to 
$ 2 d $-Stokes waves along 
the direction $j $. 
%They contain all 
%the $ 2 d $ Stokes wave theory, cf. \cite{BBMM}.   
\\[1mm]
{\bf Construction of $ c(v) $.}
In  Section \ref{sec:cv} 
 we construct 
the speed $ c(v) $ close to $c_*$ solving 
the system 
\begin{equation}\label{difphi}
\di_v \Phi (c, v) [\nabla_v \cI (v)] = 
        \begin{pmatrix}
        \di_v \Phi (c, v) [\nabla_v \cI_1 (v)]   \\
        \di_v \Phi (c, v) [\nabla_v \cI_2 (v)] 
        \end{pmatrix} = 0 \, , \quad 
         \forall v \in B_{r'}^V 
\, ,
\end{equation}
for some $ 0 < r' < r $. 
We solve the two equations \eqref{difphi}  in the two unknowns $ c \in \R^2 $,  via   contraction mapping arguments. 
The main difficulty 
is that 
its linearization   becomes singular as $ v $ approaches the 
hyperplanes 
$ {\bf V}^{2d} = \cup_{j \in \cV} V_j $
of $ 2 d $-waves. Nevertheless, in Proposition 
\ref{Construction of c close and away from PJ},  exploiting the 
$\T^2_\Gamma $- symmetry of $ \Phi(c, v) $, we are able to define $ c(v) $ 
in a full neighborhood of $ v = 0 $,  
unlike  \cite{CN} which determines $ c(v) $ only  if $ v \notin {\bf V}^{2d} $ 
and, if $ v \to 
{\bf V}^{2d} $, requires to take  $ v \to 0 $. 
  The speed  $c (v) $
  is real analytic 
  in $ B_{r'} \setminus {\bf V}^{2d} $, 
  but 
  may  be not continuous at $ {\bf V}^{2d} $, as simple examples show. 
\\[1mm]
{\bf Variational principle.}
The equations \eqref{difphi} define a natural constraint for our original variational 
problem \eqref{Bifurcation problem}-\eqref{spaziY} in the sense that 
 a critical point $ \bar v $ of  the 
Hamiltonian
$ 
H(v) := \cH (v, w(c(v),v)) $
restricted to level sets 
$ I^{-1}(a) 
$ % near $ v = 0 $
of the  reduced momentum 
\begin{equation}\label{redmo}
I(v)  := \cI (v, w(c(v),v)) = \cI (v) + \cO(\| v \|^3 ) 
\, ,
\end{equation}
gives rise to a  critical point $  \bar u = \bar v + w(c(\bar v), \bar v)$ of the unconstrained  Hamiltonian $ \cH + c (\bar v) \cdot 
\cI $, thus a Stokes waves with speed $ c(\bar v)$.

In order to properly implement such ``Lagrange multiplier principle'' we have to first understand  the topology of the level sets  $ I^{-1}(a) 
$   
as $ a \in \R^2 $ varies.
\\[1mm]
{\bf Straightening of the momentum.} 
As first step we 
  ``straighten"
 $ I (v)  $ 
 into the  quadratic momentum  $ \cI (v)  $.  
 Theorem \ref{th:Moser's trick}  actually proves the existence 
of an equivariant homeomorphism $\zeta $  
such that 
\begin{equation}\label{cambiov}
(I \circ \zeta)(v) = \cI (v) \, , \quad
\zeta (0) = 0 \, ,
\quad \forall v \in B_{r''}^V \, , 
\end{equation}
for some $ 0 < r'' < r' $. 
 This 
result  is also the key of  the non-existence results of 
Theorem \ref{th:(u,c) and cI}.
The main difficulty to prove such   vector valued 
Morse type  result 
is when $ v $ approaches  the 
hyperplanes  $ {\bf V}^{2d} $.
We employ a Moser trick deformation argument and fully exploit the $ \T^2_\Gamma $-symmetries. 
\\[1mm]
{\bf Topology of $ \MA $.}
In view of \eqref{cambiov}, 
we are reduced 
to understand the topology of 
the level sets $ \cI^{-1} (a) $
which is not empty iff 
$ a \in \cC $, since     
$ \cI (\R^2) = \cC $ by \eqref{Momuntum in coordinates}. 
The $ \T^2_\Gamma \rtimes \Z_2 $
space 
$$ 
\MA := \big\{ v \in U_2 := \zeta (B_{r''}^V) \ | \ I(v) = a   \big\} 
$$ 
is compact but it may not be
a manifold, which is the case  
if $ a $
is parallel  to a unique resonant wave vector $ j_0 \in \cV  $.  
In this case 
$\MA $ is the {\it disjoint} union 
\begin{equation}\label{disj}
\MA = \MA^{3d} \sqcup 
\underbrace{\MA^{2d}}_{2d\text{-Stokes}} 
\quad   \text{where} \quad  
\MA^{3d} := 
\MA \setminus {\bf V}^{2d} 
\quad  \text{and} \quad  
\MA^{2d} = \MA \cap V_{j_0}
%=  \{ v \in V_{j_0} \ | \ \| v\|= \sqrt{2} |a| \}
\end{equation}
consists of the % celebrated
% non-resonant 
$ 2 d $ Stokes orbit with momentum $ a $. We prove that $ \MA^{3d} $ is a manifold and 
$$
    V= \underbrace{T_v\MA^{\threed}}_{= \ker \di_v I (v)  }\oplus \, \text{span}(\nabla_v \cI_1(v),\nabla_v \cI_2(v))\, ,
    \quad 
\forall  v \in \MA^{3d} \,  .
$$ %Actually, if  \eqref{assumthm} holds, 
In Theorem \ref{Topology of Sa}  
we also prove that $ \MA $ is equivariantly homeomorphic to the join  topological space (see definition in  \eqref{defjoin}-\eqref{eqrela})
\begin{equation}\label{Matop}
\MA 
\cong (\underbrace{S^{2d_- -1} \times 
S^{2d_+-1}}_{\rm product \, of
\, spheres})  \star S^{1}
\qquad \text{where} 
\qquad d_\pm \geq 1
\end{equation}
are the number of resonant wave vectors of $ \cV $ 
 at the right/left of $ j_0 $ respectively. If $ a $
 is parallel to a resonant wave vector in  
 $  \pa \cC $  
 then $ \MA \cong S(V_0) $ is formed by $ 2d$-waves. If $ a \in \cC $ is not collinear to any $ j \in \cV $ (case of  \cite{CN}) then $ \MA $ is a manifold diffeomorphic to the product of two spheres.
 \\[1mm]
{\bf Existence.}
Let us now explain the 
existence of at least one genuinely $3d$  Stokes wave  in Theorem \ref{Existence of 3d solutions collinear nonresonant}. Assume \eqref{assumthm}.
On  $ \MA $ the function $ H $ attains maximum $ M $ and minimum $m $.
If $ M = m $  the function  $ H_{|\MA} $ is constant and hence {\it any}  point of $ \MA^{\threed}$ is critical. Otherwise $ m < M $ 
    %the minimum and the maximum 
    are 
distinct critical levels,  
and since the $2d$-Stokes wave orbit  
$ \MA^{2d} \subset  H^{-1}(\ell_*) $ is contained in a 
unique level set of 
 $ H $, there exists a  minimum 
 or  maximum point $ \bar v $ of  $ H $ which belongs to $ \MA^{3d} $ (recall that $ \MA $ is the disjoint union \eqref{disj}).
Thus  
$ \bar u = \bar v + w(c(\bar v), \bar v)$  is a truly $3d$-Stokes wave  
solution  with momentum  $ \cI(\bar u) =  a $.  
 \\[1mm]
{\bf   Multiplicity.}  
We implement the 
equivariant  Conley-Morse theory
of Section \ref{s:equivariant_Morse_theory},   developed in a $G$-metric space,
for the $ \T^2_\Gamma  $-invariant  Hamiltonian $ H $ on $ \MA $.  
We  first  construct in Proposition 
\ref{A flow}
a gradient like flow $ \phi^t $ for $ H $ whose 
stationary points in $  \MA^{\threed} 
$ 
are  critical points of 
$ \Phi(c(v), \cdot) $ and  thus  give rise to Stokes waves. If there are infinitely many stationary 
    $ \T^2_\Gamma $-orbits of $ \phi^t $, then there exist also infinitely many
    $ \T^2_\Gamma \rtimes \Z_2 $
    critical orbits of $ H $ and the claim of Theorem \ref{Existence of 3d solutions collinear nonresonant} is trivial.
    Then we suppose there are only finitely many stationary 
    $ \T^2_\Gamma $-orbits and,  
by L\"usternik-Schnirelmann arguments and computations in equivariant cohomology for the space $ \MA \cong (S^{2d_- -1} \times 
S^{2d_+-1})\star S^{1}$ in \eqref{Matop},
we deduce using Theorem \ref{teo:ast}
the existence  
of at 
least  
$$ 
\cuplength_{\T^2_\Gamma}(S^{2d_- -1} \times 
S^{2d_+-1})+1  
\geq  d_- + d_+ - 2  =\#\cV- 2 
$$
{\it  distinct  critical values} of $ H $,  
{\it different} from the critical level  
$ \ell_* = H (\MA^{2d}) $  the space $ \MA \cong (S^{2d_- -1} \times 
S^{2d_+-1})\star S^{1}$ in \eqref{Matop} 
(the definition of equivariant cuplenght is in  \eqref{def:CL}). 
This implies the existence of the same number of 
$ (\mathbb{T}^2 \rtimes \Z_2) $-orbits, thus geometrically distinct Stokes waves, as stated in Theorem \ref{Existence of 3d solutions collinear nonresonant}.

The proof of Theorem 
\ref{teo:ast} faces a serious difficulty: unlike the stabilizers of the action at any point in
 $\MA^{3d}$ are finite, the stabilizers at any point of the $ 2d $-Stokes wave  $ {\MA^{2d}\cong} S^1 $
  are infinite.
This requires a precise 
understanding of the annihilator of the graded 
module $H^*_{\T^2_\Gamma}\big
((S^{2d_- -1} \times 
S^{2d_+-1})\star S^1\big) $
as addressed in 
 Lemma \ref{lm:join}. 
Part \ref{part:equivariant_critical_point_theory} 
contains  a rigorous and self-contained exposition of these ideas  and techniques.  

 \medskip
\noindent 
{\bf Acknowledgments.}  
We thank Thomas Goodwillie for his help with the 
computations of Proposition~\ref{p:cohomology_sphere_product}, and Bernard Deconinck, Alberto Maspero  
and Eric Wahl\'en  for useful comments. 

\part{Equivariant critical point theory}
\label{part:equivariant_critical_point_theory}

The main result of this part is Theorem \ref{teo:ast} which provides a lower bound for the number of critical orbits for a continuous Lyapunov function defined on a 
singular 
topological space
(as \eqref{Matop}) 
invariant under the action of a $d$-torus.
After  recalling basic definitions concerning graded modules over a graded ring, 
%which are the underlying algebraic structure of equivariant cohomology, 
we report 
in Section~\ref{s:equivariant_cohomology}   the notion of equivariant cohomology, we investigate some properties of its annihilator, and compute it for some spaces 
 needed later:  
 %that will be relevant for the variational theory of Skokes waves: 
 products of spheres, and joint products with spheres. In Section~\ref{s:equivariant_Morse_theory} we provide a self-contained account of % some aspects  of 
 equivariant Morse-Conley theory for continuous Lyapunov functions on compact metric spaces in a form needed to prove Theorem \ref{teo:ast}. These topics are present in literature, e.g.,   \cite{Morse:1996aa, Milnor:1974aa,Atiyah:1983aa,Conley:1978aa,Benci1991,Salamon1985} in ordinary or equivariant form. The arguments involving the annihilator of the local cohomology of invariant sets build on the works of Fadell-Rabinowitz 
\cite{FR}, Fadell-Husseini \cite{Fadell:1988aa}, Bartsch 
\cite{Bart1}, and Bartsch-Clapp \cite{Bartsch:1990aa}.\vspace{7pt}
%\section{Modules over a ring}
%\label{s:modules}
%We first % %briefly  
%recall 
%basic definitions
%and properties of abstract modules over rings with unit. 

\paragraph{\bf Modules and algebras.}\label{ss:modules}
We  consider a ring $R$ with unit $1\in R$, without non-trivial zero divisors (i.e.~$r_1r_2\neq0$ for all $r_1,r_2\in R\setminus\{0\}$), and that is commutative up to sign (i.e.~for each $r_1,r_2\in R$, we have  $r_1r_2=r_2r_1$ or $r_1r_2=-r_2r_1$, where the sign depends on the elements $r_1,r_2$). We denote the multiplication in the ring simply by juxtaposition. An \emph{$R$-module} $M$ is an abelian group equipped with a scalar multiplication
$ R\times M\to M $, $  (r,m)\mapsto r\cdot m $, 
such that, for any $x,y\in M$ and any 
$r,s\in R$, we have
$$
(r+s)\cdot x=r\cdot x + s\cdot x \, , \quad r\cdot(x+y)=r\cdot x + r\cdot y \, , \quad 
r\cdot(s\cdot x)=(rs)\cdot x \, , \quad 1\cdot x=x \, .
$$
An $R$-module $M$ is  finitely generated, if there exist finitely many  $x_1,...,x_n\in M$ such that any element of $M$ can be written as a linear combination $r_1\cdot x_1+...+r_n\cdot x_n$ for some $r_1,...,r_n\in R$. 
The \emph{torsion submodule} of $ M $ 
is defined as
\begin{align*}
 M_\tor := \big\{ x\in M\ \big|\ r\cdot x=0\mbox{ for some }r\in R\setminus\{0\} \big\}.
\end{align*}
When $M=M_\tor$, we say that the $R$-module is purely torsion, whereas when $M_\tor=0$ we say that it is torsion-free. The quotient abelian group 
\begin{equation}\label{Mfree}
M_\free := M/M_\tor
\end{equation}
inherits an $R$-module structure, and it is called the \emph{torsion-free quotient} module, as $(M_\free)_\tor=0$.
An ideal $I\subseteq R$ is an additive subgroup of the ring such that $IR\subseteq I$ (since the ring $R$ is commutative up to sign, we have $IR=RI$).
The \emph{annihilator} of a subset $S\subseteq M$ is the ideal 
\begin{equation}\label{Sann}
 S_\ann := \big\{ r\in R\ |\ r\cdot S=0 \big\} \, . 
\end{equation}
Note that a purely torsion module $M=M_\tor$ may still have trivial annihilator $M_\ann= 0$. Nevertheless, if $M$ is finitely generated, then $M=M_\tor$ if and only if $M_\ann\neq 0$. If $M$ and $N$ are $R$-modules, a \emph{homomorphism} $\phi:M\to N$ is a $R$-linear map, meaning that 
$ \phi(r\cdot x+y)=r\cdot\phi(x)+\phi(y) $,
for all $ r \in R $, 
$ x, y \in M $. 
We denote by $\image(\phi):=\phi(M)$ its image and by $\ker(\phi) := \{x\in M\ |\ \phi(x)=0 \}$ its kernel.
If $P$ is another $R$-module, a sequence of two homomorphisms 
$ M\toup^{\phi}N\toup^{\psi} P $
is \emph{exact} when $\image(\phi)=\ker(\psi)$. 
The following important elementary properties   hold.

\begin{proposition}$ $
\label{p:modules} 
\begin{enumerate}[$(i)$]
\setlength{\itemsep}{5pt}
\item Each $R$-module homomorphism $\phi:M\to N$ satisfies $\phi(M_\tor)\subseteq N_\tor$.

\item For each exact sequence of $R$-modules
\[M\toup^{\phi}N\toup^{\psi}P,\] 
we have $ M_\ann P_\ann$ $\subseteq N_\ann $. 

\item For each exact sequence of $R$-modules
$\displaystyle M\toup^{\phi}N\toup^{\psi}P $, 
starting at a purely torsion $R$-module $M=M_\tor$ $($equivalently, $M_\free = 0 $$)$, the homomorphism $\psi$ induces an injective homomorphism 
$ \psi:N_\free\hookrightarrow P_\free $.

\end{enumerate}
\end{proposition}

\begin{proof}
$ (i)$ For each $x\in M_\tor$ there exists $r\in R\setminus\{0\}$ such that $r\cdot x=0$, and therefore $r\cdot\phi(x)=\phi(r\cdot x)=0$, which proves that $\phi(x)\in N_\tor$.

($ii$) For each $r\in P_\ann$ and $x\in N$, we have $\psi(r\cdot x)=r\cdot\psi(x)=0$. Therefore, using the exact sequence, there exists $y\in M$ such that $\phi(y)=r\cdot x$. For each $s\in M_\ann$, we have $sr\cdot x=s\cdot\phi(y)=\phi(s\cdot y)=0$, and therefore $sr\in N_\ann$.

($iii$) By ($i$) the map
$ \psi $   induces an homomorphism $ \psi : N_\free \to  P_\free $ between  torsion free quotient modules, cf.~\eqref{Mfree}. If by contradiction the induced homomorphism $\psi:N_\free\to P_\free$ is not injective,  there is % a non-torsion element 
$x\in N\setminus N_\tor$ such that $\psi(x)\in P_\tor$, and therefore there is $r\in R\setminus\{0\}$ such that $0=r\cdot\psi(x)=\psi(r\cdot x)$. By the exact sequence, we infer that $r\cdot x\in\image(\phi)$. However $r\cdot x$ is not a torsion element, and so item ($i$) implies that  $r\cdot x$ cannot lie in the image of $\phi$ whose domain is purely torsion.
%\item For each $r\in P_\ann$ and $x\in N$, we have $\psi(r\cdot x)=r\cdot\psi(x)=0$. Using the exact sequence, we have $r\cdot x\in\image(\phi)$.
\end{proof}

We also consider a special class of modules. An \emph{$R$-algebra} $A$ is a ring with unit $1\in A$, equipped with a ring homomorphism $\pi_A:R\to A$, which is a map such that $ \pi_A (1) = 1 $,  $\pi_A(r_1r_2)=\pi_A(r_1)\pi_A(r_2)$ and $\pi_A(r_1+r_2)=\pi_A(r_1)+\pi_A(r_2)$ for all $r_1,r_2\in R$. Notice that $A$ is in particular an $R$-module with the scalar multiplication
\begin{align*}
r\cdot x:=\pi_A(r)x \, ,\qquad\forall r\in R \, ,\ x\in A \, .
\end{align*}
A homomorphism of $R$-algebras $\phi:A\to B$ is an $R$-linear ring homomorphism, meaning that $\phi(r\cdot x+y)=r\cdot\phi(x)+\phi(y)$ and 
$ \phi(x \cdot y)=\phi(x) \cdot \phi(y) $, such that we have a commutative diagram
\[
\begin{tikzcd}
 A
 \arrow[rr, "\phi"]
 & & 
 B
 \\
 & 
 R \arrow[ul, "\pi_A"]\arrow[ur, "\pi_B"']
 & 
\end{tikzcd}
\]
\begin{lemma}\label{p:algebras}
The annihilator of an $R$-algebra $A$ is 
$ A_\ann=\ker(\pi_A) $. 
For each $R$-algebra homomorphism $\phi:A\to B$, we have 
$ A_\ann\subseteq B_\ann $. 
\end{lemma}

\begin{proof}
Clearly $\ker(\pi_A)\subseteq A_\ann$. 
Conversely, if $r\in A_\ann$, we have 
$ 0=r\cdot1=\pi_A(r)1=\pi_A(r) $, 
thus $\ker(\pi_A)=A_\ann$. Since 
$\pi_B=\phi\circ\pi_A$, we infer that 
$ A_\ann=\ker(\pi_A)\subseteq\ker(\pi_B)=B_\ann $. \qedhere
\end{proof}

\paragraph{\bf Gradings.}\label{ss:gradings}

In  algebraic topology, we  need rings and modules equipped with gradings. We denote by $\N=\{0,1,2,...\}$ the natural numbers (including zero). A ring $R^*$ with unit $1\in R^*$ and without non-trivial zero divisors is \emph{graded} when it can be written as the direct sum of abelian groups
\begin{align*}
R^*=\bigoplus_{d\in\N} R^d \qquad
\text{such that} \qquad R^{d_1}R^{d_2}\subseteq R^{d_1+d_2} \, , \ \forall  d_1,d_2\in\N \, . 
\end{align*}
 A non-zero element $r\in R^*\setminus\{0\}$ is called homogeneous when $r\in R^d$ for some $d := \degree{r}$, called the degree of $ r$. 
In the sequel, whenever we consider a non-zero  $r\in R^*\setminus\{0\}$ we always tacitly assume that it is homogeneous.
Note that the degree of the unit is $\degree{1}=0$. The ring $R^*$ is \emph{graded-commutative} when
$ r_1r_2=(-1)^{\degree{r_1}\,\degree{r_2}}r_2r_1 $, for any  $ r_1,r_2\in R\setminus\{0\} $.

An $R^*$-module $M^*$ is graded when it can be written as the direct sum of abelian groups
\begin{align*}
M^*=\bigoplus_{d\in\N} M^d \qquad
\text{such that}
\qquad R^{d_1}\cdot M^{d_2}\subseteq M^{d_1+d_2} \, , 
\  \forall 
d_1,d_2\in\N \, . 
\end{align*}

Homogeneous elements 
$x\in M^*\setminus\{0\}$ of $\degree{x}:=d$
are defined as for a graded ring. A homomorphism of $R$-graded modules $\phi^*:M^*\to N^*$ 
  preserves the degree of homogeneous elements, i.e. $\phi^*(M^d)\subseteq N^d$.
An $R^*$-algebra $A^*$ is graded when it is an $R^*$-graded ring with unit, equipped with a ring homomorphism $\pi_{A^*}:R^*\to A^*$ which preserves the degree of homogeneous elements.

\section{Equivariant cohomology}
\label{s:equivariant_cohomology}

Throughout this article, $G$ will denote a compact Lie group.

\paragraph{\bf Classifying spaces.}
 For any Lie group $ G $ (and indeed even general topological groups) there exists a contractible space $EG$ on which $ G $ acts freely. 
The quotient $BG:=EG/G$ is called the \emph{classifying space} of $G$. While $EG$ is not unique, $BG$ is unique up to homotopy equivalence. We will denote by $H^*(\cdot)$ the singular cohomology with coefficients in the ring of rational numbers $\Q$, and we will simply refer to it as the cohomology. 

%\begin{example}
%For $G=\Z_2=\Z/2\Z$, we can choose $E\Z_2$ to be the unit sphere in $\R^\infty$, equipped with the free $\Z_2$ action 
%$g\cdot x = (-1)^g x$. The classifying space is the infinite dimensional real projective space $B\Z_2=\RP^\infty$. 
% \end{example}
    
\begin{example}
For $G=S^1=\R/2\pi\Z$, we can choose $ES^1$ to be the unit sphere in $\C^\infty$, equipped with the free $S^1$ action $\theta \cdot z= e^{ \im \theta}z$. The classifying space  is the infinite dimensional complex projective space $BS^1=\CP^\infty$. Its cohomology is the ring of rational polynomials in one variable
$ H^*(BS^1)=\Q[u] $ 
where $u$ is a generator of $H^2(BS^1)\cong\Q$.
\end{example}

\begin{example}\label{torus}
Given two compact Lie groups $G_1$ and $G_2$ and their product $G:=G_1\times G_2$, we can choose $EG=EG_1\times EG_2$ equipped with the product action $(g_1,g_2)\cdot(e_1,e_2)=(g_1\cdot e_1,g_2\cdot e_2)$, and obtain the classifying space $BG=BG_1\times BG_2$. By the K\"unneth theorem, its cohomology is the tensor product
\begin{equation}\label{kunneth}
H^*(BG)=H^*(BG_1)\otimes H^*(BG_2) \, .
\end{equation}
We infer that the torus 
$\T^d := S^1 \times...\times S^1 $ 
($d$-fold product) has classifying space 
$ B \T^d = 
\CP^\infty \times... \times \CP^\infty $, 
whose cohomology is the ring of rational polynomials in $d$ variables
\begin{equation}\label{anntoro}
H^*(B\T^d) = \Q[u_1,...,u_d] \, ,
\end{equation}
where $u_1,...,u_d$ are generators of $H^2(B\T^d)\cong\Q^d$.
\end{example}

\begin{example}\label{ex:BF}
    If $F$ is a  finite group, its singular cohomology with $\Q$ coefficents is 
$  H^0(BF)=
            \Q $
            and 
           $ H^d(BF)= 
            0 $
            for any $ d \geq 1 $. 
\end{example}
Any compact Lie group $G$ contains a maximal subgroup isomorphic to a torus $\T^d$, where maximal means that such a subgroup is not contained in a torus subgroup of larger dimension. The dimension $d$ is called the {\it rank} of $G$. For any compact Lie group $G$ the classifying space $BG$ has cohomology ring 
$  H^*(BG)=\Q[u_1,...,u_d] $ 
with $ \degree{u_i}=d_i\in 2 \N $, 
where each $u_i\in H^{d_i}(BG)$ has {\it even} degree, see e.g.  \cite[Theorem 6.22]{M2010}. In particular, {\it $H^*(BG)$ is a ring without non-trivial zero divisors}.

\begin{remark}\label{rem:conne}
If  the compact Lie group $G$ is connected, the classifying space 
$ BG$ is simply connected. 
This follows from the long exact sequence of the homotopy groups associated with the principal $G$-bundle $EG\to BG$, see e.g.~\cite[Theorem~4.41]{Hatcher:2002aa}.
\end{remark}

\paragraph{\bf Borel quotient.}  

Let $X$ be a $G$-space, meaning a topological space equipped with a continuous $G$ action. 
If  the $G$ action is not free, there are no general techniques to compute the cohomology of the quotient $X/G$, 
and it is convenient to consider the \emph{Borel quotient}
\begin{equation}\label{Borelq}
X_G:= X \times_G EG \,. 
\end{equation}
Here, with common notation, $G$ acts diagonally on the product $X\times EG$, i.e.~$g\cdot(x,e) := (g\cdot x,g\cdot e)$, and $X \times_G EG := (X \times EG)/G $. The  Borel quotient
$ X_G $
is a natural substitute of the ordinary quotient  as $X\times EG$ is homotopy equivalent to $X$. The advantage is  that $G$ acts freely on $X\times EG$. The quotient projection $X\times EG\to X_G$ is a principal $G$-bundle, and the topology of $X_G$ can be studied via well established techniques in algebraic topology.

\paragraph{\bf Equivariant cohomology.}
Let $X$ be a $G$-space.
%namely a topological space equipped with a continuous $G$ action. 
Throughout this paper, all $G$-spaces are  assumed to be Hausdorff. 
A function $f:X\to Y$ is called $G$-{\it invariant} if  $ f(g\cdot x)=f(x)$ and $G$-{\it equivariant} if $f(g\cdot x)=g\cdot f(x)$ for any $ x\in X$.

\smallskip

Let $Y$ be a $G$-subspace of $X$. The \emph{$G$-equivariant cohomology} of the pair $(X, Y) $ is defined as the relative cohomology of their Borel quotients
\begin{align*}
H^*_G(X,Y):=H^*(X_G,Y_G) \, .
\end{align*}
Throughout this article $H^*(\cdot)$ denotes the singular cohomology with coefficients in the ring of rational numbers $\Q$.
We denote 
\[
H^*_G(X):=H^*_G(X,\emptyset) \, .
\]
The equivariant cohomology inherits many of the properties of the ordinary  cohomology,
including the compatibility with the $G$ action. More specifically, a continuous map between pairs of $G$-spaces $\phi:(X,Y)\to (W,Z)$ that is $G$-equivariant, meaning that 
$ \phi(g\cdot x)=g\cdot\phi(x) $, 
for any $ g\in G $,
$ x\in X $, 
induces a continuous map between the corresponding pairs of Borel quotients, still denoted with the same letter $\phi:(X_G,Y_G)\to (W_G,Z_G)$, $\phi([x,e])=[\phi(x),e]$, and therefore a ring homomorphism
\begin{align}\label{e:homomorphism_HG}
\phi^*:H^*_G(W,Z)\to H^*_G(X,Y) \, .
\end{align}
Two  $G$-equivariant continuous maps $\phi_0:(X,Y)\to (W,Z)$ and $\phi_1:(X,Y)\to (W,Z)$ induce the same homomorphism $\phi_0^*=\phi_1^*$ in equivariant cohomology provided they are interpolated by a $G$-equivariant homotopy, namely by a family of continuous $G$-equivariant maps $\phi_t:(X,Y)\to (W,Z)$ depending continuously on $t\in[0,1]$.

\smallskip

There is a crucial extra property that is specific to equivariant cohomology. 
While the cup product $\smallsmile$ gives a graded ring structure to the ordinary cohomology, the equivariant cohomology $H^*_G(X)$ is actually a \emph{graded $H^*(BG)$-algebra}, with the canonical homomorphism
\begin{align}\label{procan}
\pi_X^*:H^*(BG)\to H^*_G(X)
\end{align}
induced by the projection map 
\begin{equation}\label{proiecan}
\pi_X:X_G\to BG \, , \quad \pi([x,e]) :=[e] \, . 
\end{equation} 
Thus, by Lemma \ref{p:algebras}, for any $G$-space $X$,  
\begin{equation}\label{kerann}
\ker(\pi_X^*) = H^*_G(X)_\ann 
\, . 
\end{equation}

\begin{lemma}\label{l:pi_X_injective}
Let $X$ be a  $G$-space with  a fixed point $x=G\cdot x$ of the  action. Then  $ 
%\ker  \pi_X^*  = 
H^*_G(X)_\ann  
= 0 $.  
\end{lemma}

\begin{proof}
Since $x=G\cdot x$, the map $\pi_x:\{x\}_G\to BG$ is a homeomorphism, and therefore induces an isomorphism in cohomology. The inclusion $i:\{x\}\hookrightarrow X$ induces a homomorphism which fits in the commutative diagram
\[
\begin{tikzcd}[row sep=large]
 H^*_G(X)
 \arrow[r,"i^*"]
 & 
 H^*_G(\{x\})
 \\
 H^*(BG)
 \arrow[u,"\pi_X^*"]
 \arrow[ur,"\pi_x^*"',"\cong"]
 & 
\end{tikzcd}
\]
Since the diagonal isomorphism factorizes as $\pi_x^*=i^*\circ\pi_X^*$, we infer that $\pi_X^*$ is injective, and by \eqref{kerann}
we conclude that $ 
%\ker  \pi_X^*  = 
H^*_G(X)_\ann  
= 0 $. 
\end{proof}

The relative equivariant cohomology does not have an algebra structure, but nevertheless $H^*_G(X,Y)$ is a \emph{graded $H^*(BG)$-module} with scalar product
\begin{equation}\label{cuppro}
r\cdot h:=\pi_X^*(r)\smallsmile h \, ,
\qquad\forall r\in H^*(BG) \, ,
\ h\in H^*_G(X,Y) \, .
\end{equation}
Note that this is indeed well defined, as the cup product of a cohomology class in $H^*_G(X)$ with a relative cohomology class in $H^*_G(X,Y)$ is still a relative cohomology class in $H^*_G(X,Y)$. Clearly~\eqref{e:homomorphism_HG} is an $H^*(BG)$-module homomorphism, and even an $H^*(BG)$-algebra homomorphism if $Y=\emptyset$, meaning that it commutes with the scalar product, i.e.
$ \phi^*(r\cdot h)=r\cdot\phi^*(h) $. 
We will  often need the long exact sequences induced by a triple of $G$-spaces.
%and we stress that it is a sequence of $H^*(BG)$-module homomorphisms.

\begin{proposition}
\label{prop:long}

Any triple of $G$-spaces $Z\subseteq Y\subseteq X$ induces a long exact sequence of graded $H^*(BG)$-module homomorphisms
\begin{equation}
\label{incltripla}
 ...\toup^{\delta^{*-1}} H^*_G(X,Y) \toup H^*_G(X,Z) \toup H^*_G(Y,Z) \toup^{\delta^*} H^{*+1}_G(X,Y) \toup...
\end{equation}
where all the arrows  are induced by the inclusions, 
except  $\delta^*$. 
\end{proposition}

\begin{proof}
The statement follows from the long exact sequence  for the triple of Borel quotients $Z_G\subseteq Y_G\subseteq X_G$, and since the inclusions induce graded $H^*(BG)$-module homomorphisms in equivariant cohomology. The connecting homomorphism $\delta^*$ is a graded $H^*(BG)$-module homomorphism as well, as  follows by the definition of the scalar product \eqref{cuppro} and since $\delta^*$ satisfies 
$ \delta^*(i^*(k)\smallsmile h)=k\smallsmile\delta^*(h) $,
for any $ k\in H^*_G(Y)$, 
$ h\in H^*_G(X,Y) $, 
where $i:Y\hookrightarrow X$ is the inclusion.
\end{proof}

\subsection{Orbits of the action}\label{ss:orbits}
We recall that the compact Lie group $G$ acts freely on the contractible space $EG$, and therefore any compact Lie subgroup $H\subseteq G$ acts freely on $EG$ as well. This implies that we can choose $EH=EG$ and classifying space $BH=EG/H$. Notice that $BH$ is still a $G$-space, and $BG=BH/G$. Up to homeomorphism, the classifying space $BH$ depends only on the conjugacy class of $H$. Indeed, the quotient group $G/H$, endowed with the quotient topology, is a $G$-space whose Borel quotient admits a homeomorphism
\begin{align}
\label{e:G/H}
(G/H)_G \toup^{\cong} BH \, ,\qquad [gH,e]\mapsto[g^{-1}e] \, .
\end{align}
If $K=gHg^{-1}$ is another subgroup conjugate to $H$, there is a $G$-equivariant homeomorphism 
\[G/H\toup^{\cong} G/K \, ,\qquad hH\mapsto hHg^{-1}=hg^{-1}K \, ,\] 
and we infer by \eqref{e:G/H} that $BH\cong BK$ are $G$-homeomorphic.

Let $X$ be a $G$-space (which, as everywhere else in this paper, is assumed to be Hausdorff). Each point $x\in X$ has \emph{stabilizer}
\begin{equation}\label{Def Stabilizer} 
    G_x:= \big\{ g\in G\ |\ g\cdot x=x \big\} \, ,
\end{equation}
 also referred to as {\it isotropy group},  
which is a compact subgroup of $G$. 
By the elementary orbit-stabilizer theorem, the orbit $G\cdot x$ is homeomorphic to the quotient $G/G_x$ via the $G$-equivariant homeomorphism
\begin{equation}\label{orbsth}
 G\cdot x \toup^{\cong} G/G_x \, ,\qquad g\cdot x\mapsto gG_x \, .
\end{equation}
The projection
$ \pi_{G/G_x}$
in \eqref{procan}
factorizes as
$ \pi_{G/G_x} = q_x \circ \nu_x $,
where $ \nu_x : (G/G_x)_G    \stackrel{\cong} \toup BG_x$ 
is the  canonical homeomorphism  
defined in~\eqref{e:G/H}, and  $ q_x $ is the quotient map 
\[
q_x:BG_x\to BG=BG_x/G \, .
\]
The homeomorphism 
$ \nu_x $ together with \eqref{orbsth} 
induces a canonical isomorphism
\begin{equation}
\label{HBGiso}
H^*(BG_x)\cong H^*_G(G\cdot x) \, .
\end{equation}
Overall, we have the commutative diagram
\begin{equation}\label{cd:q pi}
\begin{tikzcd}[row sep=large]
   H^*(BG)
   \arrow[r, "\pi_{G\cdot x}^*"]
   \arrow[dr, "q_x^*"']
   &  
   H^*_G(G\cdot x)\\
   & 
   H^*(BG_x)
   \arrow[u, "\cong"']
\end{tikzcd}
\end{equation}
By  \eqref{kerann} we have  \begin{equation}\label{eq:Hann=ker(q*)=ker(pi*)}
    H^*_G(G\cdot x)_\ann=\ker(\pi^*_{G\cdot x}) \stackrel{\eqref{cd:q pi}}=\ker(q_x^*) \, . 
\end{equation}

\begin{lemma}\label{rmk:nonnull}
    Let $T=\T^d$ be a torus, and $X$ a $T$-space.
    If $x\in X$ is not a fixed point of the action, then $H^{*}_T(T\cdot x)_\ann\cap H^2(BT)\neq 0$.
    \end{lemma}

    \begin{proof}
%Let $d$ be the rank of $T$. 
The stabilizer $T_x$ is a closed abelian Lie-subgroup of $T$. Since $x$ is not a fixed point of the whole $T$ action, $T_x$ has rank $m<d$. Since $H^2(BT)\cong\Q^d$ and $H^2(BT_x)\cong\Q^{m}$, the homomorphism  $q_x^*:H^2(BT)\to H^2(BT_x)$ has a non-trivial kernel. This, together with \eqref{eq:Hann=ker(q*)=ker(pi*)}, implies the lemma.
\end{proof}

\begin{remark}\label{rm:finite stabilizer}
If $x$ has finite stabilizer $G_x$ then, by 
\eqref{HBGiso} and
Example \ref{ex:BF}, we 
deduce that 
$  H_G^0 (G\cdot x) \cong \Q$ and $ H_G^d (G\cdot x)=0 $ for any degree $d \geq 1 $. 
\end{remark}

\subsection{Annihilator of the equivariant cohomology}\label{ss:free_ann}
The following are basic properties of the annihilator of pairs of $G$-spaces.

\begin{proposition}\label{p:ann}
Let $X, Y $ be $G$-spaces. The following holds:  
\\[1mm]
$(i)$ \textnormal{\textbf{Monotonicity:}}   If there is a continuous $G$-equivariant map $\phi:X\to Y$, then
$ H^*_G(Y)_\ann\subseteq H^*_G(X)_\ann $.
\\[1mm]
$(ii)$
\textnormal{\textbf{Union subadditivity:}} for each $G$-space $X$ and for each pair of $G$-subspaces $Y,Z\subseteq X$ whose interiors cover $X$, i.e.~$X=\interior(Y)\cup\interior(Z)$, we have
\[H^*_G(Y)_\ann\smallsmile H^*_G(Z)_\ann\subseteq H^*_G(X)_\ann \, . \] 
$(iii)$
\textnormal{\textbf{Pair monotonicity:}} For each pair of $G$-spaces $Y\subseteq X$, we have 
$
H^*_G(X)_\ann\subseteq H^*_G(X,Y)_\ann $.  
\\[1mm]
$(iv)$ \textnormal{\textbf{Triple subadditivity:}} For each triple of $G$-spaces $Z\subseteq Y\subseteq X$, we have
\begin{equation}
\label{intriple}H^*_G(X,Y)_\ann\smallsmile H^*_G(Y,Z)_\ann \subseteq H^*_G(X,Z)_\ann \, . 
\end{equation}
\end{proposition}

\begin{proof}

($i$) By \eqref{e:homomorphism_HG}
a continuous $G$-equivariant map $\phi:X\to Y$ induces a $H^*(BG)$-algebra homomorphism $\phi^*:H^*_G(Y)\to H^*_G(X)$ that fits into the commutative diagram, 
\[ 
\begin{tikzcd}
 H^*_G(Y)
 \arrow[rr,"\phi^*"]
 & & 
 H^*_G(X)
 \\ 
 &
 H^*(BG)
 \arrow[ul,"\pi_Y^*"]
 \arrow[ur,"\pi_X^*"']
 & 
\end{tikzcd}
\]
where $ \pi_X^*, \pi_Y^* $ 
are induced by the canonical projections \eqref{proiecan}.
This diagram implies $H^*_G(Y)_\ann=\ker(\pi_Y^*)\subseteq\ker(\pi_X^*)= H^*_G(X)_\ann$.
\\[1mm]
($ii$) 
Consider the commutative diagram
\[ 
\begin{tikzcd}
 H^*_G(X,Y)
 \arrow[r]
 & 
 H^*_G(X)
 \arrow[r]
 & 
 H^*_G(Y)
 \\
 & 
 H^*(BG)
 \arrow[u,"\pi_X^*"]
 \arrow[ur,"\pi_Y^*"']
 & 
\end{tikzcd}
\]
where the homomorphisms in the exact horizontal line are induced by the inclusions
$ (Y,\emptyset) \hookrightarrow 
(X,\emptyset) \hookrightarrow (X,Y) $.
By the exactness, for any $r\in H_G^*(Y)_\ann=\ker(\pi_Y^*)$, the cohomology class $\pi_X^*(r)$ is in the image of the homomorphism $H^*_G(X,Y)\to H^*_G(X)$. Analogously, for any $s\in H_G^*(Z)_\ann=\ker(\pi_Z^*)$, the cohomology class $\pi_X^*(s)$ is in the image of the homomorphism $H^*_G(X,Z)\to H^*_G(X)$ induced by the inclusion.
Since $X=\interior(Y)\cup\interior(Z)$, the pair $\{Y,Z\}$ is excisive and
the cup product $\smallsmile: H^{n}_G(X,Y)\smallsmile H^{m}_G(X,Z)\to H^{n+m}_G(X,Y\cup Z)$ is well defined.
By the naturality of the cup product we deduce that the cohomology class $\pi_X^*(r\smallsmile s)=\pi_X^*(r)\smallsmile \pi_X^*(s)$ is in the image of the homomorphism $H^*_G(X,Y\cup Z)\to H^*_G(X)$ induced by the inclusion. But $H^*_G(X,Y\cup Z)=H^*_G(X,X)=0$, and therefore $r\smallsmile s\in\ker(\pi_X^*)=H_G^*(X)_\ann$.
\\[1mm]
($iii$) For each $r\in H^*_G(X)_\ann=\ker(\pi_X^*)$ and $h\in H^*_G(X,Y)$, we have $r\cdot h=\pi_X^*(r)\smallsmile h=0$.
\\[1mm]
($iv$) The inclusion 
\eqref{intriple} follows by  Proposition \ref{p:modules}-($ii$) applied to the exact sequence \eqref{incltripla}.
\end{proof}

We now provide a more technical lemma involving the torsion-free quotient and the annihilator of the equivariant cohomology in the presence of a fixed point of the group action. Hereafter, for any fixed point $x=G\cdot x$ of a compact Lie group action, we will tacitly {\it identify} $H^*_G(\{x\})  $ with $H^*(BG) $ through the isomorphism  $ \pi_x^* $ induced by the canonical homeomorphism $\pi_x:\{x\}_G\to BG$.

\begin{lemma}\label{l:free_ann_technical}
Let $X$ be a $G$-space containing a fixed point $x=G\cdot x$ of the  action, and  a $G$-invariant subspace $Y\subseteq X$. If $H^*_G(Y)_\free=H_G^*(X,\{x\})_\free=0$ and the inclusion $Y\hookrightarrow X$ is $G$-equivariantly homotopic to the constant map $Y\to\{x\}$, then the inclusion of pairs $(\{x\},\emptyset)\hookrightarrow (X,Y)$ induces an  isomorphism
\[
H^*_G(X,Y)_\free\toup^{\cong} H^*_G(Y)_\ann \, .
\]
\end{lemma}

\begin{proof}
Consider the commutative diagram
\begin{equation}
\label{e:les_free_ann_technical}
\begin{tikzcd}[row sep=large]
    H_G^{*-1}(Y)
    \arrow[r, "\delta^{*-1}"]
    &  
    H_G^*(X,Y)
    \arrow[r, "i^*"]
    &  
    H_G^*(X)
    \arrow[r, "j^*"]
    &  
    H_G^*(Y)\\
    &&
    H^*(BG)
    \arrow[u,  "\pi_X^*"]
    \arrow[ur, "\pi_Y^*"']
\end{tikzcd}
\end{equation}
where the horizontal line is the exact sequence of the pair of $G$-spaces $Y\subseteq X$, 
so that $i^*$ is induced by the inclusion  $i : (X, \emptyset)  \hookrightarrow(X,Y)$. 
The homomorphism $\pi_X^*$ is injective according to Lemma~\ref{l:pi_X_injective}.  We have 
$\ker(\pi_Y^*)=H^*_G(Y)_\ann$, cf.~\eqref{kerann}, and  the exact sequence in \eqref{e:les_free_ann_technical} implies 
\begin{equation}
\label{leprimei}
\pi_X^*(H^*_G(Y)_\ann)\subseteq\image(i^*) \, .
\end{equation}
Since, by assumption,  the inclusion $j:Y\hookrightarrow X$ is $G$-equivariantly homotopic to the constant map
$Y\to\{x\}$, we have the commutative diagram
\begin{equation*}
\begin{tikzcd}
    H_G^*(X)
    \arrow[r, "j^*"]
    \arrow[dr, "l^*"']
    &  
    H_G^*(Y)\\
    & 
    H^*(BG)
    \arrow[u, "\pi_Y^*"']
\end{tikzcd}
\end{equation*}
where the diagonal homomorphism is induced by the inclusion $l:\{x\}\hookrightarrow X$.
This, together with the exact sequence in \eqref{e:les_free_ann_technical}, implies that $\image(l^*i^*)\subseteq H^*_G(Y)_\ann$.
Since $l^*\pi_X^*=\id$, we also have the opposite inclusion $H^*_G(Y)_\ann=l^*\pi_X^* (H^*_G(Y)_\ann)\subseteq\image(l^*i^*)$ by \eqref{leprimei}, and therefore
\begin{equation}
%\label{imageiso}
\image(l^*i^*)= H^*_G(Y)_\ann \, .
\end{equation}
Using 
the  first part of the 
 exact sequence in the horizontal line in  \eqref{e:les_free_ann_technical}
 and the assumption  $H_G^*(Y)_\free =0$, Proposition 
\ref{p:modules}-($iii$) implies that $i^*$ induces an injective homomorphism
$$ i^*:H_G^*(X,Y)_\free\hookrightarrow H_G^*(X)_\free \, . 
$$
Similarly, by Proposition 
\ref{p:modules}-($iii$), 
the exact sequence of the pair $\{x\}\subseteq X$, 
$$ H_G^*(X,\{x\})\toup H_G^*(X)\toup^{l^*} H^*(BG) \, ,  
$$
and the assumption 
$ H_G^*(X,\{x\})_\free=0$,
we get  that $l^*$ induces an injective homomorphism 
$ l^*: H_G^*(X)_\free\hookrightarrow H^*(BG) $. 
Thus, the inclusion $i\circ l:(\{x\},\emptyset)\hookrightarrow(X,Y)$ induces 
an injective homomorphism $H^*_G(X,Y)_\free\hookrightarrow H^*(BG)$,  
namely an
isomorphism  
\[(i\circ l)^*=l^*i^*: H_G^*(X,Y)_\free \toup^{\cong} H^*_G(Y)_\ann.
\qedhere\]
\end{proof}

\subsection{The Euler class and the Gysin sequence}

In our applications to the multiplicity of Stokes waves, we  need to compute the equivariant cohomology of spheres, products of spheres, and certain join products. The fundamental tool is the Gysin exact sequence for oriented sphere bundles, which we report in 
Theorem \ref{t:Gysin}, see  Milnor and Stasheff \cite{Milnor:1974aa} or Spanier \cite{Spanier}.
While everywhere else in the paper we employ singular cohomology with coefficients in $\Q$, in this section we will sometimes employ singular cohomology with coefficients in $\Z$, specifying the employed coefficients ring $R$ writing $H^*(\,\cdot\,;R)$.

Let $\pi:V\to B$ be an oriented real vector bundle of rank $r$,   i.e.\  each  fiber $ V_x:=\pi^{-1}(x) $ is 
a real oriented vector space of  dimension $ r $, and the orientations vary coherently with $x$. Any complex vector bundle is also a real oriented vector bundle with a canonical orientation, see e.g. \cite[Lemma 14.1]{Milnor:1974aa}.
A complex vector bundle of complex rank one is called a complex line bundle. With a common abuse of terminology, we will often refer to its total space $V$ as the ``vector bundle", at least when the projection $\pi$ on the base space $ B $ is implicit from the context. We denote by $0_V\subseteq V$ the zero-section of the bundle, which is the subspace intersecting every fiber $V_x$ at the origin.
Note that $0_V$ is homeomorphic to $B$.
The orientation of any fiber $V_x$ fixes a generator $m_x\in H^r(V_x,V_x\setminus\{0\};\Z)\cong\Z$.
The following is a fundamental result on the cohomology of vector bundles (see e.g. \cite[Theorem 9.1]{Milnor:1974aa}).

\begin{theorem}[Thom isomorphism] 
There is a unique cohomology class $m\in H^r(V,V\setminus 0_V;\Z)$, called the \emph{Thom class of $ V $}, 
that for any $ x \in B $  projects to 
$m_x \in H^r(V_x,V_x\setminus\{0\};\Z)$
under the homomorphism induced by the inclusion 
$ i_x : (V_x, V_x \setminus \{0\} ) \to (V, V \setminus 0_V )$. 
Furthermore  the map 
\begin{equation}\label{thomiso}
    H^*(V;\Z) 
    \ttoup^{\cong} 
    H^{*+r}(V,V\setminus 0_V;\Z) \, ,
    \qquad
    h\mapsto m\smallsmile h \, ,
\end{equation}
is an isomorphism.
\end{theorem}

\begin{definition}
\label{eulerc} The \emph{Euler class} $e(V)\in H^r(B;\Z)$ is the image of the Thom class $m$ under the homomorphism induced by the compositions 
\[(B,\emptyset)
\stackrel{l}
\hookrightarrow
(V,\emptyset)
\stackrel{i}\hookrightarrow (V,V\setminus0_V)\, ,
\] 
where  $ l $ sends any  
$ x \in B $ 
to the corresponding origin 
$ 0$  in $ V_x $ %\subset V $, 
and $ i $  is the inclusion. 
\end{definition}

We list  properties of the Euler class which we will employ later on:
\begin{itemize}
\setlength{\itemsep}{5pt}
\item[(i)] If  $ \pi: V\to B $ is an oriented vector bundle and $\psi:C\to B$ is a continuous map between topological spaces, then the Euler class of the 
pull-back vector bundle
$\psi^*V\to C$ is $e(\psi^*V)=\psi^* e(V) $. 
Here 
$$
 \psi^*V:=\big\{(c,x)\in C \times V \ \big| \ \psi(c)=\pi(x) \big\}
    \stackrel{\psi^* \pi} \ttoup C \, ,  
   \     \ (c,x)\mapsto
   (\psi^* \pi)(c,x) :=   c\,  ,
$$
so that each fiber $ (\psi^*V)_c$ is isomorphic to $ V_{\psi(c)} $.
\item[(ii)] The direct sum of two oriented vector bundles $V_1\to B$ and $V_2\to B$ has Euler class
$ e(V_1\oplus V_2)=e(V_1)\smallsmile e(V_2) $. Here, $V_1 \oplus V_2\to B$ is the vector bundle with fibers $(V_1\oplus V_2)_x=(V_1)_x\oplus(V_2)_x$.
\item[(iii)] If an oriented vector bundle $\pi:V\to B$ has a nowhere vanishing section (namely a continuous map $s:B\to V\setminus 0_V$ such that $\pi\circ s=\id$) then $e(V)=0$.

\item[(iv)] If $\overline V$ is the vector bundle $V$ with opposite orientation, then $e(\overline V)=-e(V)$.

\item[(v)] The tensor product of two complex line bundles  $L_1\to B$ and $L_2\to B$ has Euler class 
$ e(L_1\otimes L_2)=e(L_1)+e(L_2) $.
Here, $L_1 \otimes L_2\to B$ is the complex line bundle with fibers $(L_1\otimes L_2)_x=(L_1)_x\otimes(L_2)_x$, where the tensor product is over $\C$.
\end{itemize}

If $V$ is a finite dimensional 
vector space with a linear action of $G$ then the Borel quotient $V_G = V \times_G EG$, with canonical projection 
$ \pi_V :  V_G \to BG $, is a vector bundle (see e.g. \cite{Tu} Proposition 4.7) which is orientable  since $ G $ is connected.

\begin{example}\label{ex1c}
Consider the circle $\T=\R/\Z$ acting on $V(k):=\C$ as
$  \theta\cdot z:=e^{ \im 2\pi k\theta }z $  for any  $ \theta\in \T $, $ z\in V(k) $, 
where $k\in\Z$ is an integer. 
The Borel quotient 
$ V(k)_\T=V(k)\times_{\T} E\T $ with the canonical projection $V(k)_\T\to B\T$, $[z,f]\mapsto [f]$,  is a complex line bundle. 

\begin{lemma}\label{l:e_k}
There exists a generator $u$ of $H^2(B\T;\Z)\cong\Z$ such that, for each $k\in\Z$, the Euler class of the complex line bundle $V(k)_\T\to B\T$ of Example \ref{ex1c} is 
$ e_k=k\,u $. 
\end{lemma}

\begin{proof} 
By Definition \ref{eulerc}, 
the Euler class $ e(L) $
of $ L :=V(1)_\T $ is 
$ e_1:=l^*i^*m $ 
where 
$ m \in H^2(L,L\setminus0_L;\Z)$ is the Thom class of $ L $.
The homomorhism 
$ l^* : H^* (L;\Z) \to H^* (B\T;\Z) $ is an isomorphism because  $l:B\T\hookrightarrow L$ is an homotopic equivalence  
with homotopy inverse the base projection $ \pi_{V(1)} :L\to B\T$.
We now prove that  also 
$ i^* $ is an  isomorphism 
and therefore $e_1$ is a generator of $H^2(B\T;\Z)$.
Consider the following portion of the long exact sequence associated with the inclusion $L\setminus0_L\subset L $ 
\begin{equation}\label{pezzom}
H^{1}(L\setminus0_L;\Z)\toup H^2(L,L\setminus 0_L;\Z)\toup^{i^*} H^2(L;\Z)\to H^2(L\setminus0_L;\Z) \, . 
\end{equation}
Note that $L\setminus 0_L$ is homotopy equivalent to $E\T $. 
Indeed, 
$L\setminus 0_L$ is homotopy equivalent to the 
sphere bundle  
$\T\times_\T E\T$ of $L$ (here $ \T = \{ z \in \C 
\ | \ | z | = 1 \}$) 
and 
$\T\times_\T E\T\cong E\T$.  
Explicitly the inclusion $j:\T\times_\T E\T\hookrightarrow L\setminus 0_L$ has homotopy inverse $\ell:L\setminus 0_L\to \T\times_\T E\T:[z,f]\mapsto [\tfrac{z}{|z|},f]$.

For any integer $k\geq 1$, we have an isomorphism of complex line bundles
\begin{equation}\label{tensorp}
\underbrace{L\otimes...\otimes L}_{\times k} \ttoup^{\cong} V(k)_{\T} \, ,
\qquad
[z_1,f]\otimes...\otimes [z_k,f]\mapsto [z_1...z_k,f] \, . 
\end{equation}
Therefore, by properties (i) and (v), the Euler class of $V(k)_\T$ is
$ 
e_k =k\,e(L)=k\,e_1 $. 
Moreover, we have an isomorphism of complex line bundles
\begin{align*}
\overline L \ttoup^{\cong} V(-1)_{\T} \, ,
\qquad
[\overline{z},f]\mapsto [z,f] 
\end{align*}
where $\overline{L}$ denotes the complex conjugate vector bundle to $L$. As a real vector bundle, $L$ is equal to $\overline L$, but is equipped with the opposite complex structure, and thus has the opposite canonical orientation than $L$. Property (iv) implies that the Euler class of $V(-1)_\T$ is 
$ e_{-1}=e(\overline L)=-e(L)=- \,e_1 $, and, 
as in \eqref{tensorp} we deduce that $ e_{-k}
:= e(V(-k)_\T ) = k\, 
e_{-1} = - k e_1 $ for any $k \geq 1 $.
Finally, for $k=0$, 
the $ \T $-action in Example 
\ref{ex1c} is trivial, and 
the vector bundle $ V(0)_\T\cong \C\times BG$ is trivial. In particular, it admits a nowhere vanishing section, and property (iii) implies that its Euler class is $ e_0 = 0 $.
\end{proof}
\end{example}

\begin{example}
\label{es2}
More generally, consider a vector $k \in\Z^d$, and the $d$-torus $\T^d=\T\times...\times\T$ acting on $V(k):=\C$ as
$  \theta\cdot z:=e^{\im 2\pi \langle k,\theta\rangle }z $,
for any $  \theta\in \T^d $, $ z\in V(k) $. 
The Borel quotient $V(k)_{\T^d}$  with the canonical projection $ V(k)_{\T^d}\to B\T^d$, $[z,f]\mapsto [f]$ is a complex line bundle.

\begin{lemma}\label{lemma214}
For $k=(k_1,...,k_d)\in\Z^d$, the Euler class of the complex line bundle $V(k)\to B\T^d$ of Example \ref{es2} is 
$ e_k=k_1\,u_1+...+k_d\,u_d $, where $u_1,...,u_d$ is the canonical basis of the cohomology group $H^2(B\T^d;\Z)$, namely
$ H^*(B\T^d;\Z)=\Z[u_1,...,u_d]
$. 
\end{lemma}

\begin{proof}
Recalling that $B\T^d=B\T\times...\times B\T$ by Example \ref{torus}, 
we consider the projections 
$  p_j:B\T^d\to B\T $, $ p_j([f_1,...,f_d]):=[f_j] $ for any $ j = 1, \ldots, d $,
and 
%For each $k=(k_1,...,k_d)\in\Z^d$, we consider 
the complex line bundles 
$ L_{k_j}:=V(k_j)_\T\to B\T $. 
% L_k:=V(k)_{\T^d}\to B\T^d  \,  .
We have an isomorphism of complex line bundles
\begin{align*}
 p_1^*L_{k_1}\otimes...\otimes p_d^*L_{k_d}
 \ttoup^{\cong}
V(k)_{\T^d} \, ,\quad
 ([z_1,f_1]\otimes...\otimes[z_d,f_d], [f_1, \dots, f_d])
 \mapsto
 [z_1...z_d,f_1,...,f_d] \, .
\end{align*}
Therefore, by Lemma~\ref{l:e_k} and properties (i) and (v), the Euler class $e_k$ of $ V(k)_{\T^d} $ is 
\begin{align*}
 e_k
 =
 p_1^*e_{k_1}+...+p_d^*e_{k_d}
 =
 k_1\,p_1^*u+...+k_d\,p_d^*u =
 k_1\, u_1 +...+k_d\,u_d 
\end{align*}
where $u_j:=p_j^*u$  
% $j=1,...,d$, 
are the canonical generators of 
$ H^2(B\T^d;\Z)\cong H^2(B\T;\Z)\oplus...\oplus H^2(B\T;\Z) \cong\Z^d $. 
\end{proof}
\end{example}

\begin{example}\label{ex:Vk}
Even more generally, consider  integer vectors $k_1,...,k_n\in\Z^d$, which we group in a $d\times n$ matrix $k=(k_1,...,k_n)$, 
and the $d$-torus $\T^d$ acting on $V(k):=\C^n$ as
\begin{align}\label{eq:def action k1...kn}
 \theta\cdot z:=(e^{ \im 2\pi \langle k_1,\theta\rangle }z_1,...,e^{ \im 2\pi \langle k_n,\theta\rangle }z_n) \, ,
 \quad 
 \forall \theta\in \T^d \, ,\ z\in V(k) \, .
\end{align}
The Borel quotient $V(k)_{\T^d}$  with the canonical projection $\pi_{V(k)}  : V(k)_{\T^d}\to B\T^d$, $[z,f]\mapsto [f]$, is an oriented vector bundle of real rank $2n$. 
\end{example}

\begin{lemma}\label{lemmanew}
For a matrix $k=(k_1,...,k_n)$ with entries $k_i\in\Z^d$, the Euler class of the oriented vector bundle $V(k)_{\T^d} \to B\T^d$ is 
$e_k=e_{k_1}\smallsmile...\smallsmile e_{k_n}$ where 
each  $e_{k_i} \in H^2(B\T^d;\Z) $ is given by Lemma \ref{lemma214}.
\end{lemma}

\begin{proof}
In  terms of bundle operations 
$ V(k)_{\T^d}=V(k_1)_{\T^d}\oplus...\oplus V(k_n)_{\T^d} $. 
Therefore,  by property (ii) and Lemma \ref{lemma214}, its Euler class is  
$e_k=e_{k_1}\smallsmile...\smallsmile e_{k_n}$.
\end{proof}

We now introduce  the Gysin exact sequence. 
Let $\pi:V\to B$ be an oriented vector bundle of rank $r $, that is  numerable,  i.e. it is trivial over an open covering which admits a locally finite partition of unity.
This assumption ensures that the vector bundle is a fibration (see e.g. \cite{Spanier}[ Th. 12 page 95]) and that admits a  bundle Riemannian metric which gives an inner product and a norm $\|\cdot\|_x$ on each fiber $V_x$. The space  
\[
E:=\sph(V):=\big\{v\in V\ \big|\ 
\| v \|_{\pi(v)}=1 \big\} 
\] 
with the  restricted base projection $\pi:E\to B$ is an oriented $(r-1)$-sphere bundle, namely each fiber $ \pi^{-1}(x) $ is the unit 
sphere in 
the $r$-dimensional vector space $ V_x $. With a common abuse of notation, we denote by $e(E):=e(V)$ the Euler class of $V$, and we call it the Euler class of $E$ as well. 

The following is another fundamental result of the theory of  vector bundles, see e.g.\ \cite[Th.~3, page 483]{Spanier}.

\begin{theorem}[Gysin sequence]
\label{t:Gysin}
Let $\pi:V\to B$ be an oriented numerable vector bundle of rank $r$ with associated sphere bundle $ E = \sph(V) $.  
For each subset $B'\subseteq B$  consider the sphere bundle pair $\pi:(E,E')\to(B,B')$
where $E':=\pi^{-1}(B')$.
Then there exists a homomorphism
\[ 
\pi_{*}:H^{*+r-1}(E,E';\Z) \to H^{*}(B,B';\Z)
\] 
that lowers the degree of $r-1$ and fits in a long exact sequence 
\begin{small}
\begin{align*}
 ...\tttoup^{\pi_*} H^*(B,B';\Z)&\tttoup^{e(E)\smallsmile} H^{*+r}(B,B';\Z)\tttoup^{\pi^*} H^{*+r}(E,E';\Z) \tttoup^{\pi_*} H^{*+1}(B,B';\Z) \tttoup ^{e(E)\smallsmile} ...
\end{align*}    
\end{small}
In particular, if $B'=\emptyset$, we have a long exact sequence
\begin{align}
 ...\tttoup^{\pi_*} H^*(B;\Z)\tttoup^{e(E)\smallsmile} &H^{*+r}(B;\Z)\tttoup^{\pi^*} H^{*+r}(E;\Z) \tttoup^{\pi_*} H^{*+1}(B;\Z) \tttoup ^{e(E)\smallsmile} ...
\label{eq:gysin} 
 \end{align}
An analogous Gysin sequence for singular cohomology with coefficients in $\Q$ holds.
\end{theorem}

We apply the Gysin sequence to the  Borel quotient  $ V_G $ 
of a
vector space $ V $ with a linear action of $G$ with canonical projection $\pi_V : V_G \to BG $. Note that $ V_G $ is numerable, see  e.g. \cite{Dold} Theorem 8.

\begin{example}\label{ex:S(k)}
Consider the torus $T:=\T^d$ acting on $V(k)=\C^n$ as  in Example~\ref{ex:Vk}.
The unit sphere 
$ 
 S(k):=S^{2n-1}\subset V(k) $
is invariant under this action. The Borel quotient $S(k)_T=S(k)\times_T ET$ is the total space of the sphere sub-bundle of the vector bundle $V(k)_T\to BT$, i.e.
$  S(k)_T = \sph(V(k)_T) $. 
%The cohomology ring of the classifying space %$BT$ is 
%$H^*(BT;%\Z)=\Z[u_1,...,u_d]$%.
By Lemma 
\ref{lemmanew}
the sphere bundle $S(k)_T\to BT$, $[z,f]\mapsto [f]$,  has Euler class 
\begin{equation}\label{ekclass}
e_k=e_{k_1}\smallsmile...\smallsmile e_{k_n} \qquad \text{where} \qquad  
e_{k_i} = 
k_{i,1}u_1+...+k_{i,d}u_d  
\end{equation}
and 
$ k_i = (k_{i,1}, \ldots, k_{i,d}) \in \Z^d $ for any $ i =1, \ldots, n $. 

\begin{lemma}\label{p:cohomology_sphere}
Assume that the % $d\times n$ 
matrix $k = (k_1, \ldots, k_n ) $  has non-zero components $ k_1, \ldots, k_n \in \Z^d \setminus \{ 0 \} $. Then the $T$-equivariant cohomology of the unit sphere $S(k) \subset V(k)$ equipped with the action \eqref{eq:def action k1...kn} is  
\begin{equation}\label{isoSk}
 \frac{\Z[u_1,...,u_d]}{(e_k)}\cong H^*_T(S(k);\Z) \, ,
\end{equation}
where $(e_k):=e_k\smallsmile \Z[u_1,...,u_d]$ is the ideal generated by the Euler class 
$e_k$ in \eqref{ekclass}. The isomorphism in \eqref{isoSk} is given by $v+(e_k)\to \pi^*_{S(k)} v$. 
\end{lemma}

\begin{proof}
Since $ k_i \neq 0 $ for any $ i = 1, \ldots, n $, the Euler class $e_k$ in 
\eqref{ekclass} is non-zero and, since 
the ring $ H^{*}(BT;\Z) $ is without  
non-trivial zero divisors,  the homomorphism 
\[H^*(BT;\Z)\tttoup^{e_k\smallsmile} H^{*+2n}(BT;\Z) \]
is injective. Then we can extract from the exact Gysin sequence 
\eqref{eq:gysin} of the sphere bundle $S(k)_T\to BT$ the short exact sequence
\begin{align*}
0\ttoup H^{*-2n}(BT;\Z) \ttoup^{e_k\smallsmile}   
H^{*}(BT;\Z) \ttoup^{\pi_{S(k)}^*}   H^*_T(S(k);\Z) \ttoup 0  
\end{align*}
where $H^*(BT;\Z)=\Z[u_1,...,u_d]$ by Lemma \ref{lemma214}.
Thus 
$\pi_{S(k)}^*$
is surjective, 
$ \ker \pi_{S(k)}^* =
(e_k) $  
and  the first isomorphism theorem 
implies   \eqref{isoSk}.
\end{proof}
\end{example}

\begin{example}\label{ex:sphere_ann}
Let $ G $ be a compact connected Lie group acting on the unit sphere $S^n \subset  \R^{n+1}$. 
We can extend the $G$ action to $ \R^{n+1}$ by setting 
$ g\cdot \rho z:=\rho(g\cdot z) $, 
for any 
$ g\in G $, $ \rho \geq 0  $, $ z\in S^n $. By Remark \ref{rem:conne}
 the classifying space $BG$ is simply connected
 and 
the vector bundle $ \R^{n+1}_G\to BG$ is orientable. Let  $e\in H^{n+1}(BG;\Q)$ be its Euler class with rational coefficients. 
The Borel quotient $ S^n_G $ is precisely the total space of the sphere bundle $\sph (\R^{n+1}_G)\to BG$, and its Gysin sequence  \eqref{eq:gysin}, 
\begin{align*}
...
\ttoup
H^*(BG;\Q)
\ttoup^{e\smallsmile}
H^{*+n+1}(BG;\Q)
\ttoup^{\pi_{S^n}^*}
H^{*+n+1}_G(S^n;\Q)
\ttoup... \, , 
\end{align*}
implies that
\begin{equation}
\label{annsfera}
    H^*_G(S^n;\Q)_\ann = 
    \ker (\pi_{S^n}^*) = (e) 
    \
    \text{where} \ 
    (e)=e\smallsmile H^*(BG;\Q) \, .   
\end{equation}
 %is the ideal generated by $e$.
Note that $e$ may be zero.
This is the case if the action on $S^n$ has fixed points, by Lemma~\ref{l:pi_X_injective}.
%or alternatively 
%because there is a section of the sphere bundle $\pi_{S^n}:S_G^n\to BG$ and using property (iii) of the Euler class.
\end{example}

\subsection{Some computations of equivariant cohomology}\label{sec:computations}

From now on we use the singular cohomology with coefficients in 
$\Q$.
All computations made with coefficients in $\Z$ still hold with coefficients in $\Q$
after taking the tensor product with $\Q$.

\paragraph{\bf Product of two spheres.}
With the notation of Example~\ref{ex:S(k)}, 
 we consider unit  spheres $S(k_1) \subset \C^{n_1}$ and $S(k_2) \subset \C^{n_2} $, equipped with $ T=\T^d $ actions of the form \eqref{eq:def action k1...kn}
defined by matrices   
$k_j=(k_{j,1},...,k_{j,n_j})
\in\Z^{d\times n_j}$ for $ j = 1,2 $,  and their product
\begin{align}\label{eq:defProduct of two spheres}
S(k_1,k_2) := S(k_1)\times S(k_2) \, ,
\end{align}
equipped with the diagonal $T$ action. We make the following \emph{non-collinearity} assumption: 
for any $h_1\in\{1,...,n_1\}$ and $h_2\in\{1,...,n_2\}$,
\begin{equation}\label{eq:non-coll}
    \text{the vectors} \ k_{1,h_1},\, k_{2,h_2} \in \Z^d \ \text{are linearly independent over} \ \Q\,.
\end{equation}

%\begin{remark}
%\label{rmk:non-coll torus}
%If the torus $T=\T^d$ has dimension $d=2$, the non-collinearity assumption is equivalent to 
%$    \langle k_{1,h_1}\rangle^\bot\cap \langle k_{2,h_2}\rangle^\bot=\{0\}$, for all $  h_1\in\{1,...,n_1\} $, 
%$ h_2\in\{1,...,n_2\} $
%and thus to the requirement that all $z\in S(k_1,k_2)$ have finite stabilizers $T_z=\{\theta\in T\ |\ \theta\cdot z=z\}$.
%\end{remark}

\begin{proposition}\label{p:cohomology_sphere_product}
{\bf ($ T $-cohomology of products of spheres)}
Under the non-collinearity assumption \eqref{eq:non-coll}, the rational $T$-equivariant cohomology of the product of spheres $S(k_1,k_2)$ in \eqref{eq:defProduct of two spheres} is 
\begin{equation}\label{prop323}
H^*_T(S(k_1,k_2);\Q)\cong \frac{\Q[u_1,...,u_d]}{(e_{k_1},e_{k_2})} \, ,
\end{equation}
where $(e_{k_1},e_{k_2}):=e_{k_1}\smallsmile \Q[u_1,...,u_d]+e_{k_2}\smallsmile \Q[u_1,...,u_d] $ is the ideal generated by the Euler classes $e_{k_j}$ of the sphere bundles $ S(k_j)_T =\sph(V(k_j)_T)$, i.e. for $ j = 1,2 $,  
\begin{equation}\label{ekj}
e_{k_j}=e_{k_j,1}\smallsmile...\smallsmile e_{k_j,n_j} \, , 
\qquad 
 e_{k_j,h} = k_{j,h,1}u_1+...+k_{j,h,d}u_d\,.
\end{equation}
\end{proposition}

\begin{proof}
 We pull-back the vector bundle $\pi_2 := \pi_{V(k_2)} :V(k_2)_T\to BT$ by the projection $\pi_1 := \pi_{V(k_1)}   :S(k_1)_T\to BT$, obtaining the vector bundle $\pi_1^*V(k_2)_T\to S(k_1)_T$, whose total space is 
\begin{equation}\label{pullb1}
\pi_1^*V(k_2)_T = \big\{ (q_1,q_2)\in S(k_1)_T\times V(k_2)_T\ \big|\ \pi_1(q_1)=\pi_2(q_2) \big\}.
\end{equation}
By property (i) its Euler class is 
$ e(\pi_1^*V(k_2)_T)=\pi_1^* e_{k_2} $. 
We can rewrite the total space of  \eqref{pullb1} as
\begin{equation}\label{pi1p2}
\pi_1^*V(k_2)_T = \big\{ ([z_1,y],[z_2,y])\ \big|\ [z_1,y]\in S(k_1)_T,\ [z_2,y]\in V(k_2)_T,\ y\in ET \big\}. 
\end{equation}
We infer that the sphere sub-bundle $\sph(\pi_1^*V(k_2)_T)$   is homeomorphic to  $ S(k_1,k_2)_T$  via the homeomorphism
\begin{equation*}
\sph(\pi_1^*V(k_2)_T)  
\toup^{\cong} 
 S(k_1,k_2)_T  \, ,
 \qquad
 ([z_1,y],[z_2,y]) \longmapsto 
 [(z_1,z_2),y] \, .
\end{equation*}
We claim that the homomorphism 
(which  enters in the Gysin sequence of the sphere bundle)
\begin{equation}\label{isomG}
H^{*-2n_2}_T(S(k_1);\Q) \to H^{*}_T(S(k_1);\Q) \, , \ 
p\mapsto \pi_1^*e_{k_2}\smallsmile p \, ,  \quad \text{is injective} \, . 
\end{equation} 
Indeed, by Lemma \ref{p:cohomology_sphere}, we have the  commutative diagram
\begin{equation}\label{perora}
 {\begin{tikzcd}
{H^*_T(S(k_1);\Q)} \arrow[r,"{\pi_1^*e_{k_2}\smallsmile} "]           & {H^{*+2n_2}_T(S(k_1);\Q) }           \\
{\frac{\Q[u_1,\dots, u_d]}{(e_{k_1})}} \arrow[u,"\cong"] \arrow[r,"e_{k_2}\smallsmile"] & {\frac{\Q[u_1,\dots, u_d]}{(e_{k_1})}} \arrow[u,"\cong"]
\end{tikzcd}}
\text{where} \quad
2n_2-1=\dim(S(k_2)) \, , 
\end{equation}
and it is sufficient to show that the map $ e_{k_2} \smallsmile $ in the last row is injective. 
Note that the non-collinearity assumption  \eqref{eq:non-coll} implies that
the polynomials  
\begin{equation}\label{indep}
e_{k_1},e_{k_2} \in\Q[u_1, \ldots, u_d] \ \text{in \eqref{ekj} are 
relatively prime} \, , 
\end{equation} 
namely, for each $h_1\in\{1,...,n_1\}$ and $h_2\in\{1,...,n_2\}$, their irreducible factors $e_{k_1,h_1},e_{k_2,h_2}$ are not 
rational multiple of each other.
If $ e_{k_2} \smallsmile v+ (e_{k_1})= (e_{k_1})$
for some $ v \in 
\Q[u_1,\dots, u_d] $ then 
$ e_{k_2} \smallsmile v 
= e_{k_1} \smallsmile r $
for some $ r \in 
\Q[u_1,\dots, u_d] $
and, in view of \eqref{indep},
$ e_{k_1} $ divides $ v $,
so $ v \in (e_{k_1})$.
This proves  \eqref{isomG}. 
The Gysin sequence 
\eqref{eq:gysin} of the sphere bundle $S(k_1,k_2)_T \to S(k_1)_T: [(z_1,z_2),y]\mapsto [z_1,y]$, \eqref{isomG}
and  the diagram \eqref{perora}, 
implies the short exact sequence
\begin{align*}
    0\tttoup \frac{\Q[u_1,\dots, u_d]}{(e_{k_1})} \tttoup^{e_{k_2}\smallsmile} \frac{\Q[u_1,\dots, u_d]}{(e_{k_1})} \tttoup^{P^*} H^*_T(S(k_1,k_2);\Q) \tttoup 0 \, . 
\end{align*}
So $ H^*_T(S(k_1,k_2);\Q)\cong \frac{\Q[u_1,\dots, u_d]}{(e_{k_1})}\Big/\ker(P^*)=\frac{\Q[u_1,\dots, u_d]}{(e_{k_1})}\Big/\frac{(e_{k_1},e_{k_2})}{(e_{k_1})}\cong \frac{\Q[u_1,\dots, u_d]}{(e_{k_1},e_{k_2})}$ proving \eqref{prop323}.
\end{proof}

\begin{corollary}\label{cor:nonzero cohomology class}
We define the class $ f $ of $H^{2(n_1+n_2-2)}(B\T^d;\Q) $ as 
\begin{align}\label{eq:f cup product ei}
 f
 :=
 e_{k_1,1}\smallsmile...\smallsmile e_{k_1,n_1-1} \smallsmile
 e_{k_2,1}\smallsmile...\smallsmile e_{k_2,n_2-1}  \, .
\end{align}
Under the non-collinearity assumption \eqref{eq:non-coll},
\[
\pi_{S(k_1,k_2)}^*f\neq0 \, , 
\quad  i.e. \ f \notin (H^*_{\T^d}(S(k_1,k_2)))_\ann  \, , 
\]
where  
$\pi_{S(k_1,k_2)}:S(k_1,k_2)_T\to BT$ is the canonical projection.
\end{corollary}

\begin{proof}
In view of \eqref{prop323}, 
we have to show that $f\not\in(e_{k_1},e_{k_2})$.
Assume by contradiction  that
\begin{equation}
\label{e:f_lin_comb}
f=p_1\smallsmile e_{k_1}+p_2\smallsmile e_{k_2} 
\qquad \text{for some} \ \   p_1,p_2\in\Q[u_1,...,u_d] \, .
\end{equation}
We factorize each polynomial 
in \eqref{ekj}
as 
$e_{k_j}=f_j\smallsmile e_{k_j,n_j}$
where 
$ f_j:=e_{k_j,1}\smallsmile...\smallsmile e_{k_j,n_j-1} $.  
Note that  $f=f_1\smallsmile f_2$.  
By \eqref{e:f_lin_comb} 
and \eqref{indep} we get that $f_1$ divides $p_2$, and $f_2$ divides $p_1$, i.e.
$ p_2=f_1\smallsmile q_2$, $ 
p_1=f_2\smallsmile q_1 $ 
for some $q_1,q_2\in\Q[u_1,...,u_d]$. By dividing both sides of \eqref{e:f_lin_comb} by $f$, we obtain
\begin{align*}
1= q_1\smallsmile e_{k_{1},n_1}+q_2\smallsmile e_{k_{2},n_2} \, . 
\end{align*}
This is a contradiction, since the right-hand side does not have terms of degree 0.
\end{proof}

\paragraph{\bf Join product with a sphere.}
The \emph{join} of two topological spaces $X$ and $Y$ is the quotient space
\begin{equation}\label{defjoin}
X \star Y
:=
\frac{X\times Y\times[0,1]}\sim \, ,
\end{equation}
 where the equivalence relation $\sim$ collapses $X\times Y\times\{0\}$ to  $X$,  and $X\times Y\times\{1\}$  to $Y$, i.e.
\begin{equation}
\label{eqrela}
(x,y,0)\sim(x,y',0) \, ,\
(x,y,1)\sim(x',y,1) \, ,\quad
\forall x,x'\in X \, ,\ y,y'\in Y \, .
\end{equation}
If $X$ and $Y$ are $G$-spaces, their join is also a $G$-space with the action
$ g\cdot [x,y,t]=[g\cdot x,g\cdot y,t] $, for any $
 [x,y,t]\in X \star Y $, 
 $ g\in G $.
 From here on, we remove the coefficient field from the notation, and just write $H^*(\,\cdot\,)$ for $H^*(\,\cdot\,;\Q)$. The next result will be used in the proof of  Theorem \ref{teo:ast}.

\begin{lemma}\label{lm:join}
Let $G$ be a compact connected Lie group, $X$ a $G$-space, and $S^n$ a sphere with a $G$ action. If $H^*_G(X)_\free=H^*_G(X \star S^n)_\free=0$ then
$ H^*_G(X\star S^n)_\ann 
= 
H^*_G(X)_\ann \smallsmile H^*_G(S^n)_\ann $. 
\end{lemma}

\begin{proof}
We consider the cone of $X$, which is the join product
$CX:=X \star \{q\}$
where $\{q\}$ is a singleton regarded  as a trivial $G$-space. The space $X$ can be seen as a $G$-invariant subspace of $CX$ via the continuous inclusion $X\hookrightarrow CX$, $x\mapsto [(x,q,0)]$. 
\\[1mm]
{\sc Step 1.}
{\it There is an exact sequence 
\begin{equation}\label{eq:MVS CXSn}
\begin{aligned}
...
\ttoup
H^*_G(CX \star S^n,X &\star S^n)
\ttoup^{i^*}
H^*_G(CX,X)\ttoup^{p^*}
H^*_G(CX\times S^n,X\times S^n)
\ttoup
...
\end{aligned}
\end{equation}
where $i:(CX,X)\hookrightarrow(CX\star S^n,X \star S^n)$ denotes the inclusion and 
\begin{equation}\label{eq:p def}
p:((CX\times S^n)_G,(CX\times S^n)_G)\to ((CX)_G,(X)_G) \, ,\quad p([z,(x,w)]) :=[z,x] \, .
\end{equation}}

We decompose the join $CX \star S^n$ as $U\cup V$, where $U$ and $V$ are the open subsets 
\[U=(CX\star S^n)\setminus(CX\times S^n\times\{0\}) \, ,
\ V=(CX\star S^n)\setminus(CX\times S^n\times\{1\}) \, .
\] 
We also set $U':=U\cap (X \star S^n)$ and $V':=V\cap (X \star S^n)$. By collapsing the interval factors, the pairs $(U,U')$, $(V,V')$, and $(U\cap V,U'\cap V')$ are $G$-equivariantly homotopy equivalent to $(CX,X)$, $(S^n,S^n)$, and $(CX\times S^n,X\times S^n)$ respectively.
Note that $H^*_G(V,V')\cong H^*_G(S^n,S^n)$ is trivial. The Mayer-Vietoris exact sequence associated with the decomposition 
$ (CX \star S^n,X \star S^n)=(U\cup V,U'\cup V') $ 
reduces to \eqref{eq:MVS CXSn}.
\\[1mm]
{\sc Step 2.} 
% Regarding  $\{q\}$ as a $G$-invariant subspace of $CX$ and 
% $CX \star S^n$ via the 
{\it Considering the inclusions $j:\{q\}\hookrightarrow (CX,X)$ and $k:\{q\}\hookrightarrow (CX \star S^n,X\star S^n)$ we have 
the  commutative diagram}
\begin{equation}\label{diagk}
\begin{tikzcd}[row sep=large]
 H^*_G(CX\star S^n,X \star S^n)
 \arrow[rr, "i^*"]
 \arrow[d, twoheadrightarrow, "k^*"']
 & & 
 H^*_G(CX,X)
 \arrow[d, twoheadrightarrow, "j^*"]
 \\
 H^*_G(X\star S^n)_\ann
 \arrow[rr, hookrightarrow]
 & 
 & 
 H^*_G(X)_\ann
\end{tikzcd}
\end{equation}
{\it where} 
$j^*$ and $k^*$
{\it are surjectives
and the bottom horizontal arrow is an inclusion}.

 As usual, we identify $H^*_G(\{q\})$ with $H^*(BG)$ via the homeomorphism $\pi_q:\{q\}_G\to BG$. 
Since the cone $CX$ is $G$-equivariantly contractible, the cohomology $H^*_G(CX,\{q\})$ is trivial, and the inclusion $X\hookrightarrow CX$ is $G$-equivariantly homotopic to the constant map to $q$. Analogously, the join $CX \star S^n$ is $G$-equivariantly contractible, and therefore $H^*_G(CX \star S^n,\{q\})$ is trivial, and the inclusion $X \star S^n\hookrightarrow CX \star S^n$ is $G$-equivariantly homotopic to the constant map to $q$. Furthermore $H^*_G(X)_\free=H^*_G(X \star S^n)_\free=0$ and by Lemma~\ref{l:free_ann_technical}, the inclusions $j$ and $k$ induce surjective homomorphisms.
The bottom horizontal arrow is an inclusion, which follows by the inclusion $X\hookrightarrow X \star S^n$ and the monotonicity of the annihilator (Proposition~\ref{p:ann}($i$)).
\\[1mm]
{\sc Step 3.}
The diagram  \eqref{diagk} and 
the exact sequence \eqref{eq:MVS CXSn} imply that 
\begin{equation}\label{eq:First part}
    H^*_G(X \star S^n)_\ann
=
\image(k^*)
=
\image(j^*i^*)
=
j^*(\ker(p^*)) \, . 
\end{equation}
{\sc Step 4.} 
% \label{e:join_prop_kerp}
$ \ker(p^*)
=
(\pi^*_{CX}e)\smallsmile H_G^*(CX,X) $ 
{\it where 
$e\in H^{n+1}(BG)$ is 
the Euler class  of the oriented sphere bundle 
$\pi_{S^n}:S^n_G\to BG $.}

Note that the oriented sphere bundle $\breve p:(CX\times S^n)_G\to (CX)_G:[z,(x,w)]\mapsto [z,x]$ is the pull-back of the sphere bundle $\pi_{S^n}:S^n_G\to BG$ via the canonical projection $\pi_{CX}:CX_G\to BG:[z,x]\mapsto [z]$. By property (i) its Euler class is therefore  $\pi_{CX}^*e$.
As an oriented sphere bundle pair 
$  p:((CX\times S^n)_G,(X\times S^n)_G)\to ((CX)_G,(X)_G) $, 
its associated Gysin exact sequence (Theorem \ref{t:Gysin}) reads
\begin{align*}
...
\ttoup
H^*_G(CX,X)
\ttoup^{\pi^*_{CX}e \smallsmile}
H^{*+n+1}_G(CX,X)\ttoup^{ p^*}
H^{*+n+1}_G(CX\times S^n,X\times S^n)
\ttoup
... 
\end{align*}
which implies $ \ker(p^*)
=
(\pi^*_{CX}e)\smallsmile H_G^*(CX,X) $.
\\[1mm]
{\sc Conclusion.} By 
\eqref{eq:First part} 
and Step 4  
we have  
\begin{equation}\label{Hinte1}
\begin{aligned}
H^*_G(X \star S^n)_\ann 
= j^*(\ker(p))=j^*((\pi_{CX}^*e)\smallsmile H^*_G(CX,X)) \\
= \breve \jmath^*(\pi_{CX}^*e)\smallsmile j^*(H^*_G(CX,X))
\end{aligned}
\end{equation}
where $\breve \jmath:\{q\}\hookrightarrow CX$.
Note that  $
    \breve \jmath^*\pi_{CX}^*e=e $, so  
% \eqref{annsfera}, % \eqref{e:join_prop_kerp}, 
\begin{align*}
\breve \jmath^*(\pi_{CX}^*e)\smallsmile j^*(H^*_G(CX,X))
%j^*&(\ker(p^*))
%\stackrel{\eqref{e:join_prop_kerp}} =
%j^*\big((\pi_{CX}^*e)\smallsmile H^*_G(CX,X)\big)\
\stackrel{\eqref{diagk}} =
e\smallsmile H^*_G(X)_\ann 
\stackrel{\eqref{annsfera}} =
H_G^*(S^n)_\ann\smallsmile H^*_G(X)_\ann 
\end{align*}
and by \eqref{Hinte1} the lemma is proved.  
\end{proof}

    \section{Equivariant Morse-Conley theory}\label{s:equivariant_Morse_theory}

Let $G$ be a compact Lie group and $X$ a {\it compact} metric $G$-space, namely a  metric space equipped with a continuous $G$ action of isometries.
A $G$-equivariant flow on $X$ is an $\R$ action that commutes with the $G$ action. We see this action as a family $\phi=(\phi^t)_{t\in\R}$ of homeomorphisms  $\phi^t:X\to X$, depending continuously on time  $t\in\R$, such that $\phi^0=\id$ and $\phi^{s+t}=\phi^s\circ\phi^t$ for any $s,t\in\R$. We denote the set of stationary points of $\phi$ by
\begin{equation*}
\fix(\phi)
:=
\big\{
x\in X\ |\ \phi^t(x)=x \, ,\ \forall t\in\R
\big\}.
\end{equation*}
The set $\fix(\phi)$ is $G$-invariant and compact. The alpha-limit and the omega-limit of the $\phi$-orbit through a point $x\in X$ are the compact, not-empty, connected $\phi$-invariant sets 
\begin{equation*}%\label{alphaomega} 
\alpha(x) :=\bigcap_{t>0}\overline{\phi^{(-\infty,-t]}(x)} \, ,
\qquad
\omega(x) := 
\bigcap_{t>0}\overline{\phi^{[t,\infty)}(x)}\, .
\end{equation*}
Equivalently, $\alpha(x)$ and $\omega(x)$ are the limit sets of $\phi^t(x)$ as $t\to
\mp \infty$ respectively, i.e.
\begin{align*}
\alpha (x) 
&= 
\big\{ y \in X \ | \
\exists\ t_n 
\to + \infty \ \mbox{such that }
\phi^{-t_n}(x) \to y \big\} \, ,\\
\omega (x) 
&= 
\big\{ y \in X \ | \
\exists\ t_n 
\to + \infty \ \mbox{such that }
\phi^{t_n}(x) \to y \big\} \, .
\end{align*}
The set  $\alpha(x)$, resp.\  $ \omega(x)$, is the smallest compact set $K\subseteq X$ such that, for any neighborhood 
$U$ of $K$, there exists $ t_U\in\R$ such that $\phi^t (x) \in U$ for all $t<t_U $, resp.\ for all $ t>t_U $.

\begin{definition}
{\bf (Lyapunov function and gradient-like flow)}
A $G$-invariant continuous function 
$F:X\to\R$ 
is a \emph{strict Lyapunov function} if 
\begin{equation}\label{Fphi}
\begin{aligned}
&F(\phi^t(x))\leq F(x) \, , \, \forall 
x\in X, 
t\geq 0,  \,  \\
&\text{with equality}  \,  F(\phi^t(x))=F(x) \ \text{if and only if} \ \phi^t(x)=x \, .
\end{aligned}
\end{equation} 
We say that $\phi$ is a \emph{gradient-like flow} for $F$.
\end{definition}

 We say that $c\in\R$ is a \emph{critical value} of $F$ if $F^{-1}(c)\cap\fix(\phi)\neq\emptyset$. A real number that is not a critical value is called a \emph{regular value} of $F$. Note that these notions depend on both $F$ and $\phi$.

\begin{remark}
If the $G$-space $X$ is a  manifold equipped with a $G$-invariant Riemannian metric, then any $G$-invariant smooth function $F$ is a strict Lyapunov function with respect to  the negative gradient flow $\phi^t$ of $F$. In this case, the notions of critical and regular values are the usual ones from analysis. 
\end{remark}

For any $x\in X$, since the function $t\mapsto F(\phi^t(x))$ is monotone decreasing, 
by  \eqref{Fphi}, and 
recalling that 
$ \omega (x)$ and $\alpha(x) $ are $ \phi $-invariant,  
there exist critical values  $a\leq b$ such that 
\begin{equation}\label{aeofix}
\omega(x)\subseteq F^{-1}(a) \, , \ 
\alpha(x)\subseteq F^{-1}(b) \quad \text{and} \quad 
\alpha(x)\cup\omega(x)\subseteq\fix(\phi) \, .
\end{equation}
{\bf Notation.} For any subset $Y\subseteq X$ and $c\in\R$, we denote
$ Y^{<c} :=Y\cap F^{-1}(-\infty,c) $, 
$
    Y^{\leq c}:=Y\cap F^{-1}(-\infty,c] $, 
    $ Y^{>c} :=Y\cap F^{-1}(c,\infty) $
    and 
$ Y^{\geq c}:=Y\cap F^{-1}[c,\infty) $. 
Since $ F $ is $ G $-invariant, if $Y$ is $G$-invariant, so are 
these four subsets.

\paragraph{\bf Isolating blocks.} 
For any $W\subseteq X$, we denote by $\Inv(W)$ its maximal $\phi$-invariant subset, i.e.
\begin{align*}
    \Inv(W):=\bigcap_{t\in\R} \phi^t(W)= \big\{ x \in W \ | 
    \ 
    \phi^t(x) \in W \, , \ 
    \forall t \in \R \big\} \, .
\end{align*}
%If $W$ is closed then
% $\Inv(W)$ is closed as well. 
We assume in the sequel that $W$ is compact and so 
 $\Inv(W)$ is compact as well. 

In view of 
\eqref{aeofix}, if $\Inv(W)$ is non-empty, it contains the non-empty $\phi$-invariant set  $\fix(\phi)\cap W$
and, if $F(W)$ contains at most one critical value, then $\Inv(W)=\fix(\phi)\cap W$. If instead $F(W)$ contains at least two critical values, $W$ may also contain non-stationary $\phi$-orbits with alpha-limit and omega-limit contained in  $W\cap\fix(\phi)$.

The exit set of $W$ with respect to the gradient-like flow 
$ \phi $ is 
\begin{equation*}
W_-:= \! \big\{ x\in W\ |\ \phi^{[0,t)}(x)\not\subseteq W\ \forall t>0 \big\} \! = \! \big\{ x\in \partial W\ |\ \phi^{[0,t)}(x)\not\subseteq W\ \forall t>0 \big\}.
\end{equation*} 
A subset $ S $ of $ X $   is called 
 $(G,\phi)$-invariant if it is both  $G$-invariant and $\phi$-invariant.
\begin{definition}\label{def:isolated isolating isolating}
{\bf (Isolated set, isolating  neighborhood, isolating block)}
A compact $(G,\phi)$-invariant subset $S\subseteq X$ is  \emph{isolated} when it admits an {\em isolating neighborhood} $W\subseteq X$, namely a subset $W\subseteq X $ such that 
$ \Inv(W) = S \subseteq \interior(W)$.
An \emph{isolating block} of $S$ is a $G$-invariant compact isolating neighborhood $W$ whose exit set  satisfies 
\begin{equation}
\label{exitblock}
\overline{W_-}\subseteq W^{<c} 
\ \text{where} \ 
c:=\min F|_{K} 
= \min
F_{|S} \, , \quad  
K:=\fix(\phi)\cap W \, .
\end{equation}
\end{definition}

The equality 
$ \min F|_{K}  = \min
F_{|S}$ follows by \eqref{aeofix}.
Note that the  $ G$-space $ X $ is an isolating block of itself with empty exit set.

\begin{remark}\label{r:split_block}
Let $K:=\fix(\phi)\cap W$ be the stationary set in an isolating block $W$. For each interval $[a,b]\subseteq\R$ with endpoints $a,b\in\R\setminus F(K)$, the intersection $Z:=W\cap F^{-1}[a,b]$ is an isolating block as well with maximal $\phi$-invariant subset
$ \Inv(Z) \subseteq  \Inv(W) \cap F^{-1}(a,b)  \subseteq \interior(Z) $
and exit set 
$ Z_-=(F^{-1}(a)\cap W)\cup (F^{-1}[a,b]\cap W_-) $.  
\end{remark}

An isolated $(G,\phi)$-invariant subset $S$ does not admit an isolating block if there is a point $x\in X\setminus S$ whose  $\alpha(x)$ and $\omega(x)$ limit sets are both contained in $S$. Such an obstruction  never occurs  if $S$ is contained in a level set of the Lyapunov function $F$, and the following holds.

\begin{lemma}
\label{l:existence_isolating_block}
Let $K\subseteq \fix(\phi)\cap F^{-1}(c)$ be a non-empty isolated compact $G$-invariant set.
Then any isolating neighborhood $V$ of $K$ contains an isolating block of $K$.
\end{lemma}

\begin{proof}
Since $ K $ is compact and $G$-invariant, 
we can assume 
with no loss of generality 
that the 
isolating neighborhood $V$
of $K $ is compact and $G$-invariant. 
For any $\tau>0$, we define
the {\it compact}  $G$-invariant  sets
\begin{align}
 V^\tau & :=
\bigcap_{t\in[0,\tau]} \phi^t(V)
= 
\big\{ x \in V \ \big|  \ 
\phi^{[-\tau,0]}(x) \subseteq V \big\}\, ,  \notag \\
 V^\infty & :=\bigcap_{t\in[0,\infty)} \phi^t(V) = 
\big\{ x \in V \ \big|  \ 
\phi^{(-\infty,0]}(x) \subseteq V \big\}  \, . \label{Vinfty}
\end{align}
Clearly $ K \subseteq V^\infty \subseteq V^\tau $ and, 
by the continuity of the flow, each  $V^\tau$ is a  {\it neighborhood} of $ K $ in $ V $. 
Moreover 
 $ K \subseteq \Inv(V^\tau)  \subseteq \Inv(V) = K $  
 and 
 $ V^\tau$ is an {\it isolating neighborhood of} $ K $
 for any $ \tau > 0 $.

We now prove that for any $ \tau $  sufficiently large  $V^\tau$ is an isolating block of $ K $.
If $V^\tau_-=\emptyset $ then $V^\tau$ is an isolating block, as 
\eqref{exitblock} trivially holds.
Assume now that $V^\tau_-\neq\emptyset$.

We first claim that the compact subset $Z:=\partial V\cap 
V^{\geq c} $ has no points in  $ V^\infty $, namely  
\begin{equation}
\label{e:Z_isolating_block}
Z\cap V^\infty=\emptyset \, .  
\end{equation}
Indeed, since no point of $Z$ is on a stationary $\phi$-orbit (as $V$ is an isolating block), for any $x\in Z$ we have $F(\phi^{-t}(x))>c$ for all $t>0$. 
 If there were a point $x\in Z\cap V^\infty$, its alpha-limit $\alpha(x)$ would be contained in $ V^{>c}$. This  is impossible because,  %$\alpha(x)$ is contained in $ K \subseteq F^{-1}(c) $.  Indeed 
 for any 
 $ x\in V^\infty $, 
 $ \alpha(x) 
$ is  non-empty and, by 
\eqref{aeofix}, 
$ \alpha(x)\subseteq V \cap \fix (\phi) \subseteq \Inv(V) = K \subseteq F^{-1}(c) $, having used that $V$ is an isolating neighborhood of $ K $. 

By \eqref{e:Z_isolating_block} and \eqref{Vinfty}, for any $x\in Z$ there is a time $\tau_x>0$ such that $\phi^{-\tau_x}(x)\not\in V$. Therefore, by the continuity of the gradient-like flow, there is an open neighborhood $U_x\subseteq X$ of $x$ such that $\phi^{-\tau_x}(U_x)\cap V=\emptyset$. Since $Z$ is compact, it is contained in a finite union  $U_{x_1}\cup...\cup U_{x_n}$. 
For any $\tau\geq \max\{\tau_{x_1}, ...,\tau_{x_n}\}$, if 
$ x \in  Z $ then $ \phi^{-\tau} (x) \notin V $ and   
we conclude that
\begin{equation}
\label{ciovol}
Z\cap V^\tau=\emptyset \, ,
\quad \forall 
\tau\geq \max\{\tau_{x_1}, ...,\tau_{x_n}\} \quad 
\text{in particular} 
\quad \overline{V^\tau_-} \cap Z = \emptyset \, . 
\end{equation}
Since 
$ \overline{V^\tau_-}\subseteq\partial V $
we deduce by \eqref{ciovol} that 
$ \overline{V^\tau_-}
\subseteq
\partial V\setminus Z\subseteq V^{<c} $, 
proving that $V^\tau$ is an isolating block of $K$.
% cf. \eqref{exitblock}.
\end{proof}

\subsection{Local cohomology}\label{s:local cohomoloy}

The following key notion originated in the 
% seminal 
work of Morse \cite{Morse:1996aa} and it is also referred to  as the equivariant cohomological Conley index, or as the equivariant critical module of $S$.

\begin{definition}
{\bf (Local cohomology)}
Let $S$ be an isolated $(G,\phi)$-invariant set with an isolating block $ W$. Let $K:=\fix(\phi)\cap S$ be the stationary set therein and $c:=\min F|_K=\min F|_{S}$ the corresponding minimal critical value. The $G$-equivariant \emph{local cohomology} of $S$ is defined as
\begin{align*}
    I^*(S)
    :=
    H^*_G(W,W^{<c}) \, . 
\end{align*}
\end{definition}
 As the notation suggests, $I^*(S)$ is {\it independent} of the choice of the isolating block $W$. This is a consequence of the following lemma.

\begin{lemma}\label{l:local_cohomology_well_posed}
Let $S$ be an isolated $(G,\phi)$-invariant set admitting two isolating blocks $V$ and $W$. Let $c:=\min F|_S$. Then the intersection $Z:=V\cap W$ is also an isolating block of $S$, and the inclusions induce isomorphisms
\begin{equation}
\label{isom3}
H^*_G(V,V^{<c}) \ttoup^{\cong}
H^*_G(Z,Z^{<c})
\leftttoup^{\cong}
H^*_G(W,W^{<c}) \, .
\end{equation}
\end{lemma}

\begin{proof}
The intersection $Z=V\cap W$ is a compact $G$-invariant neighborhood of $S$, and 
$ S = \Inv(Z) \subseteq \interior(Z)$. Thus 
$ Z $ 
is an isolating neighborhood of $ S $. Furthermore
the exit set of $ Z $ satisfies
\[
\overline{Z_-} \, 
\subseteq \, 
Z\cap
\big( \overline{V_-}\cup \overline{W_-} \big) 
 \subseteq 
X^{<c}
\]
since both $ \overline{V_-} $ and
$ \overline{W_-} $ satisfy \eqref{exitblock}.
Therefore $Z$ is an isolating block of $S$ as well.  

Since $\overline{Z_-\cup W_-}\subseteq X^{<c}$, there is $\epsilon>0$ such that 
$ \overline{Z_-}\cup \overline{W_-}\subseteq X^{<c-\epsilon} $. 
We define the compact set 
\[
Y:=W^{\geq c-\epsilon}\setminus\interior(Z) \, ,
\] 
whose exit set satisfies $Y_-\subseteq F^{-1}(c-\epsilon)\cup\partial Z$. Since 
$ Y $ does not contain the invariant set $ S = \Inv(W) \subseteq  \interior(Z) $, we have  
$\Inv(Y)=\emptyset$ and, for any  $x\in Y$ there is a time $\tau_x>0$ such that 
$\phi^{\tau_x}(x)\notin Y $, thus 
$\phi^{\tau_x}(x)\in X^{<c-\epsilon}\cup \interior(Z)$. By the continuity of the flow $\phi$, there is an open neighborhood $U_x$ of $x$ such that $\phi^{\tau_x}(U_x)\subseteq X^{<c-\epsilon}\cup \interior(Z)$. 
Since the exit 
set $ Z_-
\subseteq X^{<c-\epsilon}$, the set 
$X^{<c-\epsilon}\cup Z$ is preserved by the gradient-like flow $\phi^t$ for any  $t\geq0$, and therefore $\phi^{t}(U_x)\subseteq X^{<c-\epsilon}\cup Z$ for all $t\geq\tau_x$. Since $Y$ is compact, it is contained in a finite union $U_{x_1}\cup...\cup U_{x_n}$ and we conclude that  
\begin{equation}\label{phincl}
    \phi^\tau(Y)\subseteq Z \cup 
    X^{<c-\epsilon}
    \qquad \text{where}
    \qquad 
\tau:=\max\{\tau_{x_1},...,\tau_{x_n}\} \, . 
\end{equation}
The inclusion of pairs 
\begin{equation}
\label{inclp}
i:(Z\cup X^{<c-\epsilon},Z^{<c}\cup X^{<c-\epsilon})\hookrightarrow(W\cup X^{<c-\epsilon},W^{<c}\cup X^{<c-\epsilon})
\end{equation}
is a $G$-equivariant homotopy equivalence, with homotopic inverse
\[
\psi:=\phi^\tau: \big( W\cup X^{<c-\epsilon},W^{<c}\cup X^{<c-\epsilon} \big) \to \big( Z\cup X^{<c-\epsilon},Z^{<c}\cup X^{<c-\epsilon} \big) \, ,
\] 
which is a map between pairs thanks to 
\eqref{phincl} and since $X^{<c-\epsilon}\cup Z$
and $ Z^{<c} \cup X^{<c-\epsilon} $ 
are preserved by the gradient-like flow $\phi^t$ for any  $t\geq0 $.
Indeed, both compositions $\psi\circ i$ and $i\circ\psi$ are homotopic to the identity through the corresponding restrictions of the homotopy $h_s:X\to X$, $h_s(x) :=\phi^{s\tau}(x)$, for $s\in[0,1]$. The inclusion $i$ in \eqref{inclp} induces the cohomology isomorphism
\begin{equation}
\label{ipairi}
i^*:
H_G^*(W\cup X^{<c-\epsilon},W^{<c}\cup X^{<c-\epsilon})
\toup^{\cong}
H_G^*(Z\cup X^{<c-\epsilon},Z^{<c}\cup X^{<c-\epsilon}) \, .
\end{equation}
By excision of 
the sets $ X^{<c-\epsilon} \setminus W $ and $  X^{<c-\epsilon} \setminus Z $ respectively,  the inclusion induces isomorphisms
\begin{equation}
\label{exc1}
\begin{aligned}
H_G^*(W\cup X^{<c-\epsilon},W^{<c}\cup X^{<c-\epsilon})
&\toup^{\cong}
H^*_G(W,W^{<c}) \, ,
\\
H_G^*(Z\cup X^{<c-\epsilon},Z^{<c}\cup X^{<c-\epsilon})
&\toup^{\cong}
H^*_G(Z,Z^{<c}) \, .
\end{aligned}
\end{equation}
The isomorphisms 
\eqref{ipairi} and \eqref{exc1} fit into the following commutative diagram, where 
$ j : (Z,Z^{<c}) \to (W,W^{<c})$ is the inclusion,  
and all homomorphisms are induced by inclusions, 
\begin{equation*}
\begin{tikzcd}[row sep=large]
H^*_G(W,W^{<c})
\arrow[rr,"j^*"] 
& &
H^*_G(Z,Z^{<c})
\\
H_G^*(W\cup X^{<c-\epsilon},W^{<c}\cup X^{<c-\epsilon})
\arrow[rr,"i^*","\cong"'] 
\arrow[u,"\cong"] 
& &
H_G^*(Z\cup X^{<c-\epsilon},Z^{<c}\cup X^{<c-\epsilon})
\arrow[u,"\cong"'] 
\end{tikzcd}
\end{equation*}
proving  that $j^*$ is an isomorphism. The same holds for the isolating block $ V $ and 
\eqref{isom3} is proved. 
\end{proof}

In order to study the local cohomology, we  need the following version of the classical deformation argument from Morse theory.

\begin{lemma}
\label{l:deformation}
Let $W$ be an isolating block, $K=\fix(\phi)\cap W$ the corresponding stationary set, $[b,c)\subseteq\R\setminus F(K)$ an interval, and $a\in(-\infty,b]$ a real number such that $\overline{W_-}\subseteq W^{<a}$. Then the  inclusion induces the cohomology isomorphism
\begin{equation}
\label{lemdef2}
    H^*_G(W^{<c},W^{<a})
    \ttoup^{\cong}
    H^*_G(W^{\leq b},W^{<a}) \, .
\end{equation}
\end{lemma}

\begin{proof}
The argument is slightly more involved that the one in the proof of  Lemma~\ref{l:local_cohomology_well_posed}, since $ c $ may be a critical level. We consider a monotone increasing sequence $b:=c_0<c_1<c_2<...$ such that $c_n\to c$. The family $Z_n:=W\cap F^{-1}[b,c_{n+1}]$, for $n\in\N$, is an exhaustion by compact sets of $Z:=W\cap F^{-1}[b,c)$. Since $[b,c_{n+1}] $ are all regular values  $ \Inv(Z_n)=\emptyset$, and there is a monotone increasing sequence times $0<\tau_1\leq\tau_2\leq\tau_3<...$ such that 
\begin{equation}\label{phitaun}
    \phi^{\tau_n}(Z_n)\subseteq X^{\leq b},\qquad\forall n\in\N \, .
\end{equation}
Let $ q :(-\infty,c)\to(0,\infty)$ be a monotone increasing continuous positive function such that $ q (c_n)\geq\tau_n$ for all $n\in\N$. We define the continuous function $\tau:X^{<c}\to[0,\infty)$, $\tau(x) := q (F(x))$, and the $G$-equivariant continuous map $\psi:X^{<c}\to X^{<c}$, $\psi(x) := \phi^{\tau(x)}(x)$. In view of \eqref{phitaun}, 
\begin{equation}\label{}
    \psi(Z)\subseteq X^{\leq b}.
\end{equation}
Since $\overline{W_-}\subseteq W^{<a}$, there is $\epsilon>0$ so that $\overline{W_-}\subseteq W^{<a-\epsilon}$.
We claim that the inclusion 
\begin{equation}
\label{inc4}
i:
(W^{\leq b}\cup X^{<a-\epsilon},W^{<a}\cup X^{<a-\epsilon})
\hookrightarrow
(W^{<c}\cup X^{<a-\epsilon},W^{<a}\cup X^{<a-\epsilon})
\end{equation}
is a $G$-equivariant {\it homotopy equivalence}, with homotopy inverse 
\begin{equation}
\label{psiinv}
\psi:
(W^{<c}\cup X^{<a-\epsilon},W^{<a}\cup X^{<a-\epsilon})
\to
(W^{\leq b}\cup X^{<a-\epsilon},W^{<a}\cup X^{<a-\epsilon}) \, . 
\end{equation}
Let us first check that 
$ \psi 
$ is a map between  the topological pairs  \eqref{psiinv}.  
If $ x \in W^{<c}$ then
$ c_n \leq F(x) < c_{n+1} $ for some $ n \in \N $
and so
$ \tau_n \leq 
q(c_n) \leq q(F(x))
= \tau(x)$ implying that 
$  \phi^{\tau(x)} (x) \in X^{\leq b }$ by \eqref{phitaun}. 
Furthermore, since 
the exit set $ W_- 
\subseteq W^{<a-\epsilon}$ we conclude that  $ \psi (x) = \phi^{\tau(x)} (x) \in W^{\leq b }
\cup X^{<a-\epsilon}$. 
Similarly, if $ x \in W^{<a} $
then 
$ \phi^{\tau(x)} (x) \in W^{<a} \cup X^{<a-\epsilon} $.
The compositions $\psi\circ i$ and $i\circ\psi$ are homotopic to the identity by means of suitable restrictions of the homotopy $h_s:X^{<c}\to X^{<c}$, $h_s(x) := \phi^{s\tau(x)}(x)$, for $s\in[0,1]$.

%We now conclude as in the proof of Lemma~\ref{l:local_cohomology_well_posed}. 
The inclusion $i$ 
in \eqref{inc4} 
%is a $G$-equivariant homotopy equivalence, it 
induces the cohomology isomorphism
\begin{equation}\label{priso}\small
i^*:H_G^*(W^{<c}\cup X^{<a-\epsilon},W^{<a}\cup X^{<a-\epsilon})\toup^{\cong}H_G^*(W^{\leq b}\cup X^{<a-\epsilon},W^{<a}\cup X^{<a-\epsilon}) \, ,
\end{equation}
and, by excision of the set $ X^{<a-\epsilon} \setminus  W $, the inclusions induce isomorphisms
\begin{equation}
\label{idueiso}
\begin{aligned}
H_G^*(W^{<c}\cup X^{<a-\epsilon},W^{<a}\cup X^{<a-\epsilon})
&\toup^{\cong}
H^*_G(W^{<c},W^{<a}) \, ,
\\
H_G^*(W^{\leq b}\cup X^{<a-\epsilon},W^{<a}\cup X^{<a-\epsilon})
&\toup^{\cong}
H^*_G(W^{\leq b},W^{<a}) \, .
\end{aligned}
\end{equation}
The isomorphisms \eqref{priso}
and \eqref{idueiso} fit into the commutative diagram induced by inclusions
\begin{equation*}\small
\begin{tikzcd}[row sep=large, column sep=small]
H^*_G(W^{<c},W^{<a})
\arrow[rr,"j^*"] 
& &
H^*_G(W^{\leq b},W^{<a})
\\
H_G^*(W^{<c}\cup X^{<a-\epsilon},W^{<a}\cup X^{<a-\epsilon})
\arrow[rr,"i^*","\cong"'] 
\arrow[u,"\cong"] 
&&
H_G^*(W^{\leq b}\cup X^{<a-\epsilon},W^{<a}\cup X^{<a-\epsilon})
\arrow[u,"\cong"']
\end{tikzcd}
\end{equation*}
proving \eqref{lemdef2}.
%and conclude that $j^*$ is an isomorphism.
\end{proof}

\begin{lemma}\label{l:long_exact_sequence_local_cohomology}
Let $S$ be an isolated $(G,\phi)$-invariant set with isolating block $W$. Let $K:=\fix(\phi)\cap S$ be the corresponding stationary set, and $b\in\R\setminus F(K)$. We decompose $W$ as a union of isolating blocks $W^{\leq b}$ and $W^{\geq b}$, and  consider the corresponding $(G,\phi)$-invariant sets 
$ S_1:=\Inv(W^{\leq b}) $ and 
$ S_2:=\Inv(W^{\geq b}) $.   
Then there is a long exact sequence of $H^*(BG)$-modules
\begin{align}
\label{e:les_local_cohomology}
...
\ttoup^{\delta^{*-1}}
I^*(S_2)
\ttoup
I^*(S)
\ttoup
I^*(S_1)
\ttoup^{\delta^{*}}
I^{*+1}(S_2)
\ttoup
...
\end{align}
%all of whose homomorphisms except % the connecting ones 
%$\delta^*$ are induced by the inclusions. 
In particular
\begin{equation}\label{incann12}
    I^*(S_1)_\ann \smallsmile I^*(S_2)_\ann \subseteq I^*(S)_\ann \, .
\end{equation}
\end{lemma}

\begin{proof}
The compact sets $W^{\leq b}$ and $W^{\geq b}$ are isolating blocks respectively of $ S_1 $ and $ S_2 $, by Remark \ref{r:split_block}.
We can assume that both $S_1$ and $S_2$ are non-empty, for otherwise the lemma is trivial. 
We set $c_1:=\min F|_{S_1}=\min F|_S$ and $c_2 :=\min F|_{S_2} $, so that $ c_1 < b < c_2 $. The long exact sequence of the triple $W^{<c_1}\subseteq W^{<c_2}\subseteq W$ reads (cf.~Proposition \ref{prop:long})
\begin{equation}\label{e:les_local_cohomology_PROOF}
\begin{tikzcd}[column sep=small]
...
\arrow[r,"\delta^{*-1}"] 
& 
H^*_G(W,W^{<c_2}) 
\arrow[r] 
&
\underbrace{
H^*_G(W,W^{<c_1})}_{=
I^*(S)} 
\arrow[r] 
% \equaldown 
&
H^*_G(W^{<c_2},W^{<c_1})
\arrow[r,"\delta^*"]
&
...  
% \\ && I^*(S)
\end{tikzcd}
\end{equation}
The set 
$ V:=W^{\geq b} $ is an isolating block of $ S_2$ with exit set $ V_- = W \cap F^{-1} (b)$. By excision of the set  $W^{<b}$ (note that 
$ \overline{W^{<b}}  \subset W^{<c_2}$), the inclusion induces an isomorphism
\begin{equation}
\label{liso1}
H^*_G(W,W^{<c_2})
\ttoup^{\cong}
H^*_G(V,V^{<c_2})
=
I^*(S_2) \, .
\end{equation}
The set $W^{\leq b}$ is an isolating block for $S_1$, and Lemma~\ref{l:deformation} implies that the inclusion induces an isomorphism
\begin{equation}
\label{liso2}
H^*_G(W^{<c_2},W^{<c_1})
\ttoup^{\cong}
H^*_G(W^{\leq b},W^{<c_1})
=
I^*(S_1) \, .
\end{equation}
Inserting  \eqref{liso1}-\eqref{liso2} into \eqref{e:les_local_cohomology_PROOF}, we deduce \eqref{e:les_local_cohomology} and  Proposition \ref{p:modules}-($ii$) implies \eqref{incann12}.
\end{proof}

Within a level set of the Lyapunov function, a disjoint union decomposition of a $G$-invariant compact stationary set has the following effect on the local cohomology.

\begin{lemma}
\label{l:direct_sum_local_cohomology}
{\bf (Splitting)}
Let $K_1,K_2\subseteq \fix(\phi)\cap F^{-1}(c)$ be disjoint isolated $G$-invariant compact stationary sets contained in a level set of the Lyapunov function. 
Then the inclusion induces an $H^*(BG)$-module isomorphism
$ I^*(K_1\sqcup K_2)\ttoup^{\cong} I^*(K_1)\oplus I^*(K_2) $. 
In particular
$ I^*(K_1)_\ann\cap I^*(K_2)_\ann=I^*(K_1\sqcup K_2)_\ann $. 
\end{lemma}

\begin{proof}
Since the compact stationary sets $K_1$ and $K_2$ are disjoint, they admit disjoint isolating neighborhoods, and thus disjoint isolating blocks $W_1$ and $W_2$ by  Lemma~\ref{l:existence_isolating_block}. 
The disjoint union $W_1\sqcup W_2$ is an isolating block of $K_1\sqcup K_2$.
The cohomology of a disjoint union of pairs is isomorphic to the direct sum of the cohomology of the pairs with an isomorphism induced by the pair of inclusions. 
Therefore
$$
\begin{aligned}
 I^*(K_1  \sqcup K_2)
    =
 H^*_G(W_1 \sqcup  W_2,W_1^{<c}
  \sqcup  W_2^{<c})&\\
  \cong 
 H^*_G(W_1,W_1^{<c})
 \oplus  H^*_G(W_2,W_2^{<c}) &
 =
 I^*(K_1) \oplus I^*(K_2) 
 \end{aligned}
$$
proving the lemma. 
 \qedhere
\end{proof}

The following lemma is used in 
the proofs of Lemma \ref{l:ann_orbit}
and Theorem \ref{teo:ast}.

\begin{lemma}\label{lm:cor of slice}
Any $G$-orbit $G\cdot x$ has a $G$-invariant open neighborhood $U$ such that 
\[H^*_G(G\cdot x)_\ann=H^*_G(U)_\ann.\]
\end{lemma}

\begin{proof}
The existence theorem of a slice, see e.g. \cite{Bre}[ chapter II Theorem 5.4,  Def. 4.1, Theorem 4.2], 
implies the existence of a $G$-invariant neighborhood $U$ of $G\cdot x$ and a $G$-equivariant retraction $ r:U\to G\cdot x $.
Applying the  monotonicity property of the annihilator in
Proposition \ref{p:ann}($i$)
to the map $ r $ and the inclusion $i:G\cdot x\hookrightarrow U $  
we conclude that 
$ H^*_G(G\cdot x)_\ann \subseteq H^*_G(U)_\ann$
and 
$ H^*_G(U)_\ann \subseteq H^*_G(G\cdot x)_\ann  $.
\end{proof}  

The computation of the local cohomology of an isolated 
$\phi$-invariant set 
may be complicated. 
For our applications, it will be sufficient to single out 
an $H^*(BG)$-submodule of the local cohomology.

\begin{lemma}\label{l:ann_orbit}
If $G\cdot x\subseteq \fix(\phi)$ is an isolated $(G,\phi)$-invariant set, 
then $ H^*_G(G\cdot x)_\ann \subseteq I^*(G\cdot x)_\ann $.
\end{lemma}

\begin{proof}
By Lemma \ref{lm:cor of slice}
there is a $G$-invariant open neighborhood  $U$ of the orbit $G\cdot x$ such that $H^*_G(G\cdot x)_\ann=H^*_G(U)_\ann $.
Any sufficiently small neighborhood $V\subset U$ of  $G\cdot x $ is an isolating neighborhood, and contains an isolating block $W \subset V \subset U $ of $G\cdot x$ according to Lemma \ref{l:existence_isolating_block}.
The monotonicity property of the annihilator (cf. Proposition \ref{p:ann}-($i$)) implies that 
\[
    H^*_G(G\cdot x)_\ann\subseteq H^*_G(W)_\ann\subseteq H^*_G(U)_\ann=H^*_G(G\cdot x)_\ann\, .    
\]
The  local cohomology of $G\cdot x$ is $I^*(G\cdot x)=H^*_G(W,W^{<c})$ where $c:=F(G\cdot x)$.
For any  $r\in H^*_G(G\cdot x)_\ann 
 = H^*_G (W)_\ann = \ker (\pi_W^* ) $ and $h\in I^*(G\cdot x)$, we have
$ r\cdot h  = 0 $ 
by \eqref{cuppro} 
and therefore $r\in I^*(G\cdot x)_\ann$.
\end{proof}

\subsection{A cup-length lower bound for the number of critical values}

A cohomological version of the classical L\"usternik-Schnirelmann theorem asserts that the number of critical points of a smooth function on a closed manifold is strictly larger than the cup-length of the manifold, namely the maximal number of cohomology classes of positive degree whose total cup product is non-zero. For functions invariant by the action of a compact Lie group $G$, 
%as first remarked by Fadell and Rabinowitz \cite{FR},  
it is convenient to consider the cup-length with respect to the equivariant cohomology classes coming from $H^*(BG)$. For a given $G$-space $X$, we define
\begin{equation}\label{def:CL}
\cuplength_G(X)
:=
\sup
\Big\{
k\geq1\ \big|\ 
\exists\ w_1,...,w_k\in \tilde H^{*}(BG)\mbox{ s.t. } \ w_1\smallsmile ...\smallsmile w_k\not\in H^*_G(X)_\ann
\Big\}
\end{equation}
with the convention $\sup\varnothing=0$. Here,
$\tilde H^{*}(BG)$ is the reduced cohomology 
% of the classifying space $BG$, 
\begin{align}
\label{eq:redho}
\tilde H^d(BG)
:=
\left\{
\begin{array}{@{}ll}
  0 \, ,  &  \mbox{if }d=0 \, ,\vspace{5pt}\\
  H^d(BG) \, ,  & \mbox{if }d>0 \, .
\end{array}
\right.
\end{align}
%We shall prove 
%multiplicity of $ 3d$ %Stokes waves  
%as stationary  orbits of a functional on  the join  product of $ 
%S(k_1)\times S(k_2)$ 
%(as in  \eqref{eq:defProduct of two spheres})
%with a sphere $ S^1 $,  invariant under the action of $ \T^2 $, using 
%the  following lower bound.

\begin{example}
    \label{cup-length of two spheres}
   Let $S(k_1)\subset \C^{n_1}$ and $S(k_2)\subset \C^{n_2}$ denote the unit spheres 
   with $\T^d$-actions defined by two matrices $k_1\in \Z^{d\times n_1}$ and $k_2\in \Z^{d\times n_2}$ as in Example \ref{ex:Vk},  satisfying 
the non-collinearity assumption \eqref{eq:non-coll}. Then the product  $ S(k_1)\times S(k_2)$ 
equipped with the diagonal action satisfies 
    \begin{equation} \label{eq:CL product of spheres}
        \cuplength_{\T^d}(S(k_1)\times S(k_2))\geq n_1+n_2-2\, . 
    \end{equation}
Indeed,  
$f\in H^*(B\T^d)$ in  \eqref{eq:f cup product ei} is the cup product of $n_1+n_2-2$ cohomology classes of $\tilde H^*(B\T^d)$ and does not annihilate $H^*_{\T^d}(S(k_1,k_2))$.
    From the definition of cup-length \eqref{def:CL} we deduce \eqref{eq:CL product of spheres}.
\end{example}

In the proof of Theorem  \ref{Existence of 3d solutions collinear nonresonant} we shall use  the following 
result. 
Let $X\star Y$ be  the join product  of two $G$-spaces $X$ and $Y$, introduced in 
\eqref{defjoin}. 
We recall that $X$ and $Y$ can be identified with the subspaces $ (X\times Y\times \{0\})/ \sim $ and $ ( X\times Y\times\{1\} ) / \sim $ of $X \star Y$ respectively. A $G$ action on $X$ is \emph{transitive} when the whole space $X$ is a single $G$-orbit.

\begin{theorem}\label{teo:ast}
Let $T=\T^d$ be a torus of dimension $d\geq1$, $M$ a compact metric $ T $-space 
such that every point of $M$ has finite stabilizer, and $S^n$ an $n$-sphere equipped with a  %continuous 
transitive $T$ action. Let $F:M \star S^n\to\R$ be a continuous $T$-invariant function admitting a gradient-like flow $\phi$ such that the subspace $S^n\subseteq M\star S^n$ is stationary, i.e.
$ S^n\subseteq\fix(\phi) $. 
If there are only finitely many  $T$-orbits in $\fix(\phi)\setminus S^n$, then $F$ has at least $\cuplength_{T}(M)+1$ \emph{critical values} different from~$F(S^n)$.
\end{theorem}

\begin{proof}
By assumption, the subspace $S^n\subseteq M\star S^n$ is a $T$-orbit in $\fix(\phi)$, and we denote by $ \ell_*:=F(S^n)$ its critical value. We assume that $\fix(\phi)$ consists of finitely many $T$-orbits, so that  in particular  $F(\fix(\phi))$ consists of finitely many pairwise distinct 
critical values $\ell_*,\ell_1,...,\ell_k$. We shall prove that 
\begin{equation}\label{ciochevoglio}
\underbrace{\tilde H^*(BT)
\smallsmile \ldots \smallsmile
\tilde H^*(BT)}_{\times k}
\subseteq 
 H^*_{T}(M)_\ann \, ,
 \end{equation}
 where $\tilde H^*(BT)$ is the reduced cohomology  of the classifying space $BT$ introduced in  \eqref{eq:redho}.
In view of \eqref{def:CL}, this implies that $k\geq \cuplength_{T}(M) + 1 $, and thus proves the theorem.
\smallskip 

\noindent
{\sc Step 1 (Morse decomposition).}
The whole space $M\star S^n$ is an isolating block of itself with local cohomology
$I^*(M\star S^n)=H^*_{T}(M \star S^n)$.
For each critical value $ \ell $ of $F$, we denote the corresponding stationary set by $K_{\ell}:=\fix(\phi)\cap F^{-1}(\ell)$. 
A repeated application of Lemma~\ref{l:long_exact_sequence_local_cohomology} and Remark \ref{r:split_block} implies that 
\begin{equation}\label{morsedeco}
I^*(K_{\ell_*})_\ann
\smallsmile
I^*(K_{\ell_1})_\ann 
\smallsmile ...\smallsmile 
I^*(K_{\ell_k})_\ann
\subseteq
I^*(M \star S^n)_\ann \, .
\end{equation}

\noindent
{\sc Step 2.} {\it The  annihilator of $I^*(M \star S^n)$ is given by
}
\begin{equation}
\label{e:ann_K_LS}
I^*(M \star S^n)_\ann =  e\smallsmile H^*_{T}(M)_\ann \, ,
\end{equation}
{\it where 
$e\in H^{n+1}(BT)\setminus\{0\}$ is a 
non-zero cohomology class such that} 
\begin{equation}
\label{e:ann_crit_LS_0}
H^*_{T}(S^n)_\ann=e\smallsmile H^*(BT) \, .
\end{equation}

The $T$ action on $M\star S^n$ has no fixed points; indeed, each point of $M$ has finite stabilizer, and the $T$ action on $S^n$ is transitive.
By Lemma \ref{rmk:nonnull}, any orbit $T\cdot x\subset M \star S^n$ has non-trivial annihilator $H^*_{T}( T \cdot x)_\ann \neq 0$. Moreover, by Lemma \ref{lm:cor of slice},  $T\cdot x$ has a $T$-invariant open neighborhood $U_x$ such that
\begin{equation}\label{intornoUx}
H^*_{T}(U_x)_\ann = H^*_{T}( T \cdot x)_\ann \neq 0. 
\end{equation}
By compactness, we can cover 
$M \star S^n \subset U_{x_1} \cup \ldots \cup U_{x_N}$ 
by finitely many neighborhoods as in \eqref{intornoUx},
and by  the union subadditivity property of the annihilators (Proposition \ref{p:ann}-($ii$)) 
we infer 
that $H^*_{T}(M \star S^n)_\ann \neq 0 $,
and thus 
that  $H^*_{T}(M \star S^n)_\free=0$. Similarly, $H^*_{T}(M)_\free=0$.
Thus  Lemma~\ref{lm:join} implies that 
\begin{align}
\label{e:ann_K_LS0}
I^*(M \star S^n)_\ann
=
H^*_{T}(M)_\ann\smallsmile H^*_{T}(S^n)_\ann 
\, .
\end{align}
By \eqref{annsfera}, there is a cohomology class 
$e\in H^{n+1}(BT) $ such that 
\eqref{e:ann_crit_LS_0} holds.
Once again, since the $T$ action on $S^n$ has no fixed points, Lemma \ref{rmk:nonnull} implies that $H^*_T(S^n)_\ann$ is non-trivial, and therefore $ e \neq 0 $. The identity  \eqref{e:ann_K_LS} follows by  \eqref{e:ann_K_LS0}
and \eqref{e:ann_crit_LS_0}. 
\\[1mm] 
{\sc Step 3.}
{\it The local cohomology of the critical sets $K_{\ell_*},K_{\ell_1},...,K_{\ell_k}$ satisfies} \begin{align}
\label{e:ann_Kci_LS}
e\smallsmile H^*(BT)\subseteq I^*(K_{\ell_*})_\ann
\, ,
\
\tilde H^*(BT)\subseteq I^*(K_{\ell_i})_\ann 
  \, , \ \forall i =1, \ldots, k   \,  , 
\end{align}
{\it where $\tilde H^*(BT) $
is the reduced cohomology 
 \eqref{eq:redho}.}\vspace{5pt}

\noindent
By Lemma~\ref{l:direct_sum_local_cohomology},
for any critical level $ \ell $, we have
\begin{equation}
\label{inclcr1}
I^*(K_\ell)_{\ann} = 
\bigcap_{T\cdot x\subseteq K_\ell}
I^*(T\cdot x)_\ann \,.
\end{equation} 
Lemma \ref{l:ann_orbit} implies 
\begin{equation}
\label{e:ann_inclusion_LS}
H^*_{T}(T\cdot x)_\ann
\subseteq
I^*(T\cdot x)_\ann \, , \quad \forall x \in M \star  S^n   \, .
\end{equation}
Since any point of $M$ has finite stabilizer,  any $x\in\fix(\phi)\setminus S^n$ has finite stabilizer as well. Therefore, by Remark \ref{rm:finite stabilizer},   the equivariant cohomology $H^d_T(T\cdot x)$ is trivial in every degree $d>0$. This implies  
\begin{equation}
\label{e:ann_crit_LS}
H^*_T(T\cdot x)_\ann
=
\tilde H^*(BT) \, . 
\end{equation}
If $ \ell \neq \ell_* $
we conclude, 
by \eqref{inclcr1} and   
\eqref{e:ann_inclusion_LS}-\eqref{e:ann_crit_LS},  that 
$ \tilde H^*(BT)\subseteq I^*(K_{\ell})_\ann $, 
proving  the last inclusions in 
\eqref{e:ann_Kci_LS}.
If $ \ell = \ell_* $ 
we deduce  similarly the first 
inclusion in 
\eqref{e:ann_Kci_LS} using 
also  
% \eqref{inclcr1},  
\eqref{e:ann_crit_LS_0}
where $ \deg(e) >0 $. 
%and \eqref{e:ann_inclusion_LS}, \eqref{e:ann_crit_LS}.
\\[1mm]
\noindent
{\sc Step 4.}
Substituting \eqref{e:ann_K_LS} and   \eqref{e:ann_Kci_LS} 
in \eqref{morsedeco}, we obtain
$ e\smallsmile \tilde H^*(BT)^{\smallsmile k}
\subseteq 
e\smallsmile H^*_T(M)_\ann $, 
and, since $ e \neq 0 $ and $ H^* (BT)$ has no non-trivial zero divisors, we deduce \eqref{ciochevoglio}. 
\end{proof}

The following, more ordinary,  L\"usternik-Schnirelmann's type  result can be proved in a way similar to Theorem \ref{teo:ast}.
    \begin{theorem}\label{prop:T2G}
        Let $G$ be a compact Lie group, $M$ be a 
        compact metric $G$-space such that every point of $M$ has finite stabilizer, and $F:M\to\R$ a continuous $G$-invariant function admitting a gradient-like flow $\phi$ such that $\fix(\phi)$ consists of finitely many  $G$-orbits. Then $F$ possesses at least $\cuplength_{G}(M)+1$ critical levels.
    \end{theorem}

Theorems \ref{teo:ast}  
and \ref{prop:T2G}
provide a lower bound for the number of distinct {\it critical values}, allowing in Section \ref{sec:final}
to prove
multiplicity  of 
$ ( \T^2_\Gamma \rtimes \Z_2 ) $-geometrically distinct Stokes waves.

\part{Bifurcation of $3d$ Stokes waves}
\label{part:II}

In this part we apply the previous  results to prove  multiplicity 
of  gravity-capillary Stokes waves.

\section{The variational Lyapunov-Schmidt reduction}\label{sec:LS}
By \eqref{spaziY}
the linearized operator $\cL_c:=
\di_u \cF (c, 0)$  is equal to  
$$
    \cL_c= c \cdot \partial_x + J\di_u \nablal \cH(0)=c\cdot \partial_x +J\left(\begin{matrix}
            g -\kappa \Delta & 0 \\
            0                & G(0)
        \end{matrix}\right)\, 
$$
where  $ G(0) = |D|\tanh(\tth |D| ) $ is the Fourier multiplier 
with symbol $|\xi| \tanh (\tth |\xi| )$ 
%Dirichlet-Neumann operator 
%at $ \eta = 0 $, 
if $ \tth < \infty $, and 
$ G(0) = |D| $ in infinite depth
(we denote  
$ D := \im^{-1} \pa_x  $)
%the H\"ormander derivative). 

The operator $\cL_c $ can be  symplectically 
diagonalized:  
the real phase space 
$ L^2  := L^2 (\T^2_\Gamma, \R^2 )  $  decomposes as  direct sum of infinitely many $ 2 $ dimensional real  {\it symplectic}  subspaces {\it invariant} 
under $ \cL_c $ that we now describe. 
Define  the  real vectors, for $ j \in \Gamma'\setminus\{0\}$, 
\begin{equation}\label{symplectic base}
        v_j^{(1)}(x):=  \vect{M_j\cos(j\cdot x)}{M_{j}^{-1} \sin(j \cdot x)} \, , \quad 
        v_j^{(2)}(x):= \vect{-M_j\sin(j \cdot x)}{M_{j}^{-1}\cos(j \cdot x)}\, ,
\end{equation}
where \begin{equation}\label{eq:Mj}
    M_j  := (g + \kappa |j|^2)^{-\frac14}(|j | \tanh (\tth |j|))^{\frac14}
\end{equation} 
and 
$ v_0^{(1)}:= 
    \begin{psmallmatrix}
    1 \\ 0
    \end{psmallmatrix} $ 
    and $ v_0^{(2)}:= 
    \begin{psmallmatrix} 0 \\ 1
    \end{psmallmatrix} $.
The following lemma is directly verified. 

\begin{lemma}\label{symplectic base and coordinates}
The following holds:    
\\[1mm]
$(i)$ 
        {\sc Symplecticity.}
        For any  $ j \in \Gamma'$  
        the $ 2 $-dimensional real subspace  
\begin{equation}
\label{VJsym} V_j := \big\{ \alpha_j v_j^{(1)} + 
\beta_j v_j^{(2)} \ | \ \alpha_j, \beta_j \in \R \big\}  
\end{equation} 
is symplectic. 
Each $ V_j $ 
is symplectic orthogonal to any other $ V_k$ for any  $ j \neq k $. 
\\[1mm]
$(ii)$
        {\sc Completeness.}
        The phase space admits the decomposition 
$$ 
        L^2 := L^2 (\T^2_\Gamma, \R^2)   =  
        \overline{\bigoplus_{j \in \Gamma'}V_{j}}^{L^2}
$$ 
where   
        $\{v_j^{(1)},v_j^{(2)}\}_{j\in \Gamma'}$ is a symplectic basis  of $ L^2 $ and
\begin{equation}\label{coordiantes}
    u=\sum_{j\in \Gamma'} \alpha_j\, v_j^{(1)}+\beta_j\, v_j^{(2)}\, ,
    \quad \forall u\in L^2  \, , 
\end{equation} 
with coordinates 
        \begin{equation}
        \label{abu}
           \alpha_j := \alpha_j(u) :=  \boldsymbol{\Omega}(u,v_j^{(2)})\, , \quad  \beta_j := \beta_j (u) := \boldsymbol{\Omega}(v_j^{(1)},u) 
           \, , \quad \forall j \in \Gamma'\, .
        \end{equation}
  $(iii)$ 
        {\sc Invariance.} For any $ j \in \Gamma'\setminus\{0\}$, it results 
        $  \pa_{x} v_j^{(1)}= j v_j^{(2)} $, $ \pa_{x} v_j^{(2)} = - j v_j^{(1)} $ and 
        \begin{align}
        \label{actLc}
                \cL_c v_j^{(1)}=(c\cdot j-\omega(j)) v_j^{(2)} \, , \quad 
                \cL_c v_j^{(2)}=-(c\cdot j-\omega(j)) v_j^{(1)} \, , 
        \end{align}
where  $ \omega (j) $ is the frequency defined
in \eqref{omega}, 
        and  $\cL_c v_0^{(1)}=-g\, v_0^{(2)}$.
\\[1mm]
$(iv)$
        {\sc $ \T^2_\Gamma $ and $\Z_2$ symmetries.} For any 
        $ u \in L^2  $  and any $j \in \Gamma'$  we have
        \begin{equation}
        \label{eq:trasab}
        \begin{aligned}
            &\alpha_j (\tau_\theta u)+\im \, \beta_j (\tau_\theta u)=e^{-\im \, j \cdot \theta}(\alpha_j(u)+\im \, \beta_j(u)) \, , \quad \forall \theta\in \T^2_\Gamma \, , \\
            &\alpha_j(\rho u)+\im \, \beta_j(\rho u)=
            \alpha_j(u) - \im \, \beta_j(u) \, .
        \end{aligned}
         \end{equation} 
$(v)$ {\sc Momentum.} The momentum 
        $ \cI $ in \eqref{DefI} is equal to
        \begin{equation}\label{Momuntum in coordinates}
                \cI(u)= \tfrac12 \di_u \cI(u) [u] 
                \stackrel{\eqref{Def Hamiltonian vector field}} = \tfrac 12 {\bf \Omega} (X_{\cI}(u),u)\stackrel{\eqref{XI}}=-\tfrac12 \sum_{j\in \Gamma'\setminus\{0\}} j \, (\alpha_j^2(u)+\beta_j^2(u)) \, .
        \end{equation}
\end{lemma}

In view of \eqref{actLc}, for any $c \in \R^2 $,  the Kernel of $ \cL_c $
is Fourier supported on the wave vectors 
$ j \in \Gamma' $ such that 
$ c \cdot j = \omega (j) $,  so 
$ \ker(\cL_{c})\not=0$ if and only if 
$ c $ belongs to the set $ \mathsf C $  defined in  \eqref{bifspeed}.

\smallskip 
We reduce  \eqref{Bifurcation problem} to a finite dimensional problem by a variational Lyapunov-Schmidt procedure.

\paragraph{\bf The Kernel $ V = \ker(\cL_{c_*}) $.}
Let $ c_* $ be a speed in the set $ \mathsf C $ defined in \eqref{bifspeed} so that 
$ \ker(\cL_{c_*})\not=0$. Specifically, in view of 
\eqref{actLc}, 
\begin{equation}
\label{defKer}
V := \ker(\cL_{c_*}) = 
\Big\{ c_* \cdot \pa_x v +
   \di_u X_\cH (0) v = 0 \Big\}  =
  \Big\{  
   \di_u \nablal  (\cH + c_* \cI)  
  (0) v = 0 \Big\} = 
\bigoplus_{j \in \cV} 
V_j 
\end{equation}
where 
$ V_j $
are the $2$ dimensional real  symplectic subspaces in \eqref{VJsym} and $ \cV $ is the set of resonant wave vectors
in \eqref{Def cV}. The subspace  
$V $ is  symplectic. 
Furthermore $ V $ is finite dimensional, because 
the dispersion relation  in 
\eqref{omega} satisfies 
$\omega(j)\approx \sqrt{\kappa} |j|^{\frac32}$,  and then
    \begin{equation}\label{count of Js}
    \# \cV   \leq \#\Big\{j \in \Gamma'\setminus\{0\} \ \big| \   |c|\geq \frac{\omega(j)}{|j|}\Big\}\lesssim_{\Gamma,\kappa,g,\tth} |c|^{4}\, .
    \end{equation}
    
\begin{lemma}\label{lem:coll}
{\bf (Collinear vectors in $\cV $)} For any value of surface tension   
$\kappa>0$,  gravity  $g>0$, depth 
$\tth>0$,  and any lattice $ \Gamma $, the number of collinear vectors of $ \cV$ is at most $ 2 $.
\end{lemma}

\begin{proof}
The dispersion relation $\omega (\xi) $ 
in \eqref{omega} 
is radial, i.e. $\omega(\xi)= 
\breve \omega(|\xi |)$. The function
 $ t\mapsto \frac{\breve \omega(t)}t$ has at most one local minimum and 
    $
            \#\big\{t > 0  \ \big| \ \frac{\breve \omega(t)}t=s \big\}\leq 2 $
            for any $ s > 0  
     $. 
\end{proof}

\paragraph{\bf Lyapunov Schmidt  decomposition.}
We decompose  
the phase space $ L^2 =  L^2 (\T^2_\Gamma, \R^2) $
equipped with the symplectic form \eqref{Def W} 
as 
\begin{equation}\label{L2L2VW}
L^2  = V \oplus W  
\end{equation}
where $V=\ker(\cL_{c_*})$ and $ W $ is its symplectic orthogonal 
\begin{equation}\label{siort}
W:= V^{\bot_{\boldsymbol{\Omega}}} = \overline{\bigoplus_{
j \in \Gamma' \setminus \cV} V_j }^{L^2} \, . 
\end{equation}
We denote by $ \Pi_V $ and $ \Pi_W= \text{Id} -\Pi_V$
the symplectic projectors onto $V$ and $W$ respectively.
According to  \eqref{L2L2VW} 
the space $X$  in \eqref{spaceX} 
admits the decomposition  
$ X = V\oplus (W\cap X) $ and  \eqref{Bifurcation problem} is equivalent to 
find a solution $ (c,v,w)\in \R^2 \times V \times (W\cap X) $ of 
\[
    \begin{cases}
     \Pi_{ V} {\mathcal F}(c,v+w) = 0 \qquad \quad \, \text{bifurcation equation}\\
        \Pi_{W\cap Y} {\mathcal F}(c,v+w)=0 \qquad \text{range equation}
    \end{cases}
\]
where $\Pi_{W\cap Y}$ and $\Pi_{V}$ are the projectors associated to the  decomposition  
$Y = V \oplus (W \cap Y) $.

The  range equation is solved by the implicit function theorem. We denote $B_r^V $ the ball of radius $r$ and center 0 in $V$.

\begin{lemma}\label{lem:range}
{\bf (Solution of the range equation)}
    There exists  an analytic function $w:B_{r}(c_*)\times B_{r}^V  \subset \R^2\times V\to W\cap X$
    defined in a neighborhood of $ (c_*,0) $ solving 
\begin{equation}\label{solution of the range equation}
        \Pi_{W\cap Y}
        {\mathcal F} (c, v+w(c,v)) = 0 \, , 
    \end{equation}
satisfying
\begin{equation}\label{properties of w(c,v)}
        w(c,0)=0\, , \quad \di_v w(c,0)=0 \, , \quad \forall c\in B_{r}(c_*) \, , 
    \end{equation}
and  
\begin{equation}\label{wequi}
w(c,\rho v)=\rho w(c, v) \, , \quad
   w(c, \tau_\theta v)=\tau_\theta 
   w(c,v) \, ,   
   \quad \forall 
   \theta \in\mathbb{T}^2_\Gamma \, .  
\end{equation} 
   Furthermore  
the range equation \eqref{solution of the range equation} is equivalent to
\begin{equation}\label{Variational Range}
    \di_u\Psi(c,v+w(c,v))[\hat w]=\boldsymbol{\Omega}(\cF(c,v+w(c,v)),\hat w)=0 \, , \quad \forall \hat w\in W\, .
\end{equation}
    \end{lemma}
    
\begin{proof}
The variational characterization 
\eqref{Variational Range} follows because, recalling  \eqref{Def Hamiltonian vector field}, \eqref{spaziY}, \eqref{Variational}, 
\begin{equation}\label{Variational Range0}
    \di_u\Psi(c,v+w(c,v))[\hat u]=\boldsymbol{\Omega}(\cF(c,v+w(c,v)),\hat u)=0 \, , \quad \forall \hat u\in X \, ,
\end{equation}
and the subspaces 
$ V $ and $ W $ are symplectic orthogonal, cf. \eqref{siort}. The solution of
\eqref{solution of the range equation} 
is constructed by  the analytic implicit function theorem in the Appendix.
   \end{proof}

In view of   Lemma  \ref{lem:range} 
we are reduced to 
solve the bifurcation equation
$    \Pi_V \cF (c,v+w(c,v)) = 0 $ 
where $ w(c,v)$ is the solution of \eqref{solution of the range equation}.  
This equation has still a
variational structure.

\begin{lemma}\label{lem:varia}
 {\bf (Variational 
 structure of the bifurcation equation)}
The reduced functional   $\Phi:B_{r}(c_*)\times B_{r}^V \to \R$,  
\begin{equation}\label{def reduced functional}
        \Phi(c,v):= \Psi (c, v + w(c,v)) = 
        (\cH + c\cdot \cI )(v+w(c,v))  \, , 
    \end{equation}
    is analytic and satisfies, for any   $(c,v)\in  B_r(c_*) \times  B_r^V $,
    \begin{equation}\label{symmetriesPhi}
        \Phi (c, \rho v)= \Phi (c, v) \, , \quad  
        \Phi( c, \tau_\theta v)= \Phi (c,v) \, , \quad\forall \theta\in \mathbb{T}^2_\Gamma  \, . 
    \end{equation} 
Furthermore, for any $ (c,v)\in B_{r}(c_*)\times B_{r}^V $, for any $ \hat v\in V $, 
    \begin{align}\label{X Phi}
        \di_v\Phi(c,v)[\hat v]   = 
        \di_u \Psi (c, v + w(c,v))[\hat v ] 
        & = 
        \boldsymbol{\Omega}(\Pi_{V}
        \cF (c,v+w(c,v)),\hat v)   \, .  
    \end{align}
    Therefore, for any $c\in B_{r}(c_*)$, any critical point  $v \in B_{r}^V $ of $\Phi(c,\cdot)$ gives rise to a solution $u := v+w(c,v) \in X $ of \eqref{Bifurcation problem}.
\end{lemma}

\begin{proof}
Differentiating
\eqref{def reduced functional} with respect to $ v $ we have 
    $$
        \di_v\Phi(c,v)[\hat v]  = \di_u \Psi (c, v+w(c,v))[\hat v + \underbrace{\di_v w(c,v)[\hat  v]}_{\in W}]
      \stackrel{\eqref{Variational Range0}, \eqref{Variational Range}} = \boldsymbol{\Omega}({\cal F}(c,v+w(c,v)),\hat v) 
    $$
and 
\eqref{X Phi} follows since  
$ V $ and $ W $ are symplectic orthogonal, cf. \eqref{siort}.    \end{proof}

From now on, the analysis 
becomes drastically different from \cite{CN}, and crucially exploits the symmetries of $2d$ and $3d$ waves.

\paragraph{\bf Phase space decomposition along $2d$ and $3d$ waves.} For any $j\in\Gamma'\setminus\{0\}$ we denote by  
\begin{equation}\label{1dfs}
 [j]:= 
%\Z j = \big\{ k j \, | \, k \in \Z  \big\} \, , \quad 
\big\{ j'\in \Gamma' \ | \  
j'\parallel j \big\} 
\end{equation}
the $1$-dimensional 
sub-lattice of $ \Gamma' $ formed by its vectors parallel to $j $.  The lattice 
$ [j]$ is generated by a vector $ \underline{j} $ in $ [j]$
of minimal norm (which may not be $ j $)
namely
\begin{equation}\label{minimalj}
[j] = \Z \, \underline{j} \, . 
\end{equation}
Remark that the dual of $ [j] $ is the lattice 
$ 2 \pi \underline{j} |\underline j|^{-2}  \Z  $.

\begin{definition}{\bf ($2d$ and $3d$ waves)}\label{def:cX2d}
For any  $j \in\Gamma'\setminus\{0\} $  
we denote
$$
    \cX_{[j]}^{\twod}:=\overline{\bigoplus_{j' \in [j]} V_{j'}}^{L^2}=\overline{\bigoplus_{j' \in \Gamma', j'\parallel j} V_{j'}}^{L^2}
$$
the subspace of $ 2 d$ waves along the direction $[j ]$.  
A function $ u \in L^2  $ 
which does not belong to any
$ \cX_{[j]}^{\twod} $  
is referred to as a truly $ 3d$-wave. 
\end{definition}

Clearly, if two wave vectors $ j_1, j_2 \in  \Gamma' $ 
    are parallel, then $[j_1]=[j_2]$ and thus 
   $ \cX_{[j_1]}^{\twod} = \cX_{[j_2]}^{\twod} $.
The functions in $ \cX_{[j]}^{\twod}$ are really $ 2 d$-waves as described by the following lemma. 

\begin{lemma}\label{lm:2d-Waves}
{\bf (Shape of $2d$-waves)}
A $ 2 d$-wave $ u (x)   $ in $ \cX_{[j]}^{\twod}$ is
$ 2 \pi \underline j | \underline j|^{-2}$-periodic in space where $\underline j $
is the generator of the $1$-d lattice  $[j] $ defined in 
\eqref{minimalj}, i.e.
\begin{equation}\label{eq:periodu}
u (x + 2 \pi \underline j | \underline j|^{-2} ) = u (x)  \, ,
\quad \forall x \in \R^2 \, , 
\end{equation}
and it is constant along the directions orthogonal to the wave vector $ j $, namely
\begin{equation}\label{eq:orto}
u(x) = u \big( 
(x \cdot \hat\jmath )  \hat\jmath \big)  \quad \text{where} \quad  
\hat\jmath := j |j|^{-1} \, , 
\quad \forall x \in \R^2 \, . 
\end{equation}
\end{lemma}

\begin{proof}
Let $ u   \in \cX_{[j]}^{\twod}$. 
The periodicity property \eqref{eq:periodu} follows by  the Fourier series expansion   
\begin{equation}\label{Fouux}
        u(x)=\sum_{j' \in [j]} \alpha_{j'}\, v_{j'}^{(1)}(x) +\beta_{j'}\, v_{j'}^{(2)} (x) \quad \text{with} \ 
        v_{j'}^{(1)} \, , v_{j'}^{(2)} 
        \ \text{in} \ 
        \eqref{symplectic base} \, , 
\end{equation}
since 
$ j' \cdot \underline j | \underline j|^{-2} \in \Z $ for any $ j' \in [j] = \Z \underline j  $.  We deduce 
\eqref{eq:orto}, by \eqref{Fouux},  \eqref{symplectic base}, decomposing $ x = (x \cdot \hat\jmath) \hat\jmath + 
(x \cdot \hat\jmath^\bot) \hat\jmath^\bot $  and since  $j' \cdot \hat\jmath^\bot = 0 $ for 
any $j' \in [j] $. 
\end{proof}

A $2d$-wave $ u  $ 
may be characterized  
 in terms of the isotropy group 
\begin{equation}\label{stabilu}
            (\T_\Gamma^2)_u := 
            \Big\{ \theta \in \T^2_\Gamma \ |\ \tau_\theta u= u  \Big\}
            = 
            \bigcap_{j \in A} 
            \big\{ \theta \in 
            \T^2_\Gamma \ |\ e^{\im j \cdot \theta } = 1 \big\}
\end{equation}
    where $ A := \big\{ j \in \Gamma' \setminus \{0\} \ | \ \alpha_j (u) + \im \beta_j (u)\neq 0    \big\}  $ is the set of ``active frequencies" of $ u  $, 
cf.~\eqref{abu}, 
and  the subgroup    \begin{equation}\label{eq:def H[j]}
        {\mathtt T}_{[j]}:= 
        \bigcap_{j'\in [j]}\big\{\theta\in \T^2_\Gamma
        \ | \ e^{\im j'\cdot \theta}=1\big\} 
        \, .
    \end{equation}
     \begin{lemma}\label{ortoges}
    {\bf (Characterization of $2d$-waves)}
    Any function 
            $u \in L^2  $ is a 
            $ 2 d$-wave along the direction $ [j] $, i.e.  $ u\in \cX_{[j]}^{\twod} $, if and only if 
    \begin{equation}\label{One dimensional iff simmetric}
         \tau_\theta u=u \, , \quad \forall \theta\in 
                {\mathtt T}_{[j]} \, , \quad \text{i.e.} \quad {\mathtt T}_{[j]}\subseteq (\T_\Gamma^2)_u\, .
            \end{equation}
    \end{lemma}
    
    \begin{proof}
    We Fourier  decompose 
$
    u=\sum_{j'\in \Gamma'} \alpha_{j'}(u)\, v_{j'}^{(1)}+\beta_{j'}(u)\, v_{j'}^{(2)} $
    as in 
    \eqref{coordiantes}.    
By \eqref{eq:trasab}
we have 
$$
            \alpha_{j'}(u-\tau_\theta u)+\im \beta_{j'}(u-\tau_\theta u)=\big(\alpha_{j'}(u)+\im \, \beta_{j'}(u)\big)\big(1-e^{-\im \, {j'}\cdot \theta}\big) \, . 
$$
        Thus $\tau_\theta u=u$ for any $\theta\in {\mathtt T}_{[j]}$ is equivalent to   
        $
            \alpha_{j'}(u)+\im \, \beta_{j'}(u)=0
      $  
        for any $ j' \in \Gamma' $, $j'\not\in [j]$, namely  $ j' \not \parallel j $.
        Recalling  
        Definition \ref{def:cX2d}, this means that 
        $ u \in \cX_{[j]}^{\twod} $.
    \end{proof}

\paragraph{\bf Decomposition of $ V $ along resonant wave directions.}
We  introduce on the kernel $V$ in \eqref{defKer} the  scalar product  
\begin{equation}\label{eq:def scalar product}
    \langle v, v_1 \rangle =\Re\Big\{\sum_{j\in\cV} |j| \, z_j(v)\overline{z_j(v_1)}\Big\} \, , \qquad \|v \|^2 := \sum_{j\in\cV} |j| \, |z_j(v)|^2 \, , 
\end{equation}
where $z_j(v) =\alpha_j(v)+\im \beta_j(v)$. By  \eqref{eq:trasab}
it is $ \T^2_\Gamma\rtimes \Z_2 $-invariant, i.e. 
$$ 
\langle \tau_\theta v, \tau_\theta v_1\rangle=\langle v,v_1\rangle \, , \quad 
\forall 
\theta\in \T^2_\Gamma \, ,  \quad \langle \rho v, \rho v_1\rangle=\langle v, v_1\rangle \, . 
$$
In the sequel all the gradients in the variable $v$ are taken with respect to this scalar product.
The spaces $\{V_j\}_{j\in\cV}$ are pairwise orthogonal with 
respect to the scalar product \eqref{eq:def scalar product} and we denote $ \pi_{j}:V\to V_j $ the  orthogonal projectors associated to the decomposition \eqref{defKer}.

We also  consider the orthogonal decomposition of $ V $, 
\[
V = \bigoplus_{\hat\jmath \in \cD} V_{\hat\jmath} \, , 
\]
where $ \cD $ is the set of {\it resonant wave directions}
\begin{equation}\label{decokerj}
\cD:= \Big\{ \hat\jmath  \in S^1 \ | \  j \in \cV \Big\}   \, , \quad \hat\jmath = j  |j|^{-1} \, , 
\end{equation}
and 
    \begin{equation}\label{Def piJ}
            V_{\hat\jmath} := V \cap \cX_{[j]}^{\twod} =\bigoplus_{j'\in \cV \cap [j] 
            %j'\parallel \hat\jmath
            } V_{j'}  
    \end{equation} 
is 
the subspace  of $ V $  formed by the $ 2 d$-waves along $ [j] $.
For any resonant wave direction $ {\hat\jmath} \in \cD $, we denote the corresponding orthogonal projectors
\begin{equation}\label{Def P}
        \Pi_{\hat\jmath}:V\to V_{\hat\jmath} \, , \ \Pi_{\hat\jmath} = 
        \sum_{j' \in \cV \cap [j]} \pi_j \, , 
        \quad 
        \Pi_{\hat\jmath}^\bot:= \text{Id} -\Pi_{\hat\jmath}=\sum_{\hat\jmath' \in \cD\setminus \hat\jmath} \Pi_{\hat\jmath'} \,  .
    \end{equation}
For any $ v \in V $ it results 
$\| v \|^2 = \sum_{\hat\jmath \in \cD} \| \Pi_{\hat\jmath} v \|^2 $.   We denote 
\begin{equation}\label{V2d}
{\bf V}^{\twod} := \bigcup_{\hat\jmath \in \cD} V_{\hat\jmath}
\, , \qquad 
V_{\hat\jmath} = \big\{ v \in V \ | \ 
\Pi_{\hat\jmath}^\bot v = 0  \big\} \, , 
\end{equation}
the union of subspaces in $ V $  formed by $ 2 d$-waves. 
    \begin{remark}{\bf (Resonant wave  directions)}\label{rm:j hat}
        Since the set \eqref{Def cV} of resonant wave vectors $\cV$ is contained in the half-plane \eqref{eq:ilcono}, there is a $1-1$ correspondence between the resonant wave directions $\cD=\{\hat\jmath \  | \ j\in \cV\}$ and the lattices $\{[j] \ | \ j\in \cV\}$ in \eqref{1dfs}.
        Furthermore, in view of Lemma \ref{lem:coll}, for each resonant wave direction $\hat\jmath$,  the set $ [j]\cap \cV$ consists of 
$ 1 $ vector or $ 2 $ collinear vectors, and   
$ V_{\hat\jmath }$
is respectively   $ V_j $ 
% ($2d$ non-resonant case)  
or has 
dimension $ 4 $.  
%($2d$ resonant case).
    \end{remark}

 \begin{remark}\label{rem:Vinv}
By Lemma \ref{ortoges}
an equivariant function $ f : V \to V $ 
maps each subspace $ V_{\hat\jmath} \to V_{\hat\jmath}  $,  for any 
$ \hat\jmath \in \cD  $,  and $ V \setminus {\bf V}^{\twod}
\to V \setminus {\bf V}^{\twod} $. 
\end{remark}

  In view of \eqref{Momuntum in coordinates} and \eqref{eq:def scalar product} the momentum is  
    \begin{equation}\label{eq:cI}
        \cI(v)=-\frac12\sum_{j\in \cV} \hat\jmath \|\pi_{j} v\|^2=-\frac12\sum_{\hat\jmath \in \cD} \hat\jmath \|\Pi_{\hat\jmath} v\|^2
    \end{equation}
    where $\hat\jmath=j |j|^{-1} = (\hat\jmath_1, \hat\jmath_2)^\top $ and therefore 
    \begin{equation}\label{eq:Nabla cI}
        \nabla_v\cI(v):=\begin{pmatrix}
            \nabla_v \cI_1(v)\\
            \nabla_v \cI_2(v)
        \end{pmatrix} =
        -\sum_{\hat\jmath \in \cD} \begin{pmatrix}
            \hat\jmath_1\\
            \hat\jmath_2
        \end{pmatrix} \, \Pi_{\hat\jmath} v \, , \quad 
        \Pi_{\hat \jmath}\nabla_v \cI(v)
        =-\begin{pmatrix}
            \hat\jmath_1\\
            \hat\jmath_2
        \end{pmatrix} \Pi_{\hat\jmath}v
        \, ,\quad \forall 
        \hat\jmath \in \cD \, .
    \end{equation}
    Note that $\nabla_v\cI(v)$ commutes with the projectors $\Pi_{\hat \jmath}$,  i.e. $\Pi_{\hat \jmath}\nabla_v\cI(v)=\nabla_v\cI(\Pi_{\hat \jmath}v)$.
  The following notation  will be used throughout the paper.
\begin{definition}
{\bf (Parallel and orthogonal components to $ \hat \jmath $)}
    For any 
    resonant wave direction $ \hat\jmath \in \cD $,  cf. \eqref{decokerj}, 
    we decompose any  $y\in \R^2 $
    as
 \begin{equation}\label{eq:compoc}
    y=y^{\parallel_{\hat\jmath}} \hat\jmath + 
    y^{\bot_{\hat\jmath}} \hat\jmath^\bot \, , \quad  
        y^{\parallel_{\hat\jmath}}:= y\cdot \hat\jmath \, , \
        y^{\bot_{\hat\jmath}}:=y\cdot \hat\jmath^\bot \, , \ \hat\jmath = \vect{j_1}{j_2} |j|^{-1} \, , \ \hat\jmath^\bot 
    = \vect{-j_2}{j_1} |j|^{-1}.
    \end{equation}
We set 
    \begin{equation}\label{eq:def yj}
        y^{\hat\jmath}:=\begin{pmatrix}
            y^{\parallel_{\hat\jmath}}\\
            y^{\bot_{\hat\jmath}} 
        \end{pmatrix} = Q_{\hat\jmath} y \qquad \text{where}
        \qquad Q_{\hat\jmath} := [ \hat\jmath \ | \  \hat\jmath^\bot ] \, . 
    \end{equation}
Note that $Q_{\hat\jmath} $ is an orthogonal involution,  
$ Q_{\hat\jmath}^\top = Q_{\hat\jmath}  = Q_{\hat\jmath}^{-1} $, in particular $ \det Q_{\hat\jmath} = 1 $.  
\end{definition}

By Lemma \ref{lm:2d-Waves}, any function $u (x) $ in $ \cX^{\twod}_{[j]}$ is constant along the direction $\hat\jmath^{\bot}$ and hence the Hamiltonian vector field
generated by $ \cI^{\bot_{\hat \jmath}} (u) $,   
\begin{equation}\label{momezero}
X_{\cI^{\bot_{\hat \jmath}}}(u) \stackrel{\eqref{XI}} 
= (\hat\jmath^\bot \cdot \pa_x) u =0 \, , \quad \forall 
u\in \cX^{\twod}_{[j]} \, ,
\end{equation}
vanishes on the $ 2d$-waves along the direction $ [j]$. 
By \eqref{Momuntum in coordinates}  and 
\eqref{momezero} 
we also deduce
\begin{equation}\label{momezero2}
    \cI^{\bot_{\hat \jmath}}(u)=0\, , \quad \forall 
u\in \cX^{\twod}_{[j]} \, .
\end{equation}
Note that, for any $ \hat \jmath \in \cD $,   
\begin{equation}\label{pibot}
\nabla_v \cI^{\bot_{\hat\jmath}  } (v) = 0  \, , \  \forall v \in V_{\hat\jmath}  \, , 
\qquad 
\Pi_{\hat \jmath} \nabla_v \cI^{\bot_{\hat\jmath}} (v) = 0 \, , \  \forall v \in V \, , 
\end{equation}
because  $\di_v\cI^{\bot_{\hat \jmath}}(v)[\Pi_{\hat \jmath}v_1]=\di_v\cI^{\bot_{\hat \jmath}}(\Pi_{\hat \jmath}v_1)[v]= {\bf \Omega}(X_{\cI^{\bot_{\hat \jmath}}}(\Pi_{\hat \jmath}v_1),v) = 0 $ by \eqref{momezero}
(or use \eqref{eq:Nabla cI}).

\begin{lemma}\label{Further properties of w(c,v)} {\bf ($2d$-wave symmetry)}
For any $ {\hat \jmath} \in \cD $, 
any $c\in B_r(c_*)$ and $v\in B_{r}^V 
\cap    V_{\hat \jmath}  $, 
the solution $ w(c, v) $ of the range equation constructed in Lemma
\ref{lem:range}
satisfies 
\begin{enumerate}
            \item[$(i)$]  $w(c,v) \in  \cX_{[j]}^{\twod} \cap X $
            is a $ 2 d $-wave along the direction ${\hat \jmath}$;
            \item[$(ii)$] \label{lem843} $w(c,v)$ is independent on the component $c^{\bot_{\hat \jmath}} $ of $ c $ orthogonal to $j$, cf. \eqref{eq:compoc}.
\end{enumerate}
\end{lemma}

\begin{proof}
Item  $(i) $ follows 
  by the equivariance property
  \eqref{wequi} of  $ w ( c, \cdot ) $ and   Lemma \ref{ortoges}.  By \eqref{spaziY} and 
  \eqref{momezero} we have
\begin{equation}\label{Fcuzero}
\mathcal{F}(c, u)  =
X_{\cH + c\cdot \cI}(u)  =
X_{\cH+c^{\parallel_{\hat \jmath}}\cI^{\parallel_{\hat \jmath}}}  = 
\mathcal{F}( (c\cdot \hat\jmath)\hat\jmath , u) \, 
, \quad \forall u \in \cX_{[j]}^{\twod} \, ,  
\end{equation}
where $\cI^{\parallel_{\hat \jmath}}(v)$ is the parallel components of $\cI(v)$ according to the decomposition \eqref{eq:compoc}.
  If $v\in V_{\hat \jmath}\cap B_r^V$ 
  then $v+w(c,v) \in \cX_{[j]}^{\twod} $ by item ($i$)
  and  therefore by \eqref{Fcuzero} and \eqref{solution of the range equation} we obtain
    $$
        \Pi_{W\cap Y} 
        \mathcal{F}( (c\cdot \hat\jmath)\hat\jmath , v+w(c,v)) = \Pi_{W\cap Y} 
        \mathcal{F}(c, v+w(c,v) ) = 0 \, . 
    $$
The uniqueness of the solution of the range equation implies that 
$w(c,v)=w((c\cdot \hat\jmath)\hat\jmath ,v)$.
\end{proof}

%In order to find non-zero critical orbits of $ \Phi (c, v) $ we  use the following properties of $\Phi$. 
    \begin{proposition}{\bf (Properties of $\Phi (c, v) $)}\label{lem:expa}
    For any $(c,v)\in B_{r}(c_*)\times B_{r}^V $, 
\\[1mm]
    $ \bullet $    the reduced Hamiltonian $\Phi(c,v)$ in \eqref{def reduced functional} has the form
        \begin{equation}\label{exp1Phi}
            \Phi(c,v)=(c-c_*)\cdot \cI(v)+G_{\geq 3}(c,v)
        \end{equation}
        where    
        $G_{\geq 3}(c,v) $ is an analytic $\T^2_\Gamma\rtimes \Z_2-$invariant function 
        satisfying 
\begin{equation}\label{Gprgeq3}
\di_v^\ell G_{\geq 3} (c,0) = 0 \, , \quad 
\forall \ell = 0,1,2\, , \quad \forall c \in B_r (c_*) \, . 
\end{equation}
  $ \bullet $       For any $ {\hat \jmath} \in \cD $, 
        for any $ c \in B_{r}(c_*)  $, 
        any $ v \in B_{r}^V \cap V_{\hat \jmath}$
        the reduced Hamiltonian 
        \begin{equation}
        \label{Phi is independent on c2Q if v is in V2dj}
            \Phi(c,v) 
        =    \Phi(c^{\parallel_{\hat \jmath}} 
        \hat\jmath+c^{\bot_{\hat \jmath}}_* \hat\jmath^\bot ,v) =
        (c-c_*)^{\parallel_{\hat \jmath}} \cI^{\parallel_{\hat \jmath}} (v) + G_{\geq 3} (c^{\parallel_{\hat \jmath}} \hat\jmath+c^{\bot_{\hat \jmath}}_* \hat\jmath^\bot, v) \, , 
        \end{equation}
        is independent of $ c^{\bot_{\hat \jmath}}$, 
and   any critical point 
$v $ of 
$ \Phi(c, \cdot) : V_{\hat \jmath} \cap B_r^V \to \R $ gives rise to a $2d$ Stokes  wave  $ v + w(c,v) $ in $ \cX^{\twod}_{[j]} \cap X $.
\\[1mm]  
   $ \bullet $     The 
    reduced momentum 
    \begin{equation}\label{eq:def I(c,v)}
  {\mathtt I} : B_{r}(c_*)\times B_{r}^V \to \R  \, , 
  \quad {\mathtt I} (c,v):=\cI(v+w(c,v)) \, , 
    \end{equation} 
is analytic, $\T^2_{\Gamma}\rtimes \Z_2$-invariant  and  
    \begin{equation}\label{partial c Phi}
        {\mathtt I} (c,v) = \partial_c\Phi(c,v)  \, , \quad \forall c \in 
        B_{r}(c_*) \, ,  \  
        v \in  B_{r}^V \,      
        . 
    \end{equation}
    \end{proposition}
    
    \begin{proof}
    In view of 
    \eqref{def reduced functional},  
   the expansion  \eqref{exp1Phi} holds with 
        $$
            G_{\geq 3}(c,v) =c\cdot \cI(v+w(c,v))+\cH(v+w(c,v))-(c-c_*)\cdot \cI(v)\, .
        $$
Let us prove  \eqref{Gprgeq3}. By 
\eqref{properties of w(c,v)} 
we have  $G_{\geq 3}(c,0)=0$ and $\di_v G_{\geq 3}(c,0)=0$ for any $ c \in B_r(c_*)$ and 
    \begin{equation*}
        \di_v^2 G_{\geq 3}(c,0)[\hat v]^2 =\di_v^2\big(c_*\cdot \cI + \cH \big)_{|v=0} [\hat v]^2=\boldsymbol{\Omega}(\di_v X_{c_*\cdot \cI+\cH}(0)[\hat v],\hat v)\quad \forall \hat v\in V\, .
    \end{equation*}
        Since $V =\ker (\cL_{c_*})=\ker \big(\di_v X_{c_*\cdot \cI+\cH}(0)\big)$ we also deduce $\di_v^2G_{\geq 3}(c,0)=0$.
        Using \eqref{momezero2} and Lemma \ref{Further properties of w(c,v)}-$(ii)$, by \eqref{exp1Phi} we deduce  \eqref{Phi is independent on c2Q if v is in V2dj}.
        In view of Lemma \ref{lm:AAA} we have that any critical point of $\Phi(c,\cdot)_{|V_{\hat \jmath}}$ is a critical point of $\Phi(c,\cdot)$ and thus gives rise to a $2d$-Stokes waves solution.
        Finally, differentiating  \eqref{def reduced functional}
we get 
$$
\pa_c \Phi (c, v) =   
(\pa_c \Psi) (c, v+ w(c,v)) + 
\underbrace{(\di_u \Psi) (c, v+ w(c,v))[ \pa_c w(c,v) ]}_{=0 \ \text{by}\, \eqref{Variational Range},
\ \pa_c w(c,v)  \in W}= {\mathcal I}(v+ w(c,v)) = {\mathtt I}(c,v)
$$
proving \eqref{partial c Phi}. 
    \end{proof}

\paragraph{\bf $2d$-Stokes waves.} In view of Proposition \ref{lem:expa} any critical point of $\Phi(c,\cdot)_{|V_{\hat \jmath}}$ is a critical point also for $\Phi(c,\cdot )$. 
 The 
 classical  $2d$-theory is then recovered by restricting to $V_{\hat \jmath}$ :
\\[1.5mm]
{$\bullet$ \bf ($2d$ Stokes waves: non-resonant case)}
If $V_{\hat \jmath}=V_j$ is $2$ dimensional, any $2d$-Stokes wave solution in $\cX_{j}^{2d}$ is, up to translation of $\vartheta \in S^1 $,  an analytic curve of functions
\begin{equation}\label{2dstokesNR} u(\varepsilon)=\varepsilon v_j^{(1)}+w(c^{\parallel_{\hat\jmath}}_*(\varepsilon)\hat \jmath+c_*^{\bot_{\hat \jmath}}\hat \jmath^\bot,\varepsilon v_j^{(1)})
%=\varepsilon v_j^{(1)} \cO(\varepsilon^2)
\end{equation}
with observed speed
$ c_*^{\parallel_{\hat \jmath} }(\varepsilon ) \hat \jmath = 
c_* \cdot {\hat \jmath} + \cO(\varepsilon^2)$
(see e.g. \cite{BBMM}[page 16] or \eqref{le2dNR})   
and momentum 
$ \cI(u(\varepsilon))=-\tfrac12 j \varepsilon^2+\cO(\varepsilon^3) $ (cf. \eqref{Momuntum in coordinates}). 
\\[1.5mm]
$ \bullet $ {\bf ($2d$ Stokes waves: resonant)} 
% If $ j \parallel j_* $, 
In this case 
$  V_{\hat \jmath} $ is $ 4 $-dimensional. 
This is the resonant bifurcation 
problem of  $2d$-Stokes waves fully described   
in \cite{BBMM}[Thm. 1.3, 1.4].

\paragraph{\bf $3d$ Stokes waves. Non-resonant case.} \label{dim4R} 
Suppose 
\begin{equation}\label{casoNR}
\ker(\cL_{c_*}) = V = 
V_{j_*}
\oplus V_j = \Big\{ v=  \alpha_{j_*} v_{j_*}^{(1)} +
\beta_{j_*} v_{j_*}^{(2)} +
\alpha_{j} v_{j}^{(1)} +
\beta_{j} v_{j}^{(2)} \colon \alpha_{j_*}, \alpha_{j}, \beta_{j*}, \beta_{j} \in \R\Big\}
\end{equation}
where  
$ j, j_* \in \Gamma' $ satisfy 
$ \omega(j_* )=c_*\cdot j_*  $
and 
$ \omega(j )=c_*\cdot j  $ and 
$j \not \parallel j_* $. 
The symmetries \eqref{symmetriesPhi} show,
recalling 
\eqref{eq:trasab}, that 
(for simplicity we denote
$ \alpha_* = \alpha_{j_*}$
and $ \beta_* = \beta_{j_*}$), for any  
$\theta \in \R^2 $, 
\begin{equation}\label{sim3dNR} 
\begin{aligned}
& \begin{aligned}
\Phi (c, \alpha_*, &\beta_*, \alpha_j, \beta_j ) = \Phi \Big( c,  
R(-j_* \cdot  \theta)(\alpha_*, \beta_* ),
R(-j \cdot \theta)(\alpha_j, \beta_j ) \Big) \, , 
\end{aligned} \\
&  \Phi (c, \alpha_*, \beta_*, 
\alpha_j, \beta_j) = \Phi (c,  \alpha_*, - \beta_*, \alpha_j, - \beta_j) \, .  
\end{aligned}
\end{equation}
where $R(\cdot)$ is the rotation matrix 
$
R(\varphi):=\begin{psmallmatrix}
        \cos(\varphi) & -\sin(\varphi)\\
        \sin(\varphi) & \cos(\varphi)
    \end{psmallmatrix}
$.
Here we  denoted  
$ \Phi (c, v) $ equivalently as 
$ \Phi (c, (\alpha_j,\beta_j)_{
j \in {\mathcal V}}) $.
Since $ j, j_* $ are independent,  the linear map
$ \R^2 \to \R^2 $, $ \theta \mapsto 
(j_* \cdot \theta, j \cdot \theta ) $
is invertible, and therefore 
$\Phi (c, \alpha_*, \beta_*, \alpha_j, \beta_j ) $
is  a function  of 
$ \alpha_*^2 + \beta_*^2 $ and  
$ \alpha_j^2 + \beta_j^2 $ only,  
% and even in $ \beta_* $, 
for any $c $. Thus all the 
critical points of $ \Phi (c, \alpha_*, \beta_*, \alpha_j, \beta_j  ) $
 are 
obtained by rotations of  
critical points of the function 
$ (\alpha_*, \alpha_j) \mapsto 
 \Phi (c, \alpha_*, 0, \alpha_j, 0 )  $ 
of two variables only.  
By  Proposition \ref{lem:expa}, \eqref{Momuntum in coordinates}
and the symmetry \eqref{sim3dNR},
this function has the form 
\begin{equation}\label{exparedur}
\Phi (c, \alpha_*,0, \alpha_j,0)  =
- \frac12 (c-c_*) \cdot  j_* \, \alpha_*^2 - \frac12 
(c-c_*) \cdot j \, \alpha_j^2 + 
\underbrace{G_{\geq 3}(c,\alpha_*,0,\alpha_j,0)}_{=: G(c,\alpha_*^2,\alpha_j^2)} \, ,
\end{equation}
where 
$ (c,y_1,y_2) \mapsto G(c,y_1, y_2)$ 
is an analytic function 
satisfying $ G(c,0,0) = 0 $, 
$ \partial_{y_1} G(c,0,0) = 
\partial_{y_2} G(c,0,0) = 0 $ for any $ c \in B_r (c_*) $.  
We have to find zeros  of 
$$
\begin{aligned}
& - (\pa_{\alpha_*} \Phi)(c, \alpha_*,0,\alpha_j,0) =  
  (c-c_*) \cdot j_* \, \alpha_*  -  2 (\pa_{y_1} G) (c, \alpha_*^2, \alpha_j^2) \alpha_* = 0 \\
&  - (\pa_{\alpha_j} \Phi)(c, \alpha_*,0,\alpha_j,0) =    (c-c_*) \cdot j \, \alpha_j  -  2 (\pa_{y_2} G) (c, \alpha_*^2, \alpha_j^2) \alpha_j = 0
\end{aligned}
$$
or  equivalently, dividing the first equation by 
$ \alpha_* \neq 0 $ and the second one by   $ \alpha_j \neq 0 $,  
\begin{equation}\label{le2dNR}
    \begin{aligned}
  (c-c_*) \cdot j_*   -  2 (\pa_{y_1} G) (c, \alpha_*^2, \alpha_j^2)  = 0 \, , \\
  (c-c_*) \cdot j   -  2 (\pa_{y_2} G) (c, \alpha_*^2, \alpha_j^2)  = 0 \, .
  \end{aligned}
\end{equation}
Since $ j_*, j $ are independent,  
by the  implicit function theorem
there exists 
a  unique analytic solution 
$ c(\alpha_*^2, \alpha_j^2) $ of \eqref{le2dNR}, defined for any 
$ (\alpha_*, \alpha_j) $ in a small neighborhood of 
$ (0,0) $ in $ \R^2 $, satisfying 
$ c(0,0)  =  c_*  $.  If $ \alpha_* 
\alpha_j \neq 0 $ 
these  waves are truly $ 3 d$.   
If $ \alpha_j  \alpha_* = 0 $  they %  Stokes wave 
are $ 2 d $.

This result improves \cite{CN}[Theorem 4.3] since it holds in a full  neighborhood of $ 0 $,
namely also near   $ \alpha_* 
\alpha_j = 0 $. 

\smallskip

In the {\it resonant} cases 
when  
$ \dim \ker {\cL_{c_*}} \geq 6 $ the choice of the speed vector $ c $ (which a $ 2 $ dimensional vector) does not determine uniquely a solution anymore. 
In the next section   \ref{sec:cv} we choose properly $ c := c(v) $ 
as a function of $v $. Then  in Section \ref{sec:final} variational/topological  arguments will determine the existence of suitable  $ v $'s from which non-zero Stokes wave bifurcate.

\section{Construction of $ c(v) $}\label{sec:cv}

The main result of  this section is to   construct a  speed 
$ c(v) \in \R^2 $, close to 
$ c_* $, solving 
the system 
\begin{equation}\label{EQ c Phi}
\di_v \Phi (c, v) [\nabla_v \cI (v)] = 
        \begin{pmatrix}
        \di_v \Phi (c, v) [\nabla_v \cI_1 (v)]   \\
        \di_v \Phi (c, v) [\nabla_v \cI_2 (v)] 
        \end{pmatrix} = 0
        % \forall v \in B_{r'}^V 
% \setminus {\bf V}^{\twod} 
\, , 
\end{equation}
for {\it any} $ v $ in  a 
small neighborhood  $ B_{r'}^V $
of the origin, 
        where $\Phi(c,v)$ is the reduced Hamiltonian  in \eqref{def reduced functional} and $\nabla_v \cI (v) $ is defined in \eqref{eq:Nabla cI}.

\smallskip

The construction of $ c (v) $ if very different if 
$ v $ is ``far" from $  {\bf V}^{2d}$ (this is the case addressed in \cite{CN}) or    
``near" $  {\bf V}^{2d}$. 
The difficulty is that as $ v $
 approaches  $ {\bf V}^{\twod} $ the construction
 becomes singular  as the bifurcation problem tends to the $2d$ one, and 
% In order to take into account this degeneracy i
it is necessary to exploit the $ \T^2_\Gamma$-symmetry of the problem.
% To compensate for this singularity, 
% Section \ref{sec:close} 
%and it is compulsory to exploit the symmetries induced by the $ 2 d $ existence theory.
  % (Lemma \ref{tilde cQ-c*Q d tilde cQ(v)}). 
In order to state precisely the result we define  the region 
\begin{equation}\label{defFr}
{\mathcal F}_{r'} := \Big\{ 
        v\in B_{r'}^V\setminus\{0\} \ | \  
        d(v,{\bf V}^{\twod}) \geq \sqrt{\disb/4}\, \|v\|
        \Big\} \, ,\quad 
        d(v, {\bf V}^{\twod}) := \min_{{\hat \jmath} \in \cD} \| \Pi_{\hat \jmath}^\bot v \| \, , 
\end{equation}
% for some $ r' \in (0,r) $, 
and  
%$ {\mathcal N}_{r'} \setminus {\bf V}^{2d}  $ where
\begin{equation}\label{defNr}
\cN_{r'} :=\bigcup_{\hat \jmath\in\cD} \cN_{\hat \jmath}(r') \, , \quad
%\Big\{v\in B_{r'}^V \ | \ \| \Pi_{\hat \jmath}^\bot v \| \leq \sqrt{\disb/4}\, \|v\| \Big\} \, , \qquad 
{\mathcal N}_{\hat \jmath}(r')  := \Big\{ 
        v\in B_{r'}^V \ | \  
        \|\Pi_{\hat \jmath}^\bot v\| \leq \sqrt{\disb/4}\, \|v\|
        \Big\}  \, . 
\end{equation}
In \eqref{defFr}
and \eqref{defNr} we choose 
    \begin{equation}\label{eq:def b}    
        \disb:=
        \min \Big\{(\hat\jmath'\cdot \hat\jmath^\bot)^2 \ | \ j\not\parallel j' \Big\} \, .
    \end{equation}
    Note that $\disb\in (0,1]$.
    The  role of the constant $ \disb $ enters in Lemma \ref{Properties of QaQT}.

        \begin{proposition}\label{Construction of c close and away from PJ} {\bf (Construction of $ c(v) $)}
            There exist $r'\in (0,r)$  and a $\T^2_\Gamma\rtimes \Z_2$-invariant %analytic 
            function
            $                c:B_{r'}^V
            % \setminus {\bf V}^{\twod}
            \to B_r(c_*) $ 
            solving  system \eqref{EQ c Phi},  satisfying
 \begin{equation}\label{cQ-c*Q d cQ(v)}
                    |c(v)-c_*|\lesssim \| v \| \, ,  \quad \forall   v\in B_{r'}^V \, , 
            \end{equation} 
            and the following properties:
            \begin{itemize}
            \item {\bf (Solution outside ${\bf V}^{2d}$)} The function   
            $ c : B_{r'}^V
             \setminus {\bf V}^{\twod} \to B_r(c_*) $ is analytic and it is the unique solution of \eqref{EQ c Phi}.
             Moreover 
        \begin{equation}\label{|c0-c*| |dc0|}
            \|\di_v c(v)[\hat v]\|\lesssim \|\hat v\|\, , \quad \forall v\in \cF_{r'} \, , \quad \forall 
            \hat v\in V \, , 
        \end{equation}
        and for any $v\in \cN_{\hat \jmath}(r')\setminus V_{\hat \jmath}$, $\hat \jmath \in \cD$, the parallel and orthogonal components $c^{\parallel_{\hat \jmath}}(v) = c(v) \cdot \hat \jmath $ and $c^{\bot_{\hat \jmath}}(v) = c(v) \cdot \hat \jmath^\bot $
        satisfy 
            \begin{equation}\label{derivicino}
                            |\di_v c^{\parallel_{\hat \jmath}}(v) [\hat v]|\lesssim \|\Pi_{\hat \jmath} \hat v\|+\frac{\|\Pi_{\hat \jmath}^\bot v\|}{\| v \|}\|\Pi_{\hat \jmath}^\bot  \hat v \|\, , \quad 
                                |\di_v c^{\bot_{\hat \jmath}}(v)[\hat v]|\lesssim \|\Pi_{\hat \jmath} \hat v\|+\frac{\|v\|}{\| \Pi_{\hat \jmath}^\bot v \|}\|\Pi_{\hat \jmath}^\bot  \hat v \| \, . 
            \end{equation}  
             %, and  for any ${\hat \jmath}\in \cD$  the components $ c^{\parallel_{\hat \jmath}} (v)  = c(v) \cdot \hat\jmath $,  $ c^{\bot_{\hat \jmath}} (v)  = c(v) \cdot \hat\jmath^\bot  $ satisfy, for  any $\hat v\in V $, 
            %\begin{equation}
            %    \|\di_v c(v)[\hat v]\|\lesssim \|\hat v\|\, , \quad \text{if}\quad  d(v,{\bf V}^{2d})>\tfrac12 \|v\| \quad \text{and}
            %\end{equation}
    \noindent
           \item  {\bf (Solution on ${\bf V}^{2d}$)}
            for any 
    $ \hat \jmath  \in \cD  $ the function 
    $c  : 
         (V_{\hat \jmath} \cap B_{r'}^V )\setminus\{0\}\to  \R $ 
    is analytic and
    any other solution of \eqref{EQ c Phi} is of the kind $c(v)+k\hat \jmath^\bot$ for $k\in \R$.
\item 
{\bf (Regularity of $c^{\parallel_{\hat \jmath}}(v)$)}
for any $ \hat \jmath \in \cD $
   the parallel component $c^{\parallel_{\hat \jmath}}(v)  $  is  of class 
        \begin{equation}\label{c1Q U tilde c1Q}
         C(B_{r'}^{\hat \jmath})\cap C^1(B_{r'}^{\hat \jmath}\setminus\{0\})\quad \text{where}\quad B_{r'}^{\hat \jmath} := 
         (B_{r'}^V\setminus {\bf V}^{\twod})\cup (B_{r'}^V\cap V_{\hat \jmath}). 
        \end{equation}
    \end{itemize}
\end{proposition}
%  The function $c$ may not be continuous, as simple examples show.
  Proposition  \ref{Construction of c close and away from PJ}  results from the construction of $ c_{\threed} (v) $ far from $ {\bf V}^{\twod}$ 
  (Proposition \ref{Construction of c}), near ${\bf V}^{\twod}$ (Proposition \ref{c near PJ}) and on $ {\bf V}^{2d}$ (Lemma \ref{tilde cQ-c*Q d tilde cQ(v)}). 

 \subsection{Construction of $ c (v) $ far away 
 from $ {\bf V}^{\twod}$}\label{sec:far}

In this section we prove the 
following result.

\begin{proposition}\label{Construction of c}
    {\bf (Speed  $c_{\threed}(v)$ in $\cF_{r'}$) 
    %far away from $ {\bf V}^{\twod}$)
    }
    There exists $ r'\in (0,r)$ and a $\T^2_\Gamma\rtimes \Z_2$-invariant analytic function 
    $
        c_{\threed}: \cF_{r'} \to B_{r}(c_*) 
    $ 
    which is the unique solution of \eqref{EQ c Phi}, 
    satisfying \eqref{cQ-c*Q d cQ(v)} in $\cF_{r'}$ and \eqref{|c0-c*| |dc0|}.
\end{proposition}

%Note that for any $ v $ in $ \cF_{r'}$  in \eqref{defFr},the norms  $ \| \Pi_{\hat \jmath}^\bot v\| \sim \| v \| $  are equivalent for any $ {\hat \jmath} \in \cD $, and therefore  \eqref{|c0-c*| |dc0|} amounts to  \eqref{derivicino}.\mass{togliere 6.19}

We now prove Proposition \ref{Construction of c}. 
 In view of
\eqref{exp1Phi},  
system \eqref{EQ c Phi} is equivalent to
    \begin{equation}\label{EQ' c}
    \cA(v)(c-c_*)+ \tG_{\geq 3}(c,v)
         = 0
    \end{equation}
    where $  \cA (v) $ is  the  
    symmetric matrix
    \begin{equation}\label{Def a}
        \cA (v) :=
        \begin{pmatrix}
            \cA_{11} (v)  & \cA_{12} (v)   \\
             \cA_{12} (v)  &  \cA_{22} (v) 
        \end{pmatrix}  = 
        \begin{pmatrix}
            \langle \nabla_v \cI_1(v), \nabla_v \cI_1(v)\rangle  &
            \langle \nabla_v \cI_2(v), \nabla_v \cI_1(v)\rangle \\
            \langle \nabla_v \cI_1(v), \nabla_v \cI_2(v)\rangle  &
            \langle \nabla_v \cI_2(v), \nabla_v \cI_2(v)\rangle 
        \end{pmatrix}  
    \end{equation} 
and  
$$ 
\tG_{\geq 3}(c,v) 
:= \begin{pmatrix}
\di_v G_{\geq 3}(c,v)[\nabla_v \cI_1(v)] 
\\    
\di_v G_{\geq 3}(c,v)[\nabla_v \cI_2 (v)]
\end{pmatrix} \, . 
$$ 
The function 
$ \tG_{\geq 3} : B_r (c_*)\times B_r^V \to \R^2 $ is analytic and satisfies 
\begin{equation}\label{Ggeq3}
\di_v^\ell \tG_{\geq 3} (c,0) = 0 \, , \quad 
\forall \ell = 0,1,2\, , \quad \forall c \in B_r (c_*) \, . 
\end{equation}

    \begin{lemma} \label{Properties of a}
   The matrix  $ \cA(v)$ in \eqref{Def a} is  equal to   
        \begin{equation} \label{NablaI(v)NablaI(v) in coordinates}
            \cA(v)= \sum_{j\in \cV}
            \|\pi_{j} v\|^2\, \hat\jmath\otimes \hat\jmath = \sum_{{\hat \jmath}\in \cD}
            \|\Pi_{\hat \jmath} v\|^2
            \, \hat\jmath\otimes \hat\jmath \, ,
            \quad  \forall 
            v \in V \, , 
        \end{equation}
    where $ \hat\jmath\otimes \hat\jmath $ is the symmetric, rank $ 1 $, positive semi-definite  matrix  $\hat \jmath \otimes \hat \jmath:=\hat \jmath \hat \jmath^\top$.
    For any $v\in V$
    \begin{equation}\label{detAespli}
     \det(\cA(v))=\frac12\sum_{\hat \jmath,\hat \jmath' \in \cD} \det([\hat \jmath| \hat \jmath'])^2\, \|\Pi_{\hat \jmath} v\|^2\|\Pi_{\hat \jmath'} v\|^2
    \end{equation}
    where $[\hat\jmath| \hat \jmath']$ is the matrix with columns  $\hat \jmath$ and $\hat \jmath'$.
      The matrix $ \cA(v) $ is singular    on the set 
        \begin{equation}\label{Def S*}
            \Big\{ v \in V \ |\ \det( \cA ( v))  =0 \Big\} = {\bf V}^{\twod}  = 
            \bigcup_{{\hat \jmath} \in \cD} V_{\hat \jmath} \, ,    
        \end{equation}
      % (cf. \eqref{V2d}), 
      more precisely 
    $ \det (\cA (v)) > 0 $  on  
$ d(v,{\bf V}^{\twod}) > 0 $. 
        \end{lemma}

 \begin{proof}
    Since the projectors $\Pi_{\hat \jmath}$ are orthogonal, using \eqref{eq:Nabla cI}, we deduce \eqref{NablaI(v)NablaI(v) in coordinates}.
    The determinant is multilinear and skew-symmetric in the columns and the rows, so writing
    $ \hat\jmath 
    = (\hat\jmath_1, 
    \hat \jmath_2)^\top  $, we get
        $$
            \begin{aligned}
                \det(\cA(v))&=\det\Big[\sum_{\hat \jmath\in \cD} \hat \jmath_1 \hat \jmath\|\Pi_{\hat \jmath} v\|^2\Big| \sum_{\hat \jmath'\in \cD} \hat \jmath_2' \hat \jmath'\|\Pi_{\hat \jmath} v\|^2\Big]\\
                &=\sum_{\hat \jmath,\hat \jmath'\in \cD} \hat \jmath_1\hat \jmath_2'\det[\hat \jmath|\hat \jmath']\, \|\Pi_{\hat \jmath} v\|^2 \|\Pi_{\hat \jmath'} v\|^2\\
                &=\frac12 \sum_{\hat \jmath,\hat \jmath'\in \cD} \big(\hat \jmath_1\hat \jmath_2'\det[\hat \jmath|\hat \jmath']+\hat \jmath_1'\hat \jmath_2\det[\hat \jmath'|\hat \jmath]\big)\, \|\Pi_{\hat \jmath} v\|^2 \|\Pi_{\hat \jmath'} v\|^2\\
                &=\frac12\sum_{\hat \jmath,\hat \jmath'\in \cD} \det([\hat \jmath|\hat \jmath'])^2\, \|\Pi_{\hat \jmath} v\|^2 \|\Pi_{\hat \jmath'} v\|^2 
            \end{aligned}
        $$
        proving \eqref{detAespli}. 
        Therefore $\det(\cA(v))>0$ if and only if there exists two distinct resonant wave directions $\hat \jmath,\hat \jmath'\in \cD$ such that $\|\Pi_{\hat \jmath} v\|\|\Pi_{\hat \jmath'}v\|>0$. This proves 
         \eqref{Def S*},   recalling \eqref{Def piJ}, \eqref{V2d}.
    \end{proof}
    If  $v\in V$ satisfies $d(v,{\bf V}^{\twod})\geq \sqrt{\disb/4}\|v\|$, cf. \eqref{defFr},
    there exist two distinct $\hat \jmath,\hat \jmath'\in \cD$ such that
    $
        \|\Pi_{\hat \jmath}v\|^2\geq  \tfrac{\disb/4}{\#\cD-1}\|v\|^2 $ and $ \|\Pi_{\hat \jmath'}v\|^2\geq  \tfrac{\disb/4}{\#\cD-1}\|v\|^2 $. Therefore  the determinant 
        \eqref{detAespli}
        satisfies  
\begin{equation}\label{lem:c away from S* estimate 30}
\det(\cA(v))\gtrsim \|v\|^4
\end{equation}
and, using also \eqref{NablaI(v)NablaI(v) in coordinates}, 
    \begin{equation}\label{lem:c away from S* estimate 3}
     \cA(v)^{-1}=\frac{1}{\det(\cA(v))}
        \left(\begin{matrix}
           \cA_{22}(v) &-\cA_{12}(v) \\
           -\cA_{12}(v) &\cA_{11}(v) 
        \end{matrix}\right)\lesssim \frac{1}{\|v\|^2} \, , 
        \quad \forall 
        v \in \cF_{r'} \, , 
    \end{equation}
    meaning that each component is bounded by $\lesssim \tfrac{1}{\|v\|^2}$ (actually the matrix $ \cA(v) $ in   
    \eqref{NablaI(v)NablaI(v) in coordinates} has 
    an explicit inverse,   
    that however we shall not use).

System \eqref{EQ' c} is  equivalent, 
for any $ v \in B_{r}^V 
\setminus {\bf V}^{\twod} $, to find a fixed point of
    \begin{equation}\label{Def G}
         c =c_* - \cA(v)^{-1}\tG_{\geq 3}(c,v)  =: \cG(c,v) \, .  
    \end{equation}

\begin{lemma}
There exists
$r'\in(0,r)$ such that for any 
$ v\in \cF_{r'} $, 
the map  $ \cG (\cdot, v) $ is a contraction in the ball
$ B_{r}(c_*) $.  
The unique fixed point solution $ c_{\threed}(v) $ of \eqref{Def G} 
in $ B_r(c_*) $ is analytic in $ v $ and 
    satisfies \eqref{cQ-c*Q d cQ(v)} in $\cF_{r'}$ and  \eqref{|c0-c*| |dc0|}. 
\end{lemma}

\begin{proof}
    By \eqref{Ggeq3}, it results   $\tG_{\geq 3}(c,v)=\cO(\|v\|^3)$, uniformly in $ c $, and, using also \eqref{lem:c away from S* estimate 3},  
    \begin{equation}\label{lem:c away from S* estimate 4}
        \cG(c,v)=c_*+\cO(\|v\|)\, ,\quad \partial_{c} \cG(c,v)=\cO(\|v\|)\, .
    \end{equation}
    By \eqref{lem:c away from S* estimate 4} we deduce that if $r'$ is sufficiently small, for any  $ v \in B_{r'}^V $, the map  $\cG(\cdot,v)$ is a contraction 
    in $ B_r(c_*) $.
    The unique solution $ c_{\threed}(v)$ of $
    c_{\threed}(v) = \cG(c_{\threed}(v),v) $ satisfies \eqref{cQ-c*Q d cQ(v)} and is analytic 
    in $v $ since $\cG$ is analytic. 
    Differentiating the equation $\cA(v)(c_{\threed}(v)-c_*)+\tG_{\geq 3}(c_{\threed}(v) ,v)=0$ we get
    $$
        \di_v\cA(v)[\hat v](c_{\threed}(v)-c_*)+\cA(v)\di_v c_{\threed}(v)[\hat v]
        =-\di_v\tG_{\geq 3}(c_{\threed}(v),v)[\hat v]-\partial_c\tG_{\geq 3}(c_{\threed}(v),v) \di_vc_{\threed}(v)[\hat v] \, .
    $$
    Since $\di_v\cA(v)[\hat v]=\cO(\|v\|\|\hat v\|)$, $\tG_{\geq 3}(c,v)=\cO(\|v\|^3)$ and $c_{\threed}(v)-c_*=\cO(\|v\|)$ we 
    deduce that 
    $$
        \di_vc_{\threed}(v)[\hat v]=(\cA(v)+\partial_c\tG_{\geq 3}(c_{\threed}(v),v)\big)^{-1}\cO(\|v\|^2\|\hat v\|)=\cO(\|\hat v\|) 
    $$
    which is     \eqref{|c0-c*| |dc0|}.
\end{proof}

\subsection{Construction of $ c(v) $ near $ {\bf V}^{\twod}$}
\label{sec:close}

We now construct the 
solution $ c_{\threed}(v) $ of system \eqref{EQ c Phi}  
in the region $\cN_{r'}  \setminus {\bf V}^{2d}$ (cf. \eqref{defNr})
for some $ r' \in (0,r)  $ eventually smaller than in Proposition \ref{Construction of c}. 
The main difficulty is that the matrix 
$ \cA (v)$ 
becomes singular when $ v $ approaches 
some subspace $ V_{\hat \jmath} $  
of  $ 2 d$-waves.  

\begin{proposition}\label{c near PJ}{\bf (Speed $c_{\threed}(v)$ 
in $\cN_{r'}\setminus{\bf V}^{2d}$) % near $ {\bf V}^{\twod}$)
}\label{lem:cJ}
    There is $r'\in (0,r)$ and a $\T^2_\Gamma\rtimes \Z_2$-invariant  analytic function
    $
        c_{\threed} : \cN_{r'}\setminus {\bf V}^{2d} \to B_r(c_*) $ 
 which is the unique solution of \eqref{EQ c Phi}, 
  satisfying \eqref{cQ-c*Q d cQ(v)} 
   in $\cN_{r'}\setminus {\bf V}^{2d}$   and \eqref{derivicino}.  
\end{proposition}

In order to highlight the singularity near each $ V_{\hat \jmath}$,  ${\hat \jmath}\in \cD$, we write  system \eqref{EQ c Phi} as 
\begin{equation}\label{EQ c Phij}
    \begin{pmatrix}
        \di_v\Phi(c,v)[\nabla_v\cI^{\parallel_{\hat \jmath}}(v)]\\
        \di_v\Phi(c,v)[\nabla_v\cI^{\bot_{\hat \jmath}}(v)]
    \end{pmatrix}=0 
\end{equation}
where $ \cI^{\parallel_{\hat \jmath}}$ and $\cI^{\bot_{\hat \jmath}}$ the parallel and orthogonal components of $\cI$ as defined by \eqref{eq:compoc}. Actually 
system \eqref{EQ c Phij} is obtained applying to \eqref{EQ c Phi}  the linear orthogonal matrix $ Q_{\hat \jmath}$  in 
\eqref{eq:def yj}. 
In view of \eqref{exp1Phi} and 
$ (c-c_*) \cdot  \cI (v) = 
(c-c_*)^{\parallel_{\hat \jmath}} \cI^{\parallel_{\hat \jmath}} (v) + 
(c-c_*)^{\bot_{\hat \jmath}} \cI^{\bot_{\hat \jmath}} (v) 
$, 
system  \eqref{EQ c Phij} is equivalent to
\begin{equation}\label{EQ c Linear system0}
    \cA^{\hat\jmath}(v)(c^{\hat \jmath}-c_*^{\hat \jmath})+ \tG^{\hat \jmath}_{\geq 3} (c^{\hat \jmath},v)=0
\end{equation}
where $\cA^{\hat \jmath}(v)$ is the symmetric  matrix 
\begin{equation}\label{Def aQ}
    \cA^{\hat \jmath}(v):= 
        \begin{pmatrix}
            \cA_{11}^{\hat \jmath} (v)  & \cA_{12}^{\hat \jmath} (v)   \\
             \cA_{12}^{\hat \jmath} (v)  &  \cA_{22}^{\hat \jmath} (v) 
        \end{pmatrix}
    = 
    \begin{pmatrix}
        \langle \nabla_v \cI^{\parallel_{\hat \jmath}}(v), \nabla_v \cI^{\parallel_{\hat \jmath}}(v)\rangle  &
        \langle \nabla_v \cI^{\bot_{\hat \jmath}}(v), \nabla_v \cI^{\parallel_{\hat \jmath}}(v)\rangle \\
        \langle \nabla_v \cI^{\parallel_{\hat \jmath}}(v), \nabla_v \cI^{\bot_{\hat \jmath}}(v)\rangle  &
        \langle \nabla_v \cI^{\bot_{\hat \jmath}}(v), \nabla_v \cI^{\bot_{\hat \jmath}}(v)\rangle 
    \end{pmatrix} \, , 
\end{equation}
 the vector $ c^{\hat \jmath}-c_*^{\hat \jmath}  = (c-c_*)^j = Q_{\hat \jmath} 
(c - c_* )$ (according to  \eqref{eq:def yj}) and 
\begin{equation}\label{Def tg}
    \tG_{\geq 3}^{\hat \jmath} (c^{\hat \jmath},v):=
    \vect{ \tG_{\geq 3}^{\parallel_{\hat \jmath}} (c^{\hat \jmath},v) }{\tG_{\geq 3}^{\bot_{\hat \jmath}} (c^{\hat \jmath},v)} =
    \vect{ \di_v G_{\geq 3}(c,v)[\nabla_v \cI^{\parallel_{\hat \jmath}}(v)]}{\di_v G_{\geq 3}(c,v)[\nabla_v \cI^{\bot_{\hat \jmath}}(v)]}\, .
\end{equation}
System \eqref{EQ c Linear system0} also follows by \eqref{EQ' c} since, by 
\eqref{Def aQ},  \eqref{Def a} and  
$ Q_{\hat \jmath}^\top = Q_{\hat \jmath}^{-1} $, 
\begin{equation}\label{coniugAQ}
\cA^{\hat \jmath}(v)=Q_{\hat \jmath}\cA(v)Q_{\hat \jmath}^{-1}\, .
\end{equation} 
The components of $ \cA^{\hat \jmath} (v)$ 
%are quadratic forms in 
%$ v = \Pi_{\hat \jmath} v + \Pi_{\hat \jmath}^\bot v $
have  different sizes in $ \Pi_{\hat \jmath} v $  and 
 $ \Pi_{\hat \jmath}^\bot v $.

    \begin{lemma}{\bf (Properties of $\cA^{\hat \jmath}(v) $)}\label{Properties of QaQT}
 For any  ${\hat \jmath} \in \cD$, for any $v\in V$, 
    \begin{equation}\label{Estimate on a(v,v)}
\left(\begin{matrix}
                \|v\|^2 & \| \Pi_{\hat \jmath}^\bot v \|^2 \\
                \|\Pi_{\hat \jmath}^\bot v \|^2 & \| \Pi_{\hat \jmath}^\bot v \|^2
            \end{matrix}\right)
\geq 
   \cA^{\hat \jmath} (v) \geq 
   \left(\begin{matrix}
                \|\Pi_{\hat \jmath} v\|^2 & -\| \Pi_{\hat \jmath}^\bot v \|^2 \\
                -\|\Pi_{\hat \jmath}^\bot v \|^2 &  \disb \| \Pi_{\hat \jmath}^\bot v \|^2
            \end{matrix}\right)        
        \end{equation}     
        componentwise, 
           where  $ \disb \in (0,1] $ is defined in \eqref{eq:def b}. 
        For any $ \hat v \in V $, 
\begin{align}\label{Estimates on d(A(v+w))}
            \di_v& \cA^{\hat \jmath} (v) [\hat v]=\cO\left(\begin{matrix}
                \|\Pi_{\hat \jmath} v\|\|\Pi_{\hat \jmath} \hat v\|+\|\Pi_{\hat \jmath}^\bot v\|\|\Pi_{\hat \jmath}^\bot \hat v\| & \| \Pi_{\hat \jmath}^\bot v \|\|\Pi_{\hat \jmath}^\bot \hat v\|\\
               \| \Pi_{\hat \jmath}^\bot v \|\|\Pi_{\hat \jmath}^\bot \hat v \| & \| \Pi_{\hat \jmath}^\bot v \|\|\Pi_{\hat \jmath}^\bot \hat v\|
            \end{matrix}\right) \, . 
        \end{align}
\end{lemma}

\begin{proof} 
By \eqref{coniugAQ},
\eqref{NablaI(v)NablaI(v) in coordinates}, 
$ Q_{\hat \jmath}^\top = Q_{\hat \jmath}^{-1} $,  
and recalling \eqref{eq:compoc}, \eqref{eq:def yj},
we get 
   \begin{align*}
                \cA^{\hat \jmath}& (v) 
                 = \sum_{{\hat \jmath'} \in \cD} \|\Pi_{\hat \jmath'}  v\|^2 (Q_{\hat \jmath} \hat\jmath')\otimes (Q_{\hat \jmath} \hat\jmath')
                 \, , \\
 &(Q_{\hat \jmath} \hat\jmath')\otimes(Q_{\hat \jmath} \hat\jmath')=\begin{pmatrix}(\hat\jmath'\cdot\hat\jmath)^2 & (\hat\jmath'\cdot\hat\jmath)(\hat\jmath'\cdot\hat\jmath^\bot)\\  (\hat\jmath'\cdot\hat\jmath)(\hat\jmath'\cdot\hat\jmath^\bot) & (\hat\jmath'\cdot\hat\jmath^\bot)^2  
        \end{pmatrix}.      
            \end{align*} 
    Therefore 
    \begin{align*}
       &      \|v\|^2\geq \cA^{\hat \jmath}_{11}(v)=
       \|\Pi_{\hat \jmath} v\|^2 + \sum_{{\hat \jmath'} \in \cD\setminus [j]}  (\hat\jmath'\cdot  \hat\jmath)^2 \|\Pi_{\hat \jmath'} v\|^2 \geq \|\Pi_{\hat \jmath} v\|^2 \\
        & 
            \|\Pi_{\hat \jmath}^\bot v\|^2\geq \cA^{\hat \jmath}_{21}(v)=\cA^{\hat \jmath}_{12}(v)=\sum_{{\hat \jmath'} \in \cD\setminus [j]}  (\hat\jmath'\cdot  \hat\jmath)(\hat\jmath'\cdot  \hat\jmath^\bot) \|\Pi_{\hat \jmath'} v\|^2\geq -\|\Pi_{\hat \jmath}^\bot v\|^2 \\
& 
            \|\Pi_{\hat \jmath}^\bot v\|^2\geq \cA^{\hat \jmath}_{22}(v)=\sum_{{\hat \jmath'} \in \cD\setminus [j]}  (\hat\jmath'\cdot  \hat\jmath^\bot)^2 \|\Pi_{\hat \jmath'} v\|^2 \stackrel{\eqref{eq:def b}} 
            \geq \disb \|\Pi_{\hat \jmath}^\bot v\|^2 \, . 
\end{align*}
    Formula  \eqref{Estimates on d(A(v+w))} follows directly 
    by differentiation.
\end{proof}

The lower bound \eqref{Estimate on a(v,v)} implies that, 
for any $v\in B_r^V$ satisfying $0<\|
    \Pi_{\hat \jmath}^\bot v  \|\leq \sqrt{\disb/4}\, \|v\|$,  
$$
\det(\cA^{\hat \jmath}(v))  \geq \det   
        \left(\begin{matrix}
                \|\Pi_{\hat \jmath} v\|^2 & -\| \Pi_{\hat \jmath}^\bot v \|^2 \\
                -\|\Pi_{\hat \jmath}^\bot v \|^2 & \disb  \| \Pi_{\hat \jmath}^\bot v \|^2
            \end{matrix}\right) \geq  \frac \disb 2\|v\|^2\|\Pi_{\hat \jmath}^\bot v \|^2>0\, ,
 $$
 and, using also  the upper bound  
    \eqref{Estimate on a(v,v)}, 
    \begin{equation}\label{lem:c near PJ estimate 3}
        \cA^{\hat \jmath}(v)^{-1} \! = \! \frac{1}{\det(\cA^{\hat \jmath}(v))}
        \left(\begin{matrix}
           \cA_{22}^{\hat \jmath}(v)  & -\cA_{12}^{\hat \jmath}(v) \\
           -\cA_{12}^{\hat \jmath}(v) &  \cA_{11}^{\hat \jmath}(v) 
        \end{matrix}\right) \leq \frac{2}{\disb} \left(\begin{matrix}
            \frac1{\|v\|^2} & \frac1{\|v\|^2}\\
            \frac1{\|v\|^2}  & \frac1{\|\Pi_{\hat \jmath}^\bot v \|^2} 
        \end{matrix}\right) \, . 
    \end{equation}
System \eqref{EQ c Linear system0} is thus equivalent,
for any $v \neq 0 $, to find a
fixed point of 
\[
c^{\hat \jmath} = c^{\hat \jmath}_*-\cA^{\hat \jmath} (v)^{-1} \tG^{\hat \jmath}_{\geq 3}  (c^{\hat \jmath},v) =: \cG^{\hat \jmath}(c^{\hat \jmath},v) \, . 
\]
The next key  estimates  
enable to compensate the singularity 
of $ \cA^{\hat \jmath} (v)^{-1} $ 
as $ \Pi_{\hat \jmath}^\bot v \to 0 $. 
They follow by  general properties  of  $\T^2_\Gamma$-invariant functions, proved in Appendix \ref{Appendix A}.

   \begin{lemma} 
    For any $v\in B_{r}^V$ and any $c^{\hat \jmath}\in B_{r}(c_*)$, it results
    \begin{align}  \label{Estimates on g}
  &       \tG^{\hat \jmath}_{\geq 3}(c^{\hat \jmath},v)
        =\cO\vect{\|v\|^3}{\| \Pi_{\hat \jmath}^\bot v \|^2\|v\|} \quad \text{
    uniformly in} \ c^{\hat \jmath} \, , 
         \\
         &\partial_{c^{\parallel_{\hat \jmath}}}
         \tG^{\hat \jmath}_{\geq 3}(c^{\hat \jmath},v)
                =\cO\vect{\|v\|^3}{\| \Pi_{\hat \jmath}^\bot v \|^2\|v\|} \, , \quad \partial_{c^{\bot_{\hat \jmath}}}\tG^{\hat \jmath}_{\geq 3}(c^{\hat \jmath},v)
        =\cO\vect{\| \Pi_{\hat \jmath}^\bot v \|^2\|v\|}{\| \Pi_{\hat \jmath}^\bot v \|^2\|v\|} \, , 
        \label{Partial c1 g}
    \end{align}
where 
  $ \partial_{c^{\parallel_{\hat \jmath}}}=\hat\jmath \cdot \partial_c $ and  $  \partial_{c^{\bot_{\hat \jmath}}}=\hat\jmath^\bot \cdot \partial_c $
denote directional derivatives.
    \end{lemma}
    
    \begin{proof}
        Using \eqref{Ggeq3} 
        the  component $\tG_{\geq 3}^{\parallel_{\hat \jmath}}(c^{\hat \jmath},v) $ 
        in \eqref{Def tg} 
  satisfies  $\tG_{\geq 3}^{\parallel_{\hat \jmath}}(c^{\hat \jmath},v)=\cO(\|v\|^3)$. Moreover by \eqref{momezero2}, if $ v\in  V_{\hat \jmath} $ then 
 $ \nabla_v \cI^{\bot_{\hat \jmath}}(v) =0 $, 
 and therefore
        \begin{equation}\label{G3bot}
            \tG_{\geq 3}^{\bot_{\hat \jmath}}(c^{\hat \jmath},v) \stackrel{\eqref{Def tg} }=
            \di_v G_{\geq 3}(c,v)[
            \nabla_v \cI^{\bot_{\hat \jmath}}(v) ] = 0 \, , \quad \ \ \forall v\in B_{r}^V\cap V_{\hat \jmath}\, .
        \end{equation}
        By \eqref{Ggeq3} 
        we also have 
        $\tG^{\bot_{\hat \jmath}}_{\geq 3}(c^{\hat \jmath},v) = 
        \hat\jmath^\bot \cdot \tG^{\bot_{\hat \jmath}}_{\geq 3}(c^{\hat \jmath},v) =\cO(\|v\|^3)$ and  Lemma \ref{lm:A}-$(ii)$ implies the second estimate in \eqref{Estimates on g}.
        The first estimates in \eqref{Partial c1 g} follow by \eqref{Estimates on g}.
        By Proposition \ref{lem:expa} 
        we have
$ \partial_{c^{\bot_{\hat \jmath}}}\tG_{\geq 3}^{\hat \jmath}(c^{\hat \jmath},v)=0 $ for any 
        $ v\in B_{r}^V\cap V_{\hat \jmath} $. 
        Lemma \ref{lm:A}-$(ii)$ applied to $\partial_{c^{\bot_{\hat \jmath}}}\tG_{\geq 3}^{\hat \jmath}(c^{\hat \jmath},v)$ implies the second estimates in \eqref{Partial c1 g}.
    \end{proof}

    From \eqref{Def G}, \eqref{lem:c near PJ estimate 3},  \eqref{Estimates on g},
    \eqref{Partial c1 g} we deduce that for $r'\in (0,r)$ sufficiently small, the map  
    $$
        \cG^{\hat \jmath}:B_{r}(c_*)\times[ \cN_{\hat \jmath}(r') \setminus V_{\hat \jmath} ]\to \R^2 \, , \quad  \cG^{\hat \jmath}(c^{\hat \jmath},v):=c^{\hat \jmath}_*-\cA^{\hat \jmath} (v)^{-1}\tG^{\hat \jmath}_{\geq 3}(c^{\hat \jmath},v)\, , 
    $$
    satisfies
    \begin{equation}\label{lem:c near PJ estimate 4}
                |\cG^{\hat \jmath}(c^{\hat \jmath},v)-c^{\hat \jmath}_*|\lesssim \|v\|\, , \quad |\partial_{c^{\hat \jmath}}\cG^{\hat \jmath}(c^{\hat \jmath},v)|\lesssim \|v\| \, , 
    \end{equation}
    uniformly in $c$.
    We deduce that if $r'$ is sufficiently small $\cG^{\hat \jmath}(\cdot,v)$ is a contraction in $ B_r(c_*)  $ and hence there exists a unique fixed point 
    $c^{\hat \jmath}=c_{\threed}^{\hat \jmath}(v)$ of  $\cG^{\hat \jmath}(c^{\hat \jmath},v)=c^{\hat \jmath}$.
    Since $\cG^{\hat \jmath}$ is analytic the function $v\mapsto c^{\hat \jmath}_{\threed}(v)$ is analytic.
    Since $c_{\threed}^{\hat \jmath}(v)$ is a fixed point of $\cG^{\hat \jmath}(\cdot,v)$ from \eqref{lem:c near PJ estimate 4} we deduce
    \begin{equation}\label{lem:c near PJ estimate 5}
       c_{\threed}^{\hat \jmath}(v)=c^{\hat \jmath}_*+\cO(\|v\|) \, , 
       \quad \forall v\in \cN_{\hat \jmath}(r')\setminus V_{\hat \jmath}\, .
    \end{equation}   
  Thus \eqref{cQ-c*Q d cQ(v)} holds on $B_{r'}^V\setminus {\bf V}^{2d}$.

    \begin{lemma} 
    \eqref{derivicino} hold. 
    \end{lemma}

    \begin{proof}
     Differentiating $\cA^{\hat \jmath}(v)(c_{\threed}^{\hat \jmath}(v)-c^{\hat \jmath}_*)+\tG^{\hat \jmath}_{\geq 3}(c_{\threed}^{\hat \jmath}(v),v)=0$, for any $v\in \cN_{\hat \jmath} (r') \setminus V_{\hat \jmath} $ we have
    \begin{equation}\label{eq:dc}
        \big(\cA^{\hat \jmath}(v)+\partial_{c^{\hat \jmath}}\tG^{\hat \jmath}_{\geq 3}(c_{\threed}^{\hat \jmath}(v),v)\big)\di_v c_{\threed}^{\hat \jmath}(c)[\hat v]
        =-\di_v \cA^{\hat \jmath}(v)[\hat v](c_{\threed}^{\hat \jmath}(v)-c_*^{\hat \jmath})-\di_v \tG^{\hat \jmath}_{\geq 3}(c_{\threed}^{\hat \jmath}(v),v)[\hat v]\, .
    \end{equation}
      By \eqref{Estimate on a(v,v)}, the matrix 
    \begin{equation}\label{lem:c near PJ estimate 6}
        \cA^{\hat \jmath}(v)+\partial_{c^{\hat \jmath}}\tG^{\hat \jmath}_{\geq 3}(c^{\hat \jmath},v) \stackrel{\eqref{Partial c1 g}} =
            \cA^{\hat \jmath}(v)+
            \cO\left(\begin{matrix}
                \|v\|^3 & \|\Pi_{\hat \jmath}^\bot v \|^2\|v\|\\
                \| \Pi_{\hat \jmath}^\bot v \|^2\|v\| & \|\Pi_{\hat \jmath}^\bot v \|^2\|v\|
            \end{matrix}\right)  \, , \ v\in B_{r'}^V
    \end{equation}
   is invertible for any $v\in \cN_{\hat \jmath}(r') {\setminus V_{\hat \jmath}} $ provided $r'\in (0,r)$ is sufficiently small, and 
    \begin{equation}\label{Estimate on (Partial cQ FQ[hat v])-1}
        \big(\cA^{\hat \jmath}(v)+\partial_{c^{\hat \jmath}}\tG^{\hat \jmath}_{\geq 3}(c^{\hat \jmath},v)\big)^{-1}=
        \cO\begin{pmatrix}
                \frac1{\|v\|^2} & \frac1{\|v\|^2}\\
                \frac1{\|v\|^2} & \frac1{\|\Pi_{\hat \jmath}^\bot v \|^2}
            \end{pmatrix}\, .
    \end{equation}
    Using \eqref{G3bot}, \eqref{Ggeq3} and Lemma \ref{lm:A} we deduce
    \begin{equation}\label{eq:Gjecc}
        \di_v\tG^{\hat \jmath}_{\geq 3}(c^{\hat \jmath},v)[\hat v]=
        \cO \begin{pmatrix}             \|v\|^2\|\Pi_{\hat \jmath} \hat v\|+\|v\|\|\Pi_{\hat \jmath}^\bot v\|\|\Pi_{\hat \jmath}^\bot \hat v\|\\             \|\Pi_{\hat \jmath}^\bot v\|^2\|\Pi_{\hat \jmath} \hat v\|+\|v\|\|\Pi_{\hat \jmath}^\bot v\|\|\Pi_{\hat \jmath}^\bot \hat v\|         \end{pmatrix}\, , \ v\in B_{r'}^V \, . 
    \end{equation}  
    Hence, by \eqref{Estimates on d(A(v+w))}, \eqref{lem:c near PJ estimate 5} we obtain
    \begin{equation}\label{Estimate on dv FQ[hat v]}
        -\di_v \cA^{\hat \jmath}(v)[\hat v](c^{\hat \jmath}_{\threed}(v)-c_*^{\hat \jmath})-\di_v \tG^{\hat \jmath}_{\geq 3}(c^{\hat \jmath}_{\threed},v)[\hat v] = \cO \begin{pmatrix}             \|v\|^2\|\Pi_{\hat \jmath} \hat v\|+\|v\|\|\Pi_{\hat \jmath}^\bot v\|\|\Pi_{\hat \jmath}^\bot \hat v\|\\             \|\Pi_{\hat \jmath}^\bot v\|^2\|\Pi_{\hat \jmath} \hat v\|+\|v\|\|\Pi_{\hat \jmath}^\bot v\|\|\Pi_{\hat \jmath}^\bot \hat v\|         \end{pmatrix}.
    \end{equation}
    From \eqref{eq:dc}, \eqref{Estimate on (Partial cQ FQ[hat v])-1} and \eqref{Estimate on dv FQ[hat v]} we deduce \eqref{derivicino}.
    \end{proof}

\subsection{Construction  of $c(v)$ on $ {\bf V}^{\twod}$}\label{sec:speedcv}

For any  ${\hat \jmath} \in \cD$ 
and $ v \in V_{\hat \jmath} $ system \eqref{EQ c Phi} amounts to 
$$
\di_v \Phi (c, v) \Big[ \hat \jmath  
\nabla_v \cI^{\parallel_{\hat \jmath}} (v) + \hat \jmath^\bot  
\underbrace{\nabla_v \cI^{\bot_{\hat \jmath}}(v)}_{=0 \,  \text{by} \eqref{pibot}} \Big]  \!\stackrel{\eqref{Phi is independent on c2Q if v is in V2dj}} = \! 
\di_v \Phi (c^{\parallel_{\hat \jmath}}\hat \jmath+ c^{\bot_{\hat \jmath}}_*\hat \jmath^\bot, v) \Big[ \hat \jmath  
\nabla_v \cI^{\parallel_{\hat \jmath}} (v)  \Big] \! =\! 0  
$$
and we are reduced to look for a scalar  function 
$ c^{\parallel_{\hat \jmath}}(v) $ 
solving the equation 
\begin{equation}\label{eq:mf2d}
 \di_v\Phi(c^{\parallel_{\hat \jmath}}\hat \jmath +c^{\bot_{\hat \jmath}}_*\hat \jmath^\bot,v)[\nabla_v \cI^{\parallel_{\hat \jmath}}(v)] = 
 0 \, , \quad \forall v
 \in V_{\hat \jmath} \cap B_{r'}^V   
 \setminus \{0\} \, . 
\end{equation}

\begin{remark} For any $v\in B_{r}^V\cap V_{\hat \jmath}$, if $c\in B_{r}^V$ solves \eqref{EQ c Phi} then  $c^{\parallel_{\hat \jmath}}=c\cdot \hat \jmath$ solves \eqref{eq:mf2d}.
Viceversa if $c^{\parallel_{\hat \jmath}}$ solves \eqref{eq:mf2d} then $c^{\parallel_{\hat \jmath}}\hat \jmath+k\hat \jmath^\bot$ solves \eqref{EQ c Phi} for any $k\in \R$ such that $c^{\parallel_{\hat \jmath}}\hat \jmath+k\hat \jmath^\bot\in B_{r}^V$.
\end{remark}
In view of \eqref{EQ c Phij}, 
\eqref{EQ c Linear system0}-\eqref{Def tg}
and \eqref{momezero2}, the equation \eqref{eq:mf2d} is equivalent to
\begin{equation}\label{EQ c Linear system}
    \cA^{\hat \jmath}_{11}(v)(c^{\parallel_{\hat \jmath}}-c_*^{\parallel_{\hat \jmath}})+ 
    \tG^{\parallel \hat \jmath}_{\geq 3} (c^{\parallel_{\hat \jmath}},c_*^{\bot_{\hat \jmath}},v) = 0 
\end{equation}
where $ \cA^{\hat \jmath}_{11}(v) $ is defined in 
\eqref{Def aQ}
and 
$\tG^{\parallel \hat \jmath}_{\geq 3} (c^{\hat \jmath},v)$  in \eqref{Def tg}.

\begin{lemma}\label{c on PJ}
{\bf (Speed $ c^{\parallel_{\hat \jmath}}_{\twod} (v)$)}
    There exist $ r'\in (0,r)$ and, for any ${\hat \jmath} \in \cD$, a $\T^2_\Gamma\rtimes \Z_2$-invariant analytic function
    $
        c^{\parallel_{\hat \jmath}}_{\twod} : 
         (V_{\hat \jmath} \cap B_{r'}^V )\setminus\{0\}\to \big(c^{\parallel_{\hat \jmath}}_*-r,c^{\parallel_{\hat \jmath}}_* +r\big) $, $ v \mapsto c^{\parallel_{\hat \jmath}}_{\twod} (v) $, 
    which is the unique solution of \eqref{eq:mf2d}, satisfying 
    \begin{equation}\label{tilde cQ-c*Q d tilde cQ(v)}
        |  c^{\parallel_{\hat \jmath}}_{\twod} (v)-  c^{\parallel_{\hat \jmath}}_* |\lesssim \|v\| \, , \quad |\di_v c^{\parallel_{\hat \jmath}}_{\twod} (v)[\hat v]|\lesssim \|\hat v\|\, , \quad  \forall v \in V_{\hat \jmath} \cap B_{r'}^V, \,  \hat v \in V_{\hat \jmath} \, .
    \end{equation}
\end{lemma}

\begin{proof}
To simplify notation we set
$    \tG_{\twod}(c^{\parallel_{\hat \jmath}},v):=\tG^{\parallel_{\hat \jmath}}_{\geq 3}(c^{\parallel_{\hat \jmath}},c_*^{\bot_{\hat \jmath}} ,v) $.
In view of \eqref{Estimate on a(v,v)} the equation \eqref{EQ c Linear system} is equivalent, for any $v \neq 0 $,  to find a fixed point of  
$$
\cG_{\twod}:\big(c^{\parallel_{\hat \jmath}}_*-r,c^{\parallel_{\hat \jmath}}_* +r\big)\times B_{r'}^V\to \R \, , \quad
    \cG_{\twod}(c^{\parallel_{\hat \jmath}},v):=c_*^{\parallel_{\hat \jmath}}+(\cA_{11}^{\hat \jmath}(v))^{-1}\tG_{\twod}(c^{\parallel_{\hat \jmath}},v) \, , 
$$
for some $r'\in (0,r)$.
Since $\tG_{\twod}(c^{\parallel_{\hat \jmath}},v)=\cO(\|v\|^3)$ uniformly in $c^{\parallel_{\hat \jmath}}$, for any
$c^{\parallel_{\hat \jmath}}
  \in \big(c^{\parallel_{\hat \jmath}}_*-r,c^{\parallel_{\hat \jmath}}_* +r\big)$ and  any $ v \in B_{r'}^V \setminus \{0\} $,  we have 
\begin{equation}\label{G2d is a contraction}
 |\cG_{\twod}(c^{\parallel_{\hat \jmath}},v)|\lesssim \|v\| \, , \qquad  |\partial_{c^{\parallel_{\hat \jmath}}}\cG_{\twod}(c^{\parallel_{\hat \jmath}},v)|\lesssim \|v\|\, .
\end{equation}
Therefore,  for $r'\in (0,r)$ small
enough, the map
$
    \cG_{\twod}( \cdot , v ):\big(c_*^{\parallel_{\hat \jmath}}-r,c_*^{\parallel_{\hat \jmath}}+r \big)\to \big (c_*^{\parallel_{\hat \jmath}}-r,c_*^{\parallel_{\hat \jmath}}+r \big)
$  is a contraction and 
 for any $v\in B_{r'}^V$ there exists a unique fixed point $c_{\twod}^{\parallel_{\hat \jmath}}(v)$ of $\cG_{\twod}(\cdot,v)$. 
The function $v\mapsto c^{\parallel_{\hat \jmath}}_{\twod}(v)$ is analytic by applying the implicit function theorem to the analytic function 
$ \cG_{\twod}(c^{\parallel_{\hat \jmath}},v)$. 
The first bound in \eqref{tilde cQ-c*Q d tilde cQ(v)} follows from $\cG_{\twod}(c^{\parallel_{\hat \jmath}}_{\twod}(v),v)=c^{\parallel_{\hat \jmath}}_{\twod}(v)$ and \eqref{G2d is a contraction}.
Taking the differential of  \eqref{EQ c Linear system} which is satisfied identically for $c^{\parallel_{\hat \jmath}}=c^{\parallel_{\hat \jmath}}_{\twod}(v)$ we have
\begin{equation}\label{eq:dc2d EQ}
    \big( \cA_{11}^{\hat \jmath}(v)+\partial_{c^{\parallel_{\hat \jmath}}}\tG^{\parallel_{\hat \jmath}}_{\geq 3}(c^{\parallel_{\hat \jmath}},c^{\bot_{\hat \jmath}}_*,v)\big)\di_vc^{\parallel_{\hat \jmath}}(v)[\hat v]=
    -\di_v\cA^{\hat \jmath}_{11}(v)[\hat v](c^{\parallel_{\hat \jmath}}-c_*^{\parallel_{\hat \jmath}})-\di_v\tG^{\parallel_{\hat \jmath}}_{\geq 3}(c^{\parallel_{\hat \jmath}},c^{\bot_{\hat \jmath}}_*,v)[\hat v] \, .
\end{equation}
Thus using \eqref{Estimate on a(v,v)}, \eqref{lem:c near PJ estimate 6}, $\Pi_{\hat \jmath} v=v$,  \eqref{Estimates on d(A(v+w))} ,\eqref{eq:Gjecc} and the first bound in \eqref{tilde cQ-c*Q d tilde cQ(v)} we obtain  the estimate for the derivative  in \eqref{tilde cQ-c*Q d tilde cQ(v)}.
\end{proof}

   \noindent
   {\bf Definition of $c(v)$.}
   We define the function 
   $ c : B_{r'}^V\to B_{r} (c_*), 
   v \mapsto  c(v) $ in Proposition \ref{Construction of c close and away from PJ} as 
    \begin{equation}\label{c defined on the ball}
     \begin{aligned}     c(v):=\begin{cases}
                c_{\threed}(v) & \text{if} \ v\in B_{r'}^V\setminus {\bf V}^{\twod}\\
                 c^{\parallel_{\hat \jmath}}_{\twod} (v)\hat\jmath + c_*^{\bot_{\hat \jmath}} \hat\jmath^\bot& \text{if} \ v\in (B_{r'}^V\cap V_{\hat \jmath})\setminus \{0\}\, , \ {\hat \jmath} \in \cD \, , \\
                 c_*  & \text{if} \  v=0\, ,
            \end{cases}
            \end{aligned}
    \end{equation}
where $ c_{\threed} : B_{r'}^V\setminus {\bf V}^{\twod} \to \R  $ is the analytic function  
defined in Propositions \ref{Construction of c}, \ref{c near PJ} and  
$ c^{\parallel_{\hat \jmath}}_{\twod} : B_{r'}^V\cap V_{\hat \jmath} \to \R  $ is the analytic function 
defined in Lemma \ref{c on PJ}.  
 In view of Propositions \ref{Construction of c} and \ref{c near PJ} and 
\eqref{tilde cQ-c*Q d tilde cQ(v)},
the function $c(v)$  satisfies  \eqref{cQ-c*Q d cQ(v)} and so  it is continuous at $v=0$.

  %  The function $c$ may not be continuous, as simple examples show. 
 % The parallel components 
  %  $ c^{\parallel_{\hat \jmath}} (v)$ along any resonant wave 
   % direction admit continuous extensions up to $ {\bf V}^{\twod} $. 

\subsection{Regularity properties}
To complete the proof of 
Proposition 
\ref{Construction of c close and away from PJ} we have to prove the 
regularity properties in 
\eqref{c1Q U tilde c1Q}. 

\begin{lemma}\label{lm:continuity of cparallel}
{\bf (Continuity  
of $ c^{\parallel_{\hat \jmath}} (v)$)}
    For any $ {\hat \jmath} \in \cD$, the parallel component $c^{\parallel_{\hat \jmath}}(v) = \hat\jmath \cdot c(v) $ of the speed $ c(v) $ defined in \eqref{c defined on the ball} is continuous on the set 
 $ B_{r'}^{\hat \jmath} := 
         (B_{r'}^V\setminus {\bf V}^{\twod})\cup (B_{r'}^V\cap V_{\hat \jmath}) $.
         %         \eqref{c1Q U tilde c1Q}. 
\end{lemma}

\begin{proof}
We have just to prove the continuity
of $ c^{\parallel_{\hat \jmath}}(v) $ 
at any  $ \bar v \in B_{r'}^V\cap V_{\hat \jmath}  $, $ \bar v \neq 0 $.
    Let $\{v_i\}_{i\in \N}$ be a sequence in $B_{r'}^V $ converging  to $\bar v  $.
    Since $c(v_i) $ is bounded (cf.~\eqref{cQ-c*Q d cQ(v)}), for any subsequence $ v_{i_m} $ there is $ v_{{i_m}_\ell} $ such that 
    $ c(v_{{i_m}_\ell}) \xrightarrow[\ell\to+\infty]{}\bar c $ and $\bar c\in \overline{B_{r'}(c_*)}\subset B_r(c_*)$ for $r'\in (0,r)$ sufficiently small. 
    Since $c(v_i )$ solves \eqref{EQ c Phi} (equiv. \eqref{EQ c Phij}), by continuity, we deduce that
    $
        \di_v\Phi(\bar c,\bar v)[\nabla_v \cI^{\parallel_{\hat \jmath}}(\bar v)]=0 $,  
    and hence $\bar c^{\parallel_{\hat \jmath}}$ solves
        \eqref{eq:mf2d}. 
    By uniqueness,
    Lemma \ref{c on PJ} implies that $\bar c^{\parallel_{\hat \jmath}}=c^{\parallel_{\hat \jmath}}_{\twod}(\bar v)$ and thus $\lim_{i\to +\infty}c^{\parallel_{\hat \jmath}}(v_i)=c^{\parallel_{\hat \jmath}}_{\twod}(\bar v)=c^{\parallel_{\hat \jmath}}(\bar v)$. This proves that $c^{\parallel_{\hat \jmath}}(v)  $  is continuous.
\end{proof}

Actually $ c^{\parallel_{\hat \jmath}} (v)$ is also $ C^1 $ 
on $B_{r'}^{\hat \jmath}\setminus\{0\}$ as we show in Lemma \ref{lm:C1 regularity of cparallel} below.

\begin{lemma}\label{lm:regularity}
    Let $ {\mathtt E} :B_{r}(c_*)\times B_{r}^V\to \R$ be an analytic 
    function such that, for any
   $ {\hat \jmath} \in \cD $, 
        \begin{equation}\label{laderab}
    \partial_{c^{\bot_{\hat \jmath}}}{\mathtt E} (c,v)=0 \quad \forall   v\in  V_{\hat \jmath} \cap B_{r}^V , \   c\in B_{r}(c_*)    \, .
        \end{equation}
        Then the composed functions 
        $
            v\mapsto {\mathtt E} (c(v),v)$ and $v\mapsto (\nabla_v {\mathtt E}) (c(v),v)
        $, with $ c(v) $ defined in 
        \eqref{c defined on the ball},  
        are continuous on $B_{r'}^V$, analytic on $B_{r'}^V\setminus {\bf V}^{\twod}$ and 
        restricted to each $(B_{r'}^V\cap V_{\hat \jmath})\setminus\{0\}$ for any ${\hat \jmath} \in \cD $.
\end{lemma}

\begin{proof}
    Both the functions $v\mapsto {\mathtt E}(c(v),v)$ and $v\mapsto (\nabla_v {\mathtt E}) (c(v),v)$ are analytic on $B_{r'}^V\setminus {\bf V}^{\twod}$ and restricted to $ (B_{r'}^V\cap V_{\hat \jmath})\setminus\{0\}$ for any ${\hat \jmath} \in \cD $ by the analyticity property of the function $ c(v) $.
     Let us prove the continuity of $v\mapsto {\mathtt E} (c(v),v)$.
    Since the function $c(v)$ is continuous at $v=0$ (cf. \eqref{cQ-c*Q d cQ(v)}) we only have to verify it is continuous
    at $ \bar v \in B_{r'}^V\cap V_{\hat \jmath}  $, $ \bar v \neq 0 $.
        Let $\{v_i\}_{i\in \N}$ be a sequence in $B_{r'}^V$ converging to $\bar v $. Since $c(v_i)$
        is bounded (cf.~\eqref{cQ-c*Q d cQ(v)}), 
        for any subsequence $v_{i_m}$ there exists a subsequence $v_{i_{m_\ell}}$ and $\bar c^{\bot_{\hat \jmath}} \in (c_*^{\bot_{\hat \jmath}}-r,c_*^{\bot_{\hat \jmath}}+r) $ such that
        $ c^{\bot_{\hat \jmath}}(v_{i_{m_\ell}})
        \xrightarrow[\ell\to+\infty]{}  
        \bar c^{\bot_{\hat \jmath}}$.
        Hence,  
        using also the continuity 
        of $ c^{\parallel_{\hat \jmath}} (v) $ proved in  Lemma \ref{lm:continuity of cparallel}, we deduce
$ c(v_{i_{m_\ell}})     \to   \bar c := c^{\parallel_{\hat \jmath}}_{\twod}(\bar v)\hat\jmath +\bar c^{\bot_{\hat \jmath}} \hat\jmath^\bot $ and 
        $$
   {\mathtt E} (c(v_{i_{m_\ell}}),v_{i_{m_\ell}}) \xrightarrow[\ell\to+\infty]{}  {\mathtt E} (\bar c,\bar v) \stackrel{\eqref{laderab}} = 
{\mathtt E} (c^{\parallel_{\hat \jmath}}_{\twod}(\bar v)\hat\jmath+c_*^{\bot_{\hat \jmath}}\hat\jmath^\bot ,\bar v) \stackrel{\eqref{c defined on the ball}} 
= {\mathtt E}(c(\bar v), \bar v) \, .
        $$
Hence $v\mapsto {\mathtt E} (c(v),v)$ is continuous. 
The 
continuity of  $ (\nabla_v {\mathtt E}) (c(v),v)$ is similar. 
\end{proof}

\begin{lemma}{(\bf $C^1$ regularity of $c^{\parallel_{\hat \jmath}}(v) $)}\label{lm:C1 regularity of cparallel}
For any $ {\hat \jmath} \in \cD$,
    the parallel component $c^{\parallel_{\hat \jmath}}(v)$ of the speed $ c(v) $ in \eqref{c defined on the ball}  is of class $C^1$ on $B_{r'}^{\hat \jmath}\setminus\{0\}$.
\end{lemma}

\begin{proof}
We have  just to prove that 
$ c^{\parallel_{\hat \jmath}}(v) $ 
is $ C^1 $
at any  $ \bar v \in B_{r'}^V\cap V_{\hat \jmath}  $, $ \bar v \neq 0 $.
    We show that for any direction $\hat v\in V$ the derivative $\di_v c^{\parallel_{\hat \jmath}}(v)[\hat v]$ on $B_{r'}^V\setminus {\bf V}^{\twod}$ admits a continuous extension to $B_{r'}^{\hat \jmath}\setminus\{0\}$.
    For any $\hat v\in V$, let $\ell:B_{r'}^{\hat \jmath}\setminus\{0\}\to \R$ be the function
    $$
        \ell(v):=\di_v c^{\parallel_{\hat \jmath}}(v)[\hat v] \quad \forall  \ v\in B_{r'}^V\setminus {\bf V}^{2d} \, , 
    $$
    and for any $v\in (B_{r'}\cap V_{\hat \jmath})\setminus \{0\}$ be the unique solution of the equation (cf. \eqref{eq:dc2d EQ})
    \begin{equation}\label{eq:equation for dcparallelj}
            \big( \cA_{11}^{\hat \jmath}(v)+\partial_{c^{\parallel_{\hat \jmath}}}\tG^{\parallel_{\hat \jmath}}_{\geq 3}(c^{\hat \jmath}(v),v)\big)\ell(v)
            =
            -\di_v\cA^{\hat \jmath}_{11}(v)[\hat v](c^{\parallel_{\hat \jmath}}(v)-c_*^{\parallel_{\hat \jmath}})-\di_v\tG^{\parallel_{\hat \jmath}}_{\geq 3}(c^{\hat \jmath}(v),v)[\hat v] \, .
        \end{equation}
    In view of \eqref{Estimate on a(v,v)} and \eqref{lem:c near PJ estimate 6} the solution of \eqref{eq:equation for dcparallelj} is unique.
    We now prove that $\ell(v)$ is continuous 
     at any $\bar v\in (B_{r'}^V\cap V_{\hat \jmath})\setminus\{0\}$.
    Let $\{v_i\}_{i\in \N}$ be a sequence in $B_{r'}^{\hat \jmath}\setminus\{0\}$  converging to $\bar v$.
    For any subsequence $v_{i_m}$ there exists a subsequence $v_{i_{m_\ell}}$ which is entirely contained in $(B_{r'}^V\cap V_{\hat \jmath})\setminus\{0\}$ or in $B_{r'}^V\setminus{\bf V}^{2d}$.
    In the first case taking the limit in \eqref{eq:equation for dcparallelj} by uniqueness we obtain that $\ell(v_{i_{m_\ell}})\to \ell(\bar v)$.
    Hence assume that $v_{i_{m_\ell}}\in B_{r'}^V\setminus { \bf V}^{2d}$.
    In view of \eqref{derivicino} we have that $\ell(v_{i_{m_\ell}})=\di_v c^{\parallel_{\hat \jmath}}(v_{i_{m_\ell}})[\hat v]$ is bounded.
    Hence there exists $\bar \mu\in \R$ such that 
    $   \ell(v_{i_{m_\ell}})=\di_v c^{\parallel_{\hat \jmath}} (v_{i_{m_\ell}})[\hat v]\to \bar \mu $ up to subsequence. 
    In view of \eqref{eq:dc} we have
    \begin{equation}\label{eq:dc first component}
    \begin{aligned}
        &\big(\cA^{\hat \jmath}_{11}(v)+\partial_{c^{\parallel_{\hat \jmath}}}\tG^{\parallel_{\hat \jmath}}_{\geq 3}(c^{\hat \jmath}(v),v)\big)\ell(v)+
        \underbrace{\big(\cA^{\hat \jmath}_{12}(v)+\partial_{c^{\bot_{\hat \jmath}}}\tG^{\parallel_{\hat \jmath}}_{\geq 3}(c^{\hat \jmath}(v),v)\big)\di_v c^{\bot_{\hat \jmath}}(c)[\hat v]}_{=:\boldsymbol{(I_\ell)}}\\
       & =\! - \di_v \cA^{\hat \jmath}_{11}(v)[\hat v](c^{\parallel_{\hat \jmath}}(v)-c_*^{\parallel_{\hat \jmath}})
       \! - \!\underbrace{\di_v \cA^{\hat \jmath}_{12}(v)[\hat v](c^{\bot_{\hat \jmath}}(v)-c_*^{\bot_{\hat \jmath}})}_{
        =: \boldsymbol{(II_\ell)} } \! - \di_v \tG^{\parallel_{\hat \jmath}}_{\geq 3}(c^{\hat \jmath}(v),v)[\hat v]
    \end{aligned}
    \end{equation}
    for $v=v_{i_{m_\ell}}$.
    Using \eqref{Estimate on a(v,v)}, \eqref{lem:c near PJ estimate 6} and \eqref{derivicino} we deduce $\boldsymbol{(I_\ell)}\xrightarrow[\ell\to +\infty]{} 0$ and using \eqref{Estimates on d(A(v+w))}, \eqref{cQ-c*Q d cQ(v)} that $\boldsymbol{(II_\ell)}\xrightarrow[\ell\to +\infty]{} 0$.
    The other terms converge in view of Lemmata \ref{lm:continuity of cparallel} and \ref{lm:regularity} and since $\partial_{c^{\bot_{\hat \jmath}}}\tG^{\parallel_{\hat \jmath}}_{\geq 3}(c,v)=0 $ for $v\in B_{r}^V\cap V_{\hat \jmath}$.
    Passing to the limit for $\ell \to +\infty$ in \eqref{eq:dc first component} we deduce that $\bar \mu$ solves \eqref{eq:equation for dcparallelj} at $ \bar v $. 
    Since the solution to \eqref{eq:equation for dcparallelj} is unique we deduce $\ell(\bar v)=\bar \mu$ and hence $\ell(v_i)=\di_v c^{\parallel_{\hat \jmath}}(v_i)[\hat v]$ converges to $\ell(\bar v)$.
    We have proved that $\ell(v)$ is a continuous extension of $\di_v c^{\parallel_{\hat \jmath}}(v)[\hat v]$.
    \end{proof}

   The next  lemma  will be used  in Lemmata \ref{lem:regI}
    and \ref{HisC1}. 
    
    \begin{lemma}\label{lm:regularity 1}
        Let $ {\mathtt E} :B_{r}(c_*)\times B_{r}^V\to \R$ be an analytic
    $ \T^2_{\Gamma}$-invariant 
    function such that 
\begin{equation}\label{assuC1}
         {\mathtt E} (c,0)=0  \, , \quad  \partial_{c^{\bot_{\hat \jmath}}} {\mathtt E} (c,v)=0 \, ,  \  
        \forall v\in B_{r}^V\cap V_{\hat \jmath}\, , \   c\in B_{r}(c_*) \,, \ {\hat \jmath} \in \cD\, .
        \end{equation}
        Then the composed function $v\mapsto {\mathtt E} (c(v),v)$, with $ c(v) $ defined in 
        \eqref{c defined on the ball},   is $C^1 (B_{r'}^V) $.
    \end{lemma}
    
    \begin{proof}
    Note that, since ${\mathtt E}(c,\cdot)$ is $\T^2_\Gamma$-invariant,   Lemma \ref{lm:AAA} 
    implies that 
    \begin{equation}\label{assuC1 2}
        \di_v {\mathtt E}(c,0)=0\quad \forall c\in B_{r}(c_*)\, .
    \end{equation}
        By Lemma \ref{lm:regularity} the 
        composed function  $E(v) :=  {\mathtt E} (c(v),v)$ is analytic on $B_{r'}^V\setminus {\bf V}^{\twod}$ and 
        continuous on  $B_{r'}^V$ 
        %To prove that $E\in C^1(B_{r'}^V)$ we show that for any direction $\hat v\in V$ the derivative $\di_v E(v)[\hat v]$ on $B_{r'}^V\setminus {\bf V}^{\twod}$ admits a continuous extension on $ B_{r'}^V $.
        %Since $ E $ is analytic outside $ {\bf V}^{2d} $  
        % To prove that $E\in C^1(B_{r'}^V)$ 
        and we 
         have just to prove that $E  $ is $ C^1 $  on each 
        $ V_{\hat \jmath} \cap B_{r' }$ and  at $ v = 0 $.   
        Note that 
         $B_{r'}^V = \cup_{\hat \jmath\in \cD} \cF_{r'}\cup \cN_{\hat \jmath}(r') $ (with $\cF_{r'}$ and $\cN_{\hat \jmath}(r')$  defined in \eqref{defFr} and \eqref{defNr})
and we         
        show that for any direction $\hat v\in V$, for any $\hat \jmath \in \cD$ the restriction of $\di_v E(v)[\hat v]$ to $\cF_{r'}\cup \cN_{\hat \jmath}(r')\setminus V_{\hat \jmath}$ admits a continuous extension to $\cF_{r'}\cup \cN_{\hat \jmath}(r')$ which is equal to $0$ at $ v = 0$.
For any $ {\hat \jmath} \in \cD $, and  $v\in B_{r'}^V\setminus {\bf V}^{2d}$, we have 
        \begin{align*}\label{eq:dE(v)[hat v]}
        \di_v E(v)[\hat v] =
        \underbrace{(\di_v {\mathtt E}) (c(v),v)[\hat v]}_{=: \boldsymbol{(I)}} +
        \underbrace{(\partial_{c^{\parallel_{\hat \jmath}}} {\mathtt E}) (c(v),v)\di_v c^{\parallel_{\hat \jmath}}(v)[\hat v]}_{=: \boldsymbol{(II)} } +\underbrace{(\partial_{c^{\bot_{\hat \jmath}}}{\mathtt E}) (c(v),v)\di_v c^{\bot_{\hat \jmath}}(v)[\hat v]}_{\boldsymbol{(III)}} \, . 
        \end{align*}
        By  Lemma \ref{lm:regularity} the term $\boldsymbol{(I)}$ admits a continuous extension on $B_{r'}^V$ which is equal to $0$ at $v=0$  by \eqref{assuC1}.
        Similarly Lemma \ref{lm:regularity} implies  that  $(\partial_{c^{\parallel_{\hat \jmath}}}{\mathtt E})(c(v),v)$ is continuous on $B_{r'}^V$.
        % and using also Lemma \ref{lm:C1 regularity of cparallel} we conclude that $\boldsymbol{(II)}$ is continuous on $B_{r'}^{\hat \jmath}\setminus\{0\}$.
        Since $\di_v c^{\parallel_{\hat \jmath}}(v)$ is continuous on $B_{r'}^{\hat \jmath}\setminus\{0\}$ (by Lemma \ref{lm:C1 regularity of cparallel}) the term $\boldsymbol{(II)}$ extends with continuity to $B_{r'}^{\hat \jmath}\setminus \{0\}$.
        %$\cF_{r'}\cup \cN_{\hat \jmath}(r')\setminus\{0\}$.
        Moreover, $\di_v c^{\parallel_{\hat \jmath}}(v)[\hat v]$ is uniformly bounded on $\cF_{r'}\cup \cN_{\hat \jmath}(r')\setminus\{0\}$ 
        (by the first inequality in \eqref{derivicino} and \eqref{|c0-c*| |dc0|})
         and $(\partial_{c^{\parallel_{\hat \jmath}}}{\mathtt E}) (c,v)=\cO(\|v\|^2)$ (which follows by \eqref{assuC1} and \eqref{assuC1 2}) we deduce that 
        $$
        |\boldsymbol{(II)}_{|\cF_{r'}\cup \cN_{\hat \jmath}(r')\setminus\{0\}}| \lesssim \| v \|^2 \|\hat v\|
        $$ 
        and so $\boldsymbol{(II)}$ further extends with continuity at $0$.
        Lemma 
        \ref{lm:A}-$(ii)$ implies that $(\partial_{c^{\bot_{\hat \jmath}}}{\mathtt E}) (c(v),v)=\cO(\|\Pi_{\hat \jmath}^\bot v\|^2)$ and using also the second inequality in \eqref{derivicino} and \eqref{|c0-c*| |dc0|} we deduce that $$|\boldsymbol{(III)}_{|\cF_{r'}\cup \cN_{\hat \jmath}(r')\setminus V_{\hat \jmath}}|\lesssim \|\Pi_{\hat \jmath}^\bot v\|\|v\|\|\hat v\|$$ 
        and hence $\boldsymbol{(III)}$ extends with continuity to $\cF_{r'}\cup \cN_{\hat \jmath}(r')$. 
        \end{proof}
    
    \section{Straightening of the reduced momentum}\label{sec:diritto}

    Let $ I(v)$  be the reduced momentum 
\begin{equation}\label{Def I dritto}
I : B_{r'}^V \to \R^2 \, , \qquad 
    I(v) := {\mathtt I} (c(v),v) =
    \cI(v+w(c(v),v)) \, , 
\end{equation}
where $ {\mathtt I} (c,v) $ is defined in \eqref{eq:def I(c,v)} and $c(v)$ in Proposition 
\ref{Construction of c close and away from PJ}. 

\begin{lemma}\label{lem:regI}
The reduced momentum  $ I $ in 
\eqref{Def I dritto}   
is $\T^2_\Gamma\rtimes \Z_2$-invariant, $C^1(B_{r'}^V)$, analytic  in 
    $ B_{r'}^V \setminus  {\bf V}^{\twod}$ and analytic restricted to each  $(B_{r'}^V\cap V_{\hat \jmath})\setminus\{0\}$ for any $ {\hat \jmath} \in \cD$.
\end{lemma}

\begin{proof}
Since ${\mathtt I}(c,v)$ and  $c(v)$ are $\T^2_\Gamma\rtimes \Z_2$-invariant in $v$ we deduce that $I(v)$ is $\T^2_\Gamma\rtimes \Z_2$-invariant.
The  function ${\mathtt I} (c,v) =
\pa_c \Phi(c,v) $ in \eqref{eq:def I(c,v)}-\eqref{partial c Phi} satisfies the assumptions of Lemma \ref{lm:regularity 1},  by 
Proposition \ref{lem:expa}, and thus  
$ I $ in $ C^1 
(B_{r'}^V ) $. The other claimed analyticity properties of $ I(v) $ directly follow
by those of the speed function  $ c(v) $ in Proposition \ref{Construction of c close and away from PJ}.
% \eqref.
\end{proof}

The main result of this section is 
    the following ``Morse type"  result:  the reduced
    momentum  $ I(v) $ can be locally
    rectified, close to $ 0 $, 
   into the  purely quadratic momentum $ \cI (v) $.  
    \begin{theorem}\label{th:Moser's trick} {\bf (Straightening)}
        There exist open neighborhoods
        $ U_1, U_2  $ of $0$ in $ V $ with $ U_1, U_2 \subset B_{r'}^V $, and a $\T^2_\Gamma$-equivariant homeomorphism $\zeta:U_1\to U_2$
        satisfying  $ \zeta(0)=0 $, such that 
    \begin{equation}\label{Izeta}
    I\circ \zeta(v)=\cI(v) \, , \quad  \forall 
    v\in U_1 \, ,
    \end{equation}
where $I (v) $ is defined in \eqref{Def I dritto} and $ \cI (v) $ in \eqref{eq:cI}.  
    Furthermore 
    \begin{align}\label{diffeoint}
    & \zeta_{|U_1\setminus 
            {\bf V}^{\twod}}:U_1\setminus 
            {\bf V}^{\twod} \to U_2\setminus 
            {\bf V}^{\twod} \quad \text{ is a diffeomorphism}\, , \\
    & \label{diffeobordo}
    \zeta_{|(U_1\cap V_{\hat \jmath})\setminus\{0\}}:(U_1\cap V_{\hat \jmath})\setminus\{0\}\to (U_2\cap V_{\hat \jmath})\setminus\{0\} \, \ 
    \forall {\hat \jmath} \in \cD  \quad \text{ is a diffeomorphism}\, .
    \end{align}
    \end{theorem}

    \begin{remark}\label{rmk:U2 subset Br''}
        Actually $U_2\subset B_{r''}^V$ where $r''\in (0,r')$ is fixed in Lemma \ref{Estimates on A, AQ, tilde AQ11} (cf. \eqref{def:U2}).
    \end{remark}
    
    Theorem \ref{th:Moser's trick} is proved in Section \ref{sec:MT}
    using estimates proved in the 
    next section.

\subsection{The reduced momentum $ I(v) $ and the matrix $ A(v)$}
\label{sec:prel}

For any $ {\hat \jmath} \in \cD  $, we set 
    \begin{equation}\label{Def I dritto j}
    I^{\hat \jmath} (v) =\vect{I^{\parallel_{\hat \jmath}}(v)}{I^{\bot_{\hat \jmath}}(v)}
    \stackrel{\eqref{eq:def yj}} {:=} 
    Q_{\hat \jmath}I(v)  \, , 
    \quad \cI^{\hat \jmath}(v) = \vect{\cI^{\parallel_{\hat \jmath}}(v)}{\cI^{\bot_{\hat \jmath}}(v)} = Q_{\hat \jmath}\cI(v) \, .
    \end{equation}
The reduced momentum $ I(v) $ is a perturbation of the quadratic momentum $ \cI (v) $ in \eqref{eq:cI}. 
    
    \begin{lemma}\label{Properties of I dritto}
     For any ${\hat \jmath} \in \cD$ and any $v\in
    B_{r'}^V $ 
     \begin{align}
         \label{Estimate on Ij}
         &     I^{\hat \jmath}(v)=  \cI^{\hat \jmath}(v)+\cO\vect{\|v\|^3}{\|\Pi_{\hat \jmath}^\bot v\|^2 \|v \|}\, ,
         \end{align}
         \begin{equation}\label{Estimate on d I1Q dritto}
        \begin{aligned}
            &\di_v I^{\hat \jmath}(v)[\hat v]=\di_v \cI^{\hat \jmath}(v)[\hat v]+\cO(\|v\|^2\|\hat v\|) \, ,  \qquad \qquad \qquad \qquad \quad \qquad \ \ \,   \forall  v\in \cF_{r'}  \ \text{cf.} \  \eqref{defFr} \, ,  \\
            &\di_v I^{\hat \jmath}(v)[\hat v]=
            \di_v \cI^{\hat \jmath}(v)[\hat v]+\cO\begin{pmatrix}
                \|v\|(\|v\|\|\Pi_{\hat \jmath} \hat v\|+\|\Pi_{\hat \jmath}^\bot v\|\|\Pi_{\hat \jmath}^\bot \hat v\|)\\
                \|\Pi_{\hat \jmath}^\bot v\|(\|\Pi_{\hat \jmath}^\bot v\|\|\Pi_{\hat \jmath} \hat v\|+\|v\|\|\Pi_{\hat \jmath}^\bot \hat v\|)
            \end{pmatrix} \ \forall  v\in \cN_{\hat \jmath}(r')\, ,\   \hat \jmath \in \cD \, .
        \end{aligned}
    \end{equation}
    \end{lemma}
    
\begin{proof}
For any $ {\hat \jmath} \in \cD $,  we write 
$ g(v) :=  I^{\hat \jmath}(v) - \cI^{\hat \jmath}(v)  = 
{\mathtt g}(c(v),v) 
$ 
where
\begin{equation}\label{gcv}
{\mathtt g}(c,v) 
= {\mathtt I}^{\hat \jmath}(c,v)- \cI^{\hat \jmath} (v) 
\stackrel{\eqref{partial c Phi},\eqref{exp1Phi}} 
=  \pa_{c^{\hat \jmath}} G_{\geq 3}(c,v)
= \begin{pmatrix}
            \partial_{c^{\parallel_{\hat \jmath}}} G_{\geq 3}(c,v)\\
            \partial_{c^{\bot_{\hat \jmath}}} G_{\geq 3}(c,v)
        \end{pmatrix} 
\end{equation}
and  $G_{\geq 3}:B_{r}(c_*)\times B_{r}^V\to \R^2$ is   defined in \eqref{exp1Phi}.
    By \eqref{Gprgeq3} we have ${\mathtt g}(c,v) =\cO(\|v\|^3) $ uniformly in $ c $,  
    proving 
    the first estimate  \eqref{Estimate on Ij}.  
    Moreover, since $ G_{\geq 3}(c,v) $ is independent of $ c^{\bot_{\hat \jmath}} $ on $B_{r}^V\cap V_{\hat \jmath}$ (cf. Proposition \ref{lem:expa}),  it results %  \eqref{nablaIv}
    \begin{equation}\label{eq:property of g(c,v)}
        {\mathtt g}^{\bot_{\hat \jmath}}(c,v)=0 \quad \text{for any} \quad   v\in B_{r'}^V\cap V_{\hat \jmath} 
    \end{equation}
        and  \eqref{eq:A1} implies the second 
    estimate  \eqref{Estimate on Ij}. 
   Furthermore \eqref{A2ora} and \eqref{eq:A3} 
   imply 
\begin{equation}\label{eq:gs1}
        \di_v {\mathtt g} (c,v)[\hat v]=\cO\begin{pmatrix}
            \|v\|(\|v\|\|\Pi_{\hat \jmath} \hat v\|+\|\Pi_{\hat \jmath}^\bot v\|\|\Pi_{\hat \jmath}^\bot \hat v\|)\\
            \|\Pi_{\hat \jmath}^\bot v\|(\|\Pi_{\hat \jmath}^\bot v\|\|\Pi_{\hat \jmath} \hat v\|+\|v\|\|\Pi_{\hat \jmath}^\bot \hat v\|)
        \end{pmatrix}\, .
\end{equation}
    We now estimate    
\begin{equation}\label{eq:gcomp}
\di_v g(v)  =(\di_v \tg)(c(v),v) +(\partial_{c^{\parallel_{\hat \jmath}}}\tg)(c(v),v)\di_v c^{\parallel_{\hat\jmath}}(v) 
+(\partial_{c^{\bot_{\hat \jmath}}}\tg)(c(v),v)\di_v c^{\bot_{\hat\jmath}}(v) \,.
\end{equation}
    For any $ v\in B_{r'}^V\cap V_{\hat \jmath}$, by Proposition \ref{lem:expa} the term $ {\mathtt g}(c,v) $ in \eqref{gcv} satisfies
    $\partial_{c^{\bot_{\hat \jmath}}}{\mathtt g} (c,v)=0 $ and by \eqref{eq:property of g(c,v)} it satisfies $\partial_{c^{\parallel_{\hat \jmath}}}{\mathtt g}^{\bot_{\hat \jmath}} (c,v)=0$.
        Therefore  \eqref{eq:A1} and \eqref{Gprgeq3} imply
\begin{equation}\label{eq:gs2}
        \partial_{c^{\parallel_{\hat \jmath}}}{\mathtt g}(c,v)=\cO\begin{pmatrix}
            \|v\|^3\\
            \|v\|\|\Pi_{\hat \jmath}^\bot v\|^2
        \end{pmatrix} \, , \ 
        \partial_{c^{\bot_{\hat \jmath}}}{\mathtt g}(c,v)=\cO\begin{pmatrix}
            \|v\|\|\Pi_{\hat \jmath}^\bot v\|^2\\
            \|v\|\|\Pi_{\hat \jmath}^\bot v\|^2
        \end{pmatrix} \, . 
\end{equation}
 Then estimates \eqref{Estimate on d I1Q dritto} follow by
  \eqref{eq:gcomp},
  \eqref{eq:gs1}, \eqref{eq:gs2} and using \eqref{|c0-c*| |dc0|} if $v\in \cF_{r'}$, and  using \eqref{derivicino} if $v\in \cN_{\hat \jmath}(r')\setminus V_{\hat \jmath}$.
  Since $ I $ is $ C^1 (B_{r'}^V ) $ by Lemma \ref{lem:regI} the estimates
  \eqref{Estimate on d I1Q dritto} hold
  also on 
  each $ V_{\hat \jmath }$. 
\end{proof}

We now introduce the matrix valued function $A :B_{r'}^V\to \R^{2\times 2} $ as
\begin{equation}\label{Def AA}
        A(v):= 
        \begin{pmatrix} 
            \di_vI_1 (v)[\nabla_v \cI_1 (v)] & \di_vI_2  (v)[\nabla_v \cI_1 (v)] \\
            \di_v I_1 (v)[\nabla_v \cI_2 (v)] & \di_vI_2 (v)[\nabla_v \cI_2 (v)]
        \end{pmatrix}  
    \end{equation} 
whose relevance is 
highlighted by the next lemma. 
    \begin{lemma}
    \label{lem:indip}
        For  any $v\in V\setminus {\bf V}^{\twod}$ the matrix $A(v)$ is invertible if and only if
\begin{equation}\label{eq:decomposition}
            V = \ker(\di_vI(v))\oplus span(\nabla_v \cI_1(v),\nabla_v \cI_2(v))\, .
        \end{equation}
    \end{lemma}
    
    \begin{proof}
     If  $A(v)$ is invertible,
 any  $\hat v\in V$ admits the decomposition 
        $
            \hat v=\hat v_I+\mu_1\nabla_v\cI_1(v)+\mu_2\nabla_v\cI_2(v)
        $
        where $\hat v_I\in \ker(\di_v I(v) )$ and $\mu=(\mu_1,\mu_2)^\top   = A(v)^{- \top} \di_vI(v)[\hat v]$, thus   \eqref{eq:decomposition} holds.
       Viceversa,  if  \eqref{eq:decomposition} holds then 
        $$
    \rank (A(v)) = \rank(\di_vI(v))=\dim(V)-\dim(\ker(\di_v I(v)))  = 2
        $$ 
by \eqref{eq:decomposition} and, noting that $\{\nabla_v\cI_1(v),\nabla_v \cI_1(v)\}$ are linearly independent, by  Lemma \ref{Properties of a}, for   any $v\in V\setminus {\bf V}^{\twod}$. Therefore the $ 2 \times 2 $ matrix $A(v)$ is non-singular.
    \end{proof}

For any $ \hat \jmath \in \cD $, any $ v \in B_{r'}^V $, we define 
    \begin{equation}\label{Def AAQ} 
        A^{\hat \jmath}(v):=
        \begin{pmatrix} 
            \di_vI^{\parallel_{\hat \jmath}}(v)[\nabla_v \cI^{\parallel_{\hat \jmath}}(v)] & \di_vI^{\bot_{\hat \jmath}}(v)[\nabla_v \cI^{\parallel_{\hat \jmath}}(v)] \\
            \di_vI^{\parallel_{\hat \jmath}}(v)[\nabla_v \cI^{\bot_{\hat \jmath}}(v)] & \di_vI^{\bot_{\hat \jmath}}(v)[\nabla_v \cI^{\bot_{\hat \jmath}}(v)]
        \end{pmatrix}
        \, , 
    \end{equation}
   where $ I^{\parallel_{\hat \jmath}}(v), I^{\bot_{\hat \jmath}}(v) $
   are introduced  in 
   \eqref{Def I dritto j},  i.e. $A^{\hat \jmath}(v)=Q_{\hat \jmath} A(v) Q_{\hat \jmath}^{-1}$.  For any $ t \in \R $ we define  
    \begin{equation}\label{eq:def At}
        A(t,v):=tA(v)+(1-t)\cA(v) \, , \qquad A^{\hat \jmath}(t,v):=tA^{\hat \jmath}(v)+(1-t)\cA^{\hat \jmath}(v)\, , 
    \end{equation}
where the matrices $ \cA(v) $, $\cA^{\hat \jmath}(v)$ are  defined in \eqref{Def a}, \eqref{Def aQ}.
    
    \begin{lemma} {\bf (Properties of $A(t,v)$)}
    \label{Estimates on A, AQ, tilde AQ11}
        The matrix $ A(t,v) $ defined in \eqref{eq:def At} is $\T^2_\Gamma$-invariant in $v$ and is continuous on 
        $ \R \times B_{r'}^V $, analytic in $\R\times [B_{r'}^V \setminus {\bf V}^{\twod}]$ and restricted to  each $\R\times [(B_{r'}^V\cap V_{\hat \jmath})\setminus\{0\}]$   
        for any  $ {\hat \jmath}  \in \cD$ 
        (where $r'$ is defined by Proposition \ref{Construction of c close and away from PJ}) and, for any  $ {\hat \jmath}  \in \cD $, $ v \in B_{r'}^V $,
        $ t \in \R $, 
  \begin{equation}\label{AQ = cal AQ + error}
    \begin{aligned}
        & A^{\hat \jmath}(t,v) = 
            \cO(1+|t|\|v\|)\|v\|^2 \, ,  & & 
            \forall (t,v)\in \R\times \cF_{r'} \, , \\
        & A^{\hat \jmath}(t,v) = 
            \cO(1+|t|\|v\|)\begin{pmatrix}
                \|v\|^2 & \|
                \Pi_{\hat \jmath}^\bot v\|^2\\
                \| \Pi_{\hat \jmath}^\bot v \|^2 & \| \Pi_{\hat \jmath}^\bot v \|^2
            \end{pmatrix}  & & 
            \forall 
            (t,v)\in \R\times \cN_{\hat \jmath}(r')\, , \ \hat \jmath\in\cD \, .
            \end{aligned}
        \end{equation}     
     There exists $r''\in (0,r')$ such that for any $(t,v)\in (-2,2)\times [B_{r''}^V\setminus {\bf V}^{\twod}]$ we have
    \begin{equation}\label{detpospe}
        \begin{aligned}
            & \det(A(t,v)) \gtrsim \|v\|^4  & & \text{if} \quad v\in \cF_{r''}  \ \text{cf.} \  \eqref{defFr} \, ,  \\
            &\det(A(t,v))=\det(A^{\hat \jmath}(t,v)) \geq \tfrac \disb 4\|v\|^2\|\Pi_{\hat \jmath}^\bot v\|^2  & & \text{if} \quad v\in \cN_{\hat \jmath}(r'')\setminus V_{\hat \jmath}\, ,
            \   \hat \jmath \in \cD \, , \\
            &
            (A^{\hat \jmath}(t,v))_{11}\gtrsim\|\Pi_{\hat \jmath}v\|^2=\|v\|^2  & & \text{if} \quad v\in  B_{r''}^V\cap V_{\hat \jmath}\, ,
            \   \hat \jmath \in \cD\, .
        \end{aligned}
    \end{equation}
    \end{lemma}
    
    \begin{proof}
Since $\cI(v)$ and  $I(v)$ 
are $\T^2_\Gamma$-invariant (cf. Lemma \ref{lem:regI}), 
        their gradients are $ \T^2_\Gamma $-equivariant  and so the matrix $A(t,v)$ is $\T^2_\Gamma$-invariant.
The  regularity properties of 
$ A (t,v) $ are the same  of 
 the matrix  $ A(v)$ 
  in \eqref{Def AA} which, in turn, follow by  Lemma \ref{lem:regI}.
In view of 
 \eqref{Def AAQ}, \eqref{Def aQ}, \eqref{Estimate on d I1Q dritto}, 
\eqref{eq:Nabla cI}, \eqref{pibot},
that $ \Pi^\bot_{\hat \jmath}  $ commutes with  $ \nabla_v \cI^{\parallel_{\hat \jmath}} (v)  $ and $ \nabla_v \cI^{\bot_{\hat \jmath}} (v)  $,
 the matrix
$R^{\hat \jmath}(v):=A^{\hat \jmath}(v) - \cA^{\hat \jmath}(v)$ satisfies
\begin{equation}\label{eq:inside lemma AQ = cal AQ + error}
    \begin{aligned}
        & R^{\hat \jmath}(v) =
            \cO(\|v\|^3) \, , & & 
            \forall v\in \cF_{r'} \, , \\
        &  R^{\hat \jmath}(v)=
            \cO(\|v\|)\begin{pmatrix}
                \|v\|^2 & \|
                \Pi_{\hat \jmath}^\bot v\|^2\\
                \| \Pi_{\hat \jmath}^\bot v \|^2 & \| \Pi_{\hat \jmath}^\bot v \|^2
            \end{pmatrix} \, , & & 
            \forall 
            v\in \cN_{\hat \jmath}(r')\, , \ \hat \jmath\in\cD \, .
            \end{aligned}
\end{equation}
and hence $ A^{\hat \jmath}(t,v) = \cA^{\hat \jmath} (v) +  t \, (A^{\hat \jmath}(v) - \cA^{\hat \jmath}(v))=\cA^{\hat \jmath}(v)+tR^{\hat \jmath}(v)$ satisfies 
 \eqref{AQ = cal AQ + error} by 
\eqref{Estimate on a(v,v)}.
Let us prove \eqref{detpospe}.  
        Let $ R(v):=A(v)-\cA(v)=Q_{\hat \jmath}^{-1} R^{\hat \jmath}(v) Q_{\hat \jmath}$ (cf. \eqref{Def AAQ} and \eqref{coniugAQ}).
        For any  $ v\in \cF_{r''} $ 
        we have       
        $d(v,{\bf V}^{\twod})\geq \sqrt{\disb/4}\|v\|>0$ and then recalling  \eqref{eq:def At}, %\eqref{Def AAQ}, \eqref{eq:def At} and \eqref{coniugAQ}, 
        $$
            \det (A(t,v))= 
            \underbrace{\det (\cA(v))}_{ \gtrsim \|v\|^4 \, \text{by} \, \eqref{lem:c away from S* estimate 30}}
            \det \Big( \text{Id} +
            \underbrace{t\cA^{-1}(v)R(v)}_{= \cO(|t|\|v\|) \, \text{by} \,
            \eqref{lem:c away from S* estimate 3}, %\eqref{A= cal A + O(v3)}
            \eqref{eq:inside lemma AQ = cal AQ + error}
            } \Big)   \gtrsim \|v\|^4 
        $$
     for any $ t \in (-2,2) $ taking $ r'' \in (0,r')$ small enough, proving the first lower bound in \eqref{detpospe}.  Next
     we prove the second bound in \eqref{detpospe}. 
    By \eqref{Estimate on a(v,v)}
    and 
    \eqref{eq:inside lemma AQ = cal AQ + error} the matrix entries of 
$ A^{\hat \jmath}(t,v) $ satisfy
\begin{equation}\label{boundet}
    \begin{aligned}
    &         (A^{\hat \jmath}(t,v))_{11}\geq \|\Pi_{\hat \jmath} v\|^2-C|t|\|v\|^3\, , \quad  \ \,  (A^{\hat \jmath}(t,v))_{22}\geq \big(\disb -C|t|\|v\|\big)\|\Pi_{\hat \jmath}^\bot v\|^2 \, , \\
    & 
| (A^{\hat \jmath}(t,v))_{12}| , \ 
|(A^{\hat \jmath}(t,v))_{21} | \leq  
            \big(1+C|t|\|v\|\big)\|\Pi_{\hat \jmath}^\bot v\|^2 
    \end{aligned}
   \end{equation}
   for some  $C>0$.
   We deduce
$    (A^{\hat \jmath}(t,v))_{11} (A^{\hat \jmath}(t,v))_{22} 
    %& \geq  (\|\Pi_{\hat \jmath} v\|^2-C|t|\|v\|^3)(\disb -C|t|\|v\|)\|\Pi_{\hat \jmath}^\bot v\|^2\\ & 
    \geq 
    \disb \|\Pi_{\hat \jmath} v\|^2\|\Pi_{\hat \jmath}^\bot v\|^2-(\disb+1)C|t|\|v\|^3\|\Pi_{\hat \jmath}^\bot v\|^2 $ 
    and, assuming $C|t|\|v\|\leq 1 $, that 
   $ (A^{\hat \jmath}(t,v))_{12}^2
        %\leq (1+3C|t|\|v\|)\|\Pi_{\hat \jmath}^\bot v\|^4
        \leq \|\Pi_{\hat \jmath}^\bot v\|^4+3C|t|\|v\|^3\|\Pi_{\hat \jmath}^\bot v\|^2 $ 
so that  the determinant of 
        $ A(t,v) = 
        Q_{\hat \jmath}^{-1} A^{\hat \jmath}(t,v) Q_{\hat \jmath} 
        $ 
        satisfies 
 \begin{equation}\label{detAvok}
            \det(A(t,v))
             =\det(A^{\hat \jmath}(t,v))\geq (\disb \| \Pi_{\hat \jmath} v\|^2- \|\Pi_{\hat \jmath}^\bot v\|^2)\|\Pi_{\hat \jmath}^\bot v\|^2
            -(\delta+4)C|t|\|v\|^3\|\Pi_{\hat \jmath}^\bot v
            \|^2 \, .
            \end{equation}
    Recalling 
     \eqref{defNr},   
     if $v\in \cN_{\hat \jmath}(r'')$ then
    \begin{equation}\label{stimedis}
        \disb \|\Pi_{\hat \jmath} v\|^2- \|\Pi_{\hat \jmath}^\bot v\|^2\geq \tfrac \disb 2 \|v\|^2 \, , \quad \disb \in (0,1] \, .
     \end{equation}
        Hence by \eqref{detAvok} 
        and \eqref{stimedis} we conclude that 
        $$
            \det(A(t,v))
            \stackrel{ \eqref{stimedis}} \geq 
           \Big(\tfrac\disb 2-(\disb+4)C|t|\|v\|\Big)\|v\|^2\|\Pi_{\hat \jmath}^\bot v\|^2\geq \tfrac\disb3 \|v\|^2\|\Pi_{\hat \jmath}^\bot v\|^2
    $$
    for any  $ v \in B_{r''}^V $ with $ r'' \in (0,r')$ small enough.
    Finally the first inequality in \eqref{boundet} implies also
    the last bound in \eqref{detpospe} 
    for $r''\in (0,r')$ small enough,  since $ v = \Pi_{\hat \jmath} v  $
    for any $ v \in V_{\hat \jmath} $.
    \end{proof}
    
    The next key lemma  is used to prove Theorem \ref{th:Moser's trick} and Proposition \ref{A flow}.
    \begin{lemma}{\bf (Solution of  $A(t,v)^\top \mu(v)=b(v)$)}\label{lm:linear equation}
        Let  $b:B_{r'}^V\to \R^2$ be a $\T^2_\Gamma$-invariant continuous function, analytic in $B_{r'}^V\setminus {\bf V}^{\twod}$ and restricted to each $(B_{r'}^V\cap V_{\hat \jmath})\setminus\{0\}$, ${\hat \jmath} \in \cD $, satisfying 
        \begin{equation}\label{estbj}
            b^{\hat \jmath} (v)=Q_{\hat \jmath} b(v)=\cO\begin{pmatrix}
                \|v\|^2\\
                \|\Pi_{\hat \jmath}^\bot v\|^2
            \end{pmatrix} \, , \quad 
                   \forall  v\in B_{r'}^V \, , \ {\hat \jmath} \in \cD \, , 
        \end{equation}
        where $Q_{\hat \jmath}$ is defined in \eqref{eq:def yj}.
        Then, 
        there exists a $\T^2_\Gamma$-invariant function $\mu: \JJJ \times B_{r''}^V \to \R^2$ solving the system 
        \begin{equation}\label{eq:linear system}
            A(t,v)^\top\mu(t,v)=b(v) \, , 
            \quad \forall (t,v)\in \JJJ\times B_{r''}^V \, , 
        \end{equation}
        where the matrix $ A(t,v) $ is defined in \eqref{eq:def At} and $r''>0$ is the same of Lemma \ref{Estimates on A, AQ, tilde AQ11}.
        Moreover
        the vector field 
        \begin{equation}\label{defZ}
        Z:\JJJ \times B_{r''}^V\to V \, , \quad       Z(t,v):=\mu(t,v)\cdot \nabla_v  \cI(v) \, , 
        \end{equation}
        is $\T^2_\Gamma$-equivariant and satisfies the properties \eqref{list:Properties of Z} (stated in  Lemma \ref{lm:B1}).
    \end{lemma}
    
    \begin{proof}
    By Lemma \ref{Estimates on A, AQ, tilde AQ11} there is $r''\in (0,r')$ such that 
  the matrix $A(t,v)^\top$ is invertible for any $(t,v)\in \JJJ\times [B_{r''}^V\setminus{\bf V}^{\twod}]$.
    Then the solution of system 
    \eqref{eq:linear system} is 
% $\mu:J\times B_{r''}^V\to \R^2$,  
\begin{equation}\label{def:mu}
\mu:\JJJ\times B_{r''}^V\to \R^2 \, , 
\ 
            \mu(t,v):=\begin{cases}
                \mu_{\threed}(t,v) & t \in \JJJ \, , \, v\in B_{r''}^V\setminus {\bf V}^{\twod}\\
                \mu_{\twod}^{\parallel_{\hat \jmath}}\hat\jmath & t \in \JJJ \, , \, v\in (B_{r''}^V\cap V_{\hat \jmath})\setminus\{0\} \, , \ {\hat \jmath} \in \cD \\
                0 & t \in \JJJ \, , \, v=0\, , 
            \end{cases} 
        \end{equation}
where 
        \begin{align}\label{defmu3}
     &    \mu_{\threed} : \JJJ\times [B_{r''}^V\setminus {\bf V}^{\twod}] \to \R^2 \, ,\quad    (t,v) \mapsto  \mu_{\threed}(t,v):=(A(t,v))^{-\top}b(v)\,  
     \\ 
&   \mu^{\parallel_{\hat \jmath}}_{\twod}:\JJJ\times [(B_{r''}^V\cap V_{\hat \jmath})\setminus\{0\}]\to \R \, , \quad (t, v) \mapsto          \mu^{\parallel_{\hat \jmath}}_{\twod}(t,v):=(A^{\hat \jmath}(t,v))_{11}^{-1}b^{\parallel_{\hat \jmath}}(v)\, , \ {\hat \jmath} \in \cD \, , \label{defmu2}
    \end{align}  
where $ (A^{\hat \jmath}(t,v))_{11} 
\gtrsim \| v \|^2 $ for $v\in (B_{r''}^V\cap V_{\hat \jmath})\setminus\{0\}$ by  \eqref{detpospe}.
    \\[1mm]
 {\sc Step 1:} {\it The function $\mu(t,v)$ satisfies the following properties:
        \begin{itemize}
            \item[$(i)$] $\mu(t,v)$ is analytic in $\JJJ\times[B_{r''}^V\setminus {\bf V}^{\twod}]$ and restricted to   $\JJJ\times[(B_{r''}^V\cap V_{\hat \jmath})\setminus\{0\}]$, for any $ {\hat \jmath} \in \cD $; 
            \item[$(ii)$] $\mu(t,v)$ is  $\T^2_\Gamma $-invariant
            and bounded; 
            \item[$(iii)$]  $\mu^{\parallel_{\hat \jmath}}(t, v)$ for any  ${\hat \jmath} \in \cD $ is continuous on $\JJJ\times [B_{r''}^{{\hat \jmath}}\setminus\{0\}]$, cf. \eqref{c1Q U tilde c1Q}.
        \end{itemize}}
The functions $ \mu_{\threed} $
and $ \mu^{\parallel_{\hat \jmath}}_{\twod} $ in \eqref{defmu3}-\eqref{defmu2} are  analytic by 
the analyticity properties of 
$ A(t,v) $ of 
     Lemma \ref{Estimates on A, AQ, tilde AQ11} and 
the analyticity 
assumptions on $ b(v) $.  
Furthermore  $\mu(t,v)$ is $\T^2_\Gamma $-invariant as $b(v) $ (by assumption) and $ A(t,v)$ are invariant as well.
The estimates   
    \eqref{AQ = cal AQ + error}-\eqref{detpospe} imply that
        \begin{equation}\label{eq:inverse di A}
        \begin{aligned}
        &     A(t,v)^{-\top}=\cO\begin{pmatrix}\frac1{\|v\|^2}\end{pmatrix} & &
            \forall (t,v)\in \JJJ\times \cF_{r''}\, , \\
         &    (A^{\hat \jmath}(t,v))^{-\top}=\cO\begin{pmatrix}
                \frac1{\|v\|^2} & \frac1{\|v\|^2}\\
                \frac1{\|v\|^2} & \frac1{\|\Pi_{\hat \jmath}^\bot v \|^2}
            \end{pmatrix} & &
            \forall 
            (t,v)\in \JJJ\times [\cN_{\hat \jmath}(r'')\setminus V_{\hat \jmath}]\, , \ \hat \jmath\in\cD \, ,  \\
        &     (A^{\hat \jmath}(t,v))_{11}^{-1}=\cO\begin{pmatrix}\frac1{\|\Pi_{\hat \jmath} v\|^2}\end{pmatrix} = \cO\begin{pmatrix}\frac1{\|v\|^2}\end{pmatrix} & &
            \forall (t,v)\in \JJJ\times [(B_{r''}^V\cap V_{\hat \jmath})\setminus\{0\}]\, ,
        \end{aligned}
        \end{equation}
        and jointly with  
        the estimates \eqref{estbj} on $b(v)$, we deduce  
        that $\mu_{\threed}(v)$ and
        $        \mu^{\parallel_{\hat \jmath}}_{2d} (v)$ 
        in \eqref{defmu3}-\eqref{defmu2} 
        are uniformly bounded.
 Let us prove $(iii)$.
        Let $(t_i,v_i)\to (\bar t,\bar v)\in 
        \JJJ\times [(B_{r''}^V\cap V_{\hat \jmath})\setminus\{0\}]$.
        For any subsequence $(t_{i_m},v_{i_m})$ there exists another subsequence $(t_{i_{m_\ell}},v_{i_{m_\ell}})$ and $ \bar \mu\in \R^2 $ such that $\mu(t_{i_{m_\ell}},v_{i_{m_\ell}})\to \bar \mu $.
        Passing to the limit in \eqref{eq:linear system} we deduce that $\bar \mu$ solves  
       $ A(\bar t,\bar v)^\top
       \bar \mu =b(\bar v)$, thus 
       $ A^{\hat \jmath}(\bar t,\bar v)^\top \bar \mu^{\hat \jmath} = b^{\hat \jmath} (\bar v)$ for any 
       $ \hat \jmath \in \cD $, thus, by \eqref{AQ = cal AQ + error}, 
        $$
            (A(\bar t, \bar v))_{11}\bar \mu^{\parallel_{\hat \jmath}}=b^{\parallel_{\hat \jmath}}(\bar v)\quad \text{where} \quad \bar \mu^{\parallel_{\hat \jmath}}=\bar \mu\cdot \hat\jmath \quad \text{(cf. \eqref{eq:compoc})}
        $$
        and hence $\bar \mu^{\parallel_{\hat \jmath}}=\mu_{\twod}^{\parallel_{\hat \jmath}}(\bar t, \bar v)$ in \eqref{defmu2}.
        The proof of $(iii)$ is concluded.
\\[1mm]
{\sc Step 2:}
{\it The  vector field $Z(t,v)$ 
        in \eqref{defZ}
        satisfies properties \eqref{list:Properties of Z}. } 
        Since $\mu(t,v)$ is $\T^2_\Gamma$-invariant
        and $ \nabla_v  \cI(v) $ is 
        $\T^2_\Gamma$-equivariant,  the vector field $Z(t,v)$ in \eqref{defZ} is $\T^2_\Gamma$-equivariant as well.
        In view of \eqref{eq:Nabla cI} we have
        \begin{equation}\label{eq:Z explicit formula}
            Z(t,v)=-\sum_{{\hat \jmath}\in \cD}\mu(t,v)\cdot \hat\jmath\,  \Pi_{\hat \jmath}v=-\sum_{{\hat \jmath}\in \cD}\mu^{\parallel_{\hat \jmath}}(t,v)\Pi_{\hat \jmath}v\, .
        \end{equation}
        Item
        $(iii)$ implies that 
        % the function  $v\mapsto \mu^{\parallel_{\hat \jmath}}(t,v)\Pi_{\hat \jmath}v$ 
        $ Z(t,v) $ is continuous on $ \JJJ\times (B_{r''}^V \setminus \{0\})$ and
         $(ii)$ implies that 
         $ \| Z(t,v) \| \lesssim \| v \| $ thus it is continuous
         also at $ v = 0 $. 
        Properties \eqref{list:Properties of Z}-$(1)$ and $(3)$ are proved.
        Property \eqref{list:Properties of Z}-$(2)$ follows from $(i)$, and  \eqref{list:Properties of Z}-$(4)$ from \eqref{eq:Z explicit formula}.
    \end{proof}

\subsection{Rectification of the  momentum }\label{sec:MT}

To prove Theorem \ref{th:Moser's trick}
we implement a suitable Moser's trick.
For any $t\in \R $ and $v\in B_{r'}^V$ let 
\begin{equation}\label{Def ItQ}
    I(t,v):= t I(v)+(1-t)\cI(v)\, .
\end{equation}
We look for  a family of equivariant homeomorphisms $ \zeta (t,\cdot)  $ satisfying 
\eqref{diffeoint}, \eqref{diffeobordo} and 
\begin{equation}\label{eq:forma}
    I(t, \cdot)  \circ \zeta (t, \cdot)  := I(0, \cdot) 
    \stackrel{\eqref{Def ItQ}} = \cI (\cdot )   \, , \quad \forall t \in [0,1] \, , \quad \zeta (0, \cdot)  = \text{Id} \, , 
\end{equation}
so that at $ t = 1 $ the map 
$ \zeta(1, \cdot)   $
solves \eqref{Izeta}. 
We look for $ \zeta(t, \cdot)  $ as the flow 
$$
\pa_t \zeta (t, v) = X(t,\zeta (t, v)) \, , \quad 
\zeta (0, v) = v  \, , 
$$
generated by a time dependent vector field of the form 
\begin{equation}\label{Def X}
X:\JJJ\times B_{r''}^V\to V \, , \quad
    X(t,v):=\mu(t,v)\cdot \nabla_v  \cI(v)  \, , 
\end{equation}
for some $ \mu:\JJJ\times B_{r''}^V \to  \R^2 $
and   $r''\in (0,r')$ sufficiently small. 
By differentiating in $ t $,   \eqref{eq:forma} is equivalent to
\begin{equation}\label{Moser's trick}
        \di_v I(t,v)[X(t,v)] 
        + (I- \cI)(v) = 0  
        \, , 
\end{equation} 
that, by  \eqref{Def X} and recalling \eqref{eq:def At},
\eqref{Def AA}, \eqref{Def AAQ}, \eqref{Def ItQ},  amounts to the system 
\begin{equation}\label{EQ k}
        A^\top(t,v) \mu(t,v)=\cI(v)-I(v) 
\end{equation}
where $ A(t,v) $ is the matrix defined 
in \eqref{eq:def At}. 
 Lemmata \ref{lem:regI},  \ref{Properties of I dritto} imply that 
$ b(v) := \cI(v)-I(v) $ satisfies the assumptions of Lemma \ref{lm:linear equation}
and 
  we deduce the  existence of $ r''\in (0,r')$ and 
    $ \mu:\JJJ\times B_{r''}^V \to  \R^2 $, 
        such that the vector-field \eqref{Def X} solves %\eqref{Moser's trick} (i.e 
        \eqref{EQ k} and satisfies the assumptions of Lemma \ref{lm:B1}.
    Therefore Lemma \ref{lm:B1} implies  that the vector-field $ X(t,v) $ in \eqref{Def X} generates 
     flow $\zeta (t, v) $ on $B_{r''}^V$ locally in time, well defined  in $[0,1] \times B_{r'''}^V $ 
    for some  $r'''\in(0,r'')$.
    We set $U_1:=B_{r'''}^V$.
    By  the Brouwer theorem on invariance of domains we deduce that 
    \begin{equation}\label{def:U2}
        U_2:=\zeta(1, U_1)\subset B_{r''}^V\quad \text{is open} \, . 
    \end{equation}
    In view of Remark \ref{rem:Vinv} and the regularity properties of $X(t,v)$ (specifically \eqref{list:Properties of Z}-$(2)$) we deduce \eqref{diffeoint}-\eqref{diffeobordo}. The  proof 
    %of Theorem \ref{th:Moser's trick} 
    is concluded.

\section{Topology of $ \MA$}

In this section we carefully 
describe the topology of  
% $ \MA $ 
%the next Section %\ref{sec:final} we shall 
%prove existence and multiplicity of $3d$-Stokes waves as  stationary points of a gradient-like flow 
the  level sets 
\begin{equation}\label{defSa}
     \MA:=
    \Big\{ v\in U_2  \ | 
    \ I(v)=a \Big\}  \, , \quad a \in \R^2 \, , 
\end{equation}
where $ I(v)$ is the reduced momentum in \eqref{Def I dritto}
and
$U_2 $ is the open neighborhood of $ 0 $ defined  in Theorem \ref{th:Moser's trick}.
To  start 
note that 
%with simple 
%considerations. 
%Clearly 
$ \MA $
  is a $ \T^2_\Gamma $-space 
equipped with the linear action $ (\tau_\theta ) $
in \eqref{T2 symmetry}.
Furthermore 

\begin{lemma}\label{rmk:nablaI12}
   $ \MA^{3d} := \MA \setminus {\bf V}^{\twod} $ is a manifold and
\begin{equation}\label{supplem}
    V= \underbrace{T_v\MA^{\threed}}_{= \ker \di_v I (v)  }\oplus \, \text{span}(\nabla_v \cI_1(v),\nabla_v \cI_2(v))\, ,
    \quad 
\forall  v \in \MA^{3d} \,  .
\end{equation}
\end{lemma}

\begin{proof}
  The set  $\MA^{\threed}\subset U_2\subset B_{r''}^V$  by  Remark \ref{rmk:U2 subset Br''} and, for any 
  $ v \in \MA^{3d} $,   Lemma \ref{Estimates on A, AQ, tilde AQ11} implies that the matrix $A(v)=A(1,v)$ in \eqref{Def AA}, \eqref{eq:def At}  is invertible.
    Therefore $
     \{\di_v I_1(v),\di_v I_2(v)\}$  are linearly independent and by Lemma \ref{lem:indip} we deduce \eqref{supplem}.
\end{proof}

%We now provide the topological classification of $ \MA $. 
For any $ a \in \R^2 $ we consider the orthogonal 
decomposition  
\begin{equation}\label{eq:def V-0+}
    V = V_- \oplus V_0 \oplus V_+   
 \end{equation}   
where
\begin{equation}\label{eq:def cD-0+}
V_-:=\bigoplus_{\hat \jmath \in \cD_- }V_{\hat \jmath}\, , \quad  V_0:=\bigoplus_{\hat \jmath \in \cD_0} V_{\hat \jmath} \, , \quad  V_+:=\bigoplus_{\hat \jmath \in \cD_+ }V_{\hat \jmath}  \, , \end{equation}
and
$$
     \cD_-:= \big\{ \hat \jmath\in \cD \, | \, \hat \jmath 
     \cdot a^\bot <0 \big\} \, ,      \ 
        \cD_0:= \big\{ \hat \jmath\in \cD \, | \, 
        \hat \jmath \cdot a^\bot =0 \big\}\, , \ \cD_+:= \big\{ \hat \jmath\in \cD \, | \, 
        \hat \jmath \cdot a^\bot  >0 
        \big\}\, .
$$
The subspaces 
$ V_-, V_0, V_+ $ are  invariant
under the linear action $ \tau_\theta $ in \eqref{T2 symmetry}.

\begin{remark}\label{lem:linese}
If $a\in \interior(\cC)$  is in the interior of
the  cone  defined   in \eqref{defC} 
  then  
  $\cD_\pm \not=\emptyset$ and so 
  $\dim(V_\pm ) \geq 2$.
 % Here  $\dim(\cdot)$ denotes the real dimension of a vector space. 
%Let $\cC$ be    
 %[$(\beth)$]  {\bf (Boundary case)}  
%    
If 
$a \in \partial \cC\setminus\{0\} $ 
 then 
  $\cD_0 \not =\emptyset $, so  
$\dim(V_0)\geq 2 $, and   
%since $\cC$ is  convex, 
either $\cD_- =\emptyset$ or $\cD_+ =\emptyset$.   
\end{remark}

We denote $S(V_-) $, $S(V_0) $, $S(V_+) $ the unit spheres 
in $ V_-, V_0, V_+  $.

\begin{theorem}\label{Topology of Sa}
{\bf (Topology of $\MA$)}
    If $a\not \in \cC$ 
    %does not 
    %belong to  the convex cone $\cC$ defined   in \eqref{defC} 
    then $ \MA =\emptyset $.
    If $a=0$ then $\MA=\{0\}$. 
    Moreover there exists $\varepsilon>0$ such that for any $a\in \cC\cap \overline{B_\varepsilon} \setminus \{0\} $ the following holds:   
    \begin{itemize}
        \item[$(\beth)$] {\bf (Boundary case)} If $a\in \partial \cC\setminus\{0\}$ then   
        $ \MA \subset  V_0 $ is a submanifold of $ V_0 $ which is $\T^2_{\Gamma}$-equivariantly diffeomorphic to the unit sphere 
        $ S(V_0) $, i.e. 
        $   \MA \cong S(V_0) $.

        \item[$(\daleth)$] {\bf (Interior case)} If $a\in \interior(\cC)$  then  two cases may occur:
        \begin{enumerate}
            \item[$(\daleth$1)] {\bf (Non-collinear case)} If $a 
            \not \parallel \hat \jmath $  for any $ \hat \jmath \in \cD $,  then  $\MA$ is a closed submanifold of $V$  contained in $V\setminus 
            {\bf V}^{\twod} $ which is 
            $\T^2_{\Gamma}$-equivariantly diffeomorphic to   $ S(V_-)\times S(V_+) $. 
        
            \item[$(\daleth$2) ] {\bf (Collinear case)} If $a 
             \parallel \hat \jmath_0 $ is parallel to some resonant wave direction $ \hat \jmath_0 \in \cD $ then  
            \begin{equation}\label{Sa3dSa2d}
                           \MA =  \MA^{\threed} \, \sqcup  \, \MA^{\twod} 
\end{equation}
where  
$ \MA^{\threed}:=
 \MA\setminus {\bf V}^{\twod} = \MA\setminus V_0 $ is a submanifold of $ V $ and 
    $ \MA^{\twod}:=\MA\cap V_0
            $ is a submanifold of $ 
                V_0 $  diffeomorphic to the unit sphere    $ S(V_0)$  
(note that $V_0=V_{\hat \jmath_0}$ and $ \dim (V_0) \geq 2 $). There is a $\T^2_{\Gamma}$-equivariant homeomorphism           \begin{equation}\label{gamma homeomorphism}
         \quad            \gamma:\MA\toup^{\cong}
                    [S(V_-)\times S(V_+)] \star S(V_0)\, ,
                    \quad \gamma(\MA^{\twod})=S(V_0) \, , 
                \end{equation}
                where $\star$ denotes the {\sc join product} \eqref{defjoin} and $\dim(V_\pm ) \geq 2 $.
            \end{enumerate} 
    \end{itemize}
\end{theorem}

    The case $(\daleth$1) was  proved in \cite{CN}. The novel cases  are $(\beth) $ and $(\daleth$2). 

    We now prove  
    Theorem \ref{Topology of Sa}.
     The straightening Theorem \ref{th:Moser's trick}  implies that $I(v)=\cI ( \zeta^{-1}(v))  
 $ for any $v\in U_2$. 
    By \eqref{eq:cI} the image of $ \cI : V \to \R^2 $ is the cone  $\cC$  in \eqref{defC}
    and therefore, if $a\not \in \cC$ we deduce that   $ \MA=\emptyset$.
    Furthermore 
    if $a=0$ then $ \cI^{-1} (a) = 0 $ and so 
    $\MA=\{0\}$.

%\begin{lemma}\label{a nel cono}
%    There exists $\varepsilon>0$ such that, for any $ |a| \leq \varepsilon  $, 
%\begin{equation}\label{eq:Depsilon subsetsubset U2}
%        \cI^{-1}(a) \subset U_1 \quad \text{and} \quad \overline{D_\varepsilon}\subset U_2\quad \text{where} \quad D_\varepsilon:=\{v\in U_2 \, | \, |I(v)|\leq \varepsilon\}\, . 
%    \end{equation}
%    For any   $a\in \cC\cap \overline{B_\varepsilon}$
%    the set $\MA $ in \eqref{defSa} is the non-empty compact set $ \MA = \zeta ( \cI^{-1}(a))  $.   
%\end{lemma}
%\begin{proof}
%If $ v \in \cI^{-1}(a) $ then, by \eqref{eq:cI} and \eqref{Def cV} we deduce   
%$ \sum_{{\hat \jmath}\in \cD} \frac{\omega(j)}{2|j|} \|\Pi_{\hat \jmath} v\|^2 \leq |a | |c_* | $.
%Since $ \omega (j) > 0 $ we deduce  that \eqref{eq:Depsilon subsetsubset U2} holds for any $ |a| \leq \varepsilon $ small enough.  
%\end{proof}

%\begin{lemma}\label{a nel cono}
%    For any open neighborhoods 
%    $ U_1'  \subseteq U_1  $ and $U_2'\subseteq U_2$ of $0$ in $V$, there exists $\varepsilon>0$ such that, for any $ |a| \leq \varepsilon  $, 
%\begin{equation}\label{eq:Depsilon subsetsubset U2}
%        \cI^{-1}(a) \subset U_1' \quad \text{and} \quad \overline{D_\varepsilon}\subset U_2'\quad \text{where} \quad D_\varepsilon:=\{v\in U_2 \, | \, |I(v)|\leq \varepsilon\}\, . 
%    \end{equation}
%\end{lemma}
\begin{lemma}\label{a nel cono}
    There exists $\varepsilon>0$ such that, for any $ |a| \leq \varepsilon  $, 
\begin{equation}\label{eq:Depsilon subsetsubset U2}
        \cI^{-1}(a) \subset U_1 \quad \text{and} \quad \overline{D_\varepsilon}\subset U_2\quad \text{where} \quad D_\varepsilon:=\{v\in U_2 \, | \, |I(v)|\leq \varepsilon\}\, . 
    \end{equation}
\end{lemma}
\begin{proof}
If $ v \in \cI^{-1}(a) $ then,
by \eqref{eq:cI} and \eqref{Def cV}
we deduce   
$ \sum_{{\hat \jmath}\in \cD} \frac{\omega(j)}{2|j|} 
        \|\Pi_{\hat \jmath} v\|^2 
        \leq |a | |c_* | $. Since $ \omega (j) > 0 $ 
we deduce   \eqref{eq:Depsilon subsetsubset U2}  
for any $ |a| \leq \varepsilon $ small enough.  
\end{proof}
    %We choose $\varepsilon>0$ such that \eqref{eq:Depsilon subsetsubset U2} holds with $U_1'=U_1$ and $U_2'=U_2 $. 
   % Thus f
   For any   $a\in \cC\cap \overline{B_\varepsilon}$
    the set $\MA $ in \eqref{defSa} is the non-empty compact set $ \MA = \zeta ( \cI^{-1}(a))  $, and 
    it is then sufficient to characterize  the level sets of the quadratic momentum 
    $ \cI (v) $ in \eqref{eq:cI}.

        \begin{lemma}{\bf (Level sets of $\cI$)}\label{prop:levelsets of cI}
           If  $a\in \cC \setminus \{0\} $ then 
    \\[1mm]     
        $(\beth)$
               If $a\in \partial \cC\setminus\{0\}$ then  $\cI^{-1}(a) \subset V_0 $ is $\T^2_\Gamma$-equivariantly diffeomorphic to $S(V_0)$.
               \\[1mm]
  $(\daleth)$
                If $a\in \interior(\cC)\setminus\{0\}  $ then the following exclusive cases may occur:
                \begin{itemize}
                    \item[$(\daleth$1)] 
                     If $a\not \parallel \hat \jmath$ for any $\hat \jmath\in \cD$ then  $\cI^{-1}(a)$ is a closed submanifold of $V$  contained in $V\setminus 
                    {\bf V}^{\twod} $ which is 
                    $\T^2_\Gamma$-equivariantly diffeomorphic to $S(V_-)\times S(V_+) $.
                    
                    \item[$(\daleth$2)] 
                     If $a\parallel \hat \jmath_0$ for some resonant wave direction $\hat \jmath_0\in \cD$
                    then  
 $       \cI^{-1}(a)\setminus {\bf V}^{\twod}  = 
                     \cI^{-1}(a)\setminus  V_0  $
is a submanifold of $  V $, 
       $  \cI^{-1}(a)\cap {\bf V}^{\twod}  = \cI^{-1}(a)\cap V_0 $ is a submanifold of  $ V_0 = V_{\hat \jmath_0} $, diffeomorphic to 
       $ S(V_0) $ 
                      and there is a $\T^2_\Gamma$-equivariant homeomorphism 
                    \begin{equation}\label{gamma homeomorphism_2}
                        \gamma_1:\cI^{-1}(a)\toup^{\cong}[S(V_-)\times S(V_+)] \star S(V_0)\, ,
                        \quad \gamma_1(\cI^{-1}(a) {\cap {\bf V}^{2d}})=S(V_0) \, .
                    \end{equation}
                \end{itemize}
        \end{lemma}    
        \begin{proof} 
       
    Projecting  $ \cI (v) $ in \eqref{eq:cI}  
       along  $ a $ and 
       $ a^\bot $, 
    we write
    \begin{equation}
    \label{J-1a}
    \cI^{-1} (a) 
    = \Big\{ v\in V \ | \
    \sum_{\hat \jmath\in \cD}\frac{-\hat \jmath \cdot a}{2|a|^2}\|\Pi_{\hat \jmath}v\|^2=1 \, ,  \ \sum_{\hat \jmath\in \cD}\frac{-\hat \jmath \cdot a^\bot}{2|a|^2}\|\Pi_{\hat \jmath}v\|^2=0 
    \Big\} \, . 
\end{equation}
By \eqref{eq:ilcono}, 
we have $a\cdot c_*<0 $ and
 summing the second equation in \eqref{J-1a} multiplied by $\frac{ a^\bot\cdot c_*}{a\cdot c_*}$ to the first equation, 
 and decomposing  $  c_* = 
a (c_* \cdot a)|a|^{-2} + a^\bot 
(c_* \cdot a^\bot) |a|^{-2} $, we deduce that 
\begin{equation}\label{J-2a}
  \cI^{-1} (a)  = \Big\{ v\in V \ | \
    \cB(v)=1 \, ,  \ \cK(v)=0 
    \Big\}
    \end{equation}
    where
     \begin{equation}\label{cBcK}
            \cB(v):=\frac{1}{2|a\cdot c_*|}\sum_{\hat \jmath\in \cD} c_*\cdot \hat \jmath \, \|\Pi_{\hat \jmath} v\|^2 \, , \quad 
            \cK(v):=\sum_{\hat \jmath \in \cD} \hat \jmath \cdot a^\bot \|\Pi_{\hat \jmath} v\|^2\, .
    \end{equation}
    Note that $c_*\cdot \hat \jmath  >0$ for any $ \hat \jmath \in \cD$ (cf. \eqref{Def cV}) and hence $ \sqrt{\cB(v)}$ is a norm on $ V $. 

    We define the $\T^2_\Gamma$-equivariant diffeomorphism
\begin{equation}\label{diffeoxi}
\xi:V\setminus\{0\}\to V\setminus\{0\} \, , \quad
            \xi(v):= g(v) \psi(v)
        \, ,  \quad  
g(v) :=  \frac{\vvv v\vvv}{\sqrt{\cB( \psi(v))}}     \, , 
\end{equation}
where $ \psi:V\to V $
            is the linear isomorphism    
      \begin{equation}\label{defpsiso}
         \psi(v) := 
         \sum_{\hat \jmath\in \cD_-} |\hat \jmath \cdot a^\bot|^{-\frac12}\, \Pi_{\hat \jmath} v  +
        \sum_{\hat \jmath\in \cD_0} \Pi_{\hat \jmath} v
         + \sum_{\hat \jmath\in \cD_+} (\hat \jmath\cdot {a^\bot})^{-\frac12}\, \Pi_{\hat \jmath} v 
        \end{equation}
  (it is well defined in view of \eqref{eq:def cD-0+}) 
  and
\begin{equation}\label{def3bar}
\vvv v\vvv^2:=\tfrac 12\|\Pi_- v\|^2+\tfrac 12\|\Pi_+ v\|^2+\|\Pi_0 v\|^2 \, , 
\end{equation}
denoting 
$ \Pi_\pm :V\to V_\pm $ and $\Pi_0:V\to V_0$ 
 the orthogonal projectors associated to the decomposition \eqref{eq:def V-0+}.   
Note that $ g(v) $ in \eqref{diffeoxi} is a  
positive,  $ C^\infty (V \setminus \{0\}, \R) $
homogenous function  
of degree $ 0 $.
Under the diffeomorphism  $ \xi $ in \eqref{diffeoxi}
the functions $ \cB $ and $ \cK $ in \eqref{cBcK} transform into
\begin{equation}\label{futras}
            (\cB \circ \xi)(v) =\vvv v\vvv^2\, ,\quad 
            (\cK \circ \xi)(v) =g^2(v)\big(\|\Pi_+ v\|^2 -\|\Pi_- v\|^2\big)\, ,
    \quad \forall 
    v\in V\setminus\{0\} \,            .
\end{equation}
As a consequence of \eqref{futras} and \eqref{def3bar}, the map   
$\xi$ defines a 
        $ \T^2_\Gamma $-equivariant diffeomorphism which maps the set $\cI^{-1}(a)$ 
        in \eqref{J-2a}  onto 
\begin{equation}\label{eq:Sigma}
        \Sigma := \Big\{ v \in V 
        \ | \  
            \tfrac12\|\Pi_- v\|^2+\tfrac 12\|\Pi_+ v\|^2+\|\Pi_0 v\|^2=1 \, , \|\Pi_+ v\|^2 =\|\Pi_- v\|^2 \Big\}  \, . 
        \end{equation} 
       $(\beth)$ {\sc Case $a\in \partial \cC\setminus\{0\}$}.
        In view of  Remark \ref{lem:linese},
        either $ \cD_- $ or $ \cD_+ $ is empty, say $ \cD_- = \emptyset $. Then   
        the set $ \Sigma $ in \eqref{eq:Sigma} 
        reduces to
        $\Sigma= \{ v \in V \ | \  \Pi_+ v = 0 \, , \ 
        \| \Pi_0 v \|^2 = 1 \} \equiv S(V_0) $ and  
        $\xi:
           S(V_0) \to \cI^{-1}(a) 
           = \xi  (S(V_0)) \subset V_0 $ 
           (the map $ \psi$ in \eqref{defpsiso}, thus $ \xi $,  maps $  V_0\setminus\{0\} \to V_0\setminus\{0\} $) is a $\T^2_\Gamma$-equivariant diffeomorphism.
           \\[1mm]
     $(\daleth)$ {\sc Case  $a\in \interior(\cC)$}.
        Remark \ref{lem:linese} shows that 
        $\dim(V_\pm )  \geq 2$. 
        We distinguish two  sub-cases:
        \begin{itemize}
            \item[$(\daleth$1)] Assume $ a \not \parallel \hat \jmath $ for any $\hat \jmath\in \cD$. Then $\cD_0=\emptyset$ and 
           \eqref{eq:Sigma} reduces to
        $\Sigma= \{ v \in V \ | \  
        \| \Pi_- v \|^2 = 
        \| \Pi_+ v \|^2 = 1 \} 
           $
           $ \equiv S(V_-)\times S(V_+)\subset V\setminus{\bf V}^{\twod}$. Hence $\xi: S(V_-) \times S(V_+) \to \cI^{-1}(a) $ 
             is a $\T^2_\Gamma$-equivariant diffeomorphism and Remark \ref{rem:Vinv}  implies  $ \cI^{-1}(a)\subset V\setminus{\bf V}^{\twod}$.
            \item[$(\daleth$2)] Assume $a  \parallel \hat \jmath_0 $ for some $\hat \jmath_0 \in \cD $.
Since the diffeomorphism 
$ \xi $ in \eqref{diffeoxi} maps each 
$ V_{\hat \jmath} \setminus \{0  \} $ in itself and  
$ V \setminus  {\bf V}^{\twod} $ in itself, we have
$$
\xi : 
\underbrace{\Sigma \cap 
{\bf V}^{\twod}}_{= S(V_0) \, \text{by} \,  \eqref{eq:Sigma}} \to \underbrace{\cI^{-1}(a)\cap {\bf V}^{\twod}}_{= \cI^{-1}(a)\cap V_0} 
\, , \quad \xi : \Sigma \setminus  
{\bf V}^{\twod} \to \cI^{-1}(a)\setminus  {\bf V}^{\twod} =
\cI^{-1}(a)\setminus V_0\, , 
% = \xi (\Sigma \cap {\bf V}^{\twod}) 
$$         
and  $ \cI^{-1}(a)\setminus {\bf V}^{\twod}$  is a sub-manifold of $V$ of codimension two because, 
for any $v\in V\setminus {\bf V}^{\twod}$ the vectors  $\{\nabla_v \cI_1(v), \nabla_v \cI_2 (v)\}$ are linearly independent 
by Lemma \ref{Properties of a}.        
         We now prove that 
         $ \Sigma $
         is homeomorphic to the join 
         topological space $ M \star Y $
         where $ M := S(V_-)\times S(V_+) $ and 
         $ Y := S(V_0) $. 
        The map
        $$
        f : M \times   Y \times [0,1] \to \Sigma \, , \quad  (v_1, v_0,t) \mapsto 
        f(v_1, v_0, t) := \sqrt{1-t^2}\, v_1 + {t}\, v_0 \, ,
        $$
        is 
        continuous, surjective, $ \T^2_\Gamma $-equivariant,  
and defines a map $ \tilde f $ on the quotient space
\[
\begin{tikzcd}[row sep=large]
 M \times Y \times [0,1] 
 \arrow[d,"\pi"']  \arrow[r,"f"]
 & 
\Sigma 
 \\
M \star Y = \frac{M \times Y\times[0,1]}\sim 
 \arrow[ur,"\tilde f "',"\cong"]
 & 
\end{tikzcd}
\]
where $ \sim $ is the equivalence relation defined in \eqref{eqrela}.
The map $ \tilde f $ is continuous, bijective and since $ M, Y $ are compact it is a $ \T^2_\Gamma $-equivariant homeomorphism between 
$ M \star Y $ and $ \Sigma $. The homeomorphism \eqref{gamma homeomorphism_2} is 
$ \gamma_1 := \tilde f^{-1} \circ \xi^{-1} $ (note that $ \tilde f (Y) = S(V_0) $).        
        \qedhere
        \end{itemize}
        \end{proof}
        
\section{Existence and multiplicity of $3d$ Stokes waves}
\label{sec:final}

In this section we finally prove Theorem \ref{Existence of 3d solutions collinear nonresonant} regarding 
existence and multiplicity of truly $ 3d $ Stokes waves 
having  
momentum 
    $    \cI (u) = a$, 
   {\it collinear} with exactly {\it one} resonant wave vector $j\in \cV$. Specifically 
   we take any 
   \begin{equation} \label{comeprendoa}  
   a\in  \interior(\cC) \cap \overline{B_\varepsilon} 
   \quad \text{where} \quad   
   \varepsilon>0 \
\text{is fixed as in Theorem }  \ref{Topology of Sa} 
   \end{equation} 
 (i.e. as in Lemma 
 \ref{a nel cono}), and
\begin{equation} \label{comeprendoa1}
a \parallel j_0 \ \text{for some} \ j_0 \in \cV \, , \quad 
a \not \parallel j \, , \ \forall  j \in \cV \setminus \{ j_0  \} \, . 
\end{equation}
In this case 
the $\T^2_\Gamma$-space 
   $\MA$ in 
\eqref{defSa} is $\T^2_\Gamma$-equivariantly homeomorphic to
    $[S(V_-)\times S(V_+)] \star S^1$
    where $ S^1 \equiv S(V_0)   $ because 
    $ \dim (V_0 ) = 2 $,  
and  Theorem 
\ref{Topology of Sa}-($\daleth 2$) implies that
\begin{itemize}
\item {\bf (topology of $\MA $)}
the compact $ \T^2_\Gamma$-space $\MA $ is the disjoint union 
\begin{equation}\label{MaS1}
    \MA = 
\underbrace{\MA^{3d}}_{\stackrel{\gamma} \cong 
S(V_-) \times S(V_+)} \sqcup \underbrace{\MA^{2d}}_{\stackrel{\gamma} 
\cong 
S(V_0) \equiv S^1} \, ,\quad \MA^{3d} = \MA \setminus {\bf V}^{\twod} \, ,   
\ \MA^{2d} =  \MA \cap V_0  \, ,
\end{equation}
where  $ V_0 = V_{j_0} $ and 
$  \dim (V_\pm) \geq 2 $. 
The set  of $ 3d$-waves 
$ \MA^{3d}  $ is a manifold  
and \eqref{supplem} holds.
The linear action $ (\tau_\theta )_{\theta \in \T^2_\Gamma}$ in \eqref{T2 symmetry} 
acts transitively 
on  $ \MA^{2d} $ (it 
consists of only {\it one} 
$ \T^2_\Gamma $-orbit) and, for any
$ v $ in $ \MA^{\threed} $ 
the  stabilizer  $ (\T_\Gamma^2)_v $, cf. \eqref{stabilu}, are finite. 
\end{itemize}

\begin{remark}\label{sonosta}
 For any $ v \in \MA^{2d} $ then $ u = v + w(c_{\twod}^{\parallel_{\hat \jmath}}(v)\hat \jmath+c_*^{\bot_{\hat \jmath}}\hat \jmath^{\bot},v) $  is the celebrated $ 2d $ Stokes wave   constructed  by the Crandall-Rabinowitz bifurcation theorem having momentum $ \cI (u) = a $. 
\end{remark}

We first  prove the existence of at least one $3d$ Stokes wave by a simple variational argument. 
Then, to prove multiplicity,  we 
 construct  in Proposition  \ref{A flow}  a gradient-like flow of the metric space $ \MA $,  whose stationary points 
are  Stokes waves and 
we exploit the rich equivariant topology of $ \MA $  described in \eqref{MaS1} via the  abstract  Theorem \ref{teo:ast}.
Clearly 
the  existence result is  covered by Theorem \ref{Existence of 3d solutions collinear nonresonant}, but we find instructive to include it.
%for its simplicity.  
Finally  we 
prove Theorem \ref{th:(u,c) and cI}. 

\paragraph{\bf Existence of at least one $ 3d$-Stokes wave.} % \label{sec:ex}
We define the reduced Hamiltonian
\begin{equation}\label{eq def:Hdiritto}
H:B_{r'}^V\to \R \, , \quad 
 v \mapsto    H(v):= {\mathtt H} (c(v),v)\, ,
\end{equation}
where $c(v)$ is defined in Proposition \ref{Construction of c close and away from PJ} and, for any $ 
    (c,v)\in B_{r}(c_*)\times B_{r}^V $, 
\begin{equation}\label{def:H(c,v)}
\begin{aligned}
    \mathtt{H}(c,v):=\cH(v+w(c,v))
    \stackrel{\eqref{def reduced functional}, \eqref{eq:def I(c,v)}}  = \Phi(c,v)-c\cdot {\mathtt I} (c,v)    \, .
\end{aligned}
\end{equation}

\begin{lemma}\label{HisC1}
The reduced Hamiltonian  $H(v)$ in \eqref{eq def:Hdiritto} is $ \T^2_\Gamma \rtimes \Z_2$-invariant and  of class $C^1 (B_{r'}^V)$. 
\end{lemma}

\begin{proof}
The function $c(v)$ in 
Proposition \ref{Construction of c close and away from PJ}  is $\T^2_\Gamma\rtimes \Z_2$-invariant and ${\mathtt H} (c,v) 
$ is $\T^2_\Gamma\rtimes \Z_2$-invariant and 
satisfies the assumptions of Lemma \ref{lm:regularity 1}, by Proposition \ref{lem:expa}.
\end{proof}  

%The set $ \MA $ in 
%\eqref{defSa} is a 
% for the functional $ \Phi (c(v), v) $. 

\begin{lemma}\label{thm:Lagrange}
   {\bf  (Natural constraint)} If $ \bar v \in {\MA^{3d} } $ is a constrained critical point of $H_{|\MA} : \MA \to \R $ 
    then $ \bar u := \bar v+  w( c(\bar v), \bar v)$ is a 3d-Stokes wave solution of $\cF(c(\bar v), \bar u)=0 $ 
    with momentum $ \cI (\bar u) = a $.
\end{lemma}

\begin{proof}
  %      By Lemma \ref{lm:regularity} the function $\nabla_v\Phi(c(v),v)$ is continuous.
        We differentiate the function $ H(v) $  in \eqref{eq def:Hdiritto} 
        that,  
        in view of \eqref{def:H(c,v)} and \eqref{Def I dritto}, is equal to 
         $H(v)=\Phi(c(v),v)-c(v)\cdot I(v)$.
        For any $v\in B_{r'}^V\setminus {\bf V}^{\twod}$ and any $\hat v\in V$ we get
\begin{equation}\label{eq:the usual computations}
            \di_v H(v)[\hat v]=\underbrace{\big(\partial_c\Phi(c(v),v)-I(v)\big)}_{\stackrel{\eqref{partial c Phi},\eqref{Def I dritto}}=0}\cdot \di_v c(v)[\hat v]+(\di_v \Phi)(c(v),v)[\hat v]-c(v)\cdot \di_v I(v)[\hat v]\, .
        \end{equation}
Since $\bar v \in \MA^{\threed} = \MA\setminus {\bf V}^{\twod} $ is a  critical point of $H_{|\MA} $ 
and $ T_{\bar v}\MA^{\threed} = \ker \di_v I (\bar v) $, cf. \eqref{supplem}, we have
\begin{equation}\label{diftanzero}
   0 = \di_vH(\bar v)_{|T_{\bar v}\MA^{\threed}} \stackrel{\eqref{eq:the usual computations}} = 
   (\di_v\Phi)(c(\bar v),\bar v)_{|T_{\bar v}\MA^{\threed}} \, . 
\end{equation}
   By the splitting \eqref{supplem}, the fact that  
   $ c(v) = c_{3d} (v) $ solves \eqref{EQ c Phi} by   Proposition \ref{Construction of c close and away from PJ},  and
\eqref{diftanzero} 
        we conclude that  $(\di_v\Phi)(c(\bar v), \bar v) =  0 $. Then 
        Lemma \ref{lem:varia} implies that 
$\cF(c( \bar v), \bar v+ 
w(c(\bar v ), \bar v)  )=0$.
\end{proof}

    By Lemma \ref{thm:Lagrange}, it is enough to find a critical point $\bar v\in \MA^{\threed}$ of $H_{|\MA} $.
    Since $\MA^{\twod}$
    is a 
    $ \T^2_\Gamma$-orbit (see  comments below \eqref{MaS1}) and $ H $ is $\T^2_\Gamma$-invariant,  it is contained in a level of $H_{|\MA}$, namely 
    $H (\MA^{\twod} ) = \ell_* $. 
    Since $ \MA $  is compact, $ H_{|\MA} $ attains 
    minimum, say $m $, and  maximum, say $ M $,  on $\MA $. 
    If $m=M$ the function  $ H_{|\MA} $ is constant and hence {\it any}  point of $ \MA^{\threed}$ is critical.
    If $m<M$ then, either  $\ell_* \not=m$ or 
    $ \ell_* \not= M $,  and so 
 there exists at least {\it one}  critical  
 point $ \bar v $ of $ H_{|\MA} $  in $\MA^{\threed}$.

\paragraph{\bf Gradient-like flow.} 

In order to prove multiplicity results, 
we first construct a 
 gradient-like flow for the  
reduced Hamiltonian $H $ defined in \eqref{eq def:Hdiritto}.

\begin{proposition}\label{A flow}
{\bf (Gradient-like flow)}
There exists a global flow $ \phi :\R\times D_{\varepsilon}\to D_{\varepsilon} $ where  $D_{\varepsilon}$ is defined in \eqref{eq:Depsilon subsetsubset U2}  
such that, for any $ t\in \R $,  
\\[1mm]
$(i)$
  $  I( \phi^t(v)) = I(v) $ for
any $ v\in D_{\varepsilon} $, any $ t \in \R $;
\\[1mm]
$(ii)$ 
$ \phi^t$ is $\T^2_\Gamma $-equivariant and
\begin{equation}\label{eq:Ftv}
    \phi^t(  D_{\varepsilon}\setminus {\bf V}^{\twod} ) = D_{\varepsilon}\setminus {\bf V}^{\twod} \, , \quad \phi^t( D_{\varepsilon}\cap V_{\hat \jmath})=
    D_{\varepsilon}\cap V_{\hat \jmath} \, , \quad 
    \forall {\hat \jmath} \in \cD\, ;
\end{equation}
$(iii)$
for any $v\in D_{\varepsilon} $
            \begin{equation}\label{H-HFt}
                H(\phi^t(v))=H(v)-\int_0^t \Big\|(\nabla_v\Phi)
                \big(c(\bar v),\bar v \big)\big|_{\bar v=\phi^\tau (v)}\Big\|^2\di \tau \, .
            \end{equation}
            Moreover $H$ is strictly decreasing on non-constant trajectories of $\phi$. 
\end{proposition}

In view of $ (i) $ the set $ \MA $ in \eqref{defSa} is   invariant under $ \phi^t $ and $(iii)$ implies the following corollary.  

\begin{lemma}\label{lem:cor}
If $ \bar v \in \MA^{\threed} $ 
is a stationary point  of the gradient-like flow  $ \phi^t_{|\MA} $,   
    then $ \bar u:= \bar v+w(c( \bar v), \bar v)$ is a $3d$ Stokes wave solution of $\cF(c( \bar v), \bar u)=0$ with momentum $ \cI (\bar u) = a $.
\end{lemma}

To  prove Proposition \ref{A flow} we look for $\phi^t(\cdot) $ as the flow generated by a vector field  of the form 
\begin{equation}\label{Def Y field}
    Y(v):=
        - (\nabla_v\Phi) (c(v),v)+\mu(v)\cdot \nabla_v \cI(v) \, , 
        \quad \forall 
        v\in B_{r'}^V \, , 
\end{equation}
where 
$\mu(v)$ is determined by the condition 
\begin{equation}\label{I is conserved}
    \frac{\di}{\di t} I(\phi^t(v)) =0\quad \text{i.e.} \quad \di_v I(v)[Y(v)]=0\, .
\end{equation}
 By \eqref{Def Y field} 
system \eqref{I is conserved} is equivalent to 
\begin{equation}\label{EQ mu}
        A^\top(v)\mu(v)=b(v) \qquad \text{where}\qquad b(v):=\di_v I(v)[(\nabla_v \Phi)(c(v),v)] 
        \
\end{equation}
and $A(v)$  is the matrix in 
\eqref{Def AA}. 

\begin{lemma}\label{lm:mu and Y}
%{\bf 
%Solution of \eqref{I is conserved}.}
    There is 
    a solution $\mu:B_{r''}^V\to \R^2$ of system  \eqref{EQ mu} where $r''>0$ is fixed in Lemma \ref{Estimates on A, AQ, tilde AQ11}.
    The corresponding vector field $Y(v)$ in \eqref{Def Y field} satisfies the properties \eqref{list:Properties of Z} of Lemma \ref{lm:B1}. 
\end{lemma}

\begin{proof}
We solve \eqref{EQ mu} applying Lemma \ref{lm:linear equation} with $ t = 1 $.
    Let us verify 
    its assumptions. The $\T^2_\Gamma $ invariance of $ b(v) $ follows 
    form the invariance of $ I (v) $ and $ \Phi (c, v) $. Moreover  
    Proposition \ref{lem:expa} and Lemma  \ref{lm:regularity} imply that the map $v\mapsto (\nabla_v\Phi) (c(v),v)$ is continuous on $B_{r'}^V$, analytic in $B_{r'}^V\setminus {\bf V}^{\twod}$ and restricted to each $(B_{r'}^V\cap V_{\hat \jmath})\setminus\{0\}$, for any $ {\hat \jmath} \in \cD $
    (thus  satisfies properties \eqref{list:Properties of Z}-$(1)$-$(2)$).   
    Therefore, using also Lemma \ref{lem:regI}, the map 
    $ b:B_{r'}^V\to \R^2$ satisfies the same properties as well. Let us verify also \eqref{estbj}. Since the function $ \Phi(c,v) $ satisfies the assumptions of Lemma 
    \ref{lm:A} with $ m = 2 $, we deduce by 
    \eqref{A2ora} that, 
     for any $v\in B_{r'}^V$, \begin{equation}\label{eq:estimate on nabla Phi}
    \|(\nabla_v\Phi)(c(v),v)\|\lesssim \|v\|\, , \ \|\Pi_{\hat \jmath}^\bot\nabla_v\Phi(c(v),v)\|\lesssim \|\Pi_{\hat \jmath}^\bot v\|\, ,  \quad \forall  {\hat \jmath}\in \cD \, , 
    \end{equation}
   in particular  $ (\nabla_v\Phi) (c(v),v)$ satisfies \eqref{list:Properties of Z}-$(3)$-$(4)$.
    Hence, for any $  {\hat \jmath}\in \cD $, $ v\in B_{r'}^V $, we deduce
    $$ 
    b^{\hat \jmath}(v) = 
    Q_{\hat \jmath} b(v) 
    =
    \di_v I^{\hat \jmath}(v)[(\nabla_v\Phi)(c(v),v)] =\cO\begin{pmatrix}
                \|v\|^2\\
                \|\Pi_{\hat \jmath}^\bot v\|^2
            \end{pmatrix} 
    $$
    by \eqref{eq:Nabla cI}, \eqref{Estimate on d I1Q dritto} and \eqref{eq:estimate on nabla Phi}.
    Lemma \ref{lm:linear equation} implies the existence  of  $\mu(v):=\mu(1,v)$ solving \eqref{EQ mu} and such that $\mu(v)\cdot \nabla \cI(v)$ satisfies the assumptions \eqref{list:Properties of Z} of Lemma \ref{lm:B1}.
    As we already verified that $(\nabla_v\Phi)(c(v),v)$ satisfies {\eqref{list:Properties of Z}}, the vector field $Y(v)$ satisfies properties \eqref{list:Properties of Z} as well. 
\end{proof}

The flow of $ Y(v)  $ is a  Lyapunov-function for $ H(v) $ in \eqref{eq def:Hdiritto}.  

\begin{lemma}
    The vector field $Y(v)$ in \eqref{Def Y field} with $\mu(v)$ constructed in Lemma \ref{lm:mu and Y}
   satisfies
    \begin{equation}\label{X H decreases 2}
        \di_v H(v)[Y(v)]=- \|(\nabla_v \Phi)(c(v),v)\|^2\, ,
        \quad \forall v\in B_{r''}^V \, . 
    \end{equation}
  %  where $H$ is the $ C^1 $ function defined in \eqref{eq def:Hdiritto}.
\end{lemma}

\begin{proof}
   For any  $v\in B_{r''}^V {\setminus {\bf V}^{2d} } \subset B_{r'}^V {\setminus {\bf V}^{2d} }  $, inserting $\hat v=Y(v)$ in 
\eqref{eq:the usual computations} we get 
\begin{equation}
\label{graps} 
    \begin{aligned}
     \di_v H(v)[Y(v)]& = (\di_v \Phi)(c(v),v)[Y(v)] {-}  c(v)\cdot \di_v I(v)[Y(v)]  \\
     \stackrel{\eqref{I is conserved},\eqref{Def Y field}}=& (\di_v\Phi) (c(v),v)[{-} (\nabla_v \Phi)(c(v),v)+\mu(v)\cdot \nabla_v\cI(v)] 
      \stackrel{\eqref{EQ c Phi}}=-
\| \nabla_v\Phi(c(v),v) \|^2 
    \end{aligned}
    \end{equation}
proving \eqref{X H decreases 2} for any $v\in B_{r''}^V\setminus {\bf V}^{\twod}$.
    Since $ \di_v H (v) $, $ Y(v) $ and $ (\nabla_v\Phi)(c(v),v) $ are continuous in $ v $, we deduce by  \eqref{graps} that 
    \eqref{X H decreases 2} 
    holds for any  
    $ v \in  B_{r''}^V $.
\end{proof}

\noindent 
{\bf Proof of Proposition \ref{A flow}.}
    By Lemma \ref{lm:B1},  the vector field $Y(v)$ 
     in Lemma \ref{lm:mu and Y} admits 
     a 
     $\T^2_\Gamma$-equivariant local in time 
      flow   $\phi^t(v)$ for any $ v \in B_{r''}^V $,  that is globally defined on $ D_\varepsilon $. Indeed, 
  by \eqref{I is conserved} 
    the restricted  
    momentum $I(\cdot)$ 
    is conserved by the flow $ \phi $. 
    Furthermore 
     $\overline{D_{\varepsilon}} \subset U_2\subset  B_{r''}^V$
    by \eqref{eq:Depsilon subsetsubset U2} and Remark \ref{rmk:U2 subset Br''} ($\varepsilon $ is fixed in Lemma \ref{a nel cono}). 
    Since   the set $D_{\varepsilon} $ is {\it compact} and {\it invariant} under $ \phi^t $ the flow  $ \phi $ is globally defined on $ D_{\varepsilon} $ and
$(i)$ holds.
Remark \ref{rem:Vinv} implies \eqref{eq:Ftv}.
  %   We pass to $(iii)$.
    Finally  \eqref{X H decreases 2} implies \eqref{H-HFt}.
    Moreover, if $H(\phi^{t_1}(v))=H(\phi^{t_2}v))$ for some $t_2>t_1$ then  \eqref{H-HFt} implies that $ (\nabla_v\Phi)(c(v),v)$  vanishes on $ \phi^t(v) $ for any $t\in [t_1,t_2]$ as well as  $ Y(v) $ in \eqref{Def Y field}
    (if $ (\nabla_v\Phi)(c(v),v) =  0 $ then $ \mu (v) = 0 $ by  \eqref{EQ mu} and $ Y(v) = 0 $). $ \hfill \square$
    
\paragraph{\bf Multiplicity of $3d$-Stokes waves: proof of Theorem \ref{Existence of 3d solutions collinear nonresonant}. } % \label{sec:MC}

    Let 
    $a $ as in 
    \eqref{comeprendoa}-\eqref{comeprendoa1}. 
    In view of Lemma
    \ref{lem:cor}, stationary points in $  \MA^{\threed} $ of the gradient-like flow 
    $ \phi^t( \cdot ) $  of $ H $  restricted to  $ \MA $ give rise to truly $3d $ Stokes waves.  
    If there are infinitely many stationary 
    $ \T^2_\Gamma $-orbits of the gradient-like flow, then there exist also infinitely many
    $ \T^2_\Gamma \rtimes \Z_2 $
    critical orbits of $ H $, thus geometrically distinct $ 3d $-Stokes waves.
    Otherwise  
we apply Theorem \ref{teo:ast}
to the $ \T^2_\Gamma $-invariant continuous 
function 
$$ 
F : \underbrace{M}_{= S(V_-)\times S(V_+)} \star 
\underbrace{S^1}_{\equiv S(V_0)}  \to \R \, , \quad 
F := H \circ \gamma^{-1}  \, , 
$$ 
where $ \gamma $ is the  
$ \T^2_\Gamma $-equivariant homeomorphism of  Theorem 
\ref{Topology of Sa}-($\daleth 2$). 
By the comments below  \eqref{MaS1}, the action  $ (\tau_\theta)_{\theta \in \T^2_\Gamma}$  acts transitively 
on the orbit $S^1$,
and 
the  stabilizers  
at any point  
of  $ M $ are finite.  
The function $ F $ has the  gradient-like flow $  \gamma  \circ \phi^t_{| \MA} \circ \gamma^{-1}$ for which 
(cf. Remark \ref{sonosta}) 
all $ S^1 $ is stationary.
Therefore Theorem \ref{teo:ast} 
proves the existence  
of at least $ \cuplength_{\T^2_\Gamma}(M)+1 $
{\it critical values} of $ F$ 
% thus $ H $, 
{\it different} from the critical level  
$ \ell_* = F(S^1) $,
%$ F(\gamma (\MA^{\twod} ) ) = H (\MA^{\twod})  
% $$
% of the $ 2d$-Stokes wave $ S^1 $, 
and thus 
the same number of 
$ \T^2_\Gamma  \rtimes \Z_2 $-critical orbits of $ H $ 
at a level different from 
$ \ell_* = H (\MA^{\twod})  $.
In view of \eqref{MaS1} these orbits are contained in $ \MA^{\threed} $, thus are truly $ 3d $-waves. 

\begin{lemma}
$ \cuplength_{\T^2_\Gamma}(S(V_-)\times S(V_+)) \geq \tfrac12(\dim(V_-)+\dim(V_+))-{2}=\#\cV- {3} $. 
\end{lemma}

\begin{proof}
The  lower bound for $ \cuplength_{\T^2_\Gamma}(M) $
is provided by Example \ref{cup-length of two spheres}
with $n_1=\tfrac12 \dim(V_-)$
and $n_2=\tfrac12 \dim(V_+)$.
 Let us verify the non-collinearity assumption \eqref{eq:non-coll}. 
The spheres $S(V_-)$ and $S(V_+)$ correspond to $S(k_1)$ and $S(k_2)$ with matrices $k_l =\lattice^{\top}(j_{l,1},\dots, j_{l, n_l})\in \Z^{2\times n_l }$ for $ l =1,2$ where $\lattice$ is the matrix defining the lattice $\Gamma$ (cf. \eqref{eq:Gamma'}) and  
    \begin{equation}\label{eq:cV+cV-}
    \cV_-:=\{j_{1, i}\}_{i=1}^{n_1}:=\{j\in \cV \ | \ j\cdot a^\bot<0\}\, , \quad \cV_+:=\{j_{2, i}\}_{i=1}^{n_2}:=\{j\in \cV \ | \ j\cdot a^\bot>0\}\, .
    \end{equation}
    Any pair of vectors in $ \cV_- $ and $ \cV_+ $ 
    are  independent 
    (non-collinear) 
    and thus  \eqref{eq:non-coll} holds. 
    The lemma follows by 
    \eqref{eq:CL product of spheres} and 
    $ \dim (V_-) + \dim (V_+)   = 2 \#\cV - \dim (V_0) =
    2 (\#\cV - 1)  $. 
\end{proof}
    
    We conclude the  existence of  at least $\#\cV-{2}$ distinct geometrically 
    distinct $3d$-Stokes waves.

    \paragraph{\bf 
    Non-collinear case: proof of Theorem \ref{th:noncollinear}.}
    
        In view of Theorem \ref{Topology of Sa}-$(\daleth 1)$ the set $\MA $ is a manifold of $ 3d$-waves 
        in $  V\setminus{\bf V}^{2d}$,  
        equivariantly diffeomorphic to $S(V_-)\times S(V_+)$,       and any stabilizer $ (\T_\Gamma^2)_v $, $ v \in \MA $, is finite (cf. \eqref{stabilu}).
      %  Then
       %The manifold $\MA $ corresponds to $S(k_1)\times S(k_2)$ where $k_l =\lattice^{\top}(j_{l,1},\dots, j_{l, n_l})\in \Z^{2\times n_l }$ for $ l =1,2$ and $j_{l,s}$ are defined as in \eqref{eq:cV+cV-} above.
      %  Since $\cV_+\sqcup \cV_-=\cV$ we deduce $\cuplength_{\T^2_\Gamma}(\MA)\geq \#\cV-2$ by \eqref{eq:CL product of spheres}.
      Then   Theorem \ref{th:noncollinear}
        follows similarly by 
         Theorem \ref{prop:T2G}
         and
\eqref{eq:CL product of spheres}.

\paragraph{\bf Proof of Theorem 
\ref{th:(u,c) and cI}.}

        Let $r>0$ be as Lemma \ref{lem:range}, $r'\in (0,r)$ as in the definition 
        \eqref{c defined on the ball} of $c(v)$  and $\varepsilon>0$ be fixed by  Theorem \ref{Topology of Sa}.
        We recall that the set 
        $ D_\varepsilon $
        in \eqref{eq:Depsilon subsetsubset U2}
        and the range $ U_2 $ of the 
        straightening 
        homeomorphism $ \zeta $  
        %(cf. Lemma \ref{a nel cono} and 
        (cf. Theorem \ref{th:Moser's trick}) satisfy
   $
             D_\varepsilon
        \subset U_2\subset B_{r'}^V\subset B_r^V $. 
        We take $\delta \in (0,{r'}) $ 
        small such that 
        \begin{equation}\label{eq:def delta}
        \Pi_V B_{\delta}^X\subset D_\varepsilon
        \qquad
        \text{where} \qquad  B_\delta^X:= \big\{ u\in X \ |\ \|u\|_X<\delta \big\} 
        \end{equation}
         and $\Pi_V$ is the projector onto $V$ associated to the decomposition \eqref{L2L2VW}.
The Lyapunov-Schmidt reduction describes {\it all} the small amplitude Stokes waves with speed near $ c_* $. 
     
        \begin{lemma}\label{lem:uvw}
        Let  
        $(c, u)\in B_\delta(c_*) \times B_\delta^X $     be 
        a Stokes wave solution of  \eqref{Bifurcation problem}. Then 
$ u= v+w(c( v),  v)$ where 
  $  v:=\Pi_V u\in D_\varepsilon $ and 
$c(v)$ is defined in \eqref{c defined on the ball}.
        \end{lemma}

        \begin{proof}
        Lemmata \ref{lem:range} and \ref{lem:varia}       imply that \begin{equation}\label{eq: (bar c,bar u)Bifurcation equation}
            u= v+w( c,  v)\, , \quad v = \Pi_V u
            \,
            \, , \quad \di_v\Phi( c,  v)=0\, \quad \text{where}
            \ v  
        \in D_\varepsilon 
        \end{equation}
      by \eqref{eq:def delta}.  
        If $ v\in V\setminus{\bf V}^{\twod}$  we deduce that $c\in B_{{r'}}(c_*)$ solves \eqref{EQ c Phi} and thus  Lemma \ref{Construction of c close and away from PJ} implies that  $ c=c_{\threed}( v)$.
        Hence $ u= v+w(c( v),  v)$.
        Otherwise, if $ v\in  V_{\hat \jmath}$ for some 
        %resonant wave direction 
        $ \hat \jmath \in \cD $, we have $\di_v\Phi (c,v)=\di_v\Phi(c^{\parallel_{\hat \jmath}}\hat \jmath,v)$ and by \eqref{eq: (bar c,bar u)Bifurcation equation} we deduce that  $ c^{\parallel_{\hat \jmath}}\in (c_*^{\parallel_{\hat \jmath}}- {r'}, c_*^{\parallel_{\hat \jmath}}+{r'})$ solves \eqref{eq:mf2d}. Thus by Lemma \ref{c on PJ}, we conclude that  $ c^{\parallel_{\hat \jmath}}\hat \jmath=c^{\parallel_{\hat \jmath}}_{\twod}( v)\hat \jmath =c( v)$ and
       % by \eqref{eq: (bar c,bar u)Bifurcation equation} we have 
       $ u=v+w(c^{\parallel_{\hat \jmath}}\hat \jmath,v)=v+w(c( v), v)$  by 
       Lemma \ref{Further properties of w(c,v)}-$(ii)$. 
\end{proof}
        
        The momentum of the Stokes wave $ u $ is  equal to 
        $$
        a :=\cI( u) \stackrel{\text{Lemma} \, \ref{lem:uvw}}= \cI (v + w(c(v),v)) \stackrel{ \eqref{Def I dritto}} = I( v)
        \, , \quad v\in 
        D_\varepsilon   \subset U_2 \, , 
        $$ 
        and therefore 
        $ v\in  \MA $ 
         (cf.  definition \eqref{defSa}).
        Theorem \ref{Topology of Sa} implies that  $a\in \cC$, proving \eqref{Identroa}. 
        Since  $v\in D_\varepsilon$ then  $|a|\leq \varepsilon$.
        The proof of Theorem \ref{th:(u,c) and cI} is concluded  by  the following two cases:
        \\[1mm]
    $(\mathfrak{B})$] If $a=0$ then,  Theorem \ref{Topology of Sa} implies that  $\MA=\{0\}$,  thus $ v=0$ and $ u=0$.
            If $a\in \partial \cC\setminus\{0\}$ then,  Theorem \ref{Topology of Sa}-$(\beth)$ implies that   $ v\in \MA\subset {\bf V}^{\twod}$
            and we deduce by  
Lemma \ref{Further properties of w(c,v)}
 that 
$ u = v + w(c(v),v)$ is a $2d$ wave.
    \\[1mm]
    $(\mathfrak{I})$] If $a\in \interior(\cC)$,  and $ a $ is not collinear with  any $\hat\jmath \in \cD$ then,   Theorem \ref{Topology of Sa}-$(\daleth 1)$ implies that  $ v\in \MA\subset V\setminus {\bf V}^{\twod}$ is a $ 3 d$ wave, and thus $ u = v + w(c(v),v)$ is a $3d$ wave.

\begin{appendix}
        \section{Invariant functions and equivariant flows}\label{Appendix A}

        We  give some  properties of $\T^2_\Gamma $-invariant functions defined on the Euclidean $\T^2_\Gamma $-space $V$
        in \eqref{defKer}.
        
        \begin{lemma}\label{lm:AAA}
        Let $ f : B_r^V \to \R $  be a $ \T^2_\Gamma $-invariant function
        of class $ C^1 $. Then, for any 
        $ \hat \jmath \in \cD $,    \begin{equation}\label{Tailo2}
                \di_v f(\Pi_{\hat \jmath}v)  \Pi_{\hat \jmath}^\bot = 0 \, ,  \quad         \forall  v \in B_r^V \, . 
        \end{equation} 
        \end{lemma}
        
        \begin{proof}       
Since $ f $ is $ \T^2_\Gamma $-invariant,   $ \di_v f( \tau_\theta v) \, \tau_\theta = \di_v f( \tau_\theta v) $ for any $ v \in B_r^V $, $ \theta \in \T^2_\Gamma $,   
and therefore, for any $ \hat \jmath \in \cD $,   for any $ \theta $ in the subgroup $ {\mathtt T}_{[j]} \subset \T^2_\Gamma $ defined in \eqref{eq:def H[j]},    % $ j' \in \cV $, 
\begin{equation}\label{divtj}
            \di_v f(\Pi_{\hat \jmath}v) \pi_{j'}  \stackrel{\eqref{One dimensional iff simmetric}} = \di_v f( \tau_\theta \Pi_{\hat \jmath}v) \pi_{j'}   = \di_v f(\Pi_{\hat \jmath}v) \pi_{j'}\tau_\theta    
        \end{equation}
where $\pi_{j'}:V\to V_{j'}$ 
is the projector. 
        %(into irreducible representations) 
        By 
        \eqref{divtj} 
         the linear operator $\di_v f(\Pi_{\hat \jmath}v) \pi_{j'}:V_{j'}\to V$ is $\mathtt{T}_{[j]}$-invariant.
If   $j'\in \cV\setminus[j]$ then $ V_{j'} $ is an irreducible representation of  $\mathtt{T}_{[j]}  $ and,            since $\di_v f(\Pi_{\hat \jmath}v) \pi_{j'}$ is not injective, Schur lemma implies that 
\begin{equation}
\label{dalschu}
\di_v f(\Pi_{\hat \jmath}v) \pi_{j'} = 0 \, , \quad 
\forall j \in \cV \setminus [j] \, . 
\end{equation}
By \eqref{dalschu}, recalling \eqref{Def P},  we deduce \eqref{Tailo2}.
        \end{proof}

By the  previous lemma a $ \T^2_\Gamma $-invariant function $ f $ satisfies  $\di_v f(0)=0$ and 
a critical point $v \in V_{\hat \jmath} $ of $ f_{|V_{\hat \jmath}}$ is actually a critical point of 
$ \di_v f (v) = 0 $. 
        \begin{lemma}\label{lm:A}
        {\bf (\texorpdfstring{$\T^2$-}{T^2_Gamma-}invariant functions)}
            Let $f:B_r(c_*)\times B_r^V\to \R$ be a  smooth function with bounded derivatives at any order,  
            satisfying $f(c,v)= \cO (\|v\|^m)$ uniformly in $ c $,   for some $m\geq 2$, 
and $ f(c, \cdot ) $ is              $\T^2_\Gamma$-invariant for any $ c \in B_r(c_*) $.             Then,  for any $ \hat \jmath \in \cD $, the following  holds:
\\[1mm]
$(i)$ 
            for any  $ v \in B_r^V $, $ \hat v \in V $, 
            \begin{equation}\label{A2ora}
                    \di_v f(c,v)[\hat v]= \cO(\|v\|^{m-1}\|\Pi_{\hat \jmath} \hat v\|+\|v\|^{m-2}\|\Pi_{\hat \jmath}^\bot v\|\| \Pi_{\hat \jmath}^\bot \hat v\| )
                    \quad
                    \text{ uniformly in} \ c \, . 
            \end{equation} 
$(ii)$ If   
            $
                f(c,v)=0 $ for any $ v\in B_r^V\cap V_{\hat \jmath}
            $
            then
            \begin{align}\label{eq:A1}
              &   f(c,v)= \cO (\|\Pi_{\hat \jmath}^\bot v\|^2 \|v\|^{m-2}) 
              \end{align}
            uniformly in $c$, 
            and if  $m\geq 3$ then
              \begin{align}
        & \label{eq:A3}
                \di_v f(c,v)[\hat v]= \cO (\|v\|^{m-3}\|\Pi_{\hat \jmath}^\bot v\|^2\|\Pi_{\hat \jmath} \hat v\|+\|v\|^{m-2}\|\Pi_{\hat \jmath}^\bot v\|\| \Pi_{\hat \jmath}^\bot \hat v\| )
            \quad
                    \text{ uniformly in} \ c \, .\end{align}
            
        \end{lemma}
        
        \begin{proof}        
        %We claim that,   
        %\begin{equation}\label{Tailo2}
        %        \di_v f(c,\Pi_{\hat \jmath}v) \circ \Pi_{\hat \jmath}^\bot = 0 \, , \quad       \forall         \hat \jmath \in \cD \, ,        \         v \in B_r^V \, . 
        %\end{equation}
        %Indeed $ \di_v f(c, \tau_\theta v) \, \tau_\theta = \di_v f(c, \tau_\theta v) $ for any $ \theta \in \T^2_\Gamma $, and therefore, for any $ \theta \in {\mathtt T}_{[j]}$, 
        %\begin{equation}\label{divtj}
        %    \di_v f(c,\Pi_{\hat \jmath}v)   \stackrel{\eqref{One dimensional iff simmetric}} = \di_v f(c, \tau_\theta \Pi_{\hat \jmath}v)   = \di_v f(c,\Pi_{\hat \jmath}v)\tau_\theta  \, . 
        %\end{equation}
        %As a consequence \eqref{Tailo2} follows because, for any $ \hat v \in V $,  
        %$$
        %        \di_v f(c,\Pi_{\hat \jmath}v)[\Pi_{\hat \jmath}^\bot\hat v] \stackrel{\eqref{divtj}}  =\int_{{\mathtt T}_{[j]}}\di_v f(c,\Pi_{\hat \jmath}v)[\tau_\theta \Pi_{\hat \jmath}^\bot\hat v] \di \theta\stackrel{\eqref{eq:observation 1}}=\di_v f(c,\Pi_{\hat \jmath}v)[\Pi_{\hat \jmath} \Pi_{\hat \jmath}^\bot\hat v]=0  \, . 
        %$$
 Since  $ f (c, \cdot) $  is $ \T^2_\Gamma $-invariant for any $ c \in B_r (c_*) $, by     \eqref{Tailo2} 
           the Taylor expansion of $ f(c, \cdot) $
          at any $ \Pi_{\hat \jmath} v $ in the direction
            $ \Pi_{\hat \jmath}^\bot v $ is  
          \begin{equation}\label{Tailor3}
                f(c,v)=f(c,\Pi_{\hat \jmath} v)+R(c,v)[\Pi_{\hat \jmath}^\bot v]^2\, , 
                \qquad \forall v \in B_r^V \, ,  
            \end{equation}
           with 
           quadratic integral Taylor remainder 
            $  R(c,v)[h ]^2:=\int_0^1(1-s)\di_v^2 f(c,\Pi_{\hat \jmath} v+s\Pi_{\hat \jmath}^\bot v)[h ]^2\,\di s $ for any  $ h \in V $.
        Taking the differential in \eqref{Tailor3} we get 
            \begin{equation}\label{diffok}
            \di_v f(c,v)[\hat v]=\di_v f(c,\Pi_{\hat \jmath} v)[\Pi_{\hat \jmath} \hat v]+\di_v R (c,v)[\hat v] [\Pi_{\hat \jmath}^\bot v]^2+
               2 R (c,v)[\Pi_{\hat \jmath}^\bot v,\Pi_{\hat \jmath}^\bot \hat v]\, .
            \end{equation}         
            Now $(i)$ follows by \eqref{diffok}
            since $|\di_v^{\ell} f(c,v)|\lesssim_m \|v\|^{m-\ell} $ for any $  0\leq \ell \leq m$, 
        which follows because $ f (c, v) = \cO (\|v\|^m)$ uniformly in $ c $,  is smooth, and has bounded derivatives
        (for $ m =2 $ just use that $ \di_v^2 f(c,v) $ is bounded). 
Item  $(ii)$ follows similarly by \eqref{Tailor3} and \eqref{diffok}  because 
$ f(c,\Pi_{\hat \jmath} v) = 0 $  and  
    $ \di_v f(c,\Pi_{\hat \jmath} v) \Pi_{\hat \jmath} = 0 $ 
    for any $ v\in B_r^V\cap V_{\hat \jmath} $. 
        \end{proof}

        The following lemma is used twice in the paper (for proving Theorem \ref{th:Moser's trick} and
        Proposition 
         \ref{A flow}). 
       
        \begin{lemma}\label{lm:B1}
        {\bf (Continuous equivariant flow)}
            Let 
            $J\subseteq\R$ be an open interval. Assume that the vector field  $Z:J\times B_{r''}^V\to V$ is
            $\T^2_\Gamma$-equivariant for any $ t \in J $, and   
            \begin{align}
            \label{list:Properties of Z}\tag{$Z$}
            \begin{minipage}{0.9\linewidth}
            \begin{enumerate}
                \item[(1)] $Z(t,v)$ is continuous on $ J \times B_{r''}^V $;
                \item[(2)] $Z(t,v)$ is analytic on $J\times  [B_{r''}^V\setminus {\bf V}^{\twod}]$
                      and restricted to  $J\times [(B_{r''}^V\cap V_{\hat \jmath})\setminus\{0\}]$ for any ${\hat \jmath} \in \cD $;
                \item[(3)] $\|Z(t,v)\|\lesssim \|v\|$ for any $(t,v)\in J\times B_{r''}^V$;
                \item[(4)] For any ${\hat \jmath}\in\cD$  we have 
                $\|\Pi_{\hat \jmath}^\bot Z(t,v)\|\lesssim 
                {\|\Pi_{\hat \jmath}^\bot v\|} $ for any $(t,v)\in J\times B_{r''}^V$. 
            \end{enumerate}
            \end{minipage}
            \end{align}
        Then $Z(t,v)$ admits a  flow $\zeta(t,v)$ 
            defined for any $ v \in B_{r''}^V $, locally in time,
which is 
continuous and  
            $\T^2_\Gamma$-equivariant.  
            Moreover, for any closed interval $J'\subset J $  there is $ {r'''}\in(0,r'')$ such that $\zeta$ is defined on $J'\times B_{r'''}^V$.          
        \end{lemma}
        
        \begin{proof}
           For any $v_0\in B_{r''}^V $  let $v(t)$ be a solution of 
            \begin{equation}\label{eq:ode}
                    \dot v(t)=Z(t,v(t)) \, ,  \ 
                    v(0)=v_0 \, , \quad 
                    t\in J(v) \, , 
            \end{equation}
            defined on the open interval $J(v)\subseteq J$ containing $0$. By the regularity Assumptions \eqref{list:Properties of Z}, 
           problem \eqref{eq:ode} has always a solution that is  unique for any $ v_0 \in  B_{r''}^V \setminus {\bf V}^{2d}$. 
           {However  $ Z (t, v) $ does not satisfy a Lipschitz property near the subspaces $  V_{\hat \jmath } $,  $ \hat \jmath \in \cD $.}  
           We  deduce 
           uniqueness of the solution  by the following estimates.  
            \\[1mm]
            {\sc Step 1:} {\it there is 
            $C >  0 $ such that any solution  
            $ v:J(v)\to B_{r''}^V$ of \eqref{eq:ode}  satisfies}
            \begin{equation}\label{uniqueness: step1}
                \|v(t)\|\geq e^{-C|t|} \|v_0\| \, , \quad  d(v(t),{\bf V}^{\twod})\geq e^{-C|t|} d(v_0,{\bf V}^{\twod}) \, ,  \qquad \forall  t\in J(v)\, .
            \end{equation}
            The function  $ \varphi(t) := \|v(t)\|^2$
            satisfies $\varphi(0) = {\| v_0 \|^2} $ and, by \eqref{eq:ode}  and Assumption ${(3)}$, 
    \begin{equation}\label{apriori}
       { |\varphi'(t)| =  
            2 |\langle Z(t,v(t)),v(t)\rangle |
            \leq 2 C  \varphi(t)} \, . 
    \end{equation}
        Thus by comparison  
        $           {\varphi(t) \geq  \varphi(0) e^{- 2 C |t|}} 
        $ for any $t \in J(v) $ and we deduce the first inequality in \eqref{uniqueness: step1}.  Let us prove the other one.       
        For any ${\hat \jmath}\in \cD $,  the function $
        \varphi_{\hat \jmath}(t) := \|\Pi_{\hat \jmath}^\bot v(t)\|^2$
        satisfies  $\varphi_{\hat \jmath}(0) {= \|\Pi_{\hat \jmath}^\bot v_0 \|^2 } $ and
        by \eqref{eq:ode}  and Assumption $ {(4)}$, 
        $
            |\varphi_{\hat \jmath}'(t)|= {2 |\langle \Pi_{\hat \jmath}^\bot Z(t,v(t)),\Pi_{\hat \jmath}^\bot v(t)\rangle |
            \leq 2 C  \varphi_{\hat \jmath}(t)}$. 
        By comparison we get  $ \varphi_{\hat \jmath}(t) \geq 
        \varphi_{\hat \jmath}(0) e^{-2C |t|}$ and thus 
        $ \|\Pi_{\hat \jmath}^\bot v(t)\|\geq e^{-C |t|}\|\Pi_{\hat \jmath}^\bot v_0\|  \geq e^{-C |t|} d({v_0},{\bf V}^{\twod})
        $
(recall \eqref{defFr}). 
Taking the minimum in $ {\hat \jmath} \in \cD $ 
we deduce also the second inequality in \eqref{uniqueness: step1}.
\\[1mm]
        {\sc Existence of the flow.} 
        By the regularity Assumption ($2$) the solutions of 
     \eqref{eq:ode} are unique in
        $ B_{r''} \setminus {\bf V}^{2d}$. 
        Assumptions ($3$)-($4$)
imply that   $ Z(t,0) = 0 $ and $ Z(t, \cdot)  : V_{\hat \jmath } \to V_{\hat \jmath } $ 
and thus, if $ v_0 = 0 $, a solution of the Cauchy problem \eqref{eq:ode} is $ v(t) = 0 $ and for any $ v_0 \in V_{\hat \jmath } \setminus \{0\} $, 
a solution of  \eqref{eq:ode} lives in 
$ V_{\hat \jmath } \setminus \{0\} $. These are the unique solutions because, 
by \eqref{uniqueness: step1},  if $v_0 \in B_{r''}^V 
        \setminus {\bf V}^{2d} $  then the solution $ v(t) $
        of  \eqref{eq:ode} 
        will never reach in finite time  $  {\bf V}^{2d} $. 
        This proves that, for any $ v \in B_r^V $ the flow  $ v_0  \mapsto \zeta (t,v) := v(t) $ of $Z(t,v) $ is well defined locally in time near $ t = 0 $.
        The equivariance of $ Z(t, v) $ and   the uniqueness of the solutions of \eqref{eq:ode} also imply that  
        $ \zeta (t, \cdot) $ 
        is $\T^2_\Gamma$-equivariant. Furthermore for any closed  $ J' \subset J $ the flow $ \zeta (t, v) $ is defined for any $ v \in B_{r'''}^V $ sufficiently small. Indeed  \eqref{apriori} implies the upper bound $ \varphi(t) \leq  \varphi(0) e^{ 2 C |t|}  $ and therefore 
        $ \|v (t)\| \leq \| v_0 \| e^{2C T }
     $
        for any $ t \in J' $  with  $ T := \max_{t \in J'}{|t|} $. Then it is sufficient to take $ r''' < r'' e^{-2C T} $.  
        \\[1mm]
        {\sc Continuity of the flow.}
        Let $D\subseteq J\times B_{r''}^V$ be the domain of definition of the flow $\zeta$ i.e. the set of pairs $(t,v_0)\in J\times B_{r''}^V$ such that the maximal solution of \eqref{eq:ode} is defined at time $t$. We claim that the uniqueness of the solutions implies also the  continuity of the flow $\zeta:D\to B_{r''}^V$ at any $(\bar t, \bar v)\in D$. Let $\{(t_i, v_i)\} \subset D $ such that $(t_i, v_i)\to (\bar t, \bar v)$.
        Note that, since $ Z $ is bounded (cf. \eqref{list:Properties of Z}-$(3)$) we have
        \begin{equation}\label{eq:lip}
            |\zeta(v,t)-\zeta(v,t_1)|\leq \sup_{(s,u)\in J\times B_{r''}^V}\|Z(s,v)\| |t-t_1|  \lesssim |t-t_1| \,
\quad \forall (v,t),(v,t_1)\in D\, .
        \end{equation}  
        If $\bar t=0$, since $Z(t,v)$ is uniformly bounded, we have $| \zeta(\bar t,\bar v)-\zeta(t_i, v_i)|\leq | \zeta(0,\bar v)-\zeta(0, v_i)| +| \zeta(0,v_i)-\zeta(t_i, v_i)|\lesssim |\bar v-v_i|+|t_i| \to 0$ by \eqref{eq:lip}. 
        Consider the case  $\bar t>0$ (for $ \bar t < 0 $ the argument is the same).
        For any $\tau\in [0,\bar t)$ we have $t_i>\tau$ for $ i  $ large enough, 
each 
        $\zeta(\cdot, v_i)$ and $\zeta(\cdot ,\bar v)$ are well defined for $ [0, \tau ]$, and 
        we claim that % the uniform convergence 
        \begin{equation}\label{eq:a claim}
            \sup_{t\in[0,\tau]}|\zeta (t,v_i) -\zeta(t,\bar v) |\to 0\quad \text{as}\quad i\to +\infty\, .
        \end{equation}   
        For any subsequence $i_m$, since  the functions $\{\zeta(\cdot,v_{i_m}):[0,\tau]\to B_{r''}^V\}$ are uniformly equicontinuous (cf. \eqref{eq:lip} ), by Ascoli-Arzel\'a theorem there exists another subsequence $i_{m_\ell} $ such that $\zeta(\cdot, {i_{m_\ell}})$ converges uniformly to a function $\bar {\mathsf{z}}\in C([0,\tau],\overline{B_{r''}^V})$.
        The function $\bar {\mathsf{z}} $ is a solution of \eqref{eq:ode} with $v_0=\bar v$ on 
        $\mathfrak{p}:=\{t\in [0,\tau] \ |  \ [0,t]\subset {\mathsf{\bar z}}^{-1}
        (B_{r''}^V) \}$. The set $\mathfrak{p}$ is open in $[0,\tau]$.
        By uniqueness, $  \mathsf{\bar z}(t)=\zeta(t , \bar v)$ for any $t\in \mathfrak{p}$ and hence also for $t\in \overline{\mathfrak{p}}$.
        By the definition  of $\mathfrak{p}$ we deduce that $\overline{\mathfrak{p}}\subset \mathfrak{p}$ and thus $\mathfrak{p}=[0,\tau]$ i.e. $\mathsf{\bar z}([0,\tau])\subset B_{r''}^V$.
        We deduce \eqref{eq:a claim}.    
By triangular inequality $
            |\zeta(\bar t,\bar v)-\zeta(t_i,v_i)|\leq
            |\zeta(\bar t,\bar v)-\zeta(\tau,\bar v)| + 
            |\zeta(\tau,\bar v)-\zeta(\tau,v_i)|+
            |\zeta(\tau,v_i)-\zeta(t_{i},v_i)|
        $
        and \eqref{eq:lip}, \eqref{eq:a claim} we deduce $\zeta(t_i,v_i)\to \zeta(\bar t,\bar v)$.
        Thus $\zeta$ is continuous on $D$.
        \end{proof}

        \paragraph{Solution of the range equation.} 
       We  prove Lemma \ref{lem:range}.
    The range equation \eqref{solution of the range equation} amounts to solve 
$ \cG(c,v,w) = 0 $ where 
$\cG:\R^2\times 
    V\times (W\cap X) \to W\cap Y $ is 
    $
    \cG(c,v,w):= \Pi_{W\cap Y}\cF(c,v+w) $.
The map $ \cG $ is analytic in a small neighborhood of 
$ (c_*,0,0) $ by  \cite{BMV2}[Theorem 1.2]
about the analyticity  of the Dirichlet-Neumann operator 
and the algebra properties of $H^{\sigma,s}(\T^2_\Gamma)$. 
It results  
\begin{equation}\label{dG}
        \cG(c,0,0) = 0 \, , \quad  \di_w\cG(c,0,0)=\Pi_{W\cap Y}\cL_{c|W\cap X}\, .
    \end{equation}
    We now  construct the inverse $L : W\cap Y\to W\cap X $ of  $\di_w\cG(c_*,0,0) $.
    Let $\mathscr{W}$ be the linear span generated by $\{v_j^{(1)},v_j^{(2)} \ |\ j \in \Gamma'\setminus(\cV\cup\{0\})\}$ and let $\mathscr{W}_X:=\mathscr{W}\cup\{v_0^{(1)} \}$ and $\mathscr{W}_Y:=\mathscr{W}\cup\{v_0^{(2)}\}$.
    %Let $\mathscr{W}_X$ and $\mathscr{W}_Y$ be  the linear spans 
    %$$
    %  \mathscr{W}_X  \equiv  \{v_j^{(1)},v_j^{(2)}|\ j \in \Gamma'\setminus(\cV\cup\{0\})\}\cup\{v_0^{(1)} \} \, ,  \quad 
    %  \mathscr{W}_Y  \equiv \{v_j^{(1)},v_j^{(2)} \ |\ j \in \Gamma'\setminus(\cV\cup\{0\})\}\cup\{v_0^{(2)}\}.
    %$$  
    We consider  the linear operator $L:\mathscr{W}_Y\to \mathscr{W}_X$,  
    \begin{equation}\label{Def L}
        \begin{cases}
            L v_j^{(1)}:=\frac{-1}{c_*\cdot j-\omega(j)} v_j^{(2)}\\
            L v_j^{(2)}:=\frac{1}{c_*\cdot j-\omega(j)} v_j^{(1)}
        \end{cases}\quad \forall j \not \in \cV\cup \{0\} \, , \qquad  Lv_0^{(2)}:=-\frac1g v_0^{(1)}\, ,
    \end{equation}
which is well defined since $c_*\cdot j - \omega (j) \neq 0 $ for any $j\in \Gamma'\setminus (\cV\cup \{0\})$. 
    By Lemma \ref{symplectic base and coordinates} $(iii)$, \eqref{dG} and \eqref{Def L} we have
    \begin{equation}\label{L is an inverse}
            L \di_w\cG(c_*,0,0)_{|\mathscr{W}_X}= \text{Id}_{\mathscr{W}_Y}
            \, , \quad 
            \di_w\cG(c_*,0,0)  L= \text{Id}_{\mathscr{W}_X}\, .
    \end{equation}
   
   \begin{lemma}
\label{continuity of L}
       $  \|Lu\|_{H^{\sigma,s}\times H^{\sigma,s}}\lesssim \|u\|_{H^{\sigma,s-1}\times H^{\sigma,s-2}}$ 
       for any $ u \in \mathscr{W}_Y $. Thus $L$ extends to a continuous operator $W\cap Y\to W\cap X$ (still denoted $L$) that is
    an inverse of $ \di_w\cG(c_*,0,0) $.  
    \end{lemma}
    
    \begin{proof}
The analytic norms $ \| \ \|_{H^{\sigma,s}} $ defined in  \eqref{def:Hs} of the components of
          $ u  = \begin{psmallmatrix}
          \eta \\ \psi 
          \end{psmallmatrix}
        \in \mathscr{W}$ expanded as in  
\eqref{coordiantes} are, recalling \eqref{symplectic base}, 
\begin{equation}
        \begin{aligned}
        \label{norm of eta}
            &\|\eta\|^2_{H^{\sigma,s}}=\tfrac{1}{4}\sum \big(|\alpha_j(u)+\alpha_{-j}(u)|^2+|\beta_j(u)-\beta_{-j}(u)|^2\big)\, M_j^2\, | j|^{2s}e^{2\sigma |j|_1}\\
            &\|\psi\|^2_{H^{\sigma,s}}=\tfrac{1}{4}\sum \big(|\alpha_j(u)-\alpha_{-j}(u)|^2+|\beta_j(u)+\beta_{-j}(u)|^2\big)\, M_j^{-2}\, | j|^{2s}e^{2\sigma |j|_1}
        \end{aligned}
        \end{equation}
        where the sums are over $ j \in \Gamma'\setminus(\cV\cup \{0\})$.
       By \eqref{Def L},    the operator $ L $ 
       acts, in the coordinates $\{(\alpha_j,\beta_j)\}_{j}$ in \eqref{abu},  as
        $$      \alpha_j(Lu)=
        \tfrac{\beta_j(u)}{c_*\cdot j-\omega(j)} \, , \quad 
                \beta_j(Lu)
                =\tfrac{-\alpha_j(u)}{c_*\cdot j-\omega(j)} \, , 
            \quad \forall j \in \Gamma'\, , \  \forall u\in \mathscr{W} \, . 
      $$
     In view of  \eqref{norm of eta},   in order to estimate $\|Lu\|_{H^{\sigma,s}\times H^{\sigma,s}}$ we consider
%       $(\alpha_j(Lu)\pm \alpha_{-j}(Lu))$ and $(\beta_j(Lu)\pm \beta_{-j}(Lu))$ as follows
 %   and thus 
       \begin{equation}\label{alpha  pm alpha beta pm beta}
           \begin{aligned}
               &\alpha_j(Lu) \pm \alpha_{-j}(Lu)=
                    \tfrac{c_*\cdot j}{(c_*\cdot j)^2-\omega(j)^2}(\beta_j(u) \mp \beta_{-j}(u))+
                    \tfrac{\omega(j)}{(c_*\cdot j)^2- \omega(j)^2}(\beta_j(u) \pm \beta_{-j}(u))\\
                &\beta_j(Lu)\pm \beta_{-j}(Lu)=
                    -\tfrac{c_*\cdot j}{(c_*\cdot j)^2-\omega(j)^2}(\alpha_j(u) \mp \alpha_{-j}(u))
                    -\tfrac{\omega(j)}{(c_*\cdot j)^2-\omega(j)^2}(\alpha_j(u) \pm \alpha_{-j}(u)) \, . 
           \end{aligned}
       \end{equation}
By \eqref{omega}  with positive surface tension $ \kappa > 0 $ we have
$           |(c_*\cdot j)^2-\omega(j)^2| \approx |j|^{3} $ 
for any $ j \in \Gamma'\setminus(\cV\cup\{0\}) $
    and, since $
            M_j \approx |j|^{-\frac 14} $ by  \eqref{eq:Mj},   
            then 
        \begin{equation}\label{a uniform bound}
            \Big|\frac{(c_*\cdot j)\, | j|}{(c_*\cdot j)^2-\omega(j)^2}\Big|\, , \
            \Big|\frac{\omega(j)\, M_j^2\,  | j|^2}{(c_*\cdot j)^2-\omega(j)^2}\Big|\, , \
            \Big|\frac{(c_*\cdot j)\,  | j|^2}{(c_*\cdot j)^2-\omega(j)^2}\Big|
            \, , \
            \Big|\frac{\omega(j)\, M_j^{-2}\,  | j|}{(c_*\cdot j)^2-\omega(j)^2}\Big|\leq K\, .
        \end{equation} 
        In view of  \eqref{norm of eta}, \eqref{alpha  pm alpha beta pm beta} and \eqref{a uniform bound} the analytic norms of the components of
        $ Lu = \begin{psmallmatrix}
            \eta_1\\
            \psi_1
        \end{psmallmatrix} $ satisfy 
        \[
                \|\eta_1\|^2_{H^{\sigma,s}}\leq K^2\big(\|\eta\|_{H^{\sigma,s-1}}^2+\|\psi\|_{H^{\sigma,s-2}}^2\big)\, , \quad 
                \|\psi_1\|^2_{H^{\sigma,s}}\leq K^2\big(\|\eta\|_{H^{\sigma,s-1}}^2+\|\psi\|_{H^{\sigma,s-2}}^2\big) \, . 
        \]
By \eqref{L is an inverse}
the extended continuous operator $ L : W\cap Y\to W\cap X $ 
    is an inverse of $ \di_w\cG(c_*,0,0)$. 
    \end{proof}
    
%Lemma \ref{continuity of L}  
By  the Implicit function theorem there is $r>0$ and an analytic function $w:B_{r}(c_*)\times B_{r}^V\to W\cap X$ solving 
    $ 
        \cG(c,v,w(c,v))=0 $ for all 
        $ (c,v)\in B_{r}(c_*)\times B^V_r $
        and satisfying 
    $w(c,0)=0$ for any $c\in B_{r}(c_*)$
    because  $\cG(c,0,0) \equiv 0$. The equivariance properties \eqref{wequi} follow by 
those of $ {\cal F }(c, u) $ inherited by 
\eqref{Z2 symmetry}, \eqref{T2 symmetry} and 
uniqueness. Finally differentiating $ 0 = \cG(c,v,w(c,v)) $ 
at $ v = 0 $ we obtain, since $ w(c,0) = 0 $
    $$
    0 = \underbrace{\di_v\cG(c,0,0)}_{\Pi_{W\cap Y} {\cL_c}_{|V} = 0 \ \text{as} \, \cL_cV\subseteq V } + \, \di_w\cG(c,0,0) \, \di_v w(c,0) = \di_w\cG(c,0,0) \, \di_v w(c,0) \, .
    $$
    Since $\di_w\cG(c,0,0)$ is invertible we deduce $\di_v w(c,0)=0$.
    The proof is concluded.
    \end{appendix}

  \begin{footnotesize}

% \newpage

\vspace{7pt}
\noindent 
Tommaso Barbieri\\ 
SISSA, Via Bonomea 265, 34136, Trieste, Italy\\ 
\texttt{tbarbier@sissa.it}

\vspace{3pt}

\noindent 
Massimiliano Berti\\ 
SISSA, Via Bonomea 265, 34136, Trieste, Italy\\ 
\texttt{berti@sissa.it}

\vspace{3pt}

\noindent
Marco Mazzucchelli\\ 
Sorbonne Université, Université Paris Cité, CNRS, IMJ-PRG, F-75005 Paris, France\\ 
\texttt{marco.mazzucchelli@imj-prg.fr}.

\end{footnotesize}

\newpage

\makeatletter
\providecommand\@dotsep{5}
\renewcommand{\listoftodos}[1][\@todonotes@todolistname]{%
  \@starttoc{tdo}{#1}}
\makeatother
% \listoftodos


\begin{thebibliography}{999}                                         

%\bibitem{AR} 
%    Ambrosetti A., Rabinowitz P., {\it Dual Variational Methods in Critical Point Theory and Applications}. Journ. Func. Anal, 14, 349-381, 1973.

\bibitem{AFT}
    Amick C.,  Fraenkel L.,  Toland J., {\it On the Stokes conjecture for the wave of extreme form}. Acta Math. 148, 193–214, 1982.


%\bibitem{Atiyah:1984aa}
%Atiyah M. F., Bott R., \mass{non e' citato, serve? }
%{\it The moment map and equivariant cohomology.} 
%Topology, 1, 1--28, 23, 1984.

%\bibitem{AGN}
%Ahmad R, 
%· M. D. Groves, Nilsson D., %%A Resonant Lyapunov Centre Theorem with an Application
%to Doubly Periodic Travelling Hydroelastic Waves

\bibitem{Atiyah:1983aa}
Atiyah M. F.,  Bott R.,
{\it The {Y}ang-{M}ills equations over {R}iemann surfaces.} 
Philos. Trans. Roy. Soc. London Ser. A, 
1505, 523--615, 308, 1983.

\bibitem{BG} 
Bagri G. S.,  Groves M. D., 
{\it A Spatial Dynamics Theory for Doubly Periodic
Travelling Gravity-Capillary Surface Waves on Water
of Infinite Depth}. 
J. Dyn. Diff. Equat., 27:343-370, 2015. 

\bibitem{BBMM}
Barbieri T., Berti M., Maspero A., Mazzucchelli M.,
{\it Bifurcation of gravity-capillary Stokes waves with constant vorticity}.
J. Differential Equations 451, 113753, 2026.






\bibitem{Bartsch:1993aa} Bartsch T., {\it Topological methods for variational problems with symmetries}.
Lecture Notes in Math., 1560
Springer-Verlag, Berlin, x+152 pp.,
1993.

\bibitem{Bart1}
Bartsch T.
{\it A generalization of the Weinstein-Moser theorems on periodic orbits of a Hamiltonian system near an equilibrium}. 
Ann. Inst. H. Poincar\'e C 
Anal. Non Lin\'eaire 14, no. 6, 691-718, 1997.


\bibitem{Bartsch:1990aa}
Bartsch T., Clapp M.,
{\it Bifurcation theory for symmetric potential operators and the equivariant cup-length.}  
Mathematische Zeitschrift,
3, 341--356, 204, 1990.

\bibitem{Benci1991} Benci V., {\it A New Approach to the Morse-Conley Theory
and Some Applications.} Annali di Matematica pura ed applicata
(IV), Vol. CLVIII, 1991.

%\bibitem{Berti}
 %   Berti M.,   {\it Nonlinear Oscillations of Hamiltonian PDEs}. {Progress in Nonlinear Differential Equations}. {\sc book},  Birkh\"auser, 1-180 pages, Boston, ISBN-13: 978-0-8176-4680-6, 2008. 

\bibitem{BFM}
	Berti M., Franzoi L.,  Maspero A., \emph{Traveling quasi-periodic water waves with constant vorticity.} Arch. Rat. Mech. Anal., 240, 99-202, 2021. 

  \bibitem{BFM2}
Berti M., Franzoi L., Maspero A.,
 {\it  Pure gravity traveling quasi-periodic water waves with constant vorticity}.
 Comm. Pure Applied Math., 77(2): 990--1064, 2024.
	
\bibitem{BMV2}  Berti M.,   Maspero A.,  Ventura P., 
\emph{On the analyticity of the Dirichlet-Neumann operator and Stokes waves}. 
Rend. Lincei Mat. Appl., %doi 10.4171/RLM/983, 
33, 611--650, 2022.


\bibitem{Bre}  Bredon G. E., {\it Introduction to Compact Transformation Groups}. Academic Press, New York and London, 1972.

%\bibitem{BuTo} Buffoni B., Toland J., {\it Analytic Theory of Global Bifurcation: An Introduction.} Princeton Series in Applied Mathematics, 2003.


%\bibitem{Chang:1993aa}
%Chang C., 
	%Infinite-dimensional {M}orse theory and multiple solution problems,  x+312, 
 %   Birkh\"{a}user Boston, Inc., Boston, MA,
	%Progress in Nonlinear Differential Equations and their Applications, 6, 1993.

%\bibitem{ChG}  Ghoussoub N., 
%{\it The Conley Index and the critical groups via an extension of Gromoll-Meyer
%Theory}. Topological Methods in Nonlinear Analysis, % Journal of the Juliusz Schauder Center 
%Volume 7,  77–93, 1996. 

%\bibitem{Conley:1971aa} 
%Conley C., Easton, R.,
%{\it Isolated invariant sets and isolating blocks.}  
%Transactions of the American Mathematical Society,
%	158, 1971.

\bibitem{Conley:1978aa} 
Conley C., 
{\it  Isolated invariant sets and the {M}orse index.}, iii+89,
	American Mathematical Society, Providence, RI, CBMS Regional Conference Series in Mathematics,
	38, 1978.
    
%\bibitem{const_book}
%Constantin A., 
%{\it Nonlinear Water Waves with %Applications to Wave-Current Interaction and Tsunamis.}
%CBMS-NSF Regional Conf, Series in Applied Math., 81,  SIAM, 2011.


	%\bibitem{CIP}
	%Constantin A., Ivanov R.I., Prodanov E.M., 
	%\emph{Nearly-Hamiltonian Structure for Water Waves with Constant Vorticity.}
	%J. Math. Fluid Mech., 
	%10, 224–237, 2008.

%\bibitem{CSt}
%Constantin A., Strauss W., 
%{\it Exact steady periodic water waves with vorticity}, 
%Comm. Pure Appl. Math. 57,  4, 481-527, 2004. 

\bibitem{CSV}
Constantin A., Strauss W., Varvaruca E., 
{\it Global bifurcation of steady gravity water waves with critical layers}.
Acta Math.   217,  no. 2, 195–262, 2016.

%\bibitem{CHOS}
%Craig W., 
%Henderson D., 
%Oscamou M., 
%Segur H., {\it Stable three-dimensional waves of nearly permanent form on deep water}, Mathematics and Computers in Simulation
%Volume 74, Issues 2-3, 7 2007, Pages 135-144, 2007. 


\bibitem{CN}
	Craig W., Nicholls  D., 
	\emph{Traveling two and three dimensional capillary gravity water waves}.  SIAM J. Math. Anal., 
	32, 323-359, 2000.

\bibitem{CN2}
Craig W., Nicholls D.,
\newblock {\it Traveling gravity water waves in two and three dimensions.}
\newblock { Eur. J. Mech. B Fluids}, 21(6):615-641, 2002.


 
	\bibitem{CS}
	Craig W., Sulem  C., 
	\emph{Numerical simulation of gravity water waves}. J. Comput. Phys., 
	108, 73-83, 1993.

    \bibitem{Dold}  Dold A., 
    \emph{Partitions of Unity in the Theory of Fibrations}. Annals of Mathematics, Vol. 78, No. 2, 1963.  

	\bibitem{dubreil}
		Dubreil-Jacotin M.-L.,
		\newblock{\em Sur la d\'etermination rigoureuse des ondes permanentes
			p\'eriodiques d'ampleur finie}.
		\newblock{ J. Math. Pures Appl. 13}, 217-291, 1934.


\bibitem{Fadell:1988aa}
Fadell  E.,  Husseini S., 
{\it An ideal-valued cohomological index theory with applications to {B}orsuk-{U}lam and {B}ourgin-{Y}ang theorems.} Ergodic Theory and Dynamical Systems,  73--85,
1988. 


\bibitem{FR} Fadell E., Rabinowitz P., {\it Generalized 
cohomological index theories for the group actions 
with an application to bifurcation questions  for Hamiltonian
systems}. Inventiones Math. 45, 139-174, 1978.






	\bibitem{FG}
		Feola R.,  Giuliani F.,
		\emph{Quasi-periodic traveling waves on an infinitely deep fluid under gravity.}
		Memoires AMS, Volume 295,  164 pp., 2024.


\bibitem{FMT} 
Feola R., Montalto R., Terracina S.,  
{\it Time  quasi-periodic 3d traveling gravity water waves.}
Arxiv 2509.10318, 2025. 


\bibitem{Floer:1988aa} 
Floer  A., Zehnder  E., 
{\it The equivariant {C}onley index and bifurcations of periodic solutions of {H}amiltonian systems.} 
Ergodic Theory Dynam. Systems,
87--97, 8$^*$, 1988.

%\bibitem{gerstner}
%		Gerstner F., 
%		\newblock{\em Theorie der Wellen}. 
%		\newblock{ Abh. K\"onigl. B\"ohm. Ges. Wiss}, 1802.

\bibitem{goyon}
		Goyon R., 
		\newblock{\em Contribution \`a la th\'eorie des houles}.
		\newblock{ Ann. Sci. Univ. Toulouse 22}, 1-55, 1958.

%\bibitem{GrM}
%Gromoll D., Meyer W., 
%{\it On differentiable functions with isolated critical points}. Topology 8, 361–369, 1969.

\bibitem{GrovesHaragus}
Groves M., Haragus M.,  {\it A Bifurcation Theory for Three-Dimensional Oblique Travelling Gravity-Capillary Water Waves.} J. Nonlinear Sci. 13, 397–447, 2003.

\bibitem{GrovesMielke} Groves M., Mielke A., {\it A spatial dynamics approach to three-dimensional gravity-capillary steady water waves.} Proceedings of the Royal Society of Edinburgh: Section A Mathematics.  131(1):b 83-136, 2001.

\bibitem{GNPW}
 Groves M., Nilsson D., Pasquali S., Wahlén E., {\it Analytical study of a generalised Dirichlet–Neumann operator and application to three-dimensional water waves on Beltrami flows}.
J. Diff. Eq., 413(25): 129-189, 2024.

\bibitem{HHS}
Hammack J.L.,  Henderson
D.M., Segur H., 
{\it Progressive waves with persistent, two-dimensional surface patterns in deep water}. J. Fluid Mech.
532,  1–51, 2005.


    \bibitem{Hatcher:2002aa}
	Hatcher A., 
    {\it Algebraic topology.} xii+544,
Cambridge University Press,  2002.

\bibitem{HHSTWWW}
Haziot S.,  Hur V.M., Strauss W., Toland J., 
Wahlen E.,  Walsh S.,  Wheeler M.,
{\it Traveling water waves -the ebb and flow of two centuries}. 
Quarterly of Applied Mathematics, 80, 2, 317-401, 2022. 

%\bibitem{Hsiang:1975aa}
%Hsiang Wu-yi, 
%{\it Cohomology theory of topological transformation groups}, x+164, 
%Springer-Verlag, 
%Ergebnisse der Mathematik und ihrer Grenzgebiete, 1975.

\bibitem{IP-Mem-2009}
Iooss G., Plotnikov P.,
\newblock {\it Small divisor problem in the theory of three-dimensional water
  gravity waves}. 
\newblock {Mem. Amer. Math. Soc.}, 200(940):viii+128, 2009.

\bibitem{IP2}
Iooss G., Plotnikov P.,
\newblock {\it Asymmetrical tridimensional traveling gravity waves}. 
\newblock { Arch. Rat. Mech. Anal.}, 200(3):789--880, 2011.



\bibitem{JT}
    Jones M., Toland J., {\it The bifurcation and secondary bifurcation of capillary gravity waves.} Proc. Royal Soc. London Ser. A, 399, 391–417, 1985. 

\bibitem{JT2}
    Jones M., Toland J., {\it Symmetry and bifurcation of capillary gravity waves.} Arch. Rat. Mech. Anal., 96, 29–53, 1986. 

%\bibitem{KN}
 %   Keady G., Norbury J., {\it On the existence theory for irrotational water waves.} Math. Proc. Cambridge Philos. Soc. 83, no. 1, 137-157, 1978. Arch. Rat. Mech. Anal. 247(98), 2023. 

\bibitem{KK}
    Kozlov V., Lokharu  E. {\it Global Bifurcation and Highest Waves on Water of Finite Depth.} Arch. Ration. Mech. Anal. 247, 5, 98, 2023.

\bibitem{LC}
    Levi-Civita  T., 
    \newblock {\it D\'etermination rigoureuse des ondes permanentes d' ampleur finie.} 
    \newblock { Math. Ann.}, 93, 264-314, 1925.

\bibitem{LSW}
    Lokharu E.,  Seth D.,  Wahlén E., {\it An Existence Theory for Small-Amplitude Doubly Periodic Water Waves with Vorticity}. Arch Rational Mech Anal 238, 607–637, 2020.


%\bibitem{Lyusternik:1947aa}
%Lyusternik L. and \v{S}nirelman L.,
%{\it Topological methods in variational problems and their application to the differential geometry of surfaces},  Uspehi Matem. Nauk (N.S.), 1(17),
%166--217, 2, 1947.

        
%\bibitem{MS}
 %   Maelhlen O., Svensson Seth D.,  {\it Asymmetric travelling wave solutions of the capillary-gravity Whitham Equation}. SIAM J. Math. Anal. 56, no. 6, 8096-8124, 2024.


\bibitem{M}
	Martin C.I., \emph{Local bifurcation and regularity for steady periodic capillary-gravity water waves with constant vorticity}. Nonlinear Anal.: Real World Applications, \textbf{14}, 131-149, 2013.

%\bibitem{ML}
 %   McLeod J. B., {\it The Stokes and Krasovskii conjectures for the wave of greatest height.} Stud. Appl. Math. 98, no. 4, 311-333, 1997.


\bibitem{Milnor:1974aa} 
Milnor J. W., Stasheff J. D.,
{\it Characteristic classes}, 
vii+331, Princeton Univ. Press,
Annals of Math. Studies,  76,
1974.


\bibitem{Morse:1996aa}
Morse M., 
{\it The calculus of variations in the large.} Reprint of the 1932 original,
xii+368, 
American Mathematical Society, Providence, RI,
American Mathematical Society Colloquium Publications, 18, 1996.


\bibitem{Mo}
    Moser J., {\it Periodic orbits near an Equilibrium and a Theorem by Alan Weinstein}.  Comm. Pure  Appl. Math., XXIX, 1976.

\bibitem{M2010} McCleary J., {\it A User's Guide to Spectral Sequences}. Cambridge University Press, 2010.

\bibitem{Nek}
		Nekrasov A. I.,  {\it On steady waves}.  Izv. Ivanovo-Voznesenk. Politekhn. 3, 1921.

\bibitem{Nil}
Nilsson D., 
{\it Three-dimensional internal gravity-capillary waves in
finite depth}. Mathematical methods in the applied sciences, 
4113-4145, Volume 42, Issue 2,  2019.


%\bibitem{NR}
 %   Nicholls D., Reitich F.,  \emph{On analyticity of travelling water waves}. Proc. R. Soc. Lond. Ser. A Math. Phys. Tech. Sci. Inf. Sci. 461 (2057), 1283-1309, 2005. 	

   % \bibitem{PSV} Piccinini L. C., Stampacchia G., Vidossich G., {\it Ordinary Differential Equations in $R^n$}. Applied Mathematical Sciences, Springer-Verlag New York Inc. 1984

\bibitem{P1982} 
    Plotnikov P. I., {\it A proof of the Stokes conjecture in the theory of surface waves}. Dinamika Sploshn Sred, 57, 41–76, 1982.

%\bibitem{R3} 
 %   Rabinowitz P., {\it A bifurcation theorem for potential operators}. J. Func. Anal., 25,  412-424, 1977.
 
	%\bibitem{R}
	%Rabinowitz P.H., \emph{Minimax Methods in Critical Point Theory with Applications to Differential Equations}. AMS-CBMS 65, 1986.
    
\bibitem{RS1}
     Reeder J., Shinbrot M., 
{\it On Wilton Ripples, II: Rigorous Results}, 
Arch. Rational Mech. Anal. 77, 321–347, 1981.

\bibitem{RS}
     Reeder J., Shinbrot M., \newblock{\it Three-dimensional, nonlinear wave interaction in water of constant depth.} \newblock{ Nonlinear Anal.,} T.M.A., 5(3), 303--323, 1981.

\bibitem{Salamon1985} Salamon D., {\it Connected simple systems and the Conley index of isolated invariant sets.} Transactions of the AMS, Volume 291. Number 1, 1985.

\bibitem{SVW}
Seth D., 
Varholm K., 
Wahl\'en E., 
{\it Symmetric doubly periodic gravity-capillary waves
with small vorticity}. 
Advances in Math. 447,  109-683, 2024. 


	\bibitem{Spanier}
		Spanier E., {\it Algebraic topology.} % Corrected 
        reprint of the 1966, % original, 
        Springer-Verlag, % New York,  
        xvi+528 pp., ISBN:0-387-94426-5, 1995.


  

	\bibitem{stokes}
		Stokes G., {\it On the theory of oscillatory waves}. Trans. Cambridge Phil. Soc. 8, 441-455, 1847.
  
	\bibitem{Struik}
		Struik D.,  {\it D\'etermination rigoureuse des ondes irrotationelles p\'eriodiques dans un canal \'a profondeur finie}.  Math. Ann. 95, 595-634, 1926.

  %  \bibitem{Seth} Svensson Seth D., {\it On small-amplitude asymmetric water waves}.  Water Waves,   
% 7, 407-430, 2025.

%\bibitem{To}
 %   Toland J. F., {\it On the existence of a wave of greatest height and Stokes conjecture.} Proc. Roy. Soc. London Ser. A 363, 1715, 469-485, 1978.

\bibitem{Tu} Tu L. W., {\it Introductory lectures on equivariant cohomology}. Princeton University Press, 2021.

%\bibitem{W}
%	Wahl\'{e}n E., 
	%\emph{Steady Periodic Capillary‐Gravity Waves with Vorticity}. SIAM J. Math. Anal. 
	%38, 921-943, 2006.

\bibitem{Wh0}
    Wahl\'en E., {\it Steady periodic capillary-gravity waves with vorticity}. SIAM J. Math. Anal. 38, 921-943, 2006.
 

%\bibitem{Wh}
 %   Wahl\'en E., {\it A Hamiltonian formulation of water waves with constant vorticity}. Letters in Math. Physics, 79, 303-315, 2007. 


\bibitem{WW24}
	Wahl\'en E.,  Weber J., {\it Large-amplitude steady gravity water waves with general vorticity and critical layers.} Duke Math. J. 173 (11), 2197-2258, 2024. 

%\bibitem{We1}
 %   Weinstein A., {\it Lagrangian submanifolds and Hamiltonian systems}. Ann. Math., 98, 377-410, 1973.

\bibitem{We2}
    Weinstein A., {\it Normal modes for non-linear Hamiltonian systems}. Inv. Math., 20, 47-57, 1973. 

%\bibitem{Yang:1957aa}
%Yang C.T., \mass{non citato}
%{\it On a problem of {M}ontgomery.} 
%Proc. Amer. Math. Soc.,
%255--257, 8, 1957.

\bibitem{Z}
	Zakharov V.E., \emph{Stability of periodic waves of finite amplitude on the surface of a deep fluid}. J. Appl. Mech. Tech. Phys., 9, 73-83, 1968.

\bibitem{Zei}
    Zeidler E., {\it Existenzbeweis f\"ur cnoidal waves unter Ber\"ucksichtigung der Oberfl\"achen spannung.} Arch. Rational Mech. Anal, 41, 81-107, 1971.

%\bibitem{Zei2}
%Zeidler E.,
%\newblock{\it Existenzbeweis f\"ur permanente Kapillar-Schwerewellen mit allgemeinen %Wirbelverteilungen},
%\newblock{ Arch. Ration. Mech. Anal.}, 50, 34--72, 1973.

\end{thebibliography}
\end{document}